\documentclass[11pt]{amsart}
\input{glyphtounicode}\pdfgentounicode=1
\usepackage{amsmath,amssymb,amsthm,mathtools,booktabs,enumitem,tabularx}
\usepackage[margin=1.15in,marginparwidth=0.95in,marginparsep=0.1in]{geometry}
\usepackage{hyperref}
\hypersetup{colorlinks=true,linkcolor=blue!50!black,citecolor=green!40!black,urlcolor=blue!50!black,
  bookmarksnumbered=true}
\usepackage{xcolor}
\usepackage{tikz}
\usepackage[en-GB,showseconds=false,showzone=false]{datetime2}
\DTMsettimestyle{default}

\newtheorem{theorem}{Theorem}[section]
\newtheorem{proposition}[theorem]{Proposition}
\newtheorem{lemma}[theorem]{Lemma}
\newtheorem{corollary}[theorem]{Corollary}
\newtheorem{mainthm}{Theorem}

\theoremstyle{definition}
\newtheorem{definition}[theorem]{Definition}
\newtheorem{remark}[theorem]{Remark}

\numberwithin{equation}{section}

\newcommand{\R}{\mathbb{R}}
\newcommand{\Z}{\mathbb{Z}}
\newcommand{\Q}{\mathbb{Q}}
\providecommand{\C}{\mathbb{C}}
\newcommand{\RP}{\mathbb{RP}}
\newcommand{\Hb}{\mathbb{H}}
\DeclareMathOperator{\adj}{adj}
\DeclareMathOperator{\diag}{diag}
\DeclareMathOperator{\tr}{tr}
\newcommand{\mub}{\bar\mu}

\newcommand{\AF}{\textup{AF}}
\newcommand{\AEt}{\textup{AE}}
\newcommand{\Ctor}{\mathcal C_{\mathrm{toric}}}
\newcommand{\cA}{\mathcal A}

\title[On the Geroch conjecture for toric gravitational instantons]{On the Geroch conjecture for toric gravitational
instantons}
\author{Lars Andersson}
\address{Department of Mathematics, KTH Royal Institute of Technology, 100 44 Stockholm, Sweden}
\email{larand@kth.se}
\author{Mattias Dahl}
\address{Department of Mathematics, KTH Royal Institute of Technology, 100 44 Stockholm, Sweden}
\email{dahl@kth.se}
\author{Gustav Nilsson}
\address{Mathematisches Institut, Universit\"at M\"unster, Einsteinstrasse 62, 48149 M\"unster, Germany}
\email{gustav.nilsson@uni-muenster.de}
\date{\today}

\begin{document}
\begin{abstract}
Geroch conjectured that the local solutions of the vacuum Einstein equations with two commuting Killing fields and vanishing twist
constants are generated from flat space by an infinite-dimensional group of transformations, and Hauser and Ernst proved this for
stationary axisymmetric spacetimes near the axis. The solutions generated from flat space by the transformations with finitely many
poles in the spectral parameter are
the solitons of Belinski and Zakharov. We prove a solitonic form of the conjecture in Riemannian signature, for toric gravitational
instantons.
The proof combines Weinstein's uniqueness argument, a closed formula for the soliton metrics in Hankel determinants and a
degree argument. It also applies to the toric Ricci-flat metrics with conical singularities constructed by Li and Sun.
\end{abstract}
\maketitle
\setcounter{tocdepth}{1}
\tableofcontents

\section{Introduction}\label{sec:intro}
Geroch introduced an infinite-dimensional group of transformations of the local solutions of the vacuum Einstein equations with
two commuting Killing fields, generated by the Ehlers transformations of the Ernst potential and the linear changes of the Killing
fields \cite{Geroch1971,Geroch1972}, and he conjectured that it acts transitively: ``it is reasonable to suppose that any two exact
solutions with a pair of commuting Killing fields (and whose two constants vanish) will, at least locally, be related by one of the
transformations described here'' \cite[p.~395]{Geroch1972}. The two constants are the twist constants of the Killing fields, which vanish
as soon as a non-zero combination of the Killing fields vanishes somewhere, for instance on an axis; this is the case of the
Weyl--Papapetrou form of \S\ref{subsec:WP}. Kinnersley and Chitre \cite{KC1977,KC1978} gave an explicit realisation of the group, and
Breitenlohner and Maison \cite{BM1987} identified it as a loop group acting through a linear system with a spectral parameter. The
conjecture was proved by Hauser and Ernst for stationary axisymmetric spacetimes in a neighbourhood of a point of the axis at which
one Killing field is timelike \cite{HE1980,HE1981,HE1987}: the set of transformations relating two solutions is determined by the
values of their Ernst potentials on the axis. A sharper question is which solutions are generated from flat space by the
transformations with finitely many poles in the spectral parameter. These are the solitons of Belinski and Zakharov
\cite{BZ1978,BZ1979}, among them the Kerr solution. In Lorentzian signature a soliton transformation with an even number of poles is
a product of Harrison's B\"acklund transformations \cite{Harrison1978}, and one with an odd number of poles is such a product
composed with the Kramer--Neugebauer involution \cite{KN1968} \cite{Cosgrove1980,Cosgrove1982}, \cite{KKV2013}; Harrison's
transformations lie in the Geroch group \cite{Cosgrove1981}. For stationary vacuum black holes in four dimensions it was proved, by
integrating the spectral problem along the boundary of the orbit space, that every solution with the right boundary conditions
belongs to a soliton family with finitely many parameters \cite{Varzugin1997,NH2012}; the boundary data were determined in this way
also for charged configurations \cite{Hennig2020} and in five dimensions \cite{LT2020}. Both questions make sense in Riemannian signature, for Ricci-flat $4$-manifolds with a
torus action, where the solutions of interest are the gravitational instantons and the flat seeds are $\R^4$ and $\R^3\times S^1$,
according to the asymptotic type. We prove the solitonic statement in this setting: every smooth simply connected toric gravitational instanton
whose torus action has finitely many fixed points, and at least one, is a gravitational soliton, and its metric is given by a closed formula in finitely many
parameters.

We begin with the concepts involved. A \emph{gravitational instanton} is a complete, non-compact, oriented Ricci-flat Riemannian
$4$-manifold with quadratic curvature decay, $|\mathrm{Rm}|=O(\mathfrak r^{-2})$ in terms of the distance $\mathfrak r$ from a base
point, or equivalently with finite curvature energy \cite{Hawking1977,GH1979}, \cite[Def.~1.1]{LS2026} (Definition~\ref{def:GI}). It is
\emph{toric} if it carries an effective isometric action of the torus $T^2$ in which every point has a connected stabiliser
\cite[Def.~2.1]{LS2025} (Definition~\ref{def:toricRF}); since compact toric Ricci-flat manifolds are flat \cite{Bochner1946},
non-compactness is no restriction. A toric Ricci-flat metric is encoded by an axisymmetric harmonic map from $\R^3$ minus an
axis into the hyperbolic plane, the Riemannian form of the Ernst equation \cite{Weyl1917,Ernst1968} (\S\ref{subsec:WP}),
together with its boundary data on the axis, the \emph{rod structure} \cite{Harmark2004,LS2025}: the fixed points of the torus
appear as \emph{turning points} on the axis, and the stabilisers of the intervals between them, the \emph{rods}, as lines in the
Lie algebra of the torus, the \emph{rod lines}. Unlike the Lorentzian case there are several asymptotic types: asymptotically
locally Euclidean (ALE), asymptotic to $\R^4/\Gamma$; asymptotically locally flat (ALF), asymptotic to a circle bundle over $\R^3$
like Taub--NUT; and asymptotically flat (AF), asymptotic to $(\R^3\times\R)/\Z$ like Euclidean Kerr (Definition~\ref{def:asymptype},
following \cite[Def.~4.30]{LS2025}).

A \emph{gravitational soliton} is a toric Ricci-flat metric that is obtained from the flat metric of its asymptotic type, $\R^4$
for ALE and $\R^3\times S^1$ for ALF and AF, by finitely many soliton transformations of Belinski and Zakharov with poles on the axis
\cite{BZ1978,BZ1979,BV2001}
(Definition~\ref{def:soliton}). A soliton transformation adds a pole to the spectral data of the linear system associated with
the harmonic map, and the new metric is obtained from the seed and the poles by linear algebra. Poles that are not turning
points are allowed; they are the \emph{phantom solitons} of \cite{Chen2016}, first used by Chen and Teo \cite{CT2011}, and their
positions are not parameters of the metric (\S\ref{subsec:soliton}). A phantom soliton is a pole of the dressing matrix that leaves
no trace in the metric, so the dressing matrix is not determined by the metric (compare the freedom in the linear system,
\cite[\S2]{Cosgrove1982}, \cite[(3.12)]{KKV2013}), and in the soliton representation of Theorem~\ref{thm:Gfull} the number of
poles is the number of turning points plus a non-negative number of phantom solitons determined by the rod structure, through its
degree index. The statement that every toric gravitational instanton is a gravitational soliton is a solitonic form of Geroch's
conjecture in Riemannian signature. It differs from the conjecture in three ways: the seed is the flat metric of the asymptotic
type, the transformation has finitely many poles, and the statement is global, while the conjecture and the theorem of Hauser and
Ernst are local. We do not identify the operations of Definition~\ref{def:soliton} with elements of a Riemannian analogue of the
Geroch group, so we do not claim that the theorem implies the conjecture for toric gravitational instantons.

\begin{mainthm}[Toric gravitational instantons are solitons]\label{thm:G}
Let $(M,g)$ be a simply connected toric gravitational instanton whose torus action has finitely many fixed points, at least
one. Then the end of $M$ is of one of the types ALE, ALF$^\pm$ or AF$_\beta$, and $g$ is a gravitational soliton.
\end{mainthm}
A detailed statement, with the number of poles, the parameters of the soliton and the closed formula for the metric, is
Theorem~\ref{thm:Gfull} at the beginning of \S\ref{sec:proofs}, where both theorems are proved.

The proof gives more than the statement. The metric of a soliton is a ratio of determinants, and the central step is a
factorisation of the determinants that describe the solitons on $\R^4$, pointed out to us by Aksteiner \cite{Aksteiner2022}
and proved by him through the Cauchy--Pfaffian identity \cite[Prop.~2.2]{Aksteiner2026}, which together with classical determinant identities turns the metric of every instanton in Theorem~\ref{thm:G} into a
closed formula: the Gram matrix of the Killing fields and the conformal factor of the orbit space are Hankel determinants of
a measure on the turning points (Theorem~\ref{thm:Gfull}, Remark~\ref{rem:GGclosed}). The parameters of the soliton, its
\emph{soliton data}, are the positions of the turning points, a positive weight at each of them, and in the ALF and AF cases
one further positive parameter, the weight of a pole at infinity; the number of poles is determined by an integer invariant
of the rod structure, the degree index (\S\ref{subsec:cells}). The proof produces the data by a surjectivity argument and does
not show that they are unique (Remark~\ref{rem:injectivity}). A first consequence of the explicit rod lines is the toric case of a
uniqueness conjecture of Li and Sun: a simply connected toric gravitational instanton whose torus action has finitely many fixed
points, at least one, and whose end is asymptotically flat without twist is the Euclidean Schwarzschild metric (Corollary~\ref{cor:AF0}). The closed formula also gives new explicit
instantons, among them a toric ALF instanton with four fixed points that is Hermitian for neither orientation
(Remark~\ref{rem:compexamples}).

The hypotheses are used as follows. The type of the end is determined by a result of Li and Sun \cite{LS2026b}, which
enters the proof only at this point (Proposition~\ref{prop:endtype}); if the type is assumed, the proof uses no result of
\cite{LS2026b}, and global Weyl coordinates, in which the orbit space is a half-plane with the turning points on its boundary,
follow from an argument of Weinstein's type \cite{Weinstein1990,KL2021,KL2026} (Lemma~\ref{lem:globalweyl}). Simple
connectivity is equivalent to the finiteness of the fundamental group and to the simple connectivity of the orbit space, and it holds
automatically when the volume growth is Euclidean (Remark~\ref{rem:pi1}); without it global Weyl coordinates can fail
(\S\ref{subsec:WP}). The finiteness of the set of fixed points
is an assumption; compare the conjecture that gravitational instantons have finite topological type \cite[\S7.2]{LS2026}.
Conversely, a complete toric Ricci-flat manifold with an end of one of the three types is a gravitational instanton with
finitely many fixed points (Remark~\ref{rem:asymptypes}).

Many toric gravitational instantons were known to be solitons: Euclidean Kerr is a two-soliton solution \cite{BZ1979}, and
Taub--NUT \cite{CT2010}, the Chen--Teo instanton \cite{CT2011,Chen2016} and Teo's instanton on the manifold $X_3$ of Li and Sun
\cite{Teo2026} were written as solitons. On the general side, Li and Sun constructed toric Ricci-flat metrics for every rod
structure of each asymptotic type, by solving the harmonic map problem with prescribed boundary data \cite{LS2025}, and Kunduri
and Lucietti studied the uniqueness of toric instantons with given rod structure \cite{KL2021,KL2026}, extending Weinstein's
argument for black holes \cite{Weinstein1990}. Dunajski, Mason and Tod approach the uniqueness question for toric instantons by
twistor methods \cite{DMT2026}. These results give existence and uniqueness for a given rod structure, but not
the metric itself. Aksteiner \cite{Aksteiner2026} studies the solitons of type AF directly, in the Belinski--Zakharov
parametrisation: integral restrictions on their rod data, a convex potential whose critical points are the smooth metrics, an
existence criterion under comparison hypotheses, and a catalogue through twelve turning points.

The proof is formulated in terms of harmonic maps rather than metrics. In global Weyl coordinates the metric is determined
by the map $\Phi$ into the hyperbolic plane and a conformal factor obtained from it by quadrature, and the smoothness of the
metric along the axis becomes a boundary behaviour of $\Phi$ that Li and Sun call local tameness by the rod structure. We work
in the class $\Ctor$ of maps that are locally tamed by a rod structure and tend to the model map of one of the three types
at infinity, uniformly in the angle (Definition~\ref{def:Ctor}). The class does not prescribe the cone angles along the axis,
and the rod lines may be arbitrary lines. The main step is Proposition~\ref{prop:GG}: every element of $\Ctor$, and every
toric Ricci-flat metric built from it with any period lattice, is a gravitational soliton. This includes the harmonic maps
that Li and Sun construct for every rod structure \cite[Thm.~4.24]{LS2025}, and toric Ricci-flat metrics with conical
singularities along the axis whose maps lie in $\Ctor$ (Remark~\ref{rem:Ctorstatus}); we state only the smooth case as a
theorem. The instantons of Theorem~\ref{thm:G} lie in $\Ctor$ (Propositions~\ref{prop:endtype} and~\ref{prop:smoothCtor}),
which gives the theorem.

We now discuss the main ideas of the proof of Theorem~\ref{thm:G}. It has three parts: uniqueness, regularity of the
solitons, and surjectivity.

The uniqueness argument is that of Weinstein \cite{Weinstein1990} and Li and Sun \cite{LS2025}, stated with minimal
hypotheses (Lemma~\ref{lem:uniq}). The hyperbolic distance between two harmonic maps is subharmonic. If two maps are locally
tamed by the same rod structure and tend to the same model at infinity, this distance is bounded, extends across the axis,
and tends to $0$ at infinity, so it vanishes.

The second part shows that the soliton metrics lie in $\Ctor$. The solitons on $\R^4$ have a closed formula
(Theorem~\ref{thm:CF}), which comes from the factorisation of the interpolation determinants that express the dressing. The
solitons of the types ALF and AF are limits in which one pole moves to infinity: in an adapted frame the closed formula is
analytic in the inverse of its position, and, by a displacement identity for Cauchy-type matrices, the limit is a single
Belinski--Zakharov dressing of the flat seed $\R^3\times S^1$ (\S\ref{sec:far}). After the phantom solitons are removed
(Lemma~\ref{lem:Cnull}), both kinds of solitons are given by one formula in terms of Hankel determinants of a measure on the
turning points (Theorem~\ref{thm:family}), whose denominators are positive off the turning points, so that these metrics are
non-degenerate and definite on the whole half-plane (Lemma~\ref{lem:famid}). On a rod the formula reduces to the rod form,
with the rod lines given by Hankel determinants; near a turning point a square root gives the turning-point form
(Proposition~\ref{prop:axis}); and at infinity the terms of the determinants cluster, which gives an explicit bound to the
model for ALE (Proposition~\ref{prop:Cinfty}) and an $O(R^{-1})$ bound for ALF and AF (Propositions~\ref{prop:farinfty}
and~\ref{prop:farCtor}).

The third part is surjectivity. The map from positive soliton data on fixed turning points to rod lines, normalised by the
model at infinity, is proper and has degree $\pm1$; so every admissible configuration of rod lines is attained
(\S\ref{sec:surj}). Given a metric in $\Ctor$, the third part gives a soliton with the same turning points, rod lines and
model, the second shows that it lies in $\Ctor$, and the first shows that the two coincide.

The paper is organised as follows. \S\ref{sec:toric} sets up the Weyl--Papapetrou reduction, rod structures, the models at
infinity and the class $\Ctor$. \S\ref{sec:uniq} proves the uniqueness lemma, shows that the instantons of Theorem~\ref{thm:G}
lie in $\Ctor$, and converts their rod lines into the normalised data on which the solitons are parametrised.
\S\ref{sec:bz} recalls the Belinski--Zakharov method and defines gravitational solitons, and \S\ref{sec:hankel} the Hankel
determinants and measures that describe them. \S\ref{sec:ale} proves the closed formula for the solitons on $\R^4$,
\S\ref{sec:alerods} studies the resulting family on the axis and at infinity, and \S\ref{sec:far} treats the far-node limit,
which gives the solitons of the types ALF and AF. \S\ref{sec:surj} proves surjectivity, and \S\ref{sec:proofs} states the
detailed form of Theorem~\ref{thm:G} and proves it. \S\ref{sec:examples} works out examples; nothing in it is used in the
proofs. Appendix~\ref{app:identities} collects the determinant identities used. The role of computation and of an AI
assistant is described in the note at the end of the paper.

\section{Toric Ricci-flat metrics}\label{sec:toric}
This section sets up the framework: the reduction of a toric Ricci-flat metric to an augmented harmonic map, rod structures
and local tameness, the models at infinity and the ends, and the class $\Ctor$ (Definition~\ref{def:Ctor}) in which everything
later is formulated.

\subsection{Toric Ricci-flat manifolds and the harmonic map}\label{subsec:WP}
\begin{definition}[{\cite[Def.~2.1]{LS2025}}]\label{def:toricRF}
A \emph{toric Ricci-flat manifold} is an oriented Ricci-flat Riemannian $4$-manifold $(M,g)$ with an effective isometric
action of the torus $T=T^2$ in which every point has a connected stabiliser.
\end{definition}
Kunduri and Lucietti assume instead that no point has a discrete non-trivial stabiliser \cite[Def.~2.1]{KL2026}. For an
effective $T$-action on a $4$-manifold the two conditions are equivalent, by the slice theorem \cite[Ch.~VIII]{Borel1960}, \cite[Ch.~VI, Thm.~2.2]{Bredon1972}: a
disconnected stabiliser $S^1\times\Z_k$, $k\ge2$, acts on its $3$-dimensional slice by a character $\chi$ on a $2$-plane and trivially on a line; $\chi$ cannot be injective on $S^1\times\Z_k$, and an element of $\ker\chi$ fixes the slice, hence (as $T$ is abelian) the open set swept out by $T$ acting on the slice, pointwise, contradicting effectiveness.
The condition makes the orbit space a manifold with boundary and corners rather than an orbifold \cite[Rem.~2.2]{LS2025}.
\begin{definition}[{\cite[Def.~1.1]{LS2026}}]\label{def:GI}
A \emph{gravitational instanton} is a complete, non-compact, oriented Ricci-flat Riemannian $4$-manifold $(M,g)$ with
quadratic curvature decay, $|\mathrm{Rm}_g|=O(\mathfrak r^{-2})$, where $\mathfrak r$ is the distance from a base point; equivalently
\cite[Rem.~2.3]{LS2026}, with finite curvature energy, $\int_M|\mathrm{Rm}_g|^2\,d\mathrm{vol}_g<\infty$. A gravitational instanton is
\emph{toric} if it is a toric Ricci-flat manifold.
\end{definition}

Let $(M,g)$ be toric Ricci-flat, with Killing fields $X_v$ for $v$ in the Lie algebra $\mathfrak t\cong\R^2$ of $T$.
Fix a basis $(e_1,e_2)$ of $\mathfrak t$ and let $G_{ab}=g(X_{e_a},X_{e_b})$ be the Gram matrix. We recall the following concepts.

\emph{Orthogonal transitivity.} The twist functions $\star(X_{e_1}^\flat\wedge X_{e_2}^\flat\wedge dX_{e_a}^\flat)$
are constant, because $\operatorname{Ric}=0$, and they vanish where $X_{e_1}\wedge X_{e_2}=0$. If the action has a
non-principal orbit they vanish identically, the distribution orthogonal to the orbits is integrable, and $g=G_{ab}\,
d\phi^ad\phi^b+\hat g$ with $\hat g$ a metric on the orbit space $\mathcal B=M/T$
\cite{Papapetrou1966,KT1966,Carter1969}, \cite[Thm.~7.1.1]{Wald1984}, \cite[\S2.1]{LS2025}. The orbit space $\mathcal B$ is a
connected surface with boundary and corners \cite[\S2.3]{LS2025}; if $M$ is simply connected, so are $\mathcal B$ and its interior
$\mathcal B^\circ$ \cite[Ch.~II, Cor.~6.3]{Bredon1972}.

\emph{Weyl coordinates.} It follows from Ricci-flatness that $\rho=\sqrt{\det G}$ is harmonic for $\hat g$ on $\mathcal B^\circ$ \cite{Weyl1917,ER2002}. Where $d\rho\ne0$, a harmonic conjugate $z$ gives isothermal coordinates, $\hat g=e^{2\nu}(d\rho^2
+dz^2)$.

\emph{Orientation.} A basis $(e_1,e_2)$, with angle coordinates $\phi^1,\phi^2$, orients the
principal orbits, and with the orientation of $M$ it induces an orientation of $\mathcal B^\circ$. We choose $z$ with $dz=\star_{\hat g}d\rho$
for this orientation; where $d\rho\ne0$ this means that $d\rho\wedge dz\wedge d\phi^1\wedge d\phi^2$ is positive, as in
\cite[\S3]{LS2025}.

When $(\rho,z)$ is a global chart of $\mathcal B^\circ$ onto the open half-plane $\Hb^\circ=\{(\rho,z):\rho>0\}$ we say that $g$ has
\emph{global Weyl coordinates}. Li and Sun call such instantons \emph{proper} \cite[Def.~3.14]{LS2025}, \cite[Def.~A.2]{LS2026b}.
Properness is a genuine restriction in general \cite[Rem.~A.3]{LS2026b}: the toric gravitational instantons on $S^4\setminus T^2$
of Khuri, Reiris, Weinstein and Yamada, quotients of space-periodic metrics, have connected stabilisers and fixed points but no global
Weyl coordinates, and they are asymptotic to a Kasner metric \cite[Thms.~2(i), 4]{KRWY2024}, \cite[\S6.8]{LS2026}; their orbit spaces are
half-cylinders, on which $\star d\rho$ has a period, so their fundamental group is $\Z$ (Remark~\ref{rem:pi1}). Under the hypotheses of Theorem~\ref{thm:G} properness is automatic: this is
Proposition~\ref{prop:endtype}, and Lemma~\ref{lem:globalweyl} when the type of the end is assumed. In global Weyl coordinates
\begin{equation}\label{eq:WP}
g=G_{ab}(\rho,z)\,d\phi^ad\phi^b+e^{2\nu(\rho,z)}\bigl(d\rho^2+dz^2\bigr),\qquad \det G=\rho^2 ,
\end{equation}
and Ricci-flatness is equivalent to two conditions \cite{BZ1978}, \cite[\S8.1, (8.3)--(8.4)]{BV2001}, \cite[\S2.2]{LS2025}. The first is
\begin{equation}\label{eq:hm}
\partial_\rho\bigl(\rho\,G^{-1}\partial_\rho G\bigr)+\partial_z\bigl(\rho\,G^{-1}\partial_z G\bigr)=0 ,
\end{equation}
i.e.\ $\Phi=\rho^{-1}G$, a map into the space $\mathcal H$ of symmetric positive definite $2\times2$ matrices of determinant $1$,
is an axisymmetric harmonic map from $\R^3$ minus the $z$-axis \cite[\S2]{LS2025}. Here $\mathcal H$ carries the Riemannian metric
\begin{equation}\label{eq:Hmetric}
ds^2=\tfrac12\tr\bigl((\Phi^{-1}d\Phi)^2\bigr),
\end{equation}
which has constant curvature $-1$; the group $SL(2,\R)$ acts on $\mathcal H$ transitively and isometrically by $\Phi\mapsto B^T\Phi B$,
with stabiliser $SO(2)$ at the identity, so $\mathcal H=SL(2,\R)/SO(2)$ is the hyperbolic plane. The second condition is the quadrature
\begin{equation}\label{eq:nu}
\partial_z\nu=\tfrac\rho4\tr(J_zJ_\rho),\qquad
\partial_\rho\nu=-\tfrac1{2\rho}+\tfrac\rho8\tr\bigl(J_\rho^2-J_z^2\bigr),\qquad J=G^{-1}dG ,
\end{equation}
whose integrability condition is \eqref{eq:hm}. So $\nu$ is determined by $G$ up to an additive constant. The components of
$\operatorname{Ric}$ along the orbits give \eqref{eq:hm}, and $\operatorname{Ric}(\partial_\rho,\partial_\rho)-\operatorname{Ric}(\partial_z,\partial_z)$ and
$\operatorname{Ric}(\partial_\rho,\partial_z)$ give \eqref{eq:nu}, the term $-1/(2\rho)$ coming from $\det G=\rho^2$; the mixed components vanish by
orthogonal transitivity, and $\operatorname{Ric}(\partial_\rho,\partial_\rho)+\operatorname{Ric}(\partial_z,\partial_z)$ by the contracted Bianchi identity \cite[\S8.1]{BV2001}, \cite[\S2.2]{LS2025}. Following
\cite{LS2025}, a pair $(\Phi,\nu)$ on $\Hb^\circ$ with \eqref{eq:hm} and \eqref{eq:nu} is an \emph{augmented harmonic map}.
Given a lattice $\Lambda\subset\mathfrak t$, it defines the toric Ricci-flat metric \eqref{eq:WP} on
$\Hb^\circ\times\mathfrak t/2\pi\Lambda$, which we orient by $d\rho\wedge dz\wedge d\phi^1\wedge d\phi^2$.
\begin{remark}\label{rem:nuinv}
The quadrature \eqref{eq:nu} involves $G$ only through the traces of products of the components of $J=G^{-1}dG$. For a constant
$\kappa\ne0$ and a constant $B\in GL(2,\R)$, replacing $G$ by $\kappa B^TGB$ replaces $J$ by $B^{-1}JB$. So every solution of \eqref{eq:nu}
for $G$ is one for $\kappa B^TGB$. Two solutions for the same $G$ have the same gradient, so they differ by a constant, since $\Hb^\circ$
is connected.

For the soliton metrics the constant in $e^{2\nu}$ is fixed in Theorem~\ref{thm:family}. Every statement holds for each choice; only the
fact that it is one constant, the same on all rods, is used, and it enters the normalised rod vectors of \S\ref{subsec:rods}.
\end{remark}
\begin{remark}[density]\label{rem:density}
Let $\mathcal P$ be a topological space, $U\subset\Hb^\circ\times\mathcal P$ open, and let $G$ be a symmetric $2\times2$ matrix function and $\nu$ a
function on $U$ such that $G$, its first and second $(\rho,z)$-derivatives and $d_{(\rho,z)}\nu$ are continuous on $U$, with
$\det G\ne0$. If \eqref{eq:hm}, resp.\ \eqref{eq:nu}, holds at the points $(\rho,z,t)\in U$ with $t$ in a dense subset of $\mathcal P$, then it
holds on $U$. Indeed, these points are dense in $U$, since $U$ is open, and both sides of the equation are continuous on $U$.
\end{remark}

We write $d$ for the Riemannian distance of \eqref{eq:Hmetric} on $\mathcal H$. In terms of the matrices,
\begin{equation}\label{eq:dist}
\tr(\Phi_1\Phi_2^{-1})-2=2(\cosh d(\Phi_1,\Phi_2)-1),
\end{equation}
with the normalisation of \cite{LS2025}. To see this, note that both sides are invariant under the isometries $\Phi\mapsto B^T\Phi B$, so
we may take $\Phi_2=I$ and, by a rotation, $\Phi_1=\diag(e^t,e^{-t})$. Then the left side is $2\cosh t-2$, and the curve
$\diag(e^s,e^{-s})$, $0\le s\le t$, is a geodesic of length $t$ by \eqref{eq:Hmetric}, so $d(\Phi_1,I)=|t|$.
A change of basis $B\in GL(2,\R)$ of $\mathfrak t$ acts by $G\mapsto B^TGB$ and on $\Phi$ by an isometry of $\mathcal H$; it rescales
$(\rho,z)$ by $|\det B|$ and changes the sign of $z$ if $\det B<0$.

\emph{Half-plane coordinates.} We use the parametrisation
\begin{equation}\label{eq:PhiHA}
\Phi_{H,A}=\frac1\rho\begin{pmatrix}H^{-1}A^2+\rho^2H&H^{-1}A\\ H^{-1}A&H^{-1}\end{pmatrix},\qquad H>0.
\end{equation}
Every $G$ with $\det G=\rho^2$ and $G_{22}>0$ is $\rho\Phi_{H,A}$ with $H=1/G_{22}$ and $A=G_{12}/G_{22}$. The matrix $\Phi_{H,A}$ is the point
$(\mathfrak x,\mathfrak y)=(\rho H,A)=(\rho/G_{22},G_{12}/G_{22})$ of the half-plane $\{\mathfrak x>0\}$ with the metric
$(d\mathfrak x^2+d\mathfrak y^2)/\mathfrak x^2$ \cite[\S4.1]{LS2025}. After exchanging the two basis vectors, $\Phi_{H,A}=\Phi_\tau$ with
$\tau=\mathfrak y+i\mathfrak x=A+i\rho H$, where for $\operatorname{Im}\tau>0$
\begin{equation}\label{eq:Phitau}
\Phi_\tau=\frac1{\operatorname{Im}\tau}\begin{pmatrix}1&\operatorname{Re}\tau\\ \operatorname{Re}\tau&|\tau|^2\end{pmatrix},\qquad\text{i.e.}\qquad
\Phi_\tau(X,X)=\frac{|x+\tau y|^2}{\operatorname{Im}\tau}
\end{equation}
for $X$ with components $(x,y)$; $\det\Phi_\tau=1$. Since
$\tr(\Phi_\tau\Phi_{\tau'}^{-1})=2+|\tau-\tau'|^2/(\operatorname{Im}\tau\operatorname{Im}\tau')$, the distance \eqref{eq:dist} is given by
\begin{equation}\label{eq:coshhalf}
\cosh d(\Phi_\tau,\Phi_{\tau'})-1=\frac{|\tau-\tau'|^2}{2\operatorname{Im}\tau\operatorname{Im}\tau'}
=\frac{(\mathfrak x-\mathfrak x')^2+(\mathfrak y-\mathfrak y')^2}{2\mathfrak x\mathfrak x'}=\frac{\rho^2(H-H')^2+(A-A')^2}{2\rho^2HH'},
\end{equation}
the last form for two matrices \eqref{eq:PhiHA} at the same $\rho$.

\subsection{Rod structures and local tameness}\label{subsec:rods}
\begin{definition}[{\cite[Def.~3.1]{LS2025}}]\label{def:rodstructure}
A \emph{rod structure} $\mathfrak R$ consists of
\begin{itemize}[leftmargin=*]
\item \emph{turning points} $w_1<\dots<w_n$ on the axis $\{\rho=0\}$, $n\ge1$, which divide it into the \emph{rods}
$I_0=(-\infty,w_1)$, $I_j=(w_j,w_{j+1})$ and $I_n=(w_n,\infty)$;
\item \emph{rod lines} $L_0,\dots,L_n\in\RP^1=\mathbb P(\mathfrak t)$ with $L_{j-1}\ne L_j$;
\item an \emph{asymptotic type}: ALE, ALF$^+$ or ALF$^-$ if $L_0\ne L_n$, and AF$_\beta$, $\beta\in\R$, if
$L_0=L_n$.
\end{itemize}
\end{definition}
Li and Sun call the $L_j$ \emph{rod vectors}, as elements of $\RP^1$ \cite[Def.~3.1]{LS2025}; we call them rod lines and keep the
name rod vectors for vectors that span them. For a smooth toric instanton the turning points are the images of the fixed points, and $L_j$ is the Lie algebra of
the circle that fixes the points over $I_j$. In terms of the Gram matrix: the generator $v_j$ of that circle vanishes over $I_j$, so
$G(v_j,\cdot)=0$ there; on a rod $G$ has rank one and its kernel is the rod line $L_j$, and at a turning point both $v_{j-1}$ and
$v_j$ vanish and $G=0$ (Lemma~\ref{lem:axis}). The rod lines of a metric are thus read off from the kernel of $G$ on the axis. The line $L_j$
contains a primitive vector $v_j$ of the lattice $\Lambda$ of $2\pi$-periodic elements, and smoothness at the turning points means
$\det(v_{j-1},v_j)=\pm1$ \cite{OR1970,LS2025}. For metrics with conical singularities the lines are arbitrary real lines.

\emph{Local models} \cite[Defs.~4.1, 4.2]{LS2025}. For a line $L$ choose $v\in L$ and $\xi$ with $\det(v,\xi)=1$. The
\emph{rod model} of $L$ is the map whose Gram matrix in the basis $(v,\xi)$ is $\diag(\rho^2,1)$, the flat product
$\R^2\times(\R\times S^1)$. For two lines $L\ne L'$ and a point $w$ of the axis choose $v\in L$, $v'\in L'$ with
$\det(v,v')=\pm1$. The \emph{turning-point model} at $w$ is the map whose Gram matrix in the basis $(v,v')$ is
\begin{equation} \diag\bigl(R_w+(z-w),\,R_w-(z-w)\bigr),\qquad R_w=\sqrt{\rho^2+(z-w)^2}, \end{equation}
the flat metric on $\R^4=\C^2$. Here $v$ vanishes on the axis below $w$ and $v'$ above it. Other choices of $\xi$, of
the scale of $v$, or of $v,v'$ change the model by an isometry of $\mathcal H$ that moves it by a distance that is bounded on $\{\rho\le\rho_0\}$ for every $\rho_0$ (for $\xi\mapsto\xi+cv$ one finds $\cosh d-1=c^2\rho^2/2$). So the
following notion does not depend on the choices.

\begin{definition}[{local tameness, \cite[Def.~4.8]{LS2025}}]\label{def:tame}
A map $\Phi\colon\Hb^\circ\to\mathcal H$ is \emph{locally tamed} by a rod structure if every point of a rod $I_j$ has a
neighbourhood on which $d(\Phi,\text{rod model of }L_j)$ is bounded, and every turning point $w_j$ has a neighbourhood
on which the distance from $\Phi$ to the turning-point model of $L_{j-1},L_j$ at $w_j$ is bounded.
\end{definition}

\emph{Normalised rod vectors and cone angles.} Let $(\Phi,\nu)$ be an augmented harmonic map, $G=\rho\Phi$. For
$v\in L_j$ put $\mathfrak c_j(v)^2=\lim_{\rho\to0}\rho^2e^{2\nu}/G(v,v)$ on $I_j$, when the limit exists and is constant along
$I_j$. The circle generated by $v$ then has the cone angle $2\pi/\mathfrak c_j(v)$ along the surface over $I_j$
\cite[(2.8)]{KL2026}, \cite[Def.~4.19]{LS2025}. The \emph{normalised rod vector} $\tilde v_j\in L_j$ is defined, up to
sign, by $\mathfrak c_j(\tilde v_j)=1$ \cite[Def.~4.17]{LS2025}. Suppose that near $I_j$, in a basis $(v_j,\xi)$ with
$\det(v_j,\xi)=1$, the Gram matrix has the \emph{rod form}
\begin{equation}\label{eq:Grod}
G=\begin{pmatrix}\rho^2a&\rho^2b\\ \rho^2b&f\end{pmatrix},\qquad af-\rho^2b^2=1,
\end{equation}
with $a,b,f$ smooth in $(\rho^2,z)$ up to $\rho=0$ \cite[(2.7)]{KL2026}. Then $\adj G\to f\,v_j\otimes v_j$ and
$\mathfrak c_j(v_j)^2=e^{2\nu}/a=e^{2\nu}f$ on the rod, so
\begin{equation}\label{eq:Nrod}
\lim_{\rho\to0}e^{2\nu}\adj G=\tilde v_j\otimes\tilde v_j\qquad\text{on }I_j .
\end{equation}
By \eqref{eq:nu} and \eqref{eq:Grod}, $\partial_\rho\nu=O(\rho)$ and $\partial_z\nu=\partial_za/(2a)+O(\rho)$ near the rod, so $\nu$ extends $C^1$ to $I_j$ and
the quadrature \eqref{eq:nu} on the rod gives $\partial_z(2\nu-\log a)=0$, so this limit is constant along the rod. For a
lattice $\Lambda$ the surface over $I_j$ is smooth exactly when $\tilde v_j$ is a primitive vector of $\Lambda$.
\begin{remark}[smoothness]\label{rem:lattice}
Choose the signs of the normalised rod vectors $\tilde v_0,\dots,\tilde v_n$ so that $\det(\tilde v_{j-1},\tilde v_j)<0$ for all $j$, and
define $\kappa_j,\beta_j$ for $1\le j\le n-1$ by
\begin{equation}\label{eq:kappabeta}
\tilde v_{j+1}=\kappa_j\tilde v_j-\beta_j\tilde v_{j-1}.
\end{equation}
They do not change if a linear map of positive determinant is applied to all $\tilde v_j$, or if they are rescaled by a common factor
(a linear map of negative determinant replaces $\kappa_j$ by $-\kappa_j$), and
$\det(\tilde v_j,\tilde v_{j+1})=\beta_j\det(\tilde v_{j-1},\tilde v_j)$, so $\beta_j>0$. A metric \eqref{eq:WP} whose Gram matrix has the rod
form \eqref{eq:Grod} on every rod and the form \eqref{eq:nutmodel} of Lemma~\ref{lem:formtame} at every turning point, with smooth
coefficients, extends smoothly across all rods and turning points, for some lattice, if and only if $\beta_j=1$ and
$\kappa_j\in\Z$ for $1\le j\le n-1$; the lattice is then generated by any two consecutive $\tilde v_j$, and all $\tilde v_j$ are primitive in
it. Indeed, the metric is smooth across the surface over $I_j$ exactly when $\tilde v_j$ is primitive in the lattice \cite[(2.8)]{KL2026},
\cite[Def.~4.19]{LS2025}, and at a turning point, by \eqref{eq:nutmodel}, exactly when the primitive vectors of the two adjacent rod
lines form a basis of it \cite{OR1970}. So smoothness for $\Lambda$ means that every pair $(\tilde v_{j-1},\tilde v_j)$ is a basis of $\Lambda$.
If so, $\tilde v_{j+1}\in\Lambda$ gives $\kappa_j,\beta_j\in\Z$, and the determinant identity gives $\beta_j=\pm1$, hence $\beta_j=1$. Conversely,
if $\beta_j=1$ and $\kappa_j\in\Z$, then $\tilde v_{j+1}=\kappa_j\tilde v_j-\tilde v_{j-1}$ and $\tilde v_{j-1}=\kappa_j\tilde v_j-\tilde v_{j+1}$, so
$(\tilde v_{j-1},\tilde v_j)$ and $(\tilde v_j,\tilde v_{j+1})$ generate the same lattice; hence all consecutive pairs are bases of one lattice.
\end{remark}

\emph{Degrees} \cite[Def.~3.11]{LS2025}. Connect $L_0,\dots,L_n$ by counter-clockwise rotations by angles in
$(0,\pi)$, and close up with one more such rotation from $L_n$ to $L_0$ if $L_n\ne L_0$. Let $e(\mathfrak R)$ be the total
angle divided by $\pi$. The \emph{degree} is $d(\mathfrak R)=e(\mathfrak R)-1$ for ALE and ALF$^+$, and $d(\mathfrak R)=e(\mathfrak R)$ for ALF$^-$ and AF$_\beta$.
Reversing the orientation gives a rod structure $\bar{\mathfrak R}$ of type ALE, ALF$^\mp$, AF$_{-\beta}$, respectively
\cite[remark after Def.~4.30]{LS2025}. Li and Sun
show that $d=0$ exactly for the hyperk\"ahler (Gibbons--Hawking) metrics and $d=1$ exactly for the Hermitian non-K\"ahler
ones \cite[Thm.~7.4]{LS2025}.
\begin{remark}\label{rem:degid}
If $\mathfrak R$ has $n\ge1$ turning points, then $d(\mathfrak R)+d(\bar{\mathfrak R})=n-1$ for ALE, and $d(\mathfrak R)+d(\bar{\mathfrak R})=n$ for ALF$^\pm$ and AF$_\beta$.
To see this, note that reflection replaces each step $s\in(0,\pi)$ by $\pi-s$. For ALE and ALF there are $n+1$ steps, including the closing one,
so $e+\bar e=n+1$. For ALE, $d+\bar d=e+\bar e-2$. For ALF exactly one of $\mathfrak R,\bar{\mathfrak R}$ is ALF$^+$, so $d+\bar d=e+\bar e-1$.
For AF there are $n$ steps, so $e+\bar e=n$ and $d+\bar d=e+\bar e$.
\end{remark}

\subsection{Models and ends}\label{subsec:models}\label{subsec:asymptypes}
Write $R=\sqrt{\rho^2+z^2}$. For $h_0\in\{0,1\}$ and
$\alpha\in\R$, with $\alpha>0$ if $h_0=0$, let
\begin{equation}\label{eq:models}
\Phi^\sharp_{h_0,\alpha}=\Phi_{H,A},\qquad H=h_0+\frac{|\alpha|}{2R},\quad A=\frac{\alpha z}{2R},\qquad
\nu^\sharp=\tfrac12\log H .
\end{equation}
These are Li--Sun's model augmented harmonic maps \cite[Defs.~4.3--4.6]{LS2025}, written in one frame $(a,b)$:
\begin{center}\small
\begin{tabular}{llll}\toprule
$(h_0,\alpha)$ & type & model space & Li--Sun\\\midrule
$(0,\alpha)$, $\alpha>0$ & ALE & $\R^4$ & AE$_\alpha$\\
$(1,\alpha)$, $\alpha>0$ & ALF$^-$ & anti-Taub--NUT & TN$^-_\alpha$\\
$(1,\alpha)$, $\alpha<0$ & ALF$^+$ & Taub--NUT & TN$^+$ (see below)\\
$(1,0)$ & AF$_0$ & $\R^3\times S^1$ & AF$_0$; AF$_\beta$ after a shear of the frame, see below\\\bottomrule
\end{tabular}
\end{center}
For $\alpha<0$, $\Phi^\sharp_{1,\alpha}$ is $\Phi^\sharp_{1,|\alpha|}$ with $A$ replaced by $-A$ at fixed $(\rho,z)$, i.e.\ its image under the
isometry $\diag(1,-1)$ of $\mathcal H$, which acts on the values of $\Phi$ and is not a change of basis of $\mathfrak t$. This is Li--Sun's model
TN$^+$ for the same pair of rod lines, in the opposite order \cite[Def.~4.5]{LS2025}, the metric TN$^-_{|\alpha|}$ with the opposite orientation. The AF$_\beta$
models are $\Phi^\sharp_{1,0}$ in a frame sheared by a constant \cite[Def.~4.6]{LS2025}; $\beta$ is recorded by the lattice. As Li and Sun
do, we orient the models by $d\rho\wedge dz\wedge d\phi^a\wedge d\phi^b$. Under $(a,b)\mapsto(a,-b)$, which comes with $z\mapsto-z$
(\S\ref{subsec:WP}), $A$ becomes $-A(\rho,-z)=A$, so the model of an oriented metric does not depend on the orientation of the frame.
In the AE case $\alpha$ is a choice of scale: the Gram matrix of $\Phi^\sharp_{0,2}$ in the frame $((\alpha/2)^{1/2}a,(2/\alpha)^{1/2}b)$, which
has determinant $1$ and so does not change $(\rho,z)$, is that of $\Phi^\sharp_{0,\alpha}$ in the frame $(a,b)$.

Li and Sun build model ends from these maps \cite[\S4.4]{LS2025}: flat $\R^4$ ($M^{\AEt}$), anti-Taub--NUT and Taub--NUT
($M^{\mathrm{TN}^\mp}$), and $(\R^3\times\R)/\langle q_\beta\rangle$ ($M^{\AF_\beta}$), each with its torus action. Let $\mathfrak r$
be the distance from a point of the model.
\begin{definition}[{\cite[Def.~4.30]{LS2025}}]\label{def:asymptype}
A toric Ricci-flat end $(E,g)$ is of \emph{asymptotic type}
\begin{itemize}[leftmargin=*]
\item \emph{ALE} if a finite cover of $E$ admits a diffeomorphism $F$, orientation-preserving and equivariant for a finite covering torus $\tilde T\to T$, onto the end of
$M^{\AEt}$ (Li and Sun call $F$ a $T$-isometry) with
$|\nabla^k(g-F^*g_{\AEt})|=O(\mathfrak r^{-k-\delta})$ for some $\delta>0$ and all $k$;
\item \emph{ALF$^\pm$} if the same holds, again on a finite cover, with $M^{\mathrm{TN}^\pm}$ in place of $M^{\AEt}$;
\item \emph{AF$_\beta$} if $E$ itself admits such a map onto the end of $M^{\AF_\beta}$.
\end{itemize}
\end{definition}
Here an \emph{end} is the complement of a compact $T$-invariant set, in $M$ or in the model, and norms and covariant derivatives
are those of the pulled-back model metric. We say that $M$ has an end of a given type if the complement of some compact
$T$-invariant subset of $M$ is of that type.
\begin{remark}\label{rem:asymptypes}
The classical definition of a toric ALE end, a Ricci-flat end $\R\times S^3/\Gamma$ with $g=d\mathfrak r^2+\mathfrak r^2\gamma+h$,
$|\nabla^kh|=O(\mathfrak r^{-\tau-k})$, on which $T$ acts preserving $\gamma$ \cite[Defs.~2.2, 2.4]{KL2026}, implies
Definition~\ref{def:asymptype}: by \cite{BKN1989} one may take $\tau=4$, the Killing fields are asymptotic to rotations,
and a centre-of-mass argument makes the chart equivariant \cite[proof of Prop.~4.38]{LS2025}. Kunduri and Lucietti's
toric ALF and AF ends \cite[Def.~2.4]{KL2026}, \cite[Def.~2.2]{KL2021} are of Li--Sun type, with $\delta=1$.
Each of these ends has $|\mathrm{Rm}_g|=O(\mathfrak r^{-2-\min(\delta,1)})$: the model metrics are flat or (anti-)Taub--NUT, whose curvature
is $O(\mathfrak r^{-3})$, and the difference from the model metric and its first two derivatives contribute $O(\mathfrak r^{-2-\delta})$. So a complete toric
Ricci-flat manifold with an end of one of these types is a toric gravitational instanton (Definition~\ref{def:GI}).
Its fixed points are finitely many. They are isolated: at a fixed point $p$ the isotropy representation of $T$ on $T_pM$ is
faithful, since an isometry fixing $p$ with derivative the identity is the identity, so $T$ is a maximal torus of
$\mathrm{SO}(T_pM)$ and fixes only $0$, and the exponential map at $p$ is equivariant. They lie outside the end, since a fixed point
in the end would lift to fixed points of the connected group $\tilde T$ in the finite cover, which permutes the finite fibre, and the
model ends contain no fixed points (the origin and the nut are not in them). So they lie in a compact set.
\end{remark}

\subsection{The class \texorpdfstring{$\Ctor$}{C\_toric}}\label{subsec:Ctor}
\begin{definition}\label{def:Ctor}
$\Ctor$ is the class of augmented harmonic maps $(\Phi,\nu)$ on $\Hb^\circ$ with the following properties. Here $\Phi$,
$\rho$ and $z$ are computed in a basis of $\mathfrak t$ (\S\ref{subsec:WP}); the properties, and the type below, do not depend
on it (\S\S\ref{subsec:WP}, \ref{subsec:models}).
\begin{enumerate}[leftmargin=*,label=\textup{(T\arabic*)}]
\item\label{it:T1} $\Phi$ is locally tamed by a rod structure with $n\ge1$ turning points (Definition~\ref{def:tame}), whose asymptotic type is that of the model in \ref{it:T2}.
\item\label{it:T2} There are $(h_0,\alpha)$ as in \eqref{eq:models} and a basis $(a,b)$ of $\mathfrak t$ in which
$d(\Phi,\Phi^\sharp_{h_0,\alpha})\to0$ as $R\to\infty$, uniformly in $\theta=\arg(z+i\rho)$.
\end{enumerate}
A toric Ricci-flat metric \eqref{eq:WP}, with any lattice $\Lambda$, lies in $\Ctor$ if its augmented harmonic map does.
Its \emph{type} is the type of the model, and AF$_\beta$ when $h_0=1$, $\alpha=0$ and the lattice records the twist
$\beta$.
\end{definition}
A translation of $z$ by $z_0$ moves $\Phi^\sharp_{h_0,\alpha}$ by a hyperbolic distance $O(|z_0|/R)$, uniformly in $\theta$, so
\ref{it:T2} does not depend on the origin of $z$. Metrics in $\Ctor$ may have conical singularities along any rod; no
regularity at the turning points is assumed beyond local tameness.
\begin{lemma}\label{lem:outerlines}
Let $(\Phi,\nu)\in\Ctor$, with model $\Phi^\sharp_{h_0,\alpha}$ in the basis $(a,b)$. Then its outer rod lines are
$L_0=[(1,\alpha/2)]$ and $L_n=[(1,-\alpha/2)]$ in that basis.
\end{lemma}
\begin{proof}
Near a point of $I_n$ with large $z$, $\Phi$ is at bounded distance from the rod model of $L_n$, which tends to the
boundary point $L_n$ of $\mathcal H$ as $\rho\to0$; by \ref{it:T2}, $\Phi^\sharp=\Phi^\sharp_{h_0,\alpha}$ is within distance $o(1)$
of $\Phi$ there. As $\rho\to0$ near a point $(0,z)$ of $I_n$, $A=\alpha z/(2R)\to\alpha/2$ and $H\to h_0+|\alpha|/(2z)>0$, and for $v=(1,-\alpha/2)$, by \eqref{eq:PhiHA},
$\rho\Phi^\sharp(v,v)=\rho^2H+H^{-1}(A-\alpha/2)^2\to0$, while $\rho\Phi^\sharp_{22}=H^{-1}\not\to0$. So $\rho\Phi^\sharp$ tends to a form of
rank one with kernel $[v]$; equivalently, the half-plane point $(\mathfrak x,\mathfrak y)=(\rho H,A)$ tends to the boundary point
$\mathfrak y=\alpha/2$, which is the line $[v]$. Two maps at bounded distance tend to the same boundary point, so
$L_n=[(1,-\alpha/2)]$. Likewise $L_0=[(1,\alpha/2)]$, since $A\to-\alpha/2$ on $I_0$.
\end{proof}
\begin{remark}[status of the class]\label{rem:Ctorstatus}
Li and Sun's strongly tamed harmonic maps \cite[Defs.~4.9, 4.10, Thm.~4.24]{LS2025}, locally tamed by a rod structure and at
distance $O(R^{-1})$ from the model of its type, lie in $\Ctor$, with any augmentation, so
Proposition~\ref{prop:GG} applies to them. Whether every toric Ricci-flat metric that is smooth except for conical singularities
along the axis, with an end of one of the three types, lies in $\Ctor$ is open; it needs a regularity theory at the turning points.
\end{remark}

\section{Uniqueness and smooth instantons}\label{sec:uniq}
This section proves the uniqueness theorem (Lemma~\ref{lem:uniq}), shows that a smooth toric instanton as in
Theorem~\ref{thm:G} has global Weyl coordinates, lies in $\Ctor$ and has an end of type ALE, ALF$^\pm$ or AF$_\beta$
(Propositions~\ref{prop:smoothCtor} and~\ref{prop:endtype}), and converts the rod lines of an element of $\Ctor$ into the
normalised rays and cells on which the soliton families are parametrised (Lemma~\ref{lem:match}).

\subsection{Uniqueness}\label{subsec:uniqueness}
The following form of uniqueness is what the proofs need. It is the argument of
\cite[Thm.~4, and the proof of Thm.~2]{Weinstein1990}, \cite[\S4, Theorem, p.~763]{HY2008}, \cite[Thm.~3.1]{KL2026}, \cite[Thm.~4.24]{LS2025}, stated with minimal hypotheses.
\begin{lemma}[uniqueness]\label{lem:uniq}
Let $\Phi_0,\Phi_1\colon\Hb^\circ\to\mathcal H$ be axisymmetric harmonic maps, locally tamed by the same rod structure
(the same turning points and rod lines), such that $d(\Phi_i,\Phi^\sharp)\to0$ as $R\to\infty$, uniformly in $\theta$,
for one model $\Phi^\sharp=\Phi^\sharp_{h_0,\alpha}$ in one basis. Then $\Phi_0=\Phi_1$.
\end{lemma}
\begin{proof}
View $\Phi_i$ as maps from $\R^3\setminus\Gamma$, $\Gamma$ the $z$-axis; they are harmonic there. The proof relies on the
following observations.

The distance between two harmonic maps into a simply connected manifold of non-positive
curvature is subharmonic \cite[assertion (F) and \S8]{Hartman1967}, \cite[Prop.~1 and \S2, (2.11)]{SY1979}. So $f=d(\Phi_0,\Phi_1)$ is continuous and subharmonic on
$\R^3\setminus\Gamma$.

Next, $f$ is bounded. Near a compact part of $\Gamma$ this holds since by local tameness both maps are within a
bounded distance of the same local models on a finite cover by neighbourhoods. On $\{R\ge R_0\}$, $f$ is bounded by the
asymptotic hypothesis and the triangle inequality. On the rest it is bounded by continuity and compactness.

Finally, $f$ extends across the axis. The axis $\Gamma$ is a closed polar set: $-\log\rho$ is locally integrable and superharmonic on $\R^3$, its
Laplacian in the sense of distributions being $-2\pi$ times the length measure on $\Gamma$, and it is $+\infty$ on $\Gamma$
\cite[\S5.1]{AG2001}. A subharmonic function on the complement of a closed polar set that is bounded above near it extends
subharmonically across it \cite[Thm.~5.2.1]{AG2001}. So $f$ extends to a non-negative subharmonic function on $\R^3$.

From the above we conclude that $f$ is a non-negative subharmonic function on all of $\R^3$ which tends to $0$ at infinity.
By the maximum principle, $f\le\sup_{\{|x|=R\}}f$ on $\{|x|\le R\}$ for every $R$, and the right side tends to $0$ as
$R\to\infty$; so $f\equiv0$, that is, $\Phi_0=\Phi_1$.
\end{proof}

\begin{lemma}[local forms give local tameness]\label{lem:formtame}
Let $G$ be a real symmetric $2\times2$ matrix function on $\Hb^\circ$, written in a basis of $\mathfrak t$, with $\det G=\rho^2$, and let $\Phi=\rho^{-1}G$
be the corresponding map into $\mathcal H$ (\S\ref{subsec:WP}).
\begin{enumerate}[leftmargin=*,label=\textup{(\roman*)}]
\item If near a point of the axis, in a basis $(v,\xi)$ with $\det(v,\xi)=1$, $G$ has the rod form \eqref{eq:Grod} with
$a,b,f$ continuous up to $\rho=0$ and $f>0$, then $d(\Phi,\text{rod model of }[v])$ is bounded near that point.
\item Let $z_0\in\R$ be a point of the axis, $s=\frac12(R_{z_0}-(z-z_0))$ and $t=\frac12(R_{z_0}+(z-z_0))$, so that $s,t\ge0$, $\rho^2=4st$ and
$z-z_0=t-s$. If near $z_0$, in a basis $(v,v')$ with
$\det(v,v')=\pm1$,
\begin{equation}\label{eq:nutmodel}
G=\begin{pmatrix}tB&stC\\ stC&sA\end{pmatrix},\qquad A,B,C\ \text{continuous in }(s,t),\quad A(0)>0,
\end{equation}
then $d(\Phi,\text{turning-point model of }[v],[v']\text{ at }z_0)$ is bounded near $z_0$.
\end{enumerate}
\end{lemma}
\begin{proof}
We use the half-plane coordinates $(\mathfrak x,\mathfrak y)=(\rho/G_{22},G_{12}/G_{22})$ and \eqref{eq:coshhalf} (\S\ref{subsec:WP}),
in the basis of each item; there $G_{22}>0$.

(i) For the rod form \eqref{eq:Grod}, $(\mathfrak x,\mathfrak y)=(\rho/f,\rho^2b/f)$, and for the rod model, with $a=f=1$ and $b=0$,
$(\mathfrak x_0,\mathfrak y_0)=(\rho,0)$. So
\begin{equation}
\cosh d-1=\frac f2\Bigl(\frac{\rho^2b^2}{f^2}+\Bigl(\frac1f-1\Bigr)^2\Bigr),
\end{equation}
which is bounded. (ii) Here $\rho^2=4st$, and $G_{22}=sA>0$ near $z_0$, so $\mathfrak x=\rho/(sA)=2\sqrt{t/s}/A$ and
$\mathfrak y=stC/(sA)=tC/A$. The turning-point model has $R_{z_0}=s+t$, $R_{z_0}\pm(z-z_0)=2t,2s$, so it is \eqref{eq:nutmodel} with
$A_0=B_0=2$, $C_0=0$, i.e.\ $(\mathfrak x_0,\mathfrak y_0)=(\sqrt{t/s},0)$. Hence, by \eqref{eq:coshhalf},
\begin{equation} \cosh d-1=\frac{(\mathfrak x-\mathfrak x_0)^2+\mathfrak y^2}{2\mathfrak x\mathfrak x_0}=\frac{(2-A)^2}{4A}+\frac{stC^2}{4A}, \end{equation}
which is bounded near $z_0$ since $A(0)>0$. (Note that $\det G=\rho^2$ forces $AB-stC^2=4$, so $B(0)=4/A(0)>0$ automatically.)
\end{proof}

\subsection{Smooth toric instantons}\label{subsec:smooth}
Lemma~\ref{lem:formtame} says that the rod form \eqref{eq:Grod} and the turning-point form \eqref{eq:nutmodel} of the Gram matrix
give local tameness. The first lemma of this subsection shows that the Gram matrix of a smooth toric Ricci-flat metric has these
forms near every point of the axis, with smooth coefficients, in the coordinates determined by the Weyl coordinates; together,
the two lemmas give the local tameness of smooth metrics (Proposition~\ref{prop:smoothCtor}). The lemma is local and uses
neither global Weyl coordinates nor the end.
\begin{lemma}[the orbit space near the axis]\label{lem:axis}
Let $(M,g)$ be a toric Ricci-flat manifold, $\Lambda\subset\mathfrak t$ the lattice of $2\pi$-periodic elements, and compute $G$ and $\rho$
in a $\Z$-basis of $\Lambda$. Let $q\in M$ be a point with non-trivial stabiliser and $p$ its image in $\mathcal B$. Let $z$ be a harmonic
conjugate of $\rho$ on $U\cap\mathcal B^\circ$, for a neighbourhood $U$ of $p$ in $\mathcal B$ with $U\cap\mathcal B^\circ$ connected, and $w=z+i\rho$. The following holds.
\begin{enumerate}[leftmargin=*,label=\textup{(\roman*)}]
\item \emph{Fixed point.} If $q$ is a fixed point, there are a $\Z$-basis $(v,v')$ of $\Lambda$ and a surface $S$ with a smooth metric
$\hat g_S$ and coordinates $(x_1,x_2)$ centred at $0$, in which $\sigma_1(x_1,x_2)=(-x_1,x_2)$ and $\sigma_2(x_1,x_2)=(x_1,-x_2)$ are
isometries, such that: a neighbourhood of $p$ in $\mathcal B$ is isometric to $S/\langle\sigma_1,\sigma_2\rangle=\{x_1,x_2\ge0\}$; the Killing
field of $v$ vanishes over $\{x_1=0\}$ and that of $v'$ over $\{x_2=0\}$; and $\rho=x_1x_2h$ with $h$ smooth, even in $x_1$ and in
$x_2$, and positive. After exchanging $(v,x_1)$ and $(v',x_2)$ if necessary, $w<w(p)$ on $\{x_1=0\}\setminus\{0\}$ and $w>w(p)$ on
$\{x_2=0\}\setminus\{0\}$, where $w(p)\in\R$ is the limit of $w$ at $p$ \textup{(}see \textup{(iii))}; and with $s,t\ge0$ defined by $\rho^2=4st$
and $z-w(p)=t-s$, the Gram matrix in the basis $(v,v')$ has the form \eqref{eq:nutmodel} with $A,B,C$ smooth in $(s,t)$ and
$A(0),B(0)>0$.
\item \emph{Circle stabiliser.} Otherwise the stabiliser of $q$ is a circle, generated by a primitive $v\in\Lambda$; let $\xi\in\mathfrak t$
with $\det(v,\xi)=1$. There are $S$, $\hat g_S$ and $(x_1,x_2)$ as in \textup{(i)} such that $\sigma=\sigma_1$ is an isometry, a
neighbourhood of $p$ in $\mathcal B$ is isometric to $S/\langle\sigma\rangle=\{x_1\ge0\}$, the Killing field of $v$ vanishes over $\{x_1=0\}$,
and $\rho=x_1h$ with $h$ smooth, even in $x_1$ and positive. The map $(x_1^2,x_2)\mapsto(\rho^2,z)$ is a local diffeomorphism at $0$,
and in the basis $(v,\xi)$ the Gram matrix has the rod form \eqref{eq:Grod} with $a,b,f$ smooth in $(\rho^2,z)$ and $f>0$.
\item In both cases $w$ extends continuously to a neighbourhood of $p$ in $\mathcal B$ and maps it homeomorphically onto a
neighbourhood of a point $w(p)\in\R$ in $\{\rho\ge0\}$.
\end{enumerate}
\end{lemma}
\begin{proof}
Since $q$ has a non-trivial stabiliser, the action has a non-principal orbit, so it is orthogonally transitive and $\rho$ is
harmonic on $\mathcal B^\circ$ (\S\ref{subsec:WP}). The determinant of $G$ is the same in all $\Z$-bases of $\Lambda$, so we may compute $\rho$
in the bases below. We first treat a fixed point, which is the main case, then a point with circle stabiliser, and finally the
continuity statement (iii).

Let $q$ be a fixed point. By the slice theorem \cite[Ch.~VIII]{Borel1960}, \cite[Ch.~VI, Thm.~2.2]{Bredon1972}, cf.\ \cite[\S1.4]{OR1970}, a $T$-invariant neighbourhood of $q$ is equivariantly
diffeomorphic to a ball in $\C^2$ with a linear action, and since the stabilisers are connected, its two weights form a basis of
the dual lattice. Let $(v,v')$ be the dual basis, a $\Z$-basis of $\Lambda$. In complex coordinates $(x_1,x_2)$ the action is
$(x_1,x_2)\mapsto(e^{i\theta_1}x_1,e^{i\theta_2}x_2)$, with $(\theta_1,\theta_2)$ the coordinates of $(v,v')$, so the Killing field of $v$
vanishes on $\{x_1=0\}$ and that of $v'$ on $\{x_2=0\}$. Let $u_i=|x_i|^2$, and let $S=\{(x_1,x_2)\in\R^2\}\subset\C^2$ be the
\emph{real slice}, a surface through $q$ with the real coordinates $(x_1,x_2)$, on which the reflections $\sigma_1$ and $\sigma_2$ of the
statement act.

We shall make use of the following observation about the Gram matrix $G$ in the basis $(v,v')$. Smooth $T$-invariant functions on
$\C^2$ are smooth functions of $(u_1,u_2)$ \cite[Thm.~1]{Schwarz1975}. The Killing fields of $v$ and $v'$ are the rotations of $x_1$
and $x_2$, so by Hadamard's lemma \cite[Lemma~2.1]{Milnor1963} (applied after replacing the Schwarz extension $F$ by $F-F(0,u_2)$, which does not change it on the quadrant, and in each variable in turn)
\begin{equation}\label{eq:Gslice}
G=\begin{pmatrix}u_1\tilde B&u_1u_2\tilde C\\u_1u_2\tilde C&u_2\tilde A\end{pmatrix},\qquad \tilde A,\tilde B,\tilde C\ \text{smooth in }(u_1,u_2),\quad
\tilde A(0),\tilde B(0)>0 ,
\end{equation}
and $\rho^2=u_1u_2\tilde D$ with $\tilde D=\tilde A\tilde B-u_1u_2\tilde C^2>0$ near $0$. This is the form \eqref{eq:nutmodel}, but in
the coordinates $(u_1,u_2)$ of the linear model, not in the coordinates $(s,t)$ defined by the Weyl coordinates; the rest of the proof
of (i) passes from the one to the other. The form \eqref{eq:Gslice} is preserved by the exchange of $(v,x_1)$ and $(v',x_2)$, which
exchanges $\tilde A$ and $\tilde B$.

The real slice $S$, introduced above, meets every orbit, and the orbits of the points of the quadrant $\{x_1,x_2\ge0\}$ are
distinct. The orbit-space metric pulled back to $S$ is
\begin{equation}\label{eq:gS}
\hat g_S(X,X)=g(X,X)-\beta^TG^{-1}\beta,\qquad \beta_a=g(X,X_{e_a}),
\end{equation}
the norm of the component of $X$ orthogonal to the orbit. It is smooth: on $S$, $\beta_1$ is $x_1$ times a
smooth function, $\beta_2$ is $x_2$ times a smooth function, and
\begin{equation}\label{eq:Ginvslice}
G^{-1}=\frac1{\tilde D}\begin{pmatrix}\tilde A/u_1&-\tilde C\\-\tilde C&\tilde B/u_2\end{pmatrix},
\end{equation}
so the singular terms cancel. It is positive definite near $0$: on $S$, $\beta_1/x_1=g(X,\partial_{\operatorname{Im}x_1})$ and
$\beta_2/x_2=g(X,\partial_{\operatorname{Im}x_2})$; at $0$ the isotropy representation is orthogonal for $g$, so the $x_1$- and
$x_2$-planes, on which the circle of $v$ acts by different characters, are $g$-orthogonal, and $g$ is rotation-invariant on each.
Hence $\beta_1/x_1$ and $\beta_2/x_2$ vanish at $0$ for $X\in T_0S$, $\hat g_S(0)=g|_{T_0S}$ is positive definite, and so is
$\hat g_S$ near $0$. The reflections $\sigma_i$ are the actions on $S$ of the elements of order $2$ of $T$, so they are
isometries of $\hat g_S$. Put $h=\tilde D^{1/2}$ and $\rho_S=x_1x_2h$ on $S$. Then $h$ is smooth, positive and even in $x_1$ and in
$x_2$, $\rho_S$ is odd under each $\sigma_i$ and equal to $\rho$ on the quadrant, and $\rho_S$ is harmonic, since it is harmonic off
the axes and smooth.

The Weyl coordinate $w$ has a double zero at the turning point, and the key step is to take its square root. On the open
quadrant near $0$, $z$ and the harmonic conjugates of $\rho_S$ are harmonic conjugates of the same function; so we can orient $S$
such that $z=z_S+c$ there, with $c\in\R$ and $z_S$ the harmonic conjugate of $\rho_S$ on $S$ for this orientation with $z_S(0)=0$.
Each $\sigma_i$ reverses the orientation and $\rho_S$ is odd, so $dz_S=\star d\rho_S$ is invariant and $z_S$ is even. In an isothermal
coordinate $\zeta$ of $\hat g_S$ centred at $0$, $w_S=z_S+i\rho_S$ is holomorphic, $w_S(0)=0$ and $dw_S(0)=0$, and its second
derivative at $0$ is non-zero, because the Hessian of $\rho_S$ at $0$ is $h(0)$ times that of $x_1x_2$. So $w_S$ has a non-degenerate
double zero at $0$, $w_S/\zeta^2$ is holomorphic and non-zero near $0$, and
\begin{equation}\label{eq:omegasqrt}
\omega=\zeta\,(w_S/\zeta^2)^{1/2}
\end{equation}
is a holomorphic coordinate with $w_S=\omega^2$; it is smooth in $(x_1,x_2)$. On the open quadrant
$\rho_S=\operatorname{Im}w_S=2\operatorname{Re}\omega\operatorname{Im}\omega>0$; we choose the sign of $\omega$ with $\operatorname{Re}\omega,\operatorname{Im}\omega>0$ there.

The two axes of the slice correspond to the real and the imaginary axis of $\omega$. Near $0$ the zero set of $\rho_S$ is
$\{x_1=0\}\cup\{x_2=0\}$, and it is also $\{\operatorname{Re}\omega=0\}\cup\{\operatorname{Im}\omega=0\}$, two smooth curves crossing transversally
at $0$. After exchanging $(v,x_1)$ and $(v',x_2)$ we may assume $\{\operatorname{Re}\omega=0\}=\{x_1=0\}$ and
$\{\operatorname{Im}\omega=0\}=\{x_2=0\}$ near $0$. Hence $\operatorname{Re}\omega=x_1\varphi_1$ and $\operatorname{Im}\omega=x_2\varphi_2$ with $\varphi_i$ smooth,
$\varphi_i(0)\ne0$ since $\omega$ is a coordinate, and $\varphi_i(0)>0$ by the choice of sign. As $w_S\circ\sigma_i=\overline{w_S}$, we have
$\omega\circ\sigma_i=\pm\overline\omega$, and $\varphi_1(0)>0$, $\varphi_2(0)>0$ force $\omega\circ\sigma_1=-\overline\omega$ and $\omega\circ\sigma_2=\overline\omega$. So the
$\varphi_i$ are even in $x_1$ and in $x_2$, hence smooth functions of $(u_1,u_2)$ \cite[Thm.~1]{Schwarz1975} (for $\Z_2\times\Z_2$ acting on $\R^2$). Off $0$, $w_S=-(\operatorname{Im}\omega)^2<0$
on $\{x_1=0\}$ and $w_S=(\operatorname{Re}\omega)^2>0$ on $\{x_2=0\}$.

It remains to pass from the coordinates $(u_1,u_2)$ of \eqref{eq:Gslice} to the coordinates $(s,t)$ of \eqref{eq:nutmodel}, which
are defined by the Weyl coordinates; this is where the square root enters. On the open quadrant $w=\omega^2+c$, so $w$ extends
continuously to the quadrant, with $w(p)=c$, and $w<w(p)$ on $\{x_1=0\}\setminus\{0\}$ and $w>w(p)$ on $\{x_2=0\}\setminus\{0\}$. With
$s,t$ as in the statement, $w-w(p)=(\sqrt t+i\sqrt s)^2$, and also $w-w(p)=\omega^2$ with $\operatorname{Re}\omega,\operatorname{Im}\omega\ge0$ on the
quadrant. Squaring is injective on the closed first quadrant, so $\operatorname{Re}\omega=\sqrt t$ and $\operatorname{Im}\omega=\sqrt s$, i.e.\
$t=u_1\varphi_1^2$ and $s=u_2\varphi_2^2$. The map $(u_1,u_2)\mapsto(s,t)$ is a local diffeomorphism at $0$, since its Jacobian there is
anti-diagonal with entries $\varphi_2(0)^2$, $\varphi_1(0)^2$, and it maps $\{u_1,u_2\ge0\}$ onto $\{s,t\ge0\}$ near $0$. In the coordinates
$(s,t)$, the matrix \eqref{eq:Gslice} becomes
\begin{equation}\label{eq:Gst}
G=\begin{pmatrix}t\,\tilde B/\varphi_1^2& st\,\tilde C/(\varphi_1\varphi_2)^2\\[1mm] st\,\tilde C/(\varphi_1\varphi_2)^2& s\,\tilde A/\varphi_2^2\end{pmatrix},
\end{equation}
which is \eqref{eq:nutmodel} with $B=\tilde B/\varphi_1^2$, $A=\tilde A/\varphi_2^2$ and $C=\tilde C/(\varphi_1\varphi_2)^2$, smooth in $(s,t)$ and with
$A(0),B(0)>0$. This proves (i).

Now let the stabiliser of $q$ be a circle. The slice theorem gives a $T$-invariant neighbourhood $S^1\times\C\times(-\epsilon,\epsilon)$ of $q$, with coordinates
$(x_1,x_2)\in\C\times(-\epsilon,\epsilon)$ on the last two factors, on which $v$ rotates $x_1$, with weight $\pm1$ since the stabilisers are
connected, and an element $\xi_0\in\Lambda$ with $\det(v,\xi_0)=1$ rotates $S^1$. With $u=|x_1|^2$, \cite[Thm.~1]{Schwarz1975} and Hadamard's
lemma (as above) give, in the basis $(v,\xi_0)$,
\begin{equation}\label{eq:Gslicecircle}
G=\begin{pmatrix}u\tilde A&u\tilde C\\u\tilde C&\tilde B\end{pmatrix},\qquad \tilde A,\tilde B,\tilde C\ \text{smooth in }(u,x_2),\quad \tilde A,\tilde B>0 .
\end{equation}
Every $\xi$ with $\det(v,\xi)=1$ is $\xi_0+cv$, $c\in\R$, and replacing $\xi_0$ by $\xi_0+cv$ replaces $(\tilde C,\tilde B)$ by
$(\tilde C+c\tilde A,\tilde B+2cu\tilde C+c^2u\tilde A)$; so the same form holds in the basis $(v,\xi)$. Let $S=\{x_1\in\R\}\times(-\epsilon,\epsilon)$
at a point of $S^1$, with coordinates $(x_1,x_2)$. As at a fixed point, $\hat g_S$ is smooth, and $\beta_1/x_1=0$ at $0$, so
$\hat g_S(0)(X,X)=g(X,X)-g(X,X_{\xi_0})^2/|X_{\xi_0}|^2>0$ for $X\in T_0S\setminus0$, since $T_0S$ is transverse to the orbit; so $\hat g_S$ is
positive definite near $0$. Further, $\sigma$ is the action of the element of order $2$ of
the circle generated by $v$, a neighbourhood of $p$ is isometric to $\{x_1\ge0\}$, and $\rho^2=u\tilde D$ with
$\tilde D=\tilde A\tilde B-u\tilde C^2>0$. Put $h=\tilde D^{1/2}$ and $\rho_S=x_1h$, which is smooth, odd and harmonic, with $d\rho_S\ne0$ on
$\{x_1=0\}$; no square root is needed here. Orient $S$ such that $z=z_S+c$ on $\{x_1>0\}$, with $z_S$ the harmonic conjugate of $\rho_S$ with
$z_S(0)=0$. As before, $z_S$ is even, so $\partial_{x_1}z_S=0$ on $\{x_1=0\}$, and $dz_S=\star d\rho_S\ne0$ there. So $(\rho_S,z_S)$ are
coordinates on $S$ near $0$; $z_S$ is a smooth function of $(u,x_2)$ \cite[Thm.~1]{Schwarz1975} (for $\Z_2$ acting on $x_1$); and $(u,x_2)\mapsto(\rho^2,z_S)=(u\tilde D,z_S)$ is a local
diffeomorphism at $0$, since its Jacobian there is triangular with diagonal entries $\tilde D(0)$ and $\partial_{x_2}z_S(0)\ne0$. With
$a=\tilde A/\tilde D$, $b=\tilde C/\tilde D$ and $f=\tilde B$, the matrix \eqref{eq:Gslicecircle} is the rod form \eqref{eq:Grod}, and $a,b,f$ are smooth in $(\rho^2,z)$. This proves (ii).

Finally we prove (iii), the continuity of $w$ at $p$. At a fixed point, $\omega$ maps a neighbourhood of $0$ in the closed quadrant
homeomorphically onto a neighbourhood of $0$ in the closed first quadrant, and squaring maps that homeomorphically onto a
neighbourhood of $0$ in $\{\rho\ge0\}$; so $w=\omega^2+c$ has the stated property, with $w(p)=c$. At a point with circle stabiliser,
$(\rho_S,z_S)$ is a chart of $S$ near $0$ that maps $\{x_1\ge0\}$ onto a neighbourhood of $0$ in $\{\rho\ge0\}$, and $w=z_S+c+i\rho_S$.
\end{proof}

\begin{lemma}[the end and its model]\label{lem:endcompare}
Let $(M,g)$ be a toric Ricci-flat manifold with an end of one of the types of Definition~\ref{def:asymptype}, with the map $F$ and the
model metric $g^\sharp$ there, and let $\mathfrak r$ be the distance in the model. In the ALE case replace the end by the finite cover.
Let $g'=(F^{-1})^*g$ on the end of the model, and let $G',\rho',\hat g'$ and $G^\sharp,\rho^\sharp,\hat g^\sharp$ be the Gram matrices in one
basis of $\mathfrak t$, the Weyl radii and the orbit-space metrics of $g'$ and $g^\sharp$. Let $(\rho^\sharp,z^\sharp)$ be the Weyl coordinates of
the model, $\hat g^\sharp=e^{2\nu^\sharp}(d\rho^{\sharp2}+dz^{\sharp2})$, and $R^\sharp=|z^\sharp+i\rho^\sharp|$. Then $R^\sharp\asymp\mathfrak r^\kappa$, with $\kappa=2$ for ALE and $\kappa=1$ otherwise: on $\R^4$, $R^\sharp$ is a
constant multiple of $\mathfrak r^2$, and on the other models it is the radius of the base $\R^3$. There are $C$ and
$R_0$ such that, with $\varepsilon=C\mathfrak r^{-\delta}$, on $\{R^\sharp>R_0\}$:
\begin{enumerate}[leftmargin=*,label=\textup{(\roman*)}]
\item $(1-\varepsilon)G^\sharp\le G'\le(1+\varepsilon)G^\sharp$ as quadratic forms on $\mathfrak t$, and $(1-\varepsilon)\hat g^\sharp\le\hat g'\le(1+\varepsilon)\hat g^\sharp$. Hence
$|\rho'/\rho^\sharp-1|\le\varepsilon$, and $d(\Phi',\Phi^\sharp)\le2\varepsilon$ for $\Phi'=G'/\rho'$ and $\Phi^\sharp=G^\sharp/\rho^\sharp$.
\item A harmonic conjugate $z'$ of $\rho'$, with the sign of \S\ref{subsec:WP}, satisfies $|d(\rho'-\rho^\sharp)|+|d(z'-z^\sharp)|\le C\varepsilon$ for the
Euclidean norm in the coordinates $(\rho^\sharp,z^\sharp)$.
\end{enumerate}
\end{lemma}
\begin{proof}
In the ALE case we prove the estimates on the finite cover $\tilde E$ of Definition~\ref{def:asymptype}, with
the lifted metric and the lifted action of a finite covering torus $\tilde T\to T$ (with the same Lie algebra $\mathfrak t$). Its deck transformations are isometries that commute with the $\tilde T$-action, since $\tilde T$
is connected; so they preserve $G$ and $\rho$, map each outer rod into itself (they preserve the stabilisers, and $L_0\ne L_n$), and act on the orbit space by conformal maps that
preserve $\rho$, i.e.\ by translations of $z$, which are the identity since they have finite order. So $\tilde E$ and $E$ have the same
orbit space, Gram matrix and Weyl coordinates. We make the following observations; the first proves (i), the others
prove (ii).

(i) Let $h=g'-g^\sharp$. By Definition~\ref{def:asymptype} with $k=0,1$, after enlarging $C$, $|h|_{g^\sharp}\le\varepsilon$ and
$|\nabla^\sharp h|_{g^\sharp}\le\varepsilon/\mathfrak r$. The Killing field $X_v$ of $v\in\mathfrak t$ is the same for $g'$ and $g^\sharp$, by equivariance, so

\begin{equation}
|G'(v,v)-G^\sharp(v,v)|=|h(X_v,X_v)|\le\varepsilon|X_v|^2_{g^\sharp}=\varepsilon G^\sharp(v,v).
\end{equation}
The orbit-space metric of a tangent vector is the
infimum of the metric over its lifts, which are the same for $g'$ and $g^\sharp$, and $(1-\varepsilon)g^\sharp\le g'\le(1+\varepsilon)g^\sharp$; this gives the
bounds for $\hat g'$. The eigenvalues $\lambda_1,\lambda_2$ of $(G^\sharp)^{-1}G'$ lie in $[1-\varepsilon,1+\varepsilon]$, $\rho'/\rho^\sharp=(\lambda_1\lambda_2)^{1/2}$, and by
\eqref{eq:dist} $\cosh d(\Phi',\Phi^\sharp)=\frac12\bigl(t+t^{-1}\bigr)$ with $t=(\lambda_1/\lambda_2)^{1/2}$; so $d=|\log t|\le2\varepsilon$ for $\varepsilon\le\frac12$.

(ii) We first collect bounds that concern the model metric $g^\sharp$ alone. Work in the frame $(a,b)$ of the model and put $\omega=X_a\wedge X_b$, so that $\rho^{\sharp2}=|\omega|^2_{g^\sharp}$ and
$\rho'^2=|\omega|^2_{g'}$. We claim that
\begin{equation}\label{eq:modelbound}
\rho^\sharp/\mathfrak r+|\nabla^\sharp\omega|_{g^\sharp}\le C|d\rho^\sharp|_{\hat g^\sharp}.
\end{equation}
Indeed $|d\rho^\sharp|_{\hat g^\sharp}=e^{-\nu^\sharp}=H^{-1/2}$, which is $\asymp\mathfrak r$ for ALE and $\asymp1$ otherwise, and
$\rho^\sharp/\mathfrak r\le R^\sharp/\mathfrak r\le C\mathfrak r^{\kappa-1}$. It remains to bound $\nabla^\sharp\omega$ by Killing fields. On $\R^4$ and on
$\R^3\times S^1$ the torus is generated by rotations and translations of the flat metric, so $|X_v|\le C\mathfrak r$ and $\nabla^\sharp X_v$ is
constant. On (anti-)Taub--NUT, a Gibbons--Hawking metric $Hh_{\R^3}+H^{-1}\eta^2$ with $d\eta=\star_{\R^3}dH$, the fibre field $\xi$ has
$|\xi|=H^{-1/2}\le1$ and $|\nabla^\sharp\xi|=\frac12|d(H^{-1}\eta)|\le C\mathfrak r^{-2}$. The lift $X$ of a rotation of $\R^3$ has $|X|\le C\mathfrak r$;
since $|\mathrm{Rm}|\le C\mathfrak r^{-3}$ and $|\nabla^\sharp\nabla^\sharp X|\le|\mathrm{Rm}||X|$ for a Killing field, $|\nabla^\sharp\nabla^\sharp X|\le C\mathfrak r^{-2}$,
and integration along rays gives $|\nabla^\sharp X|\le C$. In each case $\omega$ is a constant multiple of $X\wedge Y$ for the two
generators $X,Y$ just described (two rotations; a rotation and a translation; the lift of a rotation and the fibre field), and
$|\nabla^\sharp(X\wedge Y)|\le|\nabla^\sharp X||Y|+|X||\nabla^\sharp Y|$, which is $\le C\mathfrak r$ for ALE and $\le C$ otherwise.

Next we compare the derivatives of the Weyl radii $\rho'$ and $\rho^\sharp$. Let $\mathsf k=\Lambda^2g'-\Lambda^2g^\sharp$, the difference of the induced metrics on bivectors; by (i)
$|\mathsf k|\le C\varepsilon$ and $|\nabla^\sharp\mathsf k|\le C\varepsilon/\mathfrak r$. Then $\rho'^2-\rho^{\sharp2}=\mathsf k(\omega,\omega)$, and
$d\bigl(\mathsf k(\omega,\omega)\bigr)=(\nabla^\sharp\mathsf k)(\omega,\omega)+2\mathsf k(\nabla^\sharp\omega,\omega)$. By \eqref{eq:modelbound},
$|d(\rho'^2-\rho^{\sharp2})|\le C\varepsilon\rho^\sharp\bigl(\rho^\sharp/\mathfrak r+|\nabla^\sharp\omega|\bigr)\le C\varepsilon\rho^\sharp|d\rho^\sharp|$. Since
$d\rho'-d\rho^\sharp=d(\rho'^2-\rho^{\sharp2})/(2\rho')+(\rho^\sharp/\rho'-1)\,d\rho^\sharp$ and $\rho'\ge(1-\varepsilon)\rho^\sharp$,
$|d\rho'-d\rho^\sharp|_{\hat g^\sharp}\le C\varepsilon|d\rho^\sharp|_{\hat g^\sharp}$. Here $\rho'$ and $\rho^\sharp$ are functions on the orbit space, whose norms
for $g^\sharp$ and $\hat g^\sharp$ agree.

Finally we compare the harmonic conjugates $z'$ and $z^\sharp$ of $\rho'$ and $\rho^\sharp$. On a surface the Hodge star on $1$-forms depends only on the conformal class, so by (i)
$|(\star_{\hat g'}-\star_{\hat g^\sharp})\beta|_{\hat g^\sharp}\le C\varepsilon|\beta|_{\hat g^\sharp}$. As $F$ preserves the orientation, $dz'=\star_{\hat g'}d\rho'$ and
$dz^\sharp=\star_{\hat g^\sharp}d\rho^\sharp$ for one orientation of the orbit space, and $\{R^\sharp>R_0\}$ is simply connected, so $z'$ exists there and

\begin{equation}
|dz'-dz^\sharp|_{\hat g^\sharp}\le|(\star_{\hat g'}-\star_{\hat g^\sharp})d\rho'|+|d\rho'-d\rho^\sharp|\le C\varepsilon|d\rho^\sharp|_{\hat g^\sharp}.
\end{equation}
For a $1$-form, $|\beta|_{\rm Eucl}=e^{\nu^\sharp}|\beta|_{\hat g^\sharp}$ and $e^{\nu^\sharp}|d\rho^\sharp|_{\hat g^\sharp}=1$, which gives the Euclidean bounds.
\end{proof}

\begin{lemma}[global Weyl coordinates]\label{lem:globalweyl}
Let $(M,g)$ be a complete, simply connected toric Ricci-flat manifold whose action has a fixed point and whose end is of
one of the types of Definition~\ref{def:asymptype}. Then $\rho$ has no critical points in $\mathcal B^\circ$, and $w=z+i\rho$ is a
diffeomorphism of $\mathcal B^\circ$ onto $\Hb^\circ$, which extends to a homeomorphism of $\mathcal B$ onto $\{\rho\ge0\}$ mapping $\partial \mathcal B$
onto the axis.
\end{lemma}
\begin{proof}
This is an argument of Weinstein's type \cite{Weinstein1990}, \cite[\S2]{KL2021}, \cite[Prop.~2.5]{KL2026}, carried out with
Lemma~\ref{lem:endcompare}. The statement does not depend on the basis of $\mathfrak t$ (\S\ref{subsec:WP}), so we compute $G$ and $\rho$
in a $\Z$-basis of the lattice of $2\pi$-periodic elements. The orbit space $\mathcal B$ is a connected, simply connected surface with
boundary and corners (\S\ref{subsec:WP}), and $\partial\mathcal B\ne\varnothing$ since there is a fixed point. By orthogonal transitivity $\rho$ is harmonic and positive on $\mathcal B^\circ$ and zero on $\partial\mathcal B$, so it has a harmonic conjugate $z$
on $\mathcal B^\circ$, with the sign of \S\ref{subsec:WP}, and $w=z+i\rho$ is holomorphic for the conformal structure of $\hat g$.

We first consider $w$ at the axis. By Lemma~\ref{lem:axis}(iii), $w$ extends continuously to $\mathcal B$, and it maps a neighbourhood of each point
of $\partial\mathcal B$ homeomorphically onto a neighbourhood of a point of $\R$ in $\{\rho\ge0\}$.

On the end we compare $w$ with the Weyl coordinate of the model. Pass to the finite cover in the ALE case, as in
Lemma~\ref{lem:endcompare}. By Lemma~\ref{lem:endcompare}(ii), in the coordinates $(\rho^\sharp,z^\sharp)$ of the model, $|d(w-w^\sharp)|\le C(R^\sharp)^{-\delta/\kappa}$ with
$w^\sharp=z^\sharp+i\rho^\sharp$; we may assume $\delta<1$. Integrating along a half-circle of radius $R_0$ and a radial segment, after
adding a real constant to $z$, $|w-w^\sharp|\le C\bigl(1+(R^\sharp)^{1-\delta/\kappa}\bigr)$.

We now count the preimages under $w$ of a point of the open half-plane. Let $R_1>2R_0$ be so large that $|w-w^\sharp|<R_1/4$ on the arc $\gamma_1=\{R^\sharp=R_1\}$, and let $\mathcal B_1$
be the complement in $\mathcal B$ of $\{R^\sharp>R_1\}$. It is compact, since the complement of the end is compact, and it is simply
connected by van Kampen's theorem, since $\mathcal B$ is simply connected and $\mathcal B_1$ meets the closed half-annulus $\{R^\sharp\ge R_1\}$ in
the arc $\gamma_1$; so it is a closed disc. So $\partial\mathcal B_1$
is one closed curve, made of $\gamma_1$ and an arc of $\partial\mathcal B$, on which $w$ is real. On $\gamma_1$,
$w^\sharp$ runs once over the half-circle of radius $R_1$, and $w$ and $w^\sharp$ are real at its endpoints, where $w^\sharp=\pm R_1$ and
so $w$ has modulus at least $3R_1/4$ and the sign of $w^\sharp$, since $|w-w^\sharp|<R_1/4$. Let $\operatorname{Im}q>0$ and
$|q|<R_1/2$. Then $w\ne q$ on $\partial\mathcal B_1$. On $\gamma_1$ the segments from $w^\sharp$ to $w$ lie in $\{|\cdot|>3R_1/4\}$, so the curve
$w(\partial\mathcal B_1)$ is homotopic in $\C\setminus\{q\}$ to the half-circle of radius $R_1$ closed by a path in $\R$, and its winding number
about $q$ is $\pm1$. Since $w\ne q$ on the compact set $\partial\mathcal B_1$, $w\ne q$ on a collar of $\partial\mathcal B_1$ in $\mathcal B_1$. So the
solutions of $w=q$ in $\mathcal B_1^\circ$ lie in a compact set and are finitely many, and since $w$ is holomorphic on $\mathcal B_1^\circ$, the
argument principle on a smooth closed curve in the collar, homotopic there to $\partial\mathcal B_1$, shows that their number, counted with
multiplicity, is the winding number. It is non-negative, so it is $1$: the equation $w=q$ has exactly one solution in $\mathcal B_1^\circ$,
and it is simple.

Every point of $\mathcal B^\circ$ lies in such a $\mathcal B_1^\circ$ with $|w|<R_1/2$. So $w\colon\mathcal B^\circ\to\Hb^\circ$ is bijective with $dw\ne0$, a
biholomorphism, and $\rho$ has no critical points. The extension $w\colon\mathcal B\to\{\rho\ge0\}$ is continuous and proper, since
$|w|\ge R^\sharp-|w-w^\sharp|\to\infty$ on the end; so its image is closed and contains $\Hb^\circ$, hence is $\{\rho\ge0\}$. It is injective on
$\partial\mathcal B$: two boundary points with the same image would give, by the behaviour of $w$ at the axis, two preimages of nearby points of $\Hb^\circ$.
A continuous proper bijection onto $\{\rho\ge0\}$ is a homeomorphism.
\end{proof}

\begin{lemma}[smooth ends]\label{lem:smoothend}
Let $(M,g)$ be a toric Ricci-flat manifold with global Weyl coordinates, whose end is of one of the types of
Definition~\ref{def:asymptype}. Then \ref{it:T2} holds: in the basis of $\mathfrak t$ given by the model, $d(\Phi,\Phi^\sharp)=O(R^{-\delta/\kappa})$
uniformly in $\theta$, where $\Phi^\sharp=\Phi^\sharp_{h_0,\alpha}$ is the model map of the type, $\kappa=2$ for ALE and
$\kappa=1$ otherwise.
\end{lemma}
\begin{proof}
Use Lemma~\ref{lem:endcompare} and its notation, on the finite cover in the ALE case. On the end, $(\rho',z')$ are the global Weyl coordinates, after adding a constant to $z'$. Integrating
Lemma~\ref{lem:endcompare}(ii) as in the proof of Lemma~\ref{lem:globalweyl}, $|w'-w^\sharp|\le C(1+(R^\sharp)^{1-\delta/\kappa})$ for $w'=z'+i\rho'$;
we may assume $\delta<1$. For a point $q$ of the end let $p=(\rho',z')(q)$ and $\hat p=(\rho^\sharp,z^\sharp)(q)$ be its two Weyl
points. Then $|p-\hat p|\le C(1+R^{1-\delta/\kappa})$ and $\rho(p)/\rho(\hat p)=1+O(R^{-\delta/\kappa})$, where $\rho(p)$ denotes the first coordinate
and $R=|p|\asymp R^\sharp$. We compare $\Phi(p)$ with $\Phi^\sharp(\hat p)$ and then $\Phi^\sharp(\hat p)$ with $\Phi^\sharp(p)$.

The first two are the values at the same point $q$ of the two maps: $\Phi(p)=\Phi'(q)$ and $\Phi^\sharp(\hat p)=\Phi^\sharp(q)$, so $d(\Phi(p),\Phi^\sharp(\hat p))\le2\varepsilon=O(R^{-\delta/\kappa})$ by
Lemma~\ref{lem:endcompare}(i).

For the last two we use that the model map is Lipschitz in the right sense. In the coordinates $(\mathfrak x,\mathfrak y)=(\rho H,A)$ of \eqref{eq:PhiHA}, a direct
computation from \eqref{eq:models} gives $|\partial_z\Phi^\sharp|\le C/R$ and $|\partial_\rho\Phi^\sharp|\le1/\rho+C/R$, where the term $1/\rho$
comes only from $\partial_\rho\mathfrak x/\mathfrak x$. On the segment from $\hat p$ to $p$ the first coordinate is positive and monotone, and
$R\ge R(\hat p)/2$, so integrating along it,
\begin{equation} d(\Phi^\sharp(p),\Phi^\sharp(\hat p))\le|\log(\rho(p)/\rho(\hat p))|+C|p-\hat p|/R=O(R^{-\delta/\kappa}). \end{equation}
Hence $d(\Phi(p),\Phi^\sharp(p))=O(R^{-\delta/\kappa})$, and as $q$ runs through the end, $p$ runs through a neighbourhood of infinity,
since the complement of the end in $\mathcal B$ is compact and $(\rho,z)$ is a global chart.
\end{proof}

\begin{proposition}\label{prop:smoothCtor}
Let $(M,g)$ be a complete, simply connected toric Ricci-flat manifold whose action has a fixed point and whose end is of
type ALE, ALF$^\pm$ or AF$_\beta$ in the sense of Definition~\ref{def:asymptype}. Then $g$ has global Weyl coordinates, its
augmented harmonic map lies in $\Ctor$, with the rod structure given by the fixed points and the stabilisers, and the
normalised rod vectors $\tilde v_j$ are the primitive generators of the stabilisers.
\end{proposition}
\begin{proof}
Global Weyl coordinates are Lemma~\ref{lem:globalweyl}. We compute $G$, $\rho$ and $z$ in a $\Z$-basis of the lattice $\Lambda$ of
$2\pi$-periodic elements, which is no restriction (Definition~\ref{def:Ctor}). The turning points are the images of the fixed points,
and the rod lines are the Lie algebras of the stabilisers. This is a rod structure: fixed points are isolated, being corners of
$\mathcal B$ (Lemma~\ref{lem:axis}(i)), and the end contains none (Remark~\ref{rem:asymptypes}), so there are finitely many,
$n\ge1$ of them; their images $w_1<\dots<w_n$ are distinct, since $w$ is injective on $\mathcal B$ (Lemma~\ref{lem:globalweyl}); and by
Lemma~\ref{lem:axis}(ii) the stabiliser is locally constant on $\partial\mathcal B$ minus the corners, hence constant on each rod,
which is connected. Near a point of a rod $I_j$, Lemma~\ref{lem:axis}(ii) gives the rod form
\eqref{eq:Grod} with $a,b,f$ smooth and $f>0$, so Lemma~\ref{lem:formtame}(i) applies. At a turning point $w_k$, Lemma~\ref{lem:axis}(i)
gives a $\Z$-basis $(v,v')$ of $\Lambda$, so $\det(v,v')=\pm1$, in which $G$ has the form \eqref{eq:nutmodel} at $w_k$ with $A(0)>0$; and $v$
and $v'$ span $L_{k-1}$ and $L_k$, since their Killing fields vanish over the rods below and above $w_k$. So $L_{k-1}\ne L_k$, and
Lemma~\ref{lem:formtame}(ii) applies. This is \ref{it:T1}, and \ref{it:T2} is Lemma~\ref{lem:smoothend}. By
Lemma~\ref{lem:axis}(ii) and \eqref{eq:Nrod} the normalised rod vectors exist; since $g$ is smooth along the surface
over $I_j$ for the lattice of $2\pi$-periodic elements, the cone angle of the primitive generator is $2\pi$, so it is
$\pm\tilde v_j$.
\end{proof}

The type of the end need not be assumed. This rests on the work of Li and Sun on toric Einstein metrics \cite{LS2026b}.
\begin{proposition}[the end]\label{prop:endtype}
Let $(M,g)$ be a simply connected toric gravitational instanton whose torus action has finitely many fixed points, and at
least one. Then:
\begin{enumerate}[leftmargin=*,label=\textup{(\roman*)}]
\item the Gaussian curvature $\mathcal K$ of the orbit space $(\mathcal B,\hat g)$ is non-positive;
\item $g$ has global Weyl coordinates; more precisely, $(\rho,z)$, with $z$ a harmonic conjugate of $\rho$, is a homeomorphism of
$\mathcal B$ onto $\{\rho\ge0\}$ and a conformal diffeomorphism of $\mathcal B^\circ$ onto $\Hb^\circ$;
\item the end of $M$ is of type ALE, ALF$^\pm$ or AF$_\beta$.
\end{enumerate}
\end{proposition}
\begin{proof}
Since the action has a fixed point, it is orthogonally transitive and $\partial \mathcal B\ne\varnothing$, and $\mathcal B$ is a connected surface
with boundary and corners which, like $\mathcal B^\circ$, is simply connected, since $M$ is (\S\ref{subsec:WP}). We compute $G$ and $\rho$ in a $\Z$-basis of the
lattice of $2\pi$-periodic elements; nothing below depends on the basis.

(i) This is the case $\lambda=0$ of the curvature bound of Li and Sun for toric Einstein metrics \cite[Prop.~3.2]{LS2026b}, as
they observe in \cite[Lemma~A.6]{LS2026b}. That lemma is stated for metrics with global Weyl coordinates, which its proof does
not use; we give the argument in full, following the proof of \cite[Prop.~3.2]{LS2026b}.

We first show that a positive supremum of $\mathcal K$ would be attained. Over $\mathcal B^\circ$ the orbits are principal and the projection is a Riemannian submersion with integrable horizontal
distribution, so by O'Neill's formula $\mathcal K$ is the sectional curvature of $g$ on the horizontal planes. Let $\operatorname{dist}_o$ be the distance in
$\mathcal B$ from the image of a base point $o$. It is at most the distance in $M$ from $o$, so $|\mathcal K|\le|\mathrm{Rm}_g|=O(\operatorname{dist}_o^{-2})$. Closed
bounded subsets of $\mathcal B$ are compact, since $M$ is complete and $T$ is compact. So if $\sup \mathcal K>0$, it is a maximum $\mathcal K_0>0$, attained
at a point $p$.

Next we recall the differential inequality for $\mathcal K$. On $\mathcal B^\circ$ let $u=\log\rho$, $U=\frac12|du|^2_{\hat g}$, $Q_{ij}=\frac14\tr(\Phi^{-1}\partial_i\Phi\,\Phi^{-1}\partial_j\Phi)$, a non-negative
form, $e=\frac12\tr_{\hat g}Q$ and $\det Q=\det(Q_{ij})/\det(\hat g_{ij})\ge0$. The equation $\operatorname{Ric}=0$ gives $e=\mathcal K+\frac12U$ (the trace over the orbit space of
\cite[Prop.~2.5]{LS2026b}, with $\tr\bigl((G^{-1}dG)^2\bigr)=2|du|^2+\tr\bigl((\Phi^{-1}d\Phi)^2\bigr)$), and, wherever $e>0$, the pointwise
inequality
\begin{equation}\label{eq:Kineq}
\Delta_{\hat g}\mathcal K+\frac{3U}{2e}\langle du,d\mathcal K\rangle_{\hat g}\ \ge\ \frac{|d\mathcal K|^2}{e}+\frac{3\mathcal K^2(2U+\mathcal K)}{e}+3\det Q
\end{equation}
\cite[Lemma~3.1]{LS2026b}, with $\lambda=0$. Both are local consequences of $\operatorname{Ric}=0$ on $\mathcal B^\circ$, with no global
hypothesis: these two facts, and \cite[Lemma~A.17, Thm.~A.4]{LS2026b} in (ii) and (iii), are what this proposition takes from
\cite{LS2026b}.

A maximum $\mathcal K_0>0$ cannot be attained at an interior point: if $p\in \mathcal B^\circ$, then $e(p)\ge \mathcal K_0>0$, $d\mathcal K(p)=0$ and $\Delta_{\hat g}\mathcal K(p)\le0$, while the right-hand side of
\eqref{eq:Kineq} is at least $3\mathcal K_0^3/e(p)>0$.

Nor can it be attained at a boundary point. Let $p\in\partial \mathcal B$. By Lemma~\ref{lem:axis}, near $p$, $\mathcal B$ is $S/\langle\sigma\rangle$, resp.\
$S/\langle\sigma_1,\sigma_2\rangle$, if $p$ lies on a rod, resp.\ is a turning point, and $\rho=xh$, resp.\ $\rho=xyh$, with $h$ smooth, even and
positive, where $(x,y)=(x_1,x_2)$. Since the reflections are isometries, the off-diagonal entry of $\hat g_S$ is odd in $x$, resp.\ in $x$
and in $y$, and the diagonal entries are even; after rescaling $x$ and $y$, $\hat g_S(0)=dx^2+dy^2$, so $\hat g_S^{11}(x,0)=1+O(x^2)$,
which is used in the expansions of $U$ and $e$ below. Then $\mathcal K$ is a smooth function on $S$, even in $x$, resp.\ in $x$ and
in $y$, with a local maximum at $0$, so $d\mathcal K(0)=0$, $\mathcal K_{xx}(0)\le0$ and $\mathcal K_{yy}(0)\le0$; and $u=\log|x|+\varphi$, resp.\
$u=\log|xy|+\varphi$, with $\varphi=\log h$ smooth and even in the same variables. Along $y=0$, $x\to0$, resp.\ along $x=y=t\to0$:
$u_x=1/x+O(x)$, and $u_y=O(1)$, resp.\ $u_y=1/y+O(y)$; so $U=\frac1{2x^2}+O(1)$, resp.\ $U=t^{-2}+O(1)$, $e=\frac1{4x^2}+O(1)$, resp.\
$e=\frac1{2t^2}+O(1)$, and $\frac{3U}{2e}\to3$. Since $\mathcal K_x=x\mathcal K_{xx}(0)+O(x^2)$, and $\mathcal K_y\to0$, resp.\
$\mathcal K_y=y\mathcal K_{yy}(0)+O(y^2)$, we get $\langle du,d\mathcal K\rangle\to\mathcal K_{xx}(0)$, resp.\ $\mathcal K_{xx}(0)+\mathcal K_{yy}(0)$. So the left-hand
side of \eqref{eq:Kineq} tends to $4\mathcal K_{xx}(0)+\mathcal K_{yy}(0)$, resp.\ $4\bigl(\mathcal K_{xx}(0)+\mathcal K_{yy}(0)\bigr)$, which is $\le0$.
In both cases $e>0$ near $p$, so \eqref{eq:Kineq} holds there. Since $\mathcal K>0$ near $p$, its right-hand side is at least
$6\mathcal K^2U/e$, and $2U/e\to4$, so it is at least $12\mathcal K_0^2-o(1)>0$, which contradicts the limits. Hence $\mathcal K\le0$.

(ii) By (i), the properness criterion of Li and Sun \cite[Lemma~A.17]{LS2026b} applies: $M$ is complete, the action is
orthogonally transitive, $\mathcal B^\circ$ is simply connected, $\partial \mathcal B\ne\varnothing$ and $\mathcal K\le0$.

(iii) By (ii), $(M,g)$ is a proper toric gravitational instanton in the sense of \cite[Defs.~A.1, A.2]{LS2026b}, with finitely
many turning points. By \cite[Thm.~A.4]{LS2026b} its end is of one of the asymptotic types of \cite[Def.~4.30]{LS2025}, which are
those of Definition~\ref{def:asymptype}.
\end{proof}

\begin{remark}[the fundamental group]\label{rem:pi1}
Let $(M,g)$ be a toric Ricci-flat manifold whose action has a fixed point $p$, with orbit space $\mathcal B$ and orbit map $\pi$.
Then $\pi_*\colon\pi_1(M)\to\pi_1(\mathcal B)$ is an isomorphism. It is onto by \cite[Ch.~II, Cor.~6.3]{Bredon1972}. For injectivity,
let $\alpha$ be a loop at a point $x$ of a principal orbit such that $\pi\circ\alpha$ is null-homotopic in $\mathcal B$. The
non-principal orbits form a set of codimension $2$, so $\alpha$ can be moved into the union $M^\circ$ of the principal orbits, a
principal $T$-bundle over $\mathcal B^\circ$. This bundle is trivial, since $\mathcal B^\circ$, a non-compact surface, is homotopy
equivalent to a graph, and the inclusion $\mathcal B^\circ\subset\mathcal B$ is a homotopy equivalence; so $\alpha$ is homotopic in
$M^\circ$ to a loop in the orbit of $x$. Such a loop is null-homotopic in $M$: if $\gamma$ is a path from $x$ to $p$ and $\ell$ a loop in
$T$, the loops $t\mapsto\ell(t)\cdot\gamma(s)$ deform it to the constant loop at $p$. Since $\mathcal B$ is a non-compact surface with
non-empty boundary, $\pi_1(\mathcal B)$ is free. Hence $M$ is simply connected if and only if $\pi_1(M)$ is finite, if and only if
$\mathcal B$ is simply connected; for a toric gravitational instanton with finitely many fixed points these hold if and only if
$\mathcal B$ is the closed half-plane of global Weyl coordinates (Proposition~\ref{prop:endtype}), and otherwise $\pi_1(M)$ is free of
positive rank and $\star d\rho$ has periods, as for the instantons on $S^4\setminus T^2$ of \S\ref{subsec:WP}. So
Theorem~\ref{thm:G} holds with ``simply connected'' replaced by ``with finite fundamental group''. The latter is automatic when the
volume growth is Euclidean, $\operatorname{vol}B(x,r)\ge cr^4$, since $\operatorname{Ric}=0$ \cite{Li1986,Anderson1990}; so the
hypothesis can be dropped for metrics with Euclidean volume growth, in particular for those with an ALE end.
\end{remark}

\subsection{Normalised rays and cells}\label{subsec:cells}
The rod lines of a metric in $\Ctor$ are normalised by the model at infinity. The following lemma converts them into
the cells on which the soliton families are parametrised.
\begin{lemma}[total turning]\label{lem:match}
Let $\alpha\in\R$ and let $\ell_0,\dots,\ell_n\in\R^2$ satisfy $\det(\ell_{j-1},\ell_j)<0$ for all $j$, $\ell_0=(1,\alpha/2)$ and
$\ell_n\in\R(1,-\alpha/2)$. Let $s_j\in(0,\pi)$ be the clockwise angle from $\ell_{j-1}$ to $\ell_j$.
\begin{enumerate}[leftmargin=*,label=\textup{(\roman*)}]
\item There is an integer $r$ with $\sum_js_j=r\pi+2\arctan(\alpha/2)$, and $\ell_n$ is a positive multiple of
$(-1)^r(1,-\alpha/2)$. Here $0\le r\le n-1$ if $\alpha>0$, $1\le r\le n-1$ if $\alpha=0$, and $1\le r\le n$ if $\alpha<0$.
\item Let $\alpha\ne0$ and $D_\alpha=\diag(1,2/\alpha)$. Then $D_\alpha\ell_0=(1,1)$, $D_\alpha\ell_n$ is a positive multiple of
$(-1)^r(1,-1)$, and the steps of $(D_\alpha\ell_j)$ lie in $(0,\pi)$. For $\alpha>0$ they are clockwise with sum
$(r+\frac12)\pi$. For $\alpha<0$ they are counter-clockwise with sum $(r-\frac12)\pi$; after exchanging the two
components they are clockwise with sum $(r-\frac12)\pi$, and $D_\alpha\ell_n$ becomes a positive multiple of
$(-1)^{r-1}(1,-1)$.
\end{enumerate}
\end{lemma}
\begin{proof}
(i) The clockwise angle from $\ell_0$ to $\ell_n$ is $\sum_js_j$ modulo $2\pi$. Now $\ell_0$ has the angle $\varphi=\arctan(\alpha/2)$,
and $(-1)^k(1,-\alpha/2)$ has the angle $-\varphi+k\pi$. So $\sum_js_j\equiv2\varphi+k\pi$ modulo $2\pi$ if $\ell_n$ is a positive
multiple of $(-1)^k(1,-\alpha/2)$. Take $r\equiv k$ modulo $2$ with $\sum_js_j=r\pi+2\varphi$. The bounds follow from
$0<\sum_js_j<n\pi$, since $2\varphi\in(-\pi,\pi)$ has the sign of $\alpha$, and $\varphi=0$ forces $r\ge1$.

(ii) For $c>0$ the vectors $\diag(1,c)\ell_j$ satisfy the hypotheses of (i) with $c\alpha$ in place of $\alpha$, since
$\det(\diag(1,c)\ell_{j-1},\diag(1,c)\ell_j)=c\det(\ell_{j-1},\ell_j)$. Their steps are continuous in $c$ and lie in $(0,\pi)$, so
the integer of (i), $\pi^{-1}(\sum_js_j-2\arctan(c\alpha/2))$, does not depend on $c$, and it is $r$. Take $c=2/|\alpha|$: the clockwise steps
of $m_j=\diag(1,2/|\alpha|)\ell_j$ sum to $r\pi+2\arctan(\pm1)=(r\pm\frac12)\pi$, with the sign of $\alpha$, and $m_n$ is a positive
multiple of $(-1)^r(1,\mp1)$. For $\alpha>0$, $m_j=D_\alpha\ell_j$. For $\alpha<0$, $D_\alpha=\diag(1,-1)\diag(1,2/|\alpha|)$, and $\diag(1,-1)$
reverses the orientation, so the steps of $D_\alpha\ell_j$ are those of $m_j$, counter-clockwise, with sum $(r-\frac12)\pi$, and
$D_\alpha\ell_n$ is a positive multiple of $(-1)^r(1,-1)$. Exchanging the components is the rotation $(x,y)\mapsto(-y,x)$ after
$\diag(1,-1)$, so the exchanged vectors are the rotated vectors $m_j$, with clockwise steps summing to $(r-\frac12)\pi$, and the last
one is a positive multiple of $(-1)^r(-1,1)=(-1)^{r-1}(1,-1)$. In all cases $D_\alpha\ell_0=(1,1)$, which the exchange fixes.
\end{proof}
For a metric in $\Ctor$ with model $\Phi^\sharp_{h_0,\alpha}$ in the basis $(a,b)$, Lemma~\ref{lem:outerlines} shows that its outer
rod lines are $L_0=[(1,\alpha/2)]$ and $L_n=[(1,-\alpha/2)]$. Since consecutive lines differ, there are vectors $\ell_j$ with
$[\ell_j]=L_j$, unique up to positive multiples for $j\ge1$, with $\ell_0=(1,\alpha/2)$ and $\det(\ell_{j-1},\ell_j)<0$ for all $j$:
given $\ell_{j-1}$, exactly one of the two vectors of $L_j$ of any fixed length is reached from it by a clockwise rotation by an
angle in $(0,\pi)$. Then $\ell_n\in\R(1,-\alpha/2)$, so Lemma~\ref{lem:match} applies. We call the rays of the $\ell_j$ the
\emph{normalised rays}, and the integer $r$ of Lemma~\ref{lem:match}(i) the \emph{degree index}. For ALE, $\alpha>0$ and the rays
$D_\alpha\ell_j$ lie in the cell
\begin{equation}\label{eq:cell}
C_r=\Bigl\{(s_1,\dots,s_n)\in(0,\pi)^n:\ \textstyle\sum_js_j=(r+\frac12)\pi\Bigr\},
\end{equation}
recorded by the clockwise steps.
\begin{lemma}[degree and degree index]\label{lem:degreeindex}
Let $(\Phi,\nu)\in\Ctor$ have $n$ turning points and degree index $r$ (above). Then the Li--Sun degrees of its rod structure
and of the reversed one are $\{n-1-r,r\}$ for ALE and $\{n-r,r\}$ for ALF$^\pm$ and AF.
\end{lemma}
\begin{proof}
Let $\ell_j$ be the normalised rays, $\varphi=\arctan(\alpha/2)$, so that $\sum_js_j=r\pi+2\varphi$ (Lemma~\ref{lem:match}). The
lines of $\ell_n$ and $\ell_0$ have the angles $-\varphi$ and $\varphi$ modulo $\pi$. The number $e(\mathfrak R)$ of \S\ref{subsec:rods} is the
counter-clockwise winding number of the closed loop of rod lines in $\RP^1$; each clockwise step $s_j\in(0,\pi)$ between
consecutive lines is a counter-clockwise step $\pi-s_j$. For ALE and ALF$^-$, where $\alpha>0$, the clockwise closing step from
$[\ell_n]$ to $[\ell_0]$ is $\pi-2\varphi$. So the clockwise winding with the closing step is $r+1$, and the counter-clockwise winding is
$e=n+1-(r+1)=n-r$. Hence $d=e-1=n-1-r$ for ALE and $d=e=n-r$ for ALF$^-$. For ALF$^+$, where $\alpha<0$, the clockwise closing step
is $-2\varphi$, the clockwise winding is $r$, $e=n+1-r$ and $d=e-1=n-r$. For AF there is no closing step, the clockwise winding is $r$,
$e=n-r$ and $d=e=n-r$. The degree of the reversed structure follows from Remark~\ref{rem:degid}.
\end{proof}

\section{Belinski--Zakharov solitons}\label{sec:bz}
This section recalls the Belinski--Zakharov method and defines the gravitational solitons (Definition~\ref{def:soliton}) and the two
families used in the proof: the ALE soliton metrics (\S\ref{subsec:ALEsolitons}) and the solitons on $\R^3\times S^1$ (\S\ref{subsec:MKN}).

\subsection{Dressing}\label{subsec:dressing}
We recall the inverse scattering method of Belinski and Zakharov in the form needed here; see \cite[\S\S1.3--1.4 and 8.1--8.2]{BV2001} for an
account. Throughout this section, $G$ denotes a real symmetric $2\times2$ matrix function on $\Hb^\circ$ that solves
\eqref{eq:hm}, a \emph{matrix solution}, not necessarily with $\det G=\pm\rho^2$. Equation \eqref{eq:hm} is the compatibility
condition of the linear system \eqref{eq:lax} below for a matrix function $\Psi(\lambda,\rho,z)$, the \emph{generating matrix}, which
depends on a complex \emph{spectral parameter} $\lambda$ and gives back $G$ as $\Psi(0,\rho,z)$. Given a known solution, the \emph{seed}
$(G_0,\Psi_0)$, one looks for $\Psi=\chi\Psi_0$ with a \emph{dressing matrix} $\chi$ that is rational in $\lambda$, has simple poles and
tends to the identity as $\lambda\to\infty$. Let $p$ be the number of poles of $\chi$, and $\lambda_1,\dots,\lambda_p$ the poles,
functions of $(\rho,z)$; the index $k$ below always runs over $1,\dots,p$.

Let $U=\rho\,\partial_\rho G\,G^{-1}$, $V=\rho\,\partial_zG\,G^{-1}$, $D_1=\partial_z-\frac{2\lambda^2}{\lambda^2+\rho^2}\partial_\lambda$ and
$D_2=\partial_\rho+\frac{2\lambda\rho}{\lambda^2+\rho^2}\partial_\lambda$. The linear system
\begin{equation}\label{eq:lax}
D_1\Psi=\frac{\rho V-\lambda U}{\lambda^2+\rho^2}\Psi,\qquad D_2\Psi=\frac{\rho U+\lambda V}{\lambda^2+\rho^2}\Psi
\end{equation}
is compatible if and only if $G$ solves \eqref{eq:hm}, and then one may take $\Psi(0,\rho,z)=G$ \cite{BZ1978,BZ1979},
\cite[(8.6)--(8.8)]{BV2001}. For a point $w$ of the axis put
\begin{equation}\label{eq:poles}
\mu_w=\sqrt{\rho^2+(z-w)^2}-(z-w),\qquad \mub_w=-\rho^2/\mu_w ,
\end{equation}
the two roots of $X^2+2(z-w)X-\rho^2=0$, the \emph{soliton} and \emph{anti-soliton} pole trajectories through $w$; $\mu_w$ vanishes
on the axis above $w$, and $\mub_w$ below it (Figure~\ref{fig:readd}). The linear system forces each pole to be of this form: given
a seed $(G_0,\Psi_0)$, poles $\lambda_k\in\{\mu_{w_k},\mub_{w_k}\}$ and constant row vectors $m^{(k)}_0$ (the \emph{BZ vectors}),
the dressed matrix is
\begin{equation}\label{eq:dressed}
G=G_0-\sum_{k,l}\frac{G_0m^{(k)T}\,(\Gamma^{-1})_{kl}\,m^{(l)}G_0}{\lambda_k\lambda_l},\qquad
\Gamma_{kl}=\frac{m^{(k)}G_0m^{(l)T}}{\rho^2+\lambda_k\lambda_l},\qquad m^{(k)}=m^{(k)}_0\Psi_0(\lambda_k)^{-1}.
\end{equation}
It solves \eqref{eq:hm} where $\det\Gamma\ne0$; this is the $n$-soliton theorem of Belinski and Zakharov \cite{BZ1978}, see
\cite[\S8.2, (8.19)--(8.26)]{BV2001} for a proof. Together with Theorem~\ref{thm:BZnu} below, it is the only input from the inverse scattering
literature that this paper uses.
A matrix dressed with $p$ poles has \cite[(8.27)]{BV2001}
\begin{equation}\label{eq:detdressed}
\det G=(-1)^p\rho^{2p}\prod_k\lambda_k^{-2}\,\det G_0 ,
\end{equation}
so it no longer satisfies $\det G=\pm\rho^2$; this is corrected by the normalisation of Definition~\ref{def:soliton}(d), or
avoided by Pomeransky's scheme below.

\emph{Diagonal solutions.} For a diagonal matrix $G=\diag(g_1,g_2)$, $G^{-1}\partial G$ is diagonal with the entries $\partial\log|g_1|$
and $\partial\log|g_2|$, so $G$ solves \eqref{eq:hm} if and only if $\log|g_1|$ and $\log|g_2|$ are \emph{axisymmetric harmonic}, that is, solve
$\partial_\rho(\rho\,\partial_\rho u)+\partial_z(\rho\,\partial_zu)=0$; these are the harmonic functions on $\R^3$ minus the axis that depend only on
$(\rho,z)$. The functions $\log\rho$ and $\log\mu_w$ are axisymmetric harmonic, by a direct computation, and hence so is
$\log|\mub_w|=2\log\rho-\log\mu_w$.

\emph{Removing and re-adding poles.} Pomeransky's scheme \cite{Pomeransky2006} produces solutions with $\det G=\pm\rho^2$ without a
normalisation. It starts from a diagonal seed in which, for each pole $\lambda_k$ to be used, the other root $\lambda'_k=-\rho^2/\lambda_k$
of \eqref{eq:poles} occurs as a factor of a diagonal entry, and proceeds in two steps.
\begin{enumerate}[leftmargin=*,label=\textup{(\arabic*)}]
\item \emph{Removing the pole $\lambda_k$}: replace the factor $\lambda'_k$ of the diagonal entry by $\lambda_k$. This changes the
logarithm of the entry by $\log|\lambda_k|-\log|\lambda'_k|=2\log|\lambda_k|-2\log\rho$, which is axisymmetric harmonic, so the result
is again a diagonal solution. It multiplies the determinant by $\lambda_k/\lambda'_k=-\lambda_k^2/\rho^2$; the entry changes sign,
since $\lambda_k$ and $\lambda'_k$ have opposite signs, and its zero on the axis moves from one side of $w_k$ to the other
(Figure~\ref{fig:readd}).
\item \emph{Re-adding the poles}: dress the new seed by \eqref{eq:dressed}, with the same poles $\lambda_k$ and arbitrary BZ vectors
$m^{(k)}_0$. By \eqref{eq:detdressed}, this multiplies the determinant by $\prod_k(-\rho^2/\lambda_k^2)$.
\end{enumerate}
The poles are called re-added because the dressing puts back exactly the poles taken out in step (1), and the two factors
of the determinant cancel; so the scheme keeps $\det G=\pm\rho^2$. Dressing with poles that were not removed first would change
the determinant, and the result would have to be normalised.

The case of one node shows what the two steps do. Let
\begin{equation}\label{eq:onenodeseed}
G_0=\diag(\mu_w,\rho^2/\mu_w),
\end{equation}
the flat block of $\R^4$ with its turning point at $w$: the first entry vanishes on the axis above $w$, the second below it. The
pole to be used is $\lambda=\mub_w$, whose other root $\mu_w$ is the first entry. Removing it replaces $\mu_w$ by $\mub_w=-\rho^2/\mu_w$,
which gives $G_0'=(\rho^2/\mu_w)\diag(-1,1)$, an indefinite block whose two entries both vanish on the axis below $w$; the turning
point is gone, and $\det G_0'=-\rho^4/\mu_w^2$. Re-adding $\mub_w$ with a vector $m^{(1)}=(a,b)$ in \eqref{eq:dressed}, for which
$\Gamma=-\mu_w(a^2-b^2)/(\rho^2+\mu_w^2)$, gives
\begin{equation}\label{eq:onenode}
G=\frac{\mu_w\,uu^T+(\rho^2/\mu_w)\,vv^T}{a^2-b^2},\qquad u=\begin{pmatrix}a\\-b\end{pmatrix},\quad v=\begin{pmatrix}b\\-a\end{pmatrix},
\qquad \det G=\rho^2 .
\end{equation}
For $a^2\ne b^2$ this is $\pm G_0$ written in the constant frame $(u,v)/|a^2-b^2|^{1/2}$: the one-node family consists of the
flat block alone, in all frames of determinant one. For $a=\pm b$ the vector is null for $G_0'$, $\Gamma=0$, and the dressing is
undefined; with more nodes these are the directions of the phantom solitons of \S\ref{subsec:ALEsolitons}. With several nodes the
same happens pole by pole: the trivial BZ vectors give back the seed up to a constant factor and a constant change of frame (for the
ALE seed, Lemma~\ref{lem:Cbase} and \S\ref{subsec:CF}), and the general ones give the soliton families of
\S\S\ref{subsec:ALEsolitons} and \ref{subsec:MKN}.
\begin{figure}[ht]
\centering
\begin{tikzpicture}[x=1cm,y=1cm]
\fill[gray!12] (-4.2,0) rectangle (4.2,1.5);
\node at (-2.9,1.1) {$\rho>0$};
\draw[->] (-4.2,0) -- (4.5,0) node[right] {$z$};
\draw[line width=1.6pt] (0.4,0) -- (4.2,0);
\draw[line width=1.6pt,dash pattern=on 4pt off 3pt] (-4.2,0) -- (0.4,0);
\fill (0.4,0) circle (2pt);
\node[above] at (0.4,0.06) {$w$};
\node[below,font=\small] at (2.5,-0.08) {$\mu_w=0$, \ $\mub_w=-2(z-w)$};
\node[below,font=\small] at (-2.1,-0.08) {$\mub_w=0$, \ $\mu_w=2(w-z)$};
\node at (2.3,0.85) {$\mu_w\,\mub_w=-\rho^2$};
\end{tikzpicture}
\caption{A node $w$ on the axis $\{\rho=0\}$ of the half-plane, and the two pole trajectories \eqref{eq:poles} through it:
$\mu_w$ vanishes on the axis above $w$, $\mub_w$ below it, and their product is $-\rho^2$.}
\label{fig:readd}
\end{figure}
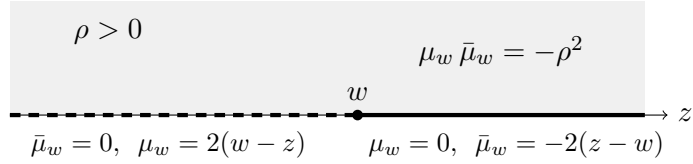

\emph{The conformal factor.} The second input from the inverse scattering literature is the conformal factor of the dressed
metric. We state it in the form in which it is used, with its hypotheses.
\begin{theorem}[the conformal factor; Belinski--Zakharov, Pomeransky]\label{thm:BZnu}
Let $G_0$ be a diagonal solution of \eqref{eq:hm} with $\det G_0=\pm\rho^2$, and let $\nu_0$ be a solution of the quadrature
\eqref{eq:nu} for $G_0$. Let $\lambda_1,\dots,\lambda_p$ be poles whose other roots $\lambda'_k=-\rho^2/\lambda_k$ occur as factors of the
diagonal entries of $G_0$. Remove these poles as in step \textup{(1)} above, and re-add them by \eqref{eq:dressed} with BZ vectors
$m^{(k)}_0$; let $G$ be the result, so that $\det G=\det G_0$, and let $\Gamma$ be its matrix \eqref{eq:dressed}. Let $\Gamma^{(0)}$ be $\Gamma$
for \emph{trivial} BZ vectors, for which the re-added poles give $\kappa T^TG_0T$ with a constant $\kappa\ne0$ and a constant
$T\in GL(2,\R)$. Then, wherever $\det\Gamma\ne0$,
\begin{equation}\label{eq:BZnu}
\nu=\nu_0+\tfrac12\log\bigl|\det\Gamma/\det\Gamma^{(0)}\bigr|
\end{equation}
solves the quadrature \eqref{eq:nu} for $G$.
\end{theorem}
This is the form in which the conformal factor of Pomeransky's scheme is used by Pomeransky \cite[\S II and (23)]{Pomeransky2006}
and stated in general by Elvang and Figueras \cite[\S2.1, (2.14)]{EF2007}. It rests on the formula of Belinski and Zakharov for
the conformal factor of a dressed matrix \cite[\S3, (3.4) and Appendix]{BZ1979}, \cite[\S8.2, (8.31)]{BV2001}: for a seed $G_0'$ with conformal factor $\nu_0'$
dressed with the poles $\lambda_k$, $e^{2\nu}=C_0\,e^{2\nu_0'}F\det\Gamma$, with a constant $C_0$ and a function $F$ of $(\rho,z)$ and the
poles only (Belinski and Zakharov derive this for $\det G_0'=\pm\rho^2$; that it holds verbatim for the collapsed seed, in either signature, is the content of \cite[\S II]{Pomeransky2006}, \cite[\S2.1]{EF2007}). Applied to the collapsed seed $G_0'$ with the trivial and with the general BZ vectors, which use the same poles, the
factor $F$ cancels in the quotient, and the trivial vectors give $\kappa T^TG_0T$, whose solutions of \eqref{eq:nu} are those of $G_0$
(Remark~\ref{rem:nuinv}); this gives \eqref{eq:BZnu} up to an additive constant, which is all that is used. Here the collapsed seed
has $\det G_0'=\pm\rho^2e^{2h}$ with $h=\frac12\sum_k\log|\lambda_k/\lambda'_k|$ axisymmetric harmonic, and for such a matrix \eqref{eq:nu} is
still integrable, since the terms by which its right side differs from that of $e^{-h}G_0'$ form the closed Weyl form
$\frac12dh+\frac\rho4\bigl((h_\rho^2-h_z^2)\,d\rho+2h_\rho h_z\,dz\bigr)$; so $\nu_0'$ is defined. We do not reproduce the proof of
\cite{BZ1979}. Theorem~\ref{thm:BZnu} is used once, in Theorem~\ref{thm:CF}(iv), and it is the only place where the regularity of the
conformal factor of the soliton metrics rests on a cited result.

\emph{Directions of the BZ vectors.} The dressed matrix, the first formula in \eqref{eq:dressed}, depends on the vectors $m^{(k)}$ only
through their directions $[m^{(k)}]\in\RP^1$. Indeed, replacing $m^{(k)}$ by $\tau_km^{(k)}$, with non-zero scalars $\tau_k$, replaces
$\Gamma_{kl}$ by $\tau_k\tau_l\Gamma_{kl}$, hence $(\Gamma^{-1})_{kl}$ by $(\Gamma^{-1})_{kl}/(\tau_k\tau_l)$, so the term of $(k,l)$ in the sum is multiplied by
$\tau_k\tau_l(\tau_k\tau_l)^{-1}=1$. Since $m^{(k)}$ is linear in $m^{(k)}_0$, the same holds for the constant vectors $m^{(k)}_0$. All seeds below are
scalar multiples $f\hat G_0$ of a matrix $\hat G_0$ with an explicit generating matrix $\hat\Psi_0$, and we take $\Psi_0=\varphi\hat\Psi_0$ with
a scalar function $\varphi$. Then $[m^{(k)}]=[m^{(k)}_0\hat\Psi_0(\lambda_k)^{-1}]$, whatever the value of $\varphi$ at $\lambda_k$.
\begin{lemma}[scalar factors]\label{lem:scalar}
Let $G_0=f\hat G_0$ with $f$ a scalar function. Dressing $G_0$ by \eqref{eq:dressed} with poles $\lambda_k$ and directions $[m^{(k)}]$ gives $f$
times the dressing of $\hat G_0$ with the same poles and directions.
\end{lemma}
\begin{proof}
In \eqref{eq:dressed}, the matrix $\Gamma$ of $G_0$ is $f$ times that of $\hat G_0$, and each numerator $G_0m^{(k)T}m^{(l)}G_0$ is $f^2$
times that of $\hat G_0$.
\end{proof}

\subsection{Static seeds and their collapse}\label{subsec:flatseeds}
Fix points $w_1<\dots<w_N$ of the axis, the \emph{nodes}, at which the poles will sit.

\emph{Static seeds.} We call a diagonal solution of \eqref{eq:hm} \emph{static}: its two Killing fields are orthogonal, and in
Lorentzian signature, with one of them timelike, the metric is static. The flat metrics are static: flat $\R^3\times S^1$ has the
Gram matrix $\diag(1,\rho^2)$. By the criterion of \S\ref{subsec:dressing}, multiplying the two diagonal entries of a diagonal
solution by $\mu_w^{\pm1}$ and $\mu_w^{\mp1}$ gives a diagonal solution. This is the operation of \emph{adding a diagonal soliton} at
$w$ (Definition~\ref{def:soliton}(a)); in this way Emparan and Reall build the generalised Weyl solutions, the static toric
solutions, from flat space \cite[\S II, (2.28), and \S III]{ER2002}. Let $\epsilon_k=(-1)^{k+1}$ and $U=\frac12\sum_k\epsilon_k\log\mu_{w_k}$. Adding diagonal
solitons with the alternating signs $\epsilon_k$ at the nodes to $\diag(1,\rho^2)$ gives the static solution
\begin{equation}\label{eq:staticseed}
G_0=\diag\bigl(e^{2U},\rho^2e^{-2U}\bigr),\qquad e^{2U}=\prod_k\mu_{w_k}^{\epsilon_k} .
\end{equation}
As $R\to\infty$, $e^{2U}\sim(R-z)^{\sum\epsilon_k}$:
\begin{itemize}[leftmargin=*]
\item for odd $N$, $\sum\epsilon_k=1$, and $G_0$ is asymptotic to $\diag(R-z,R+z)$, flat $\R^4$ (the \emph{ALE seed});
\item for even $N$, $\sum\epsilon_k=0$, and $G_0$ is asymptotic to $\diag(1,\rho^2)$, flat $\R^3\times S^1$ (the \emph{AF seed});
it is Euclidean multi-Schwarzschild.
\end{itemize}

\emph{Collapse.} Each $\mu_{w_k}$ occurs as a factor of a diagonal entry of $G_0$, so the pole $\mub_{w_k}$ can be removed as in step (1)
of \S\ref{subsec:dressing}. Removing one pole at every node turns $G_0$ into a scalar function times a constant matrix (ALE) or times
the flat block $\diag(1,\rho^2)$ (AF); we call this the \emph{collapse} of the seed. After the collapse all dependence on the nodes sits
in a scalar factor, which passes through the dressing by Lemma~\ref{lem:scalar}. So the poles can be re-added to a constant matrix,
resp.\ to the flat block, which is what makes the formulas of \S\S\ref{subsec:ALEsolitons} and \ref{subsec:MKN} explicit.
\begin{lemma}[collapse of the seed]\label{lem:collapse}
At every node remove a pole: in the diagonal entry in which $\mu_{w_k}$ appears with exponent $+1$, replace that $\mu_{w_k}$ by
$\mub_{w_k}$. The seed becomes
\begin{equation}\label{eq:collapsed}
G_0'=(-1)^m\frac{\rho^{N+1}}{\prod_k\mu_{w_k}}\diag(-1,1)\ \ (N=2m+1),\qquad
G_0'=\frac{(-\rho^2)^m}{\prod_k\mu_{w_k}}\diag(1,\rho^2)\ \ (N=2m). \end{equation}
\end{lemma}
\begin{proof}
The first entry becomes $\prod_{\epsilon_k=+}\mub_{w_k}\prod_{\epsilon_k=-}\mu_{w_k}^{-1}=(-\rho^2)^{n_+}/\prod_k\mu_{w_k}$ and the second
$\rho^2(-\rho^2)^{n_-}/\prod_k\mu_{w_k}$, with $n_\pm=\#\{k:\epsilon_k=\pm1\}$. For $N=2m+1$, $n_+=m+1$ and $n_-=m$, so
$(-\rho^2)^{n_+}=(-1)^{m+1}\rho^{2m+2}$ and $\rho^2(-\rho^2)^{n_-}=(-1)^m\rho^{2m+2}$. For $N=2m$, $n_+=n_-=m$.
\end{proof}
Removing a pole changes the sign of a diagonal entry, so the collapsed ALE seed is indefinite: up to the scalar factor it is the
constant Lorentzian block $\diag(-1,1)$, the seed that Chen uses for an odd number of solitons \cite[(1.5), (3.7)]{Chen2016}. This
is the parity rule of Chen for Euclidean signature: an odd number of soliton transformations turns an indefinite block into a
definite one and an even number preserves the signature, so the odd family of \S\ref{subsec:ALEsolitons} is dressed on the constant
indefinite block and the even family of \S\ref{subsec:MKN} on the definite block $\diag(1,\rho^2)$, Chen's seed (1.4). Positive
definiteness of the resulting Gram matrices is proved directly in Theorems~\ref{thm:CF}(iii), \ref{thm:farclosed}\ref{it:fcreg}
and~\ref{thm:mkn}.

\subsection{Gravitational solitons}\label{subsec:soliton}
\begin{definition}\label{def:soliton}
A toric Ricci-flat metric \eqref{eq:WP}, locally tamed by a rod structure with $n\ge1$ turning points and of asymptotic type ALE,
ALF or AF (Definitions~\ref{def:rodstructure}, \ref{def:tame} and~\ref{def:Ctor}), is a \emph{gravitational soliton} if, in some basis
of $\mathfrak t$, its Gram matrix $G$ coincides on $\Hb^\circ$, up to a constant non-zero factor (determined up to sign by $\det G=\rho^2$), with a
matrix obtained from the flat model of its type
($\R^4$ for ALE, $\R^3\times S^1$ for ALF and AF) by finitely many of the following operations, with $w$ a point of the axis:
\begin{enumerate}[leftmargin=*,label=\textup{(\alph*)}]
\item \emph{adding a diagonal soliton} to a diagonal matrix: multiplying its two diagonal entries by $\mu_w^{\pm1}$ and $\mu_w^{\mp1}$
\cite[\S II]{ER2002};
\item \emph{removing a pole} from a diagonal matrix: replacing a factor $\mu_w$ of a diagonal entry by $\mub_w$, or a factor $\mub_w$ by
$\mu_w$ \cite{Pomeransky2006};
\item \emph{dressing} by \eqref{eq:dressed}, with poles $\mu_w$ or $\mub_w$;
\item the \emph{normalisation} of \cite{BZ1978}, \cite[(8.29)--(8.30)]{BV2001}, which multiplies a matrix dressed with $p$ poles $\lambda_k$ by $\pm\rho^{-p}\prod_k\lambda_k$
and so restores $\det G=\pm\rho^2$.
\end{enumerate}
\end{definition}
The conformal factor is then determined by $G$ up to an additive constant, through \eqref{eq:nu} (by \eqref{eq:BZnu} for the scheme
of removing and re-adding poles), so the condition is one on $G$, or on the harmonic map $\Phi$. The normalisation multiplies $G$ by
$e^h$ with $h=\sum_k\log|\lambda_k|-p\log\rho$, which is axisymmetric harmonic, so it preserves \eqref{eq:hm}. For the ALE family it can
be avoided, since it is built by removing and re-adding poles (for the family on $\R^3\times S^1$ see Remark~\ref{rem:readd}). Only two
normal forms occur in this paper.
In the ALE case, $G$ is a constant times the dressing \eqref{eq:bordered} of the collapsed ALE seed of Lemma~\ref{lem:collapse}, which
arises from flat $\R^4$ by adding a diagonal soliton at each further node (\S\ref{subsec:flatseeds}) and removing a pole at each node; in
the ALF and AF cases, a constant times the normalised dressing \eqref{eq:BZAF} of $\diag(1,\rho^2)$. The cases with the fewest poles are
worked out in \S\ref{subsec:smallest}.

\emph{Phantom solitons.} Poles at which the soliton metric has no turning point are called \emph{phantom solitons}, after Chen
\cite[\S4.4]{Chen2016}. In Chen's solitons on flat space each pole has a real BZ parameter $C_k$; for $C_k=\pm1$ the two rods that meet at
its node have the same rod line and join, although the pole still enters the metric \cite[(4.34), (4.35)]{Chen2016}. Chen and Teo had
used this mechanism, without the name, to construct their AF instanton \cite[p.~361]{CT2011}, \cite[\S2.3]{CT2015}. In the ALE family the
phantom solitons are the poles whose BZ vector is null for the seed, $c_k\in\{0,\infty\}$ (\S\ref{subsec:ALEsolitons}); these are Chen's
$C_k=\pm1$, since for an odd number of poles his seed is $\diag(-1,1)$ and $\Gamma_{kk}\propto1-C_k^2$ \cite[(3.10)]{Chen2016}. In the family on
$\R^3\times S^1$ they are the poles with BZ parameter $b_k=\pm1$ (\S\ref{subsec:MKN}; Theorem~\ref{thm:mkn}; Proposition~\ref{prop:faraxis}\ref{it:fa3}). A phantom soliton
can be moved to any point of the axis off the other nodes, with the other parameters suitably held fixed; this changes $G$ at most by a
constant change of frame and leaves the solutions of \eqref{eq:nu} unchanged (Lemma~\ref{lem:Cnull}(b); Proposition~\ref{prop:faraxis}\ref{it:fa2},
where $F_1F_2$ is unchanged). So its
position is not a parameter of the metric.

The constant factor is harmless: the dressing \eqref{eq:dressed} commutes with constant factors of the seed, and the seeds are
defined only up to such factors (Lemma~\ref{lem:collapse}). The period lattice and the cone angles are not part of the data; they
are determined by the metric. Limits of such metrics are not needed, and are not allowed in this definition.

\subsection{The ALE soliton metrics}\label{subsec:ALEsolitons}
Let $N=2m+1$ and nodes $w_1<\dots<w_N$. Start from the ALE seed, collapsed as in Lemma~\ref{lem:collapse} to $G_0'=f\eta$, with
$f=\rho^{N+1}/\prod_k\mu_{w_k}$ and $\eta=(-1)^m\diag(-1,1)$. Here $G_0'$ is written in the basis of $\mathfrak t$ in which the seed
\eqref{eq:staticseed} is diagonal, and $\eta$ is the matrix of the quadratic form $(-1)^m(-x^2+y^2)$ on $\mathfrak t\cong\R^2$, where
$(x,y)$ are the components of a vector in this basis; the form is indefinite. For $G_0'=f\eta$ the matrices $U=\rho f_\rho/f\cdot\mathrm{Id}$ and $V=\rho f_z/f\cdot\mathrm{Id}$ of \eqref{eq:lax} are scalar, so
\eqref{eq:lax} is a scalar system and a solution with $\Psi_0'(0)=G_0'$ is $\Psi_0'=\varphi\eta$, where
$\varphi(\lambda)=(\rho^2-2(z-z_0)\lambda-\lambda^2)^{m+1}\prod_k(\mu_k-\lambda)^{-1}$ solves the scalar system with $\varphi(0)=f$; we use this one.
Hence the vectors $m^{(k)}=m^{(k)}_0\Psi_0'(\lambda_k)^{-1}$ are multiples of $m^{(k)}_0\eta^{-1}$ by the nowhere-vanishing functions $\varphi(\lambda_k)^{-1}$, which do not depend on $c$.

\emph{The null basis.} A vector $v\in\mathfrak t$ is \emph{null} for $\eta$ if $\eta(v,v)=0$, i.e.\ $x=\pm y$. We use a \emph{null basis}
$(e_0,e_1)$ of $\mathfrak t$, consisting of two null vectors normalised by $\eta(e_0,e_1)=\frac12$. Then $e_0+e_1$ and $e_0-e_1$ are
orthonormal for $\eta$, with $\eta(e_0\pm e_1,e_0\pm e_1)=\pm1$, and we fix the remaining boost freedom $(e_0,e_1)\mapsto(te_0,t^{-1}e_1)$ so that they are the basis in which the seed is diagonal;
and in the null basis
\begin{equation}\label{eq:etahat}
\eta=\hat\eta:=\begin{pmatrix}0&1/2\\1/2&0\end{pmatrix}.
\end{equation}
Write the BZ vectors in the null basis as $m^{(k)}_0=(\beta_k,\alpha_k)$. The ratio $c_k=[\alpha_k:\beta_k]\in\RP^1$, that is, $c_k=\beta_k/\alpha_k$,
with $c_k=\infty$ when $\alpha_k=0$, is the \emph{parameter} of the $k$-th soliton, and $c=(c_1,\dots,c_N)\in(\RP^1)^N$. Throughout, a pair
$[\alpha:\beta]$ denotes the value $\beta/\alpha$, so that the value $\lambda$ is $[1:\lambda]$. The null basis is convenient since in it the matrix $K$ below has entries that are
linear and symmetric in each pair $(\alpha_k,\beta_k)$; the parameters $c_k\in\{0,\infty\}$, the BZ vectors along the null directions, are
the phantom solitons (\S\ref{subsec:soliton}).

Re-add the anti-soliton poles $\lambda_k=\mub_k:=\mub_{w_k}$. Then \eqref{eq:dressed} becomes, with $\mu_k=\mu_{w_k}$,
\begin{equation}\label{eq:bordered}
G_{ab}=f\Bigl(\hat\eta_{ab}-\frac14\sum_{k,l}\frac{\hat c_{a,k}\,(K^{-1})_{kl}\,\hat c_{b,l}}{\mub_k\mub_l}\Bigr),\qquad
K_{kl}=\frac{\alpha_k\beta_l+\beta_k\alpha_l}{2(\rho^2+\mub_k\mub_l)},
\end{equation}
where $a,b\in\{0,1\}$ are indices for the null basis, and $\hat c_0=\beta=(\beta_k)_k$, $\hat c_1=\alpha=(\alpha_k)_k$ are the \emph{border
vectors}, with the entries $\hat c_{a,k}$, so that $m^{(k)}_0=(\hat c_{0,k},\hat c_{1,k})$. Indeed, write $m^{(k)}=m^{(k)}_0\hat\eta^{-1}/\sigma_k$ with scalars
$\sigma_k$. Then $G_0'm^{(k)T}=fm^{(k)T}_0/\sigma_k$, and $\Gamma=4f\,SKS$ with $S=\diag(1/\sigma_k)$, since
$m^{(k)}_0\hat\eta^{-1}m^{(l)T}_0=2(\alpha_k\beta_l+\beta_k\alpha_l)$; the $\sigma_k$ cancel. Put $D=\det K$. For finite $c_k$ we take $\alpha_k=1$,
$\beta_k=c_k$, and for $c_k=\infty$ we take $(\alpha_k,\beta_k)=(0,1)$, unless other representatives are specified.

\emph{Homogeneity.} The right side of \eqref{eq:bordered} is homogeneous of degree $0$ in each pair $(\alpha_k,\beta_k)$; this is the
invariance of the dressing under rescaling of the BZ vectors (\emph{Directions of the BZ vectors}, \S\ref{subsec:dressing}). ($D$ is
homogeneous of degree $2$ in each pair, since the pair enters row $k$ and column $k$ of $K$.) So $G$ depends only on the ratios $c_k$,
and the values $c_k=0$ and $c_k=\infty$, i.e.\ $(\alpha_k,\beta_k)=(1,0)$ and $(0,1)$, are allowed. We call $G$ the \emph{ALE
soliton matrix} with nodes $w_1,\dots,w_N$ and parameters $c_1,\dots,c_N$. Where $D\ne0$ it is a matrix solution of \eqref{eq:hm}, by
the dressing method, for every $c$ including $c_k\in\{0,\infty\}$, and the Gram matrix of the \emph{ALE soliton metric}, in the null basis, is $2(-1)^{m+1}G$;
in the frame $e_\pm/\sqrt2$ it is $(-1)^{m+1}G(e_\pm,e_\pm)$.

\begin{definition}[{admissible set}]\label{def:admissible}
Let $N=2m+1$, let $w_1<\dots<w_N$ be nodes and $c=(c_1,\dots,c_N)\in(\RP^1)^N$ the parameters of an ALE soliton matrix, and let
$P(x)=\prod_k(x-w_k)$, with the derivative $P'(w_k)=\prod_{l\ne k}(w_k-w_l)$ at the node $w_k$. The \emph{weights} of $c$ are
$a_k=c_k/P'(w_k)\in\R\cup\{\infty\}$. The \emph{admissible set} $\cA$ is the set of
$c\in(\RP^1)^N$ such that
\begin{enumerate}[leftmargin=*,label=\textup{(\alph*)}]
\item all weights lie in $[0,\infty]$, i.e.\ $(-1)^{N-k}c_k\in[0,\infty]$ for all $k$;
\item at most $m$ of the $c_k$ equal $0$, and
\item at most $m$ of the $c_k$ equal $\infty$.
\end{enumerate}
The interior of $\cA$, where all weights are positive and finite, is the \emph{positive cone} $\cA^\circ$. Every point of $\cA$ is a
limit of points of $\cA^\circ$, since a coordinate $0$ or $\infty$ can be moved into its open arc. The \emph{base point}
$c^*=((-1)^{N-k})_k$ has the weights $1/|P'(w_k)|$, so $c^*\in\cA^\circ$.
\end{definition}
By Lemma~\ref{lem:Cbase}, the BZ vectors of $c^*$ give back the seed up to a constant factor and a constant change of frame; so the
BZ vectors of $c^*$ can serve as the \emph{trivial vectors} in \eqref{eq:BZnu}: $\Gamma^{(0)}=\Gamma(c^*)$.
Conditions (b) and (c) are Chen's constraint on the number of phantom solitons \cite[\S4.4, (4.40)]{Chen2016}: for $N$ poles on flat
space, at most $\lfloor N/2\rfloor$ phantom solitons of each sign are possible, since otherwise $\Gamma$ is singular; Chen states this without proof. The
conditions are sharp: if $m+1$ of the $c_k$ are $0$ (resp.\ $\infty$), then $\beta_k=0$ (resp.\ $\alpha_k=0$) for these $k$, so the $m+1$ rows
of $K$ with these indices vanish outside the other $m$ columns; they are linearly dependent, and $D=\det K\equiv0$. Conversely, once
(a) holds, (b) and (c) suffice: $D>0$ on $\Hb^\circ$ for every $c\in\cA$ (Theorem~\ref{thm:CF}(iii)).

\subsection{Solitons on the flat seed \texorpdfstring{$\R^3\times S^1$}{R3 x S1}}\label{subsec:MKN}
Let $m\ge1$, let $w_1<\dots<w_{2m}$ be points of the axis, $\mu_k=\mu_{w_k}$, and $b_k\in\RP^1$. The number of poles is even, $2m$, since
these solitons arise in \S\ref{sec:far} as limits of ALE soliton metrics with $2m+1$ nodes, one of which is sent to infinity
(Theorem~\ref{thm:mkn}).

The flat block $G_0=\diag(1,\rho^2)$ of $\R^3\times S^1$ has the generating matrix
\begin{equation}
\Psi_0(\lambda)=\diag\bigl(1,\rho^2-2(z-z_0)\lambda-\lambda^2\bigr),
\end{equation}
for any constant $z_0$; we take $z_0\notin\{w_k\}$. Indeed, here $U=\diag(0,2)$ and $V=0$, and substitution in \eqref{eq:lax} verifies
the two equations, with $\Psi_0(0)=G_0$ (cf.\ the Lorentzian version \cite[(8.36)]{BV2001}). At a root $\lambda_k$ of $X^2+2(z-w_k)X-\rho^2=0$ this is
$\diag\bigl(1,-2(w_k-z_0)\lambda_k\bigr)$. So the BZ vector $m^{(k)}_0=\bigl(1,-2(w_k-z_0)b_k\bigr)$ gives the direction
$[m^{(k)}]=[(1,b_k/\lambda_k)]$, for the soliton and for the anti-soliton pole alike. For $b_k=\infty$ read $m^{(k)}_0=(0,1)$ and $[m^{(k)}]=[(0,1/\lambda_k)]$, and in $\Gamma_{kl}$ below use the homogeneous form: with $b_k=\beta^+_k/\beta^-_k$ the numerator $\mu_k\mu_l+\rho^2b_kb_l$ becomes $\mu_k\mu_l\beta^-_k\beta^-_l+\rho^2\beta^+_k\beta^+_l$ (the vectors are rescaled by $\beta^-_k$, which does not change the dressing). Dressing $G_0$ with the soliton poles $\mu_k$ and
these vectors, and normalising, gives
\begin{equation}\label{eq:BZAF}
G=\frac{\prod_k\mu_k}{\rho^{2m}}\Bigl(G_0-\sum_{k,l}\frac{G_0m^{(k)T}\,(\Gamma^{-1})_{kl}\,m^{(l)}G_0}{\mu_k\mu_l}\Bigr),\qquad
\Gamma_{kl}=\frac{\mu_k\mu_l+\rho^2b_kb_l}{\mu_k\mu_l(\rho^2+\mu_k\mu_l)},
\end{equation}
with $\det G=\rho^2$: the \emph{soliton on $\R^3\times S^1$} with poles $w$ and parameters $b$, the Riemannian counterpart of the
multi-Kerr--NUT solutions of Kramer and Neugebauer, which are the same dressing of the Minkowski block $\diag(-1,\rho^2)$ with real
poles \cite[(5)--(7)]{KN1980}, \cite[\S\S8.3 and 8.5]{BV2001}. Here $\Gamma_{kl}$ is \eqref{eq:dressed} for the seed $G_0$ and the vectors $m^{(k)}=(1,b_k/\mu_k)$,
\begin{equation}\label{eq:BZAFgamma}
m^{(k)}G_0m^{(l)T}=1+\frac{\rho^2b_kb_l}{\mu_k\mu_l},\qquad
\Gamma_{kl}=\frac{m^{(k)}G_0m^{(l)T}}{\rho^2+\mu_k\mu_l}=\frac{\mu_k\mu_l+\rho^2b_kb_l}{\mu_k\mu_l(\rho^2+\mu_k\mu_l)},
\end{equation}
and the factor $\prod_k\mu_k/\rho^{2m}$ is the normalisation of Definition~\ref{def:soliton}(d): by \eqref{eq:detdressed} with $p=2m$ poles
and $\det G_0=\rho^2$,
\begin{equation}\label{eq:BZAFdet}
\det\Bigl(G_0-\sum_{k,l}\frac{G_0m^{(k)T}\,(\Gamma^{-1})_{kl}\,m^{(l)}G_0}{\mu_k\mu_l}\Bigr)=\rho^{4m}\prod_k\mu_k^{-2}\cdot\rho^2,
\end{equation}
and multiplying the matrix by $\prod_k\mu_k/\rho^{2m}$ multiplies its determinant by $\prod_k\mu_k^2/\rho^{4m}$, which restores $\det G=\rho^2$. For the data of \S\ref{sec:far}, $(-1)^mG$ is
positive definite (Theorem~\ref{thm:mkn}), and the metric is built from $(-1)^mG$. Formula \eqref{eq:BZAF} uses the normalisation of
Definition~\ref{def:soliton}(d); it can also be obtained by removing and re-adding poles (Remark~\ref{rem:readd}). Euclidean Kerr is
$m=1$; the Chen--Teo instanton \cite[(2), Fig.~1]{CT2011} and Teo's AF instanton \cite[(1)--(8)]{Teo2026} are $m=2$, with one and with no phantom soliton
(Table~\ref{tab:smooth}).

\emph{Frames.} Several bases of $\mathfrak t$ are used. Table~\ref{tab:frames} lists them, where they are defined, their determinant
relative to another frame of the list, and where they are used; a change of frame of determinant $\pm1$ does not change the Weyl
coordinates (\S\ref{subsec:WP}).
\begin{table}[ht]\centering\small
\begin{tabularx}{\textwidth}{>{\raggedright\arraybackslash}p{0.26\textwidth}>{\raggedright\arraybackslash}p{0.17\textwidth}>{\raggedright\arraybackslash}p{0.25\textwidth}>{\raggedright\arraybackslash}X}\toprule
frame & defined in & determinant & used in\\\midrule
basis diagonalising the seed \eqref{eq:staticseed} & \S\ref{subsec:flatseeds} & --- & \S\S\ref{subsec:flatseeds}--\ref{subsec:ALEsolitons}\\
null basis $(e_0,e_1)$, $\eta(e_0,e_1)=\frac12$ & \S\ref{subsec:ALEsolitons} & $\pm\frac12$ rel.\ to the diagonal basis & \eqref{eq:bordered}, \S\ref{sec:ale}\\
$e_\pm=e_0\pm e_1$ & Lemma~\ref{lem:Cbase} & $-2$ rel.\ to $(e_0,e_1)$ & base point\\
Hankel frame $(e_0,-e_1)$ (ALE) & Theorem~\ref{thm:family}(a) & $-1$ rel.\ to $(e_0,e_1)$ & \S\ref{sec:alerods}, \S\ref{sec:surj}\\
$(e_\gamma,e_1)$, $e_\gamma=-e_0+\gamma e_1$ (ALF, AF) & \S\ref{sec:far} & $-1$ rel.\ to $(e_0,e_1)$ & Theorem~\ref{thm:family}(b), \S\ref{sec:far}\\
frame of the soliton \eqref{eq:BZAF} & \S\ref{subsec:MKN} & $B=\gamma^{-1/2}\mathsf J_\gamma$, $\det B=1$, rel.\ to $(e_\gamma,e_1)$ & Theorem~\ref{thm:mkn}\\
$F_\alpha$ & \S\ref{subsec:CinCtor} & $1$ rel.\ to the Hankel frame & Corollary~\ref{cor:CinCtor}, \S\ref{sec:surj}\\
$(a_\gamma,b_\gamma)$ & Proposition~\ref{prop:farCtor} & $1$ rel.\ to $(e_\gamma,e_1)$ & \S\ref{sec:surj}, \S\ref{sec:proofs}\\
$(a,b)$ of the model & Definition~\ref{def:Ctor} & --- & \S\ref{sec:uniq}, \S\ref{sec:proofs}\\
$(f_0,f_1)$ of \eqref{eq:introclosed} & Remark~\ref{rem:GGclosed} & $1$ rel.\ to $(a,b)$ & \S\ref{sec:proofs}\\\bottomrule
\end{tabularx}
\caption{The frames of $\mathfrak t$ used in the paper. The scaling $E_\Omega$ of \S\ref{sec:far} has determinant $1$.}\label{tab:frames}
\end{table}

\section{Hankel data}\label{sec:hankel}
This section collects the algebra of Hankel determinants of finitely supported measures in which the rod lines of the soliton
metrics are computed (Proposition~\ref{prop:axis}): Heine's formula, the Wronskian, the total turning and an atom at infinity.
Heine's formula, the Christoffel identity and the Desnanot--Jacobi identity are classical \cite{Szego1975,Krattenthaler1999}, but
we are not aware of earlier uses of Hankel determinants for the rod structures of Belinski--Zakharov solitons, or in their
twistor description (Remark~\ref{rem:twistor}).

\subsection{Hankel determinants and Heine's formula}\label{subsec:heine}
Let $\tilde\omega=\sum_k\tilde\omega_k\delta_{w_k}$ be a finitely supported measure on $\R$, a sum of \emph{atoms} $\tilde\omega_k\delta_{w_k}$;
the points $w_k$ are its \emph{nodes} and the numbers $\tilde\omega_k$ its \emph{weights}. A set of atoms is identified with the set of its
indices, so that the sets $S$ below are subsets of $\{1,\dots,n\}$, and ``the atom $k$'' is the atom at $w_k$. An atom at infinity,
$\gamma\delta_\infty$ with $\gamma\ge0$, is added in \S\ref{subsec:atominf}; ``the atom at infinity'' always means this point mass, and
$\gamma$ its weight. Throughout, $\tilde\omega$ denotes a measure whose weights $\tilde\omega_k$ may have either sign, and
$\omega=\sum_ka_k\delta_{w_k}$ a measure with positive weights $a_k$, written $\omega>0$; we write $w=(w_1,\dots,w_n)$ with $w_1<\dots<w_n$
and $a=(a_1,\dots,a_n)$. A subscript without parentheses on a measure, as in $\tilde\omega_k$, denotes a weight; the sign-reversed
measures defined below carry their index in parentheses. For a function $q$ on the nodes, with the values $q_k=q(w_k)$, or a vector
$q=(q_1,\dots,q_n)$, $q\tilde\omega$ is the measure with the weights $q_k\tilde\omega_k$; for instance $(t-x)\tilde\omega$, with $t$ the variable
of integration, has the weights $(w_k-x)\tilde\omega_k$. The \emph{moments} of $\tilde\omega$ are $h_p(\tilde\omega)=\sum_k\tilde\omega_kw_k^p$,
$p\ge0$, and its \emph{Hankel determinants} are $H_0(\tilde\omega)=1$ and
\begin{equation}\label{eq:Hs}
H_s(\tilde\omega)=\det\bigl(h_{i+l}(\tilde\omega)\bigr)_{i,l=0}^{s-1}
=\det\begin{pmatrix}h_0&h_1&\cdots&h_{s-1}\\ h_1&h_2&\cdots&h_s\\ \vdots&\vdots&&\vdots\\ h_{s-1}&h_s&\cdots&h_{2s-2}\end{pmatrix},
\qquad s\ge1 .
\end{equation}

For a set $S$ of atoms let
\begin{equation}\label{eq:PS}
\Delta_S^2=\prod_{k<l\in S}(w_k-w_l)^2\quad(=1\text{ if }|S|\le1),\qquad
P_S(\tilde\omega)=\prod_{k\in S}\tilde\omega_k\,\Delta_S^2 .
\end{equation}
\begin{lemma}[Heine]\label{lem:Aheine}
For every finitely supported measure $\tilde\omega$ with real or complex weights and every $s\ge1$, it holds that
\begin{equation}\label{eq:heine}
H_s(\tilde\omega)=\sum_{|S|=s}P_S(\tilde\omega),
\end{equation}
the sum over the sets $S$ of $s$ atoms.
\end{lemma}
\begin{proof}
Formula \eqref{eq:heine} is the Cauchy--Binet formula for $(h_{i+l})=V^T\diag(\tilde\omega_k)V$ with $V=(w_k^i)$
(Appendix~\ref{app:identities}), \cite[\S2.1]{Szego1975}, \cite{Andreief1886}.
\end{proof}

\emph{One atom split off.} Let $k$ be an atom of $\tilde\omega$ and $\tilde\omega'=\tilde\omega-\tilde\omega_k\delta_{w_k}$. Splitting the
sum \eqref{eq:heine} according to whether $k\in S$, and using $\Delta_S^2=\Delta_{S\setminus k}^2\prod_{l\in S\setminus k}(w_l-w_k)^2$ for $k\in S$, gives
\begin{equation}\label{eq:Hsplit}
H_s(\tilde\omega)=H_s(\tilde\omega')+\tilde\omega_k\,H_{s-1}\bigl((t-w_k)^2\tilde\omega'\bigr),\qquad s\ge1 ;
\end{equation}
applied to $q\tilde\omega$, it gives the same formula with $q\tilde\omega'$ in place of $\tilde\omega'$ and the weight $q_k\tilde\omega_k$.

\emph{Complex weights.} The moments, the Hankel determinants \eqref{eq:Hs}, the products \eqref{eq:PS}, the sums \eqref{eq:Sigmas}
and the shifted determinants \eqref{eq:shifted} below are defined in the same way for complex weights, and for distinct complex points $w_k$ in place of the
nodes. They are polynomials in the weights, and $H_s$ and $\Sigma_s$ are homogeneous of degree $s$:
$H_s(\lambda\tilde\omega)=\lambda^sH_s(\tilde\omega)$ and $\Sigma_s(\lambda\tilde\omega)=\lambda^s\Sigma_s(\tilde\omega)$ for $\lambda\in\C$. The identities of
Lemmas~\ref{lem:Aheine} and~\ref{lem:Achristoffel}, \eqref{eq:Hsplit}, \eqref{eq:splitSigma} and the determinant form of \eqref{eq:shifted}
are polynomial identities that use only that the points $w_k$ are distinct, so they hold in this generality.
\S\S\ref{sec:ale}--\ref{sec:far} use them for measures such as $q\omega$ with complex $q$. Statements involving signs or positivity, such
as Proposition~\ref{prop:Awronskian}, refer to real weights and nodes.

\emph{Sign-reversed measures.} Let $\omega=\sum_ka_k\delta_{w_k}$ be a positive measure. The nodes divide the axis into the rods
$I_0=(-\infty,w_1)$, $I_j=(w_j,w_{j+1})$ and $I_n=(w_n,\infty)$, as in Definition~\ref{def:rodstructure}. For $0\le j\le n$ let
\begin{equation}\label{eq:omegaj}
\omega_{(j)}=-\sum_{k\le j}a_k\delta_{w_k}+\sum_{k>j}a_k\delta_{w_k}
\end{equation}
be $\omega$ with the atoms to the left of the rod $I_j$ sign-reversed; so $\omega_{(0)}=\omega$ and $\omega_{(n)}=-\omega$.

\emph{Hankel families.} For $0\le r\le n-1$ the \emph{Hankel family} $(n,r)$ of $\omega$ is the sequence of vectors
\begin{equation}\label{eq:Hfamily}
V_j=\bigl(H_r(\omega_{(j)}),\,H_{r+1}(\omega_{(j)})\bigr)\in\R^2,\qquad 0\le j\le n ,
\end{equation}
one for each rod $I_j$. The name records that the components are Hankel determinants, of the consecutive orders $r$ and $r+1$.
The first index is the number of nodes. The vectors of the family $(n,r)$ of the reduced measure on the turning points span the
rod lines of the ALE soliton metrics (Theorem~\ref{thm:family}, Corollary~\ref{cor:CinCtor}(ii)); Lemma~\ref{lem:Climits} relates
families with different numbers of nodes.

\subsection{The Wronskian}
For $V,V'\in\R^2$ we write $\det(V,V')$ for the determinant of the $2\times2$ matrix with the columns $V$ and $V'$. For
consecutive vectors of a Hankel family, $\det(V_{j-1},V_j)$ is the Casorati determinant, or discrete Wronskian, of the two
sequences $H_r(\omega_{(j)})$ and $H_{r+1}(\omega_{(j)})$, and we call it the \emph{Wronskian} of the family at the node $w_j$. It
vanishes exactly when $V_{j-1}$ and $V_j$ are parallel, and its sign tells in which sense $V_{j-1}$ turns into $V_j$: if
$\det(V_{j-1},V_j)<0$, then $V_j$ is obtained from $V_{j-1}$ by a clockwise rotation by an angle in $(0,\pi)$, followed by a
positive rescaling. Proposition~\ref{prop:Awronskian} rests on the following classical identity.
\begin{lemma}[Christoffel identity]\label{lem:Achristoffel}
Let $\tilde\omega$ be a finitely supported measure with real or complex weights, $x\in\C$ and $r\ge1$, and let $t$ denote the variable of
integration, so that $(t-x)^p\tilde\omega$ has the weights $(w_k-x)^p\tilde\omega_k$. Then
\begin{equation}\label{eq:Achristoffel}
H_r(\tilde\omega)\,H_r\bigl((t-x)^2\tilde\omega\bigr)-H_{r+1}(\tilde\omega)\,H_{r-1}\bigl((t-x)^2\tilde\omega\bigr)=H_r\bigl((t-x)\tilde\omega\bigr)^2 .
\end{equation}
\end{lemma}
\begin{proof}
Heine's sum depends only on differences of nodes, so we may use the moments $\tilde h_p=\sum_k\tilde\omega_k(w_k-x)^p$. Then
$H_s(\tilde\omega)=\det(\tilde h_{i+l})$, $H_s((t-x)\tilde\omega)=\det(\tilde h_{i+l+1})$ and $H_s((t-x)^2\tilde\omega)=\det(\tilde h_{i+l+2})$, $0\le i,l\le s-1$.
For $M=(\tilde h_{i+l})_{i,l=0}^r$ the Desnanot--Jacobi identity \cite[\S2.3, Prop.~10, (2.16)]{Krattenthaler1999}, \cite[\S3.5, Thm.~3.12]{Bressoud1999} reads
\begin{equation}
\det M\,\det M^{0,r}_{0,r}=\det M^r_r\,\det M^0_0-\det M^0_r\,\det M^r_0 ,
\end{equation}
where $M^{I}_{J}$ omits the rows $I$ and the columns $J$. Here
\begin{equation}
\begin{gathered}
\det M=H_{r+1}(\tilde\omega),\qquad \det M^r_r=H_r(\tilde\omega),\qquad \det M^0_0=H_r\bigl((t-x)^2\tilde\omega\bigr),\\
\det M^{0,r}_{0,r}=H_{r-1}\bigl((t-x)^2\tilde\omega\bigr),\qquad \det M^0_r=\det M^r_0=H_r\bigl((t-x)\tilde\omega\bigr).
\end{gathered}
\end{equation}
\end{proof}

\begin{proposition}[Wronskian]\label{prop:Awronskian}
Let $\omega>0$ have $n$ atoms and $0\le r\le n-1$, and let $(V_j)$ be its Hankel family $(n,r)$. Then, for $1\le j\le n$,
\begin{equation}\label{eq:wronskian}
\det(V_{j-1},V_j)=-2a_jU_r(w_j)^2,\qquad U_r(x)=\sum_{|S|=r}\prod_{k\in S}a_k|x-w_k|\,\Delta_S^2 ,
\end{equation}
and $U_r(w_j)>0$.
\end{proposition}
\begin{proof}
$\omega_{(j-1)}$ and $\omega_{(j)}$ differ only in the sign of the atom $j$, which is positive in $\omega_{(j-1)}$. For $r=0$,
$V_j=(1,H_1(\omega_{(j)}))$, so
\begin{equation}
\det(V_{j-1},V_j)=H_1(\omega_{(j)})-H_1(\omega_{(j-1)})=-2a_j ,
\end{equation}
and $U_0=1$. Let $r\ge1$, and let $\tilde\omega$ be $\omega_{(j-1)}$ without the atom $j$. By \eqref{eq:Hsplit},
\begin{equation}
H_s(\tilde\omega\pm a_j\delta_{w_j})=H_s(\tilde\omega)\pm a_jH_{s-1}\bigl((t-w_j)^2\tilde\omega\bigr).
\end{equation}
So $V_{j-1}=A+a_jB$ and $V_j=A-a_jB$ with
\begin{equation}
A=\bigl(H_r(\tilde\omega),\,H_{r+1}(\tilde\omega)\bigr),\qquad B=\bigl(H_{r-1}((t-w_j)^2\tilde\omega),\,H_r((t-w_j)^2\tilde\omega)\bigr),
\end{equation}
and by Lemma~\ref{lem:Achristoffel}
\begin{equation}
\det(V_{j-1},V_j)=-2a_j\det(A,B)=-2a_jH_r\bigl((t-w_j)\tilde\omega\bigr)^2 .
\end{equation}
The measure $(t-w_j)\tilde\omega$ has the weights $-a_k(w_k-w_j)>0$ for $k<j$ and $a_k(w_k-w_j)>0$ for $k>j$, so
$H_r((t-w_j)\tilde\omega)=U_r(w_j)$, which is positive since there are $n-1\ge r$ atoms.
\end{proof}

\subsection{Limits and total turning}
\begin{lemma}[limits of the Hankel vectors]\label{lem:Climits}
Let $\omega=\sum_ka_k\delta_{w_k}>0$ have the nodes $w=(w_1,\dots,w_n)$ and the weights $a=(a_1,\dots,a_n)$, let $r\ge1$, and
fix an atom $k$. Let $a_k$ vary, with the other weights fixed, and let $\sigma_{kj}=\pm1$ be the sign of the atom $k$ in
$\omega_{(j)}$.
\begin{enumerate}[leftmargin=*,label=\textup{(\roman*)}]
\item As $a_k\to0$, $H_s(\omega_{(j)})\to H_s(\omega_{(j)}\setminus k)$ for every $s$ and $j$, where $\setminus k$ omits the atom $k$.
So, for $r\le n-2$, the family $(n,r)$ tends to the family $(n-1,r)$ on the nodes $w\setminus w_k$, with the rods $I_{k-1}$ and $I_k$ merged (the index $j\ge k$ becomes $j-1$).
\item As $a_k\to\infty$, $a_k^{-1}H_s(\omega_{(j)})\to\sigma_{kj}H_{s-1}(\omega^{(k)}_{(j)})$ for $s\ge1$, where
\begin{equation}\label{eq:omegak}
\omega^{(k)}=\sum_{l\ne k}a_l(w_l-w_k)^2\delta_{w_l},
\end{equation}
and $\omega^{(k)}_{(j)}$ is $\omega^{(k)}$ with each atom signed as in $\omega_{(j)}$. So the pair $(H_r,H_{r+1})$ of the family $(n,r)$
satisfies, for each $j$, $(\sigma_{kj}a_k)^{-1}V_j\to(H_{r-1},H_r)(\omega^{(k)}_{(j)})$, a vector of the family $(n-1,r-1)$ of $\omega^{(k)}$ (of index $j$ for $j<k$, and $j-1$ for $j\ge k$). The lines of
$V_{k-1}$ and $V_k$ again coincide, since $\sigma_{k,k-1}=-\sigma_{kk}$.
\end{enumerate}
\end{lemma}
\begin{proof}
The atom $k$ of $\omega_{(j)}$ has the weight $\sigma_{kj}a_k$, and $(t-w_k)^2(\omega_{(j)}\setminus k)=\omega^{(k)}_{(j)}$. So, by \eqref{eq:Hsplit},
\begin{equation}
H_s(\omega_{(j)})=H_s(\omega_{(j)}\setminus k)+\sigma_{kj}a_kH_{s-1}\bigl(\omega^{(k)}_{(j)}\bigr),
\end{equation}
where neither $H_s(\omega_{(j)}\setminus k)$ nor $H_{s-1}(\omega^{(k)}_{(j)})$ depends on $a_k$. Letting $a_k\to0$ gives (i); dividing by $a_k$ and
letting $a_k\to\infty$ gives (ii). The measures $\omega_{(k-1)}$ and $\omega_{(k)}$ differ only in the sign of the atom $k$, so
$\omega_{(k-1)}\setminus k=\omega_{(k)}\setminus k$ and $\omega^{(k)}_{(k-1)}=\omega^{(k)}_{(k)}$.
\end{proof}

\emph{Normalised families.} Let $(V_j)$ be the Hankel family $(n,r)$ of $\omega>0$, $0\le r\le n-1$. Since $\omega_{(0)}=\omega$, the
components of $V_0=(H_r(\omega),H_{r+1}(\omega))$ are positive, by \eqref{eq:heine}. The \emph{normalised Hankel family} is
\begin{equation}\label{eq:Vnorm}
\tilde V_j=\Bigl(\frac{H_r(\omega_{(j)})}{H_r(\omega)},\ \frac{H_{r+1}(\omega_{(j)})}{H_{r+1}(\omega)}\Bigr),\qquad 0\le j\le n ,
\end{equation}
a positive rescaling of each component. Since $\omega_{(n)}=-\omega$ and $H_s(-\omega)=(-1)^sH_s(\omega)$, the normalised outer
vectors are $\tilde V_0=(1,1)$ and $\tilde V_n=(-1)^r(1,-1)$, and by Proposition~\ref{prop:Awronskian} every step from $\tilde V_{j-1}$ to
$\tilde V_j$ is still a clockwise rotation by an angle $s_j\in(0,\pi)$. The sum $\sum_js_j$ is the \emph{total turning} of the family (Figure~\ref{fig:turning}).
\begin{figure}[ht]
\centering
\begin{tikzpicture}[x=1cm,y=1cm,>=stealth]
\draw[gray!60,->] (-2.4,0) -- (2.6,0);
\draw[gray!60,->] (0,-2.4) -- (0,2.6);
\foreach \a/\l/\p in {45/{\tilde V_0=(1,1)}/{above right}, -15/{\tilde V_1}/right, -85/{\tilde V_2}/below, -165/{\tilde V_3}/left, 135/{\tilde V_4=-(1,-1)}/{above left}}
  \draw[->,thick] (0,0) -- (\a:2) node[\p,font=\small] {$\l$};
\draw[->] (45:0.8) arc (45:-15:0.8);     \node[font=\scriptsize] at (15:1.05) {$s_1$};
\draw[->] (-15:1.0) arc (-15:-85:1.0);   \node[font=\scriptsize] at (-50:1.27) {$s_2$};
\draw[->] (-85:1.2) arc (-85:-165:1.2);  \node[font=\scriptsize] at (-125:1.47) {$s_3$};
\draw[->] (-165:1.4) arc (-165:-225:1.4);\node[font=\scriptsize] at (-195:1.67) {$s_4$};
\end{tikzpicture}
\caption{A normalised Hankel family with $n=4$ and $r=1$. The vectors $\tilde V_0,\dots,\tilde V_4$ turn clockwise by the steps
$s_j\in(0,\pi)$, from $\tilde V_0=(1,1)$ to $\tilde V_4=(-1)^r(1,-1)$, with total turning $\sum_js_j=(r+\frac12)\pi=\frac32\pi$
(Proposition~\ref{prop:Aturning}). Lemma~\ref{lem:match}(ii) brings the rod lines of an element of $\Ctor$ of type ALE, normalised by
the model at infinity, into the same form.}
\label{fig:turning}
\end{figure}
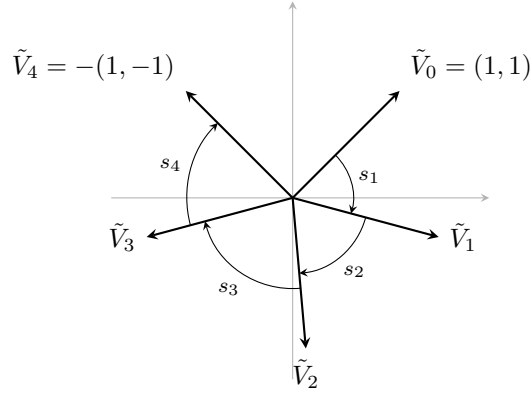
\begin{proposition}[total turning]\label{prop:Aturning}
Let $\omega>0$ have $n$ atoms and $0\le r\le n-1$. The total turning of the normalised Hankel family $(n,r)$ of $\omega$ is
\begin{equation}\label{eq:totalturning}
\sum_{j=1}^ns_j=\Bigl(r+\frac12\Bigr)\pi .
\end{equation}
In particular it depends neither on $n$ nor on the nodes and the weights.
\end{proposition}
\begin{proof}
The sum is continuous in $\omega$ and congruent modulo $2\pi$ to the fixed angle from $\tilde V_0$ to $\tilde V_n$, so it is
constant on the connected set of positive measures. For $n>r+1$ it extends continuously to a face $a_k=0$, where by
Lemma~\ref{lem:Climits}(i) the family becomes the family $(n-1,r)$ of $\omega\setminus k$, whose components $H_r,H_{r+1}$ are positive and whose steps lie in $(0,\pi)$ by Proposition~\ref{prop:Awronskian}, with one step equal to $0$. So the sum does not depend on $n$, and
it suffices to take $n=r+1$. Then the $r$-sets are the complements of single atoms $k$, with
$P_{\{1,\dots,n\}\setminus k}(\omega)=\tau_k\prod_la_l$, $\tau_k=\Delta^2_{\{1,\dots,n\}\setminus k}/a_k$. By \eqref{eq:PS} and \eqref{eq:omegaj}, the term
$P_S(\omega_{(j)})$ of Heine's formula \eqref{eq:heine} differs from $P_S(\omega)$ by a factor $-1$ for each atom of $S$ to the left of $I_j$,
i.e.\ by $(-1)^{|S\cap\{1,\dots,j\}|}$. For the single $(r+1)$-set $\{1,\dots,n\}$ this gives $H_{r+1}(\omega_{(j)})=(-1)^jH_{r+1}(\omega)$, and
for the $r$-set $\{1,\dots,n\}\setminus k$, the term of $k$, it gives the sign $(-1)^{j-[k\le j]}$. Hence
\begin{equation}
\tilde V_j=(-1)^j(1-2F_j,1),\qquad F_j=\frac{\sum_{k\le j}\tau_k}{\sum_k\tau_k},\qquad 0=F_0<F_1<\dots<F_n=1 .
\end{equation}
The vectors $u_j=(1-2F_j,1)$ turn
counter-clockwise by $\delta_j\in(0,\pi)$ from $(1,1)$ to $(-1,1)$, so $\sum\delta_j=\pi/2$. Because of the factor $(-1)^j$, the
clockwise step from $\tilde V_{j-1}$ to $\tilde V_j$ is $\pi-\delta_j$, and $\sum s_j=n\pi-\pi/2=(r+\frac12)\pi$.
\end{proof}

\emph{Total turning and the degree.} Proposition~\ref{prop:Aturning} is what matches Hankel families with rod structures. By
Lemma~\ref{lem:match}, the rod lines of an element of $\Ctor$ of type ALE, normalised by the model at infinity, determine an integer
$r$, the degree index, and their steps lie in the cell $C_r$ of \eqref{eq:cell}, the set of steps with the sum $(r+\frac12)\pi$. By
Proposition~\ref{prop:Aturning}, the steps of the normalised family $(n,r)$ lie in the same cell, whatever its nodes and weights. So
the family $(n,r)$ can only produce rod structures of degree index $r$, and \S\ref{sec:surj} shows that it produces all of them.
By Corollary~\ref{cor:CinCtor}(iii) and Lemma~\ref{lem:degreeindex}, the Li--Sun degrees of an ALE soliton metric whose rod lines
are those of the Hankel family $(n,r)$ are $n-1-r$ and $r$, one for each orientation.

\subsection{An atom at infinity}\label{subsec:atominf}
The ALF and AF soliton metrics of \S\ref{sec:far} arise from ALE soliton metrics with one more node $\Omega$, which is sent to
infinity. On the side of the measures, the atom at $\Omega$, with its weight rescaled, becomes an atom $\gamma\delta_\infty$ at the
point $\infty$ of $\RP^1$, and the Hankel determinant of order $r+1$ acquires the correction defined next.
For a measure $\tilde\omega$ on $\{w_k\}$ with real or complex weights and $s\ge0$ let
\begin{equation}\label{eq:Sigmas}
\Sigma_s(\tilde\omega)=\sum_{|S|=s}P_S(\tilde\omega)\,e_1(w_S),\qquad e_1(w_S)=\sum_{k\in S}w_k ,
\end{equation}
so that $\Sigma_0(\tilde\omega)=0$, and put $H_s=\Sigma_s=0$ for all $s<0$. By \eqref{eq:heine}, $H_s(\tilde\omega)=\Sigma_s(\tilde\omega)=0$ when $s$
exceeds the number of atoms. The sum of all nodes is written $e_1(w)=\sum_kw_k$. Splitting off an atom $k$ as in \eqref{eq:Hsplit}, with
$e_1(w_S)=w_k+e_1(w_{S\setminus k})$ for $k\in S$, gives
\begin{equation}\label{eq:splitSigma}
\Sigma_s(\tilde\omega)=\Sigma_s(\tilde\omega')+\tilde\omega_k\Bigl(\Sigma_{s-1}\bigl((t-w_k)^2\tilde\omega'\bigr)+w_kH_{s-1}\bigl((t-w_k)^2\tilde\omega'\bigr)\Bigr),
\qquad s\ge1 .
\end{equation}

For $\gamma\in\R$ and $r\ge0$, the \emph{shifted Hankel determinant} is
\begin{equation}\label{eq:shifted}
H^{(\gamma)}_{r+1}(\tilde\omega)=H_{r+1}(\tilde\omega)-2\gamma\,\Sigma_r(\tilde\omega)-\gamma^2H_{r-1}(\tilde\omega).
\end{equation}
It is the Hankel determinant of order $r+1$ with $\gamma$ added to the moment $h_{2r-1}$:
$H^{(\gamma)}_{r+1}(\tilde\omega)=\det\bigl(h_{i+l}(\tilde\omega)+\gamma\,\delta_{i+l,\,2r-1}\bigr)_{i,l=0}^{r}$. For $r=0$ both sides are $H_1(\tilde\omega)$.
For $r\ge1$, $\gamma$ is added at the entries $(r-1,r)$ and $(r,r-1)$. By multilinearity in the rows $r-1$ and $r$, the linear terms are $\gamma$
times the cofactors of these two entries, each equal to $-\Sigma_r(\tilde\omega)$ by the Cauchy--Binet formula (Appendix~\ref{app:identities}),
since $\det(w_k^l)_{k\in S,\,l\in\{0,\dots,r-2,r\}}=e_1(w_S)\Delta(w_S)$ for a set $S$ of $r$ atoms in increasing order, with the Vandermonde
determinant $\Delta(w_S)=\prod_{k<l\in S}(w_l-w_k)$ (compare the coefficients of $x^{r-1}$ in the Vandermonde determinant
$\Delta(w_S)\prod_{k\in S}(x-w_k)$ of the points $(w_S,x)$). The quadratic term is $-\gamma^2H_{r-1}(\tilde\omega)$, by the Laplace
expansion \eqref{eq:laplace} along the rows $r-1$ and $r$.

\begin{proposition}[an atom sent to infinity]\label{prop:ALFlimit}
Let $\omega=\sum_{k=1}^na_k\delta_{w_k}$ with $a_k>0$, $1\le r\le n$ and $\gamma>0$. Add an atom at $\Omega>w_n$ with the weight
$a_\Omega=\gamma\Omega^{1-2r}$, and let $\Omega\to+\infty$. Then the Hankel vectors $V^\Omega_j$ of the family $(n+1,r)$ on the first $n+1$ rods of
the nodes $w_1,\dots,w_n,\Omega$ (these are $I_0,\dots,I_{n-1}$ and $(w_n,\Omega)$) satisfy
\begin{equation}\label{eq:ALFvec}
S_\Omega V^\Omega_j=V^{(\gamma)}_j+O(\Omega^{-1}),\qquad V^{(\gamma)}_j=\Bigl(H_r(\omega_{(j)}),\ H^{(\gamma)}_{r+1}(\omega_{(j)})\Bigr),\qquad 0\le j\le n ,
\end{equation}
where $S_\Omega(X,Y)=(X,Y-\gamma\Omega X)$ is a shear of $\R^2$ of determinant $1$, the same for all $j$.
\end{proposition}
\begin{proof}
On the rods considered the atom at $\Omega$ is never sign-reversed, so the measure of the rod $j$ is $\omega_{(j)}+a_\Omega\delta_\Omega$. By
\eqref{eq:Hsplit} for the atom at $\Omega$,
\begin{equation}
H_s(\omega_{(j)}+a_\Omega\delta_\Omega)=H_s(\omega_{(j)})+a_\Omega\,H_{s-1}\bigl((t-\Omega)^2\omega_{(j)}\bigr),
\end{equation}
and by Heine's formula, since $\prod_{k\in S}(w_k-\Omega)^2=\Omega^{2|S|}\bigl(1-2\Omega^{-1}e_1(w_S)+O(\Omega^{-2})\bigr)$, for every measure
$\tilde\omega$ on the nodes $w_1,\dots,w_n$
\begin{equation}
H_{s-1}\bigl((t-\Omega)^2\tilde\omega\bigr)=\Omega^{2s-2}\bigl(H_{s-1}(\tilde\omega)-2\Omega^{-1}\Sigma_{s-1}(\tilde\omega)+O(\Omega^{-2})\bigr).
\end{equation}
Write $V^\Omega_j=(X_j,Y_j)$ and $H_s=H_s(\omega_{(j)})$, $\Sigma_s=\Sigma_s(\omega_{(j)})$. With $a_\Omega=\gamma\Omega^{1-2r}$,
\begin{equation}
X_j=H_r+\gamma\Omega^{-1}H_{r-1}+O(\Omega^{-2}),\qquad Y_j=H_{r+1}+\gamma\Omega H_r-2\gamma\Sigma_r+O(\Omega^{-1}),
\end{equation}
so $Y_j-\gamma\Omega X_j=H_{r+1}-2\gamma\Sigma_r-\gamma^2H_{r-1}+O(\Omega^{-1})=H^{(\gamma)}_{r+1}(\omega_{(j)})+O(\Omega^{-1})$.
\end{proof}

\begin{definition}[Hankel family with an atom at infinity]\label{def:ALFfamily}
For $n\ge1$, $0\le r\le n$, nodes $w_1<\dots<w_n$, weights $a_k>0$ and $\gamma\ge0$, the Hankel family $(n,r,\gamma)$ is the sequence
\begin{equation}
V_j=V^{(\gamma)}_j=(H_r(\omega_{(j)}),H^{(\gamma)}_{r+1}(\omega_{(j)}))\in\R^2,
\end{equation}
$0\le j\le n$, of \eqref{eq:ALFvec}. Its soliton data are the
positive measure $\omega+\gamma\,\delta_\infty$ on $\RP^1=\R\cup\{\infty\}$. For $r=0$ the vectors do not depend on $\gamma$, and Proposition~\ref{prop:ALFlimit} holds trivially: the added atom has the weight
$a_\Omega=\gamma\Omega$, it enters only the second component, by $\gamma\Omega X_j$, and the shear removes it. For $r=n$, $H_{n+1}(\omega_{(j)})=0$, so $V_j=\bigl(H_n(\omega_{(j)}),-2\gamma\Sigma_n(\omega_{(j)})-\gamma^2H_{n-1}(\omega_{(j)})\bigr)$; for $\gamma=0$ all second components vanish, so Proposition~\ref{prop:ALFstructure}(i) assumes $\gamma>0$ for $r=n$.
\end{definition}

\begin{proposition}[structure of the family]\label{prop:ALFstructure}
Let $\omega>0$ with $n$ atoms, $0\le r\le n$, $\gamma\ge0$, and, for $r\ge1$, $\gamma_*=\sqrt{H_{r+1}(\omega)/H_{r-1}(\omega)}$ (so $\gamma_*=0$ for
$r=n$). Let $(V_j)$ be the family $(n,r,\gamma)$, and $U_r$ as in \eqref{eq:wronskian}.
\begin{enumerate}[leftmargin=*,label=\textup{(\roman*)}]
\item\label{it:alfs1} For $1\le j\le n$,
\begin{equation}\label{eq:wronskianALF}
\det(V_{j-1},V_j)=-2a_j\bigl(U_r(w_j)+\gamma\,U_{r-1}(w_j)\bigr)^2<0
\end{equation}
(for $r=n$ assume $\gamma>0$; for $r=0$ read $U_{-1}=0$, so that $\det(V_{j-1},V_j)=-2a_j$). So all steps are clockwise.
\item\label{it:alfs2} The outer vectors satisfy
\begin{equation}\label{eq:detVoVn}
\det(V_0,V_n)=2(-1)^{r+1}H_r(\omega)\bigl(H_{r+1}(\omega)-\gamma^2H_{r-1}(\omega)\bigr)
\end{equation}
(for $r=0$ this is $-2H_1(\omega)$). For $r\ge1$, $V_0\parallel V_n$ if and only if $\gamma=\gamma_*$, and then $V_n=(-1)^rV_0$.
\item\label{it:alfs4} For $r\ge1$, as $\gamma\to\infty$, after the second component is divided by $-\gamma^2$, the family tends to the Hankel family
$(n,r-1)$ with the two components exchanged (both operations reverse the orientation, so the steps remain clockwise, consistent with (i) and Proposition~\ref{prop:Awronskian}).
\end{enumerate}
\end{proposition}
\begin{proof}
\ref{it:alfs1} For $r=0$, $V_j=(1,h_0(\omega_{(j)}))$, and $h_0(\omega_{(j)})-h_0(\omega_{(j-1)})=-2a_j$ since the atom at $w_j$ changes sign. For
$\gamma=0$ and $1\le r\le n-1$, $(V_j)$ is the Hankel family $(n,r)$, since $H^{(0)}_{r+1}=H_{r+1}$, and \eqref{eq:wronskianALF} is
Proposition~\ref{prop:Awronskian}. Let $r\ge1$ and $\gamma>0$. For $\Omega>w_n$ and $a_\Omega=\gamma\Omega^{1-2r}$, apply
Proposition~\ref{prop:Awronskian} to the measure $\omega+a_\Omega\delta_{\Omega}>0$ on $n+1$ nodes at the turning point $w_j$; the atom at $\Omega$ is
not reversed on the rods $j-1$, $j$. Let $\hat\omega$ be the measure $\omega_{(j-1)}+a_\Omega\delta_{\Omega}$ without the atom $j$, and
$\tilde\omega=\hat\omega-a_\Omega\delta_{\Omega}$. By \eqref{eq:Hsplit} for the atom at $\Omega$, which has the weight $a_\Omega(\Omega-w_j)$ in $(t-w_j)\hat\omega$,
\begin{equation}
\begin{aligned}
H_r\bigl((t-w_j)\hat\omega\bigr)&=H_r\bigl((t-w_j)\tilde\omega\bigr)+a_\Omega(\Omega-w_j)H_{r-1}\bigl((t-\Omega)^2(t-w_j)\tilde\omega\bigr)\\
&\to U_r(w_j)+\gamma U_{r-1}(w_j)\qquad(\Omega\to\infty),
\end{aligned}
\end{equation}
since $H_{r-1}((t-\Omega)^2(t-w_j)\tilde\omega)=\Omega^{2r-2}\bigl(H_{r-1}((t-w_j)\tilde\omega)+O(\Omega^{-1})\bigr)$ by the expansion in the proof of
Proposition~\ref{prop:ALFlimit}, and $H_p((t-w_j)\tilde\omega)=U_p(w_j)$ as in the proof of Proposition~\ref{prop:Awronskian}. By
Proposition~\ref{prop:ALFlimit}, the vectors $V^{\Omega}_j$ of the family on $n+1$ nodes are $S_\Omega^{-1}(V_j+O(1/\Omega))$, and $\det S_\Omega=1$, so
\begin{equation}
\det(V_{j-1},V_j)=\lim_{\Omega\to\infty}\det(V^{\Omega}_{j-1},V^{\Omega}_j)=-2a_j\bigl(U_r(w_j)+\gamma U_{r-1}(w_j)\bigr)^2 .
\end{equation}
The measure $(t-w_j)\tilde\omega$ has $n-1$ positive atoms, so $U_{r-1}(w_j)>0$, and $U_r(w_j)>0$ for $r\le n-1$; for $r=n$, $U_n=0$.
\ref{it:alfs2} Since $\omega_{(n)}=-\omega$, $H_r(\omega_{(n)})=(-1)^rH_r(\omega)$, and \eqref{eq:shifted} gives
\begin{equation}
H^{(\gamma)}_{r+1}(\omega_{(n)})=(-1)^{r+1}H^{(-\gamma)}_{r+1}(\omega)
=(-1)^{r+1}\bigl(H_{r+1}+2\gamma\Sigma_r-\gamma^2H_{r-1}\bigr),
\end{equation}
where $H_s=H_s(\omega)$ and $\Sigma_r=\Sigma_r(\omega)$, since $H_s(-\omega)=(-1)^sH_s(\omega)$ and $\Sigma_s(-\omega)=(-1)^s\Sigma_s(\omega)$. With
$V_0=(H_r,H_{r+1}-2\gamma\Sigma_r-\gamma^2H_{r-1})$ this gives \eqref{eq:detVoVn}. For $r\ge1$, $H_r>0$ and $\gamma\ge0$, so
\eqref{eq:detVoVn} vanishes if and only if $\gamma=\gamma_*$; then the second components of $V_0$ and $V_n$ are $-2\gamma\Sigma_r$ and
$(-1)^{r+1}2\gamma\Sigma_r$, so $V_n=(-1)^rV_0$.

\ref{it:alfs4} The leading term of $H^{(\gamma)}_{r+1}(\omega_{(j)})$ in $\gamma$ is $-\gamma^2H_{r-1}(\omega_{(j)})$.
\end{proof}

\begin{remark}[the twistor description]\label{rem:twistor}
Fletcher and Woodhouse encode a stationary axisymmetric or toric Ricci-flat metric by a symmetric $2\times2$ patching matrix adapted
to one rod; passing to the adjacent rod at a turning point is a local switching of the patching data, and the basis of Killing vectors
adapted to the rod changes by a constant matrix \cite[\S\S3.2, 4.2 and Thm.~4.1]{Fletcher1990}, \cite[(20), (21)]{FW1990},
\cite[\S3.4, (27)]{Tod2024}. For the Chen--Teo instanton the patching matrix is rational, with three poles \cite[\S3.1, (3.9)]{DT2024}.
These works treat one turning point at a time. Dunajski, Mason and Tod are developing a twistor construction of toric
gravitational instantons \cite{DMT2026}. The Hankel families give the rod vectors of all rods in closed form, and we are not
aware of a twistor counterpart of the total turning of Proposition~\ref{prop:Aturning}.
\end{remark}

\section{The ALE soliton metrics: the closed formula}\label{sec:ale}
This section proves the closed formula for the ALE soliton matrix (Theorem~\ref{thm:CF}) and removes the phantom solitons from it
(Lemma~\ref{lem:Cnull}).

\emph{Notation.} As in \S\ref{subsec:ALEsolitons}, there are $N=2m+1$ nodes $w_1<\dots<w_N$ with parameters
$c_k=[\alpha_k:\beta_k]\in\RP^1$, $G$ is the ALE soliton matrix \eqref{eq:bordered}, and $\mathcal A$, $\mathcal A^\circ$ and $c^*$ are the
admissible set, the positive cone and the base point of Definition~\ref{def:admissible}; the parameters $c_k\in\{0,\infty\}$ are
the phantom solitons. Hankel determinants are as in \S\ref{subsec:heine}.

\subsection{The reduced form}\label{subsec:Creduced}
Put
\begin{equation}\label{eq:yk}
y_k=\frac{\mu_k}\rho,\qquad x_k=-\frac1{y_k}\qquad(1\le k\le N),
\end{equation}
so that $\mub_k=-\rho^2/\mu_k=\rho x_k$. Since $\mu_w$ is strictly increasing in $w$, as $\partial_w\mu_w=1-(z-w)/r_w>0$, $0<y_1<\dots<y_N$ and $x_1<\dots<x_N<0$. We write
$y=(y_1,\dots,y_N)$ and $\mathsf D_y=\diag(y)$; $c$, $\alpha$ and $\beta$ are the column vectors of the soliton parameters. The \emph{reduced
matrix} $M$ and the two \emph{border vectors} $\hat c_0,\hat c_1\in\R^N$ are
\begin{equation}\label{eq:Mproj} M_{kl}=\frac{\alpha_k\beta_l+\beta_k\alpha_l}{1+y_ky_l},\qquad \hat c_0=\beta,\qquad \hat c_1=\alpha ; \end{equation}
the border vectors are those of \eqref{eq:bordered} (\S\ref{subsec:ALEsolitons}). For finite $c$ and $(\alpha_k,\beta_k)=(1,c_k)$ this is $M_{kl}=(c_k+c_l)/(1+y_ky_l)$, $\hat c_0=c$, $\hat c_1=(1,\dots,1)^T$. Let $\hat G$ be
the symmetric $2\times2$ matrix with entries
\begin{equation}\label{eq:Ghat}
\hat G_{ab}=\frac1{\prod_ky_k}\Bigl(\hat\eta_{ab}-\tfrac12\,\hat c_a^{\,T}M^{-1}\hat c_b\Bigr),\qquad a,b\in\{0,1\},
\end{equation}
wherever $\det M\ne0$; it depends only on $(y,c)$.

\begin{lemma}[reduced form]\label{lem:Creduced}
Let $M$, $\hat c_0$, $\hat c_1$ and $\hat G$ be as in \eqref{eq:Mproj} and \eqref{eq:Ghat}. Then $D=\det K=(2\rho^2)^{-N}\prod_ky_k^2\det M$ for all
$c\in(\RP^1)^N$, in any representatives. Wherever $\det M\ne0$, $G=\rho\hat G$ and $\det G=\rho^2/4$, and $\hat G$ does not change when $(\alpha_k,\beta_k)$ is
rescaled.
\end{lemma}
\begin{proof}
Since $\mub_k=-\rho/y_k$, we have $1/(\rho^2+\mub_k\mub_l)=y_ky_l/(\rho^2(1+y_ky_l))$. So $K=(2\rho^2)^{-1}\mathsf D_yM\mathsf D_y$, which gives the
formula for $D$. The border entries of \eqref{eq:bordered} are $\hat c_{a,k}/\mub_k=-y_k\hat c_{a,k}/\rho$. Where $\det M\ne0$,
$K^{-1}=2\rho^2\mathsf D_y^{-1}M^{-1}\mathsf D_y^{-1}$, so the sum in \eqref{eq:bordered} is
\begin{equation}
\tfrac14\sum_{k,l}\frac{\hat c_{a,k}(K^{-1})_{kl}\hat c_{b,l}}{\mub_k\mub_l}=\frac1{4\rho^2}\,\hat c_a^T\mathsf D_yK^{-1}\mathsf D_y\hat c_b=\frac12\hat c_a^TM^{-1}\hat c_b,
\end{equation}
and $f=\rho/\prod_ky_k$; so $G=\rho\hat G$. Since $G$ is homogeneous of degree $0$ in each pair (\S\ref{subsec:ALEsolitons}), so is $\hat G$.

For the determinant, write $\hat C=(\hat c_0,\hat c_1)$ for the $N\times2$ matrix of the border vectors. By \eqref{eq:schur}, evaluated in both orders (Sylvester's determinant identity),
\begin{equation}
\det(\hat\eta-\frac12\hat C^TM^{-1}\hat C)=\det\hat\eta\,\det(M-\frac12\hat C\hat\eta^{-1}\hat C^T)/\det M.
\end{equation}
Here $\frac12\hat C\hat\eta^{-1}\hat C^T=(\alpha_k\beta_l+\beta_k\alpha_l)_{kl}$, and
\begin{equation} M-(\alpha_k\beta_l+\beta_k\alpha_l)_{kl}=-\Bigl(\frac{(\alpha_k\beta_l+\beta_k\alpha_l)y_ky_l}{1+y_ky_l}\Bigr)_{kl}=-\mathsf D_yM\mathsf D_y . \end{equation}
Since $\det\hat\eta=-\frac14$ and $N$ is odd, $\prod_ky_k^2\det\hat G=\det(\hat\eta-\frac12\hat C^TM^{-1}\hat C)=\det\hat\eta\,\det(-\mathsf D_yM\mathsf D_y)/\det M=-\frac14(-1)^N\prod_ky_k^2=\frac14\prod_ky_k^2$.
\end{proof}

So $G$ depends on $(\rho,z)$ only through the ratios $y_k$. For a set $S$ of indices, $1_S$ denotes the vector with entries $1$ on $S$ and $0$
elsewhere, and $\prod_Sy_k=\prod_{k\in S}y_k$. The next lemma fixes the sign of $G$ in Theorem~\ref{thm:CF}.
\begin{lemma}[base point]\label{lem:Cbase}
At the base point $c=c^*$, in the basis $e_\pm$ with the null components $(1,\pm1)$, $(-1)^{m+1}G$ is the static seed
$\diag(e^{2U},\rho^2e^{-2U})$ of \eqref{eq:staticseed}; in particular it is positive definite.
\end{lemma}
\begin{proof}
Let $A_+=\{k:c^*_k=1\}$ (the odd $k$, $|A_+|=m+1$) and $A_-=\{k:c^*_k=-1\}$. We show
\begin{equation} \hat G(e_+,e_+)=(-1)^{m+1}\frac{\prod_{A_+}y_k}{\prod_{A_-}y_k},\qquad \hat G(e_-,e_-)=(-1)^{m+1}\frac{\prod_{A_-}y_k}{\prod_{A_+}y_k},\qquad
\hat G(e_+,e_-)=0 . \end{equation}
$M_{kl}=0$ if $c_k\ne c_l$, so $M=2C_{A_+}\oplus(-2)C_{A_-}$ with $C_S=[1/(1+y_ky_l)]_{k,l\in S}$. Each $C_S$ is invertible:
$1/(1+y_ky_l)=y_k^{-1}/(y_l-x_k)$ by \eqref{eq:yk}, so $C_S=\mathsf D_S^{-1}[1/(y_l-x_k)]_{k,l\in S}$ with $\mathsf D_S=\diag(y_k)_{k\in S}$, and by
Cauchy's determinant formula
$\det[1/(y_l-x_k)]_{k,l\in S}=\prod_{k<l}(y_l-y_k)(x_k-x_l)\big/\prod_{k,l}(y_l-x_k)$ (products over $S$), which is non-zero since the
$x_k$ are distinct, the $y_k$ are distinct and $x_k<0<y_l$. In particular $\det M\ne0$ at $c^*$, so $\hat G$ is defined there.
\begin{itemize}[leftmargin=*]
\item Contracting with $e_+$ gives $\hat c_0+\hat c_1=2\cdot1_{A_+}$ and $\hat\eta(e_+,e_+)=1$, so
$\prod_ky_k\cdot\hat G(e_+,e_+)=1-1_{A_+}^TC_{A_+}^{-1}1_{A_+}$.
\item Likewise $\hat c_0-\hat c_1=-2\cdot1_{A_-}$ and $\hat\eta(e_-,e_-)=-1$, so $\prod_ky_k\cdot\hat G(e_-,e_-)=-1+1_{A_-}^TC_{A_-}^{-1}1_{A_-}$.
\item $\hat G(e_+,e_-)=0$, because $\hat\eta(e_+,e_-)=0$ and $M$ is block diagonal.
\end{itemize}
By the matrix determinant lemma \eqref{eq:mdl}, $1-1^TC_S^{-1}1=\det(C_S-11^T)/\det C_S$, and $C_S-11^T=-\mathsf D_SC_S\mathsf D_S$, as in Lemma~\ref{lem:Creduced}. So $1-1^TC_S^{-1}1=(-1)^{|S|}\prod_Sy_k^2$, which gives the formulas. Finally
$G=\rho\hat G$ (Lemma~\ref{lem:Creduced}), and, since $|A_+|-|A_-|=1$ and $\epsilon_k=1$ exactly on $A_+$,
\begin{equation}
\rho\,\frac{\prod_{A_+}y_k}{\prod_{A_-}y_k}=\frac{\prod_{A_+}\mu_k}{\prod_{A_-}\mu_k}=e^{2U},\qquad
\rho\,\frac{\prod_{A_-}y_k}{\prod_{A_+}y_k}=\rho^2\,\frac{\prod_{A_-}\mu_k}{\prod_{A_+}\mu_k}=\rho^2e^{-2U}.
\end{equation}
\end{proof}

\subsection{Interpolation determinants}\label{subsec:CW}
For $2N$ points $(p_i,v_i)$, $p_i\in\R\cup\{\infty\}$, $v_i=[\alpha_i:\beta_i]$, let $W$ be the determinant of the matrix with
the rows $(\alpha_ip_i^l)_{l<N}\mid(\beta_ip_i^l)_{l<N}$. For $p_i=\infty$ the row holds the coefficients of $p^{N-1}$. The same formula
defines $W$ for complex points $p_i$ and complex pairs $(\alpha_i,\beta_i)$. Put
\begin{equation}\label{eq:Pset} \mathcal P=\bigl((x_1,-c_1),\dots,(x_N,-c_N),(y_1,c_1),\dots,(y_N,c_N)\bigr) \end{equation}
and, for $\lambda\in\R$,
\begin{equation}\label{eq:Pext} \mathcal P^{\rm ext}(\lambda)=\bigl((\infty,-\lambda),(x_1,-c_1),\dots,(x_N,-c_N),(0,\lambda),(y_1,c_1),\dots,(y_N,c_N)\bigr). \end{equation}
Here $-c_k=[\alpha_k:-\beta_k]$, and the value $\lambda$ is $[1:\lambda]$.

\begin{proposition}[interpolation determinants]\label{prop:CW}
Let $M$ and $\hat G$ be as in \eqref{eq:Mproj} and \eqref{eq:Ghat}, and let $W(\mathcal P)$ and
$W(\mathcal P^{\rm ext}(\lambda))$ be the interpolation determinants of the point sets \eqref{eq:Pset} and \eqref{eq:Pext}.
\begin{enumerate}[leftmargin=*,label=\textup{(\roman*)}]
\item For any number $N$ of nodes, odd or even, and all $c\in(\RP^1)^N$,
\begin{equation}
\det M=(-1)^{N(N-1)/2}\,W(\mathcal P)\big/\bigl(\prod_ky_k\prod_{k,l}(y_l-x_k)\bigr).
\end{equation}
\item For all $c\in(\RP^1)^N$ with $\det M\ne0$ and all real $\lambda$,
\begin{equation}
(1,\lambda)\,\hat G\,(1,\lambda)^T=W(\mathcal P^{\rm ext}(\lambda))/(2\,W(\mathcal P)).
\end{equation}
\end{enumerate}
\end{proposition}
\begin{proof}
(i) Since
\begin{equation}
(c_k+c_l)/(1+y_ky_l)=y_k^{-1}(c_l-(-c_k))/(y_l-x_k),
\end{equation}
for finite $c$ this is the identity of \cite[Thm.~1.1(a), (1.3)]{IOTZ2006} with $p=q=0$ (Appendix~\ref{app:identities}), applied to the
points $(x_k,-c_k)$ and $(y_l,c_l)$ in this order. Both sides are polynomials in $(\alpha,\beta)$, homogeneous of degree $2$ in each pair
$(\alpha_k,\beta_k)$: $\det M$ because the pair enters the row $k$ and the column $k$ of $M$, and $W(\mathcal P)$ because it enters the two
rows of the points $x_k$ and $y_k$. By homogeneity the identity, proved for finite $c$, holds wherever all $\alpha_k\ne0$, a dense set;
so it holds for all $(\alpha,\beta)$.

(ii) The row of a point with homogeneous coordinates $(X:Z)$ and the pair $[\alpha:\beta]$ is
$\bigl((\alpha X^lZ^{N-1-l})_{l<N}\mid(\beta X^lZ^{N-1-l})_{l<N}\bigr)$; for $(p:1)$ it is the row of $p$, and for $(1:0)$ that of $\infty$.
Multiply the row of the point $x_k=-1/y_k$ by $y_k^{N-1}$; it becomes the row of $(-1:y_k)$. Let $W_h(\mathcal P)=\prod_ky_k^{N-1}W(\mathcal P)$
be the determinant with these homogeneous rows. Since $\prod_ky_k\prod_{k,l}(y_l-x_k)=\prod_ky_k^{1-N}\prod_{k,l}(1+y_ky_l)$, (i) says
\begin{equation}\label{eq:CWhom}
\det M\prod_{k,l}(1+y_ky_l)=(-1)^{N(N-1)/2}\,W_h(\mathcal P).
\end{equation}
Both sides are polynomials in $(y,\alpha,\beta)$ (each term of $\det M$ has a denominator that divides $\prod_{k,l}(1+y_ky_l)$), and they
agree for all $0<y_1<\dots<y_N$ (the proof of (i) uses only these inequalities); so \eqref{eq:CWhom} holds for all real $y$, also if some $y_k=0$, and for any number of nodes.

Let $\mathbf u=\beta+\lambda\alpha$ and $\tilde M=\bigl(\begin{smallmatrix}2\lambda&\mathbf u^T\\\mathbf u&M\end{smallmatrix}\bigr)$. This is the
matrix \eqref{eq:Mproj} of $N+1$ nodes with $y_0=0$ and the pair $(\alpha_0,\beta_0)=(1,\lambda)$, put first; its factors $1+y_0y_l$ are $1$.
In the point set $\tilde{\mathcal P}$ of these $N+1$ nodes, the homogeneous point of $x_0$ is $(-1:y_0)=(-1:0)$, whose row is $(-1)^N$ times the
row of $(\infty,-\lambda)$, the row of $y_0$ is that of $(0,\lambda)$, and for $k\ge1$ the row of $(-1:y_k)$, now of degree $N$, is $y_k^N$
times the row of $x_k$. So $W_h(\tilde{\mathcal P})=(-1)^N\prod_ky_k^N\,W(\mathcal P^{\rm ext}(\lambda))$, with the rows in the order of
$\mathcal P^{\rm ext}(\lambda)$, and \eqref{eq:CWhom} for the $N+1$ nodes gives
\begin{equation}
\det\tilde M\prod_{k,l}(1+y_ky_l)=(-1)^{N(N+1)/2+N}\prod_ky_k^N\,W(\mathcal P^{\rm ext}(\lambda)).
\end{equation}
By \eqref{eq:Ghat}, $(1,\lambda)\hat\eta(1,\lambda)^T=\lambda$ and the Schur complement \eqref{eq:schur} with the pivot $M$,
$\prod_ky_k\,(1,\lambda)\hat G(1,\lambda)^T=\lambda-\frac12\mathbf u^TM^{-1}\mathbf u=\det\tilde M/(2\det M)$. Divide the last display by
\eqref{eq:CWhom}; the signs combine to $(-1)^{N(N+1)/2+N-N(N-1)/2}=(-1)^{2N}=1$, and (ii) follows, for all $(\alpha,\beta)$ with
$\det M\ne0$; then $W(\mathcal P)\ne0$ by (i).
\end{proof}

\subsection{Heine sums in projective form and split determinants}\label{subsec:split}
The interpolation determinants will be evaluated as Heine sums \eqref{eq:heine} of the weights $a_k=c_k/P'(w_k)$. To allow the values
$c_k=0$ and $c_k=\infty$, and complex multipliers of the weights, we write these sums in projective form.

\emph{Projective Heine sums.} Let $w_1,\dots,w_N$ be distinct atoms, real or complex, carrying pairs $(\alpha_l,\hat b_l)\ne(0,0)$, the
\emph{projective weights} $[\alpha_l:\hat b_l]$; we write $\omega$ for this measure in projective form. For $q=(q_1,\dots,q_N)\in\C^N$ and
$0\le s\le N$ put
\begin{equation}\label{eq:Hproj}
\mathbf H_s(q\omega)=\sum_{|T|=s}\ \prod_{l\notin T}\alpha_l\prod_{l\in T}\hat b_lq_l\ \Delta_T^2,\qquad
\mathbf H_0(q\omega)=\prod_l\alpha_l,\qquad \mathbf H_s(q\omega)=0\ \ (s<0\text{ or }s>N),
\end{equation}
the sum over the sets $T$ of $s$ atoms, with $\Delta_T^2$ as in \eqref{eq:PS}. Thus $q\omega$ is the measure with the projective weights
$[\alpha_l:\hat b_lq_l]$, as in the convention of Lemma~\ref{lem:Achristoffel}; we write $\mathbf H_s(\omega)$ for $q\equiv1$, and $\mathbf r\omega$,
$\omega/\mathbf r$, $\varpi\omega$, $\varpi^{-1}\omega$ for $q_l=r_l$, $1/r_l$, $\varpi_l$, $\varpi_l^{-1}$, where $\mathbf r$ is the function
$l\mapsto r_l$. $\mathbf H_s(q\omega)$ is a polynomial in the pairs, homogeneous of degree $1$ in each pair, in which the atoms enter only through
$\Delta_T^2$; and $\mathbf H_s(-q\omega)=(-1)^s\mathbf H_s(q\omega)$. If all $\alpha_l\ne0$, then by Heine's formula \eqref{eq:heine},
$\mathbf H_s(q\omega)=\prod_l\alpha_l\,H_s(q\tilde\omega)$
for the measure $\tilde\omega$ with the weights $\hat b_l/\alpha_l$.

For the ALE soliton the atoms are the nodes $w_1<\dots<w_N$ and the pairs are $(\alpha_l,\hat b_l)$ with $\hat b_l=\beta_l/P'(w_l)$, so that
$[\alpha_l:\hat b_l]$ is the weight $a_l=c_l/P'(w_l)$ in projective form. We use the \emph{standard representatives}
$(\alpha_k,\beta_k)=(1,c_k)$ for finite $c_k$ (so $(1,0)$ for $c_k=0$) and $(0,1)$ for $c_k=\infty$, those of \S\ref{subsec:ALEsolitons}; in
them $\hat b_k=a_k$ for finite $c_k$, and $\mathbf H_s(q\omega)=H_s(q\omega)$ when all $c_k$ are finite. Unless said otherwise, the polynomial
objects below are computed in the standard representatives. If condition (a) of Definition~\ref{def:admissible} holds, there are also
\emph{non-negative representatives}, with $\alpha_k\ge0$ and $\hat b_k\ge0$ for all $k$: the standard ones for finite $c_k$, and $(0,P'(w_k))$,
i.e.\ $\hat b_k=1$, for $c_k=\infty$. Replacing $(\alpha_k,\beta_k)$ by $\lambda_k(\alpha_k,\beta_k)$, $\lambda_k\in\R\setminus\{0\}$, multiplies every
$\mathbf H_s(q\omega)$ by $\prod_k\lambda_k$, and so a product of two of them by $\prod_k\lambda_k^2>0$. So the sign of such a product, and a
quotient of two quantities that scale alike, do not depend on the representatives.
Below we omit ``in any representatives'' for such statements; signs are proved in the non-negative representatives. These conventions are
used in \S\S\ref{sec:alerods} and~\ref{sec:far} as well.

\begin{lemma}[positivity and reduction]\label{lem:Hpos}
Let $w_1,\dots,w_N$ be distinct real atoms carrying real pairs $(\alpha_l,\hat b_l)\ne(0,0)$.
\begin{enumerate}[leftmargin=*,label=\textup{(\alph*)}]
\item Suppose that $\alpha_l,\hat b_l\ge0$ for all $l$, let $z_0=\#\{l:\hat b_l=0\}$ and $z_\infty=\#\{l:\alpha_l=0\}$, and let $q_1,\dots,q_N>0$. Then
every term of the sum \eqref{eq:Hproj} is $\ge0$, and $\mathbf H_s(q\omega)>0$ if and only if $z_\infty\le s\le N-z_0$.
\item Let $Z$ and $I$ be disjoint sets of atoms with $\hat b_l=0$ for $l\in Z$ and $\alpha_h=0$ for $h\in I$, let $\mathsf R$ be the set of the other atoms,
and let $\omega^I$ be the measure on $\mathsf R$ with the pairs $\bigl(\alpha_l,\,\hat b_l\prod_{h\in I}(w_l-w_h)^2\bigr)$. Then for every $q\in\C^N$ and every $s$,
\begin{equation}
\mathbf H_s(q\omega)=\kappa_I\,q_I\,\mathbf H^{\mathsf R}_{s-|I|}(q\omega^I),\qquad q_I=\prod_{h\in I}q_h,\qquad
\kappa_I=\prod_{l\in Z}\alpha_l\prod_{h\in I}\hat b_h\,\Delta_I^2\ne0,
\end{equation}
where $\mathbf H^{\mathsf R}$ is the sum \eqref{eq:Hproj} over the atoms of $\mathsf R$, and $q$ on the right is restricted to $\mathsf R$.
\item Let the atoms be the nodes of an ALE soliton with $c\in\mathcal A$, in non-negative representatives, and let $q,q'\in(0,\infty)^N$. Then
$\mathbf H_m(q\omega)>0$ and $\mathbf H_{m+1}(q\omega)>0$; so $\mathbf H_m(q\omega)\mathbf H_{m+1}(q'\omega)>0$ in any real representatives.
\end{enumerate}
\end{lemma}
\begin{proof}
(a) The term of $T$ is a product of non-negative factors and of $\Delta_T^2>0$. It is positive if and only if $\alpha_l>0$ for $l\notin T$ and
$\hat b_l>0$ for $l\in T$, i.e.\ if and only if $T$ contains the $z_\infty$ atoms with $\alpha_l=0$ and none of the $z_0$ atoms with $\hat b_l=0$; these
two sets are disjoint, since $(\alpha_l,\hat b_l)\ne(0,0)$. Such a $T$ with $|T|=s$ exists if and only if $z_\infty\le s\le N-z_0$.

(b) Since the pairs are non-zero, $\alpha_l\ne0$ on $Z$ and $\hat b_h\ne0$ on $I$, so $\kappa_I\ne0$. A term $T$ of $\mathbf H_s(q\omega)$ is non-zero only if
$T\supset I$ and $T\cap Z=\emptyset$ (otherwise it has a factor $\alpha_h=0$, resp.\ $\hat b_l=0$). For $T=I\cup S$ with $S\subset\mathsf R$, $\Delta_T^2=\Delta_I^2\,\Delta_S^2\prod_{l\in S,\,h\in I}(w_l-w_h)^2$ and
$\prod_{l\notin T}\alpha_l=\prod_{l\in Z}\alpha_l\prod_{l\in\mathsf R\setminus S}\alpha_l$, so the term of $T$ is $\kappa_Iq_I$ times the term of $S$ in
$\mathbf H^{\mathsf R}_{|S|}(q\omega^I)$, with $|S|=s-|I|$. Sum over $S$.

(c) By conditions (b) and (c) of Definition~\ref{def:admissible}, $z_0,z_\infty\le m$, so (a) gives the first assertion, and the second follows by the
convention above.
\end{proof}

\emph{Split determinants.} For points $\xi_1,\dots,\xi_N$, real or complex, pairs $(\alpha'_k,\beta'_k)$ and orders $s'+s=N$, let
$S_{s',s}(\xi;\alpha',\beta')$ be the determinant with the rows $\bigl((\alpha'_k\xi_k^i)_{i<s'}\mid(\beta'_k\xi_k^j)_{j<s}\bigr)$, $k=1,\dots,N$, and put
$\Delta(\xi)=\prod_{k<l}(\xi_l-\xi_k)$. Expand $S_{s',s}$ by Laplace along its first $s'$ columns. The term in which the rows of
an $s$-set $T$ meet the second block has the sign $(-1)^{I(T)}$,
\begin{equation}
I(T)=\#\{(l,l'):l\in T,\ l'\notin T,\ l<l'\},
\end{equation}
and equals $(-1)^{I(T)}\Delta(\xi_{T^c})\Delta(\xi_T)\prod_{l'\notin T}\alpha'_{l'}\prod_{l\in T}\beta'_l$, where $\Delta(\xi_S)$ is the
Vandermonde product of the points of $S$ in their order. Since
\begin{equation}
\Delta(\xi)=(-1)^{I(T)}\Delta(\xi_{T^c})\Delta(\xi_T)\prod_{l\in T,l'\notin T}(\xi_l-\xi_{l'}),
\end{equation}
we get, for distinct points,
\begin{equation}\label{eq:Csplit}
S_{s',s}(\xi;\alpha',\beta')=\Delta(\xi)\sum_{|T|=s}\frac{\prod_{l'\notin T}\alpha'_{l'}\prod_{l\in T}\beta'_l}{\prod_{l\in T,\,l'\notin T}(\xi_l-\xi_{l'})} .
\end{equation}

\begin{lemma}[split determinants in Heine form]\label{lem:split}
Let $w_1,\dots,w_N$ be distinct, real or complex, $P=\prod_k(x-w_k)$, let $\xi_k=\theta_1w_k+\theta_0$ with $\theta_1\ne0$, and let $s'+s=N$. Then
\begin{equation}
S_{s',s}(\xi;\alpha',\beta')=(-1)^{s(s-1)/2}\,\theta_1^{N(N-1)/2-ss'}\,\Delta(w)\,\mathbf H_s[\alpha',\beta'],
\end{equation}
where $\mathbf H_s[\alpha',\beta']$ is the projective Heine sum \eqref{eq:Hproj} of order $s$ of the atoms $w_l$ with the pairs $(\alpha'_l,\beta'_l/P'(w_l))$
and with all multipliers equal to $1$.
\end{lemma}
\begin{proof}
In \eqref{eq:Csplit}, $\Delta(\xi)=\theta_1^{N(N-1)/2}\Delta(w)$ and $\prod_{l\in T,l'\notin T}(\xi_l-\xi_{l'})=\theta_1^{ss'}\prod_{l\in T,l'\notin T}(w_l-w_{l'})$.
Since $\prod_{l\in T}P'(w_l)=\prod_{l\in T,\,l'\notin T}(w_l-w_{l'})\cdot(-1)^{s(s-1)/2}\Delta_T^2$ (the second factor is the product of
$w_l-w_{l'}$ over the ordered pairs of distinct $l,l'\in T$), insert this into \eqref{eq:Csplit}.
\end{proof}

\subsection{The pair involution}\label{subsec:PI}
The interpolation determinants of \S\ref{subsec:CW} factor into two determinants of half the size. The reason is the involution
$\mathcal I(A,B)(p)=p^{N-1}\bigl(A(-1/p),-B(-1/p)\bigr)$ on pairs of polynomials of degree $\le N-1$: the substitution $p\mapsto-1/p$
exchanges the two points $x_k=-1/y_k$ and $y_k$ of each node, and the sign change of $B$ exchanges their values $-c_k$ and $c_k$,
so $\mathcal I$ carries the functional of one row of a node to a multiple of the functional of the other row, and likewise for the
rows of $0$ and $\infty$ in $\mathcal P^{\rm ext}$. In a basis adapted to the two eigenspaces of $\mathcal I$, each of dimension $N$, the sum
and the difference of the two rows of a node are supported on one eigenspace each, and the determinant becomes block diagonal. The
lemma below carries this out with explicit bases and keeps track of the constants.

\begin{lemma}[pair involution]\label{lem:PI}
\begin{enumerate}[leftmargin=*,label=\textup{(\alph*)}]
\item Let $N=2m+1$, $y_1,\dots,y_N\in\C\setminus\{0\}$ and $(\alpha_k,\beta_k)\in\C^2$. Put $\sigma_k=y_k-1/y_k$ and $t_k=y_k+1/y_k$, and let $W$
be the interpolation determinant of the $2N$ points
\begin{equation}
(-1/y_1,[\alpha_1:-\beta_1]),\dots,(-1/y_N,[\alpha_N:-\beta_N]),\ (y_1,[\alpha_1:\beta_1]),\dots,(y_N,[\alpha_N:\beta_N])
\end{equation}
in this order. Then $W=2D_{\rm I}D_{\rm II}$, where
\begin{equation}\begin{gathered}
D_{\rm I}=\det\Bigl[(\alpha_k\sigma_k^i)_{0\le i\le m}\ \Big|\ (\beta_kt_k\sigma_k^j)_{0\le j\le m-1}\Bigr]_{k=1}^N,\\
D_{\rm II}=\det\Bigl[(\alpha_kt_k\sigma_k^j)_{0\le j\le m-1}\ \Big|\ (\beta_k\sigma_k^i)_{0\le i\le m}\Bigr]_{k=1}^N .
\end{gathered}\end{equation}
\item Let $N'=N+1=2m+2$, $y_1,\dots,y_N>0$, $(\alpha_k,\beta_k)\in\C^2$ and $\lambda\in\C$, and let $W^{\rm ext}(\lambda)$ be the interpolation
determinant of the $2N'$ points
\begin{equation}\begin{gathered}
(\infty,[1:-\lambda]),\ (-1/y_1,[\alpha_1:-\beta_1]),\dots,(-1/y_N,[\alpha_N:-\beta_N]),\\
(0,[1:\lambda]),\ (y_1,[\alpha_1:\beta_1]),\dots,(y_N,[\alpha_N:\beta_N]).
\end{gathered}\end{equation}
Put $\varpi_k=(y_k-i)/(y_k+i)$ and, for $\delta\in\{0,1\}$, $J_\delta=\{0\le j\le N'-1:j\equiv\delta\bmod2\}$, a set of $m+1$ elements. Let
$D_\delta(\lambda)$ be the determinant of the $N'\times N'$ matrix with the rows
\begin{equation}
\bigl((\alpha_k\varpi_k^j)_{j\in J_\delta}\ \big|\ (\beta_k\varpi_k^j)_{j\in J_{1-\delta}}\bigr)\ \ (k=1,\dots,N),\qquad\bigl(1,\dots,1\ \big|\ -\lambda,\dots,-\lambda\bigr).
\end{equation}
Then $W^{\rm ext}(\lambda)=K_1\,D_0(\lambda)D_1(\lambda)$, with a constant $K_1\ne0$ that depends only on $m$ and $y$, not on
$(\alpha,\beta)$ or $\lambda$.
\end{enumerate}
\end{lemma}
\begin{proof}
The plan is the one described above: write the interpolation determinant in a basis adapted to the eigenspaces of $\mathcal I$ (in (a) the
basis $g_i,h_j$ below, in (b) a Cayley-transformed basis), take sums and differences of the two rows of each node, and read off the
two blocks and the constant. In both parts let $d$ be the number of pairs of points, $d=N$ in (a) and $d=N'$ in (b), and $V$ the space
of pairs $(A,B)$ of polynomials of degree $\le d-1$. The row of a point $(p,[\alpha:\beta])$ is the
functional $(A,B)\mapsto\alpha A(p)+\beta B(p)$ written in the monomial basis $(p^l,0)$, $(0,p^l)$ of $V$, and the row of
$(\infty,[\alpha:\beta])$ is the functional that takes $\alpha$ times the coefficient of $p^{d-1}$ in $A$ plus $\beta$ times that in $B$. So
$W=\det F$, where $F$ is the matrix of the $2d$ functionals in the monomial basis. For another basis of $V$, with the matrix $C$ (its
columns are the new basis vectors in monomial coordinates), the matrix of the functionals in the new basis is $FC$, and
$\det(FC)=W\det C$.

(a) \emph{The basis.} Let $\sigma=p-p^{-1}$, $t=p+p^{-1}$ and
\begin{equation}
g_i=p^m\sigma^i\ (0\le i\le m),\qquad h_j=p^mt\sigma^j\ (0\le j\le m-1).
\end{equation}
These are polynomials of degree $\le2m$: the exponents of $g_i$ lie in $[m-i,m+i]$, those of $h_j$ in $[m-j-1,m+j+1]$. Let $C_A$ be
their matrix in the monomial basis of the polynomials of degree $\le 2m$, with the columns ordered $g_0,(g_1,h_0),\dots,(g_m,h_{m-1})$ and the
monomials ordered $p^m,(p^{m+1},p^{m-1}),\dots,(p^{2m},p^0)$. For $1\le i\le m$, $g_i$ and $h_{i-1}$ have the extreme terms $p^{m+i}+(-1)^ip^{m-i}$,
resp.\ $p^{m+i}+(-1)^{i-1}p^{m-i}$, and all their other exponents lie strictly between $m-i$ and $m+i$; and $g_0=p^m$. So $C_A$ is block upper
triangular, with the diagonal blocks $1$ and $\bigl(\begin{smallmatrix}1&1\\(-1)^i&(-1)^{i-1}\end{smallmatrix}\bigr)$, of determinant
$2(-1)^{i-1}$; hence $\det C_A=\pm2^m$. Use the basis $(g_i,0)$, $(h_j,0)$, $(0,g_i)$, $(0,h_j)$ of $V$, in this order; its matrix $C$ is block
diagonal with two blocks, each equal to $C_A$ up to a reordering of its rows and columns, so $\det C=(\pm\det C_A)^2=4^m$.

\emph{The rows in the new basis.} Since $\sigma(y_k)=\sigma_k$, $t(y_k)=t_k$, $\sigma(-1/y_k)=\sigma_k$, $t(-1/y_k)=-t_k$ and $(-1/y_k)^m=(-1)^my_k^{-m}$, the rows of the
two points of node $k$ are, in the four column blocks $(g_i,0)$, $(h_j,0)$, $(0,g_i)$, $(0,h_j)$,
\begin{equation}
\begin{gathered}
R_{y_k}=y_k^m\bigl(\alpha_k\sigma_k^i\mid\alpha_kt_k\sigma_k^j\mid\beta_k\sigma_k^i\mid\beta_kt_k\sigma_k^j\bigr),\\
R_{x_k}=(-1)^my_k^{-m}\bigl(\alpha_k\sigma_k^i\mid-\alpha_kt_k\sigma_k^j\mid-\beta_k\sigma_k^i\mid\beta_kt_k\sigma_k^j\bigr).
\end{gathered}
\end{equation}
With $\varkappa_k=(-1)^my_k^{2m}$,
\begin{equation}
R_{y_k}+\varkappa_kR_{x_k}=2y_k^m\bigl(\alpha_k\sigma_k^i\mid0\mid0\mid\beta_kt_k\sigma_k^j\bigr),\qquad
R_{y_k}-\varkappa_kR_{x_k}=2y_k^m\bigl(0\mid\alpha_kt_k\sigma_k^j\mid\beta_k\sigma_k^i\mid0\bigr).
\end{equation}

\emph{Signs.} Reorder the rows of $FC$ as $(x_1,y_1,\dots,x_N,y_N)$; replace each pair $(R_{x_k},R_{y_k})$ by
$(R_{y_k}+\varkappa_kR_{x_k},R_{y_k}-\varkappa_kR_{x_k})$, a row operation of determinant $2\varkappa_k$; reorder the rows as (all sums; all
differences), which has the same sign as the first reordering; and move the column block $(0,h_j)$, of $m$ columns, in front of the $2m+1$
columns of $(h_j,0)$ and $(0,g_i)$, with the sign $(-1)^{m(2m+1)}=(-1)^m$. The matrix is now block diagonal, with the blocks
$\bigl(2y_k^m(\alpha_k\sigma_k^i\mid\beta_kt_k\sigma_k^j)\bigr)_k$ and $\bigl(2y_k^m(\alpha_kt_k\sigma_k^j\mid\beta_k\sigma_k^i)\bigr)_k$. So, since $(2y_k^m)^2/(2\varkappa_k)=2(-1)^m$,
\begin{equation}
\det(FC)=(-1)^m\bigl(2(-1)^m\bigr)^ND_{\rm I}D_{\rm II}=(-1)^{m(N+1)}2^ND_{\rm I}D_{\rm II}=2^ND_{\rm I}D_{\rm II},
\end{equation}
as $m(N+1)=2m(m+1)$ is even. Hence $W=2^ND_{\rm I}D_{\rm II}/4^m=2D_{\rm I}D_{\rm II}$.

(b) Use the basis $\chi_j=(p-i)^j(p+i)^{N'-1-j}$, $0\le j\le N'-1$, of the polynomials of degree $\le N'-1$ in each component; it is a basis,
since $\chi_j=(p+i)^{N'-1}\varsigma^j$ with the Cayley coordinate $\varsigma(p)=(p-i)/(p+i)$. Let $C_1$ be the matrix of $(\chi_j)_j$ in the
monomial basis of the polynomials of degree $\le N'-1$, and $C$ the matrix of the basis $(\chi_j,0)$, $(0,\chi_j)$ of $V$; then $C=C_1\oplus C_1$,
and $\det C=(\det C_1)^2\ne0$ depends only on $N'$. The rows in the new basis, with the columns $(\chi_j,0)\mid(0,\chi_j)$, are:
\begin{itemize}[leftmargin=*]
\item point $y_k$: $(y_k+i)^{N'-1}\bigl(\alpha_k\varpi_k^j\mid\beta_k\varpi_k^j\bigr)$, since $\varsigma(y_k)=\varpi_k$;
\item point $-1/y_k$: since $\varsigma(-1/y_k)=(1+iy_k)/(1-iy_k)=-\varpi_k$ and $-1/y_k+i=y_k^{-1}\,i\,(y_k+i)$, the row is
$y_k^{1-N'}i^{N'-1}(y_k+i)^{N'-1}\bigl(\alpha_k(-\varpi_k)^j\mid-\beta_k(-\varpi_k)^j\bigr)$;
\item point $0$: $\chi_j(0)=(-i)^ji^{N'-1-j}=i^{N'-1}(-1)^j$, so the row is $i^{N'-1}\bigl((-1)^j\mid\lambda(-1)^j\bigr)$;
\item point $\infty$: $\chi_j$ is monic of degree $N'-1$, so the row is $\bigl(1\mid-\lambda\bigr)$ in all entries.
\end{itemize}
Replace the pair of rows of node $k$ by $\phi^\pm_k=R_{y_k}\mp iy_k^{N'-1}R_{x_k}$, and the pair $(R_\infty,R_0)$ by
$\phi^\pm_0=R_\infty\mp iR_0$. Since $i^{N'}=(-1)^{m+1}$, we have $1\mp i^{N'}(-1)^j=1\pm(-1)^{m+j}$, and so
\begin{equation}\begin{gathered}
\phi^\pm_k=(y_k+i)^{N'-1}\Bigl(\alpha_k\varpi_k^j\bigl(1\pm(-1)^{m+j}\bigr)\ \Big|\ \beta_k\varpi_k^j\bigl(1\mp(-1)^{m+j}\bigr)\Bigr),\\
\phi^\pm_0=\Bigl(1\pm(-1)^{m+j}\ \Big|\ -\lambda\bigl(1\mp(-1)^{m+j}\bigr)\Bigr).
\end{gathered}\end{equation}
Let $\delta\in\{0,1\}$ with $\delta\equiv m\bmod 2$. Then $1+(-1)^{m+j}$ is $2$ for $j\in J_\delta$ and $0$ for $j\in J_{1-\delta}$, and
$1-(-1)^{m+j}$ is $2$ for $j\in J_{1-\delta}$ and $0$ for $j\in J_\delta$. So the $N'$ rows $\phi^+$ vanish outside the $N'$ columns
$\{(\chi_j,0):j\in J_\delta\}\cup\{(0,\chi_j):j\in J_{1-\delta}\}$ of $D_\delta$, the $N'$ rows $\phi^-$ vanish outside the complementary $N'$
columns, those of $D_{1-\delta}$, and on their columns the node rows $\phi^+_k$, resp.\ $\phi^-_k$, are $2(y_k+i)^{N'-1}$ times the rows of
$D_\delta$, resp.\ $D_{1-\delta}$, and $\phi^\pm_0$ is $2$ times their last row. The row operations have the determinants $-2iy_k^{N'-1}$ for node $k$ (the pair
$(R_{x_k},R_{y_k})$ is multiplied by $\bigl(\begin{smallmatrix}-iy_k^{N'-1}&1\\ iy_k^{N'-1}&1\end{smallmatrix}\bigr)$) and $2i$ for the pair
$(R_\infty,R_0)$ (multiplied by $\bigl(\begin{smallmatrix}1&-i\\1&i\end{smallmatrix}\bigr)$). After a permutation of the rows and of the columns,
whose sign depends only on $m$, the matrix is block diagonal. So $W^{\rm ext}(\lambda)=K_1D_\delta(\lambda)D_{1-\delta}(\lambda)=K_1D_0(\lambda)D_1(\lambda)$, where $K_1$ is
$\pm4^{N'}\prod_k(y_k+i)^{2(N'-1)}$ divided by $\det C$ and by the determinants of the row operations; it is non-zero and independent of
$(\alpha,\beta)$ and $\lambda$.
\end{proof}

\subsection{The closed formula}\label{subsec:CF}
\emph{Geometry.} For $(\rho,z)\in\R^2$ different from the nodes $(0,w_k)$ put
\begin{equation}\label{eq:rvarpi}
r_k=\sqrt{\rho^2+(z-w_k)^2}>0,\qquad \varpi_k=\frac{w_k-z-i\rho}{r_k},\qquad |\varpi_k|=1 .
\end{equation}
Under $\rho\mapsto-\rho$, $r_k$ is unchanged and $\varpi_k$ is replaced by $\bar\varpi_k$. For $\rho>0$ and $y_k=\mu_k/\rho$ as in \eqref{eq:yk},
\begin{equation}\label{eq:yvarpi}
y_k-\frac1{y_k}=\frac{2(w_k-z)}\rho,\qquad y_k+\frac1{y_k}=\frac{2r_k}\rho,\qquad \frac{y_k-i}{y_k+i}=\varpi_k .
\end{equation}
Indeed, with $u=w_k-z$ we have $\mu_k=r_k+u$ and $\rho^2=r_k^2-u^2$, so $1/y_k=\rho/\mu_k=(r_k-u)/\rho$, which gives the first two identities. For
the third, $(y_k-i)/(y_k+i)=(r_k+u-i\rho)^2/\bigl((r_k+u)^2+\rho^2\bigr)$, where $(r_k+u)^2+\rho^2=2r_k(r_k+u)$ and $(r_k+u-i\rho)^2=2(r_k+u)(u-i\rho)$, by
$\rho^2=r_k^2-u^2$.

\emph{Notation.} Let $\omega$ be the measure of the soliton in projective form (\S\ref{subsec:split}), with the pairs $(\alpha_k,\hat b_k)$,
$\hat b_k=\beta_k/P'(w_k)$. Put
\begin{equation}\label{eq:PiCF}
\Pi(\rho,z;c)=\mathbf H_m(\mathbf r\omega)\,\mathbf H_{m+1}(\omega/\mathbf r),
\end{equation}
and let $\mathbf v=(\mathbf v_0,\mathbf v_1)\in\C^2$ be the vector with the components, in the null basis,
\begin{equation}\label{eq:vCF}
\mathbf v_0=\mathbf H_{m+1}(\varpi\omega),\qquad \mathbf v_1=\mathbf H_m(\varpi\omega).
\end{equation}
For real representatives $\Pi$ is real, and by \S\ref{subsec:split} its sign, the matrix $\operatorname{Re}(\mathbf v\bar{\mathbf v}^T)/\Pi$ and
$d_{(\rho,z)}\log|\Pi|$, in real representatives, do not depend on the representatives; unless said otherwise, $\Pi$ and $\mathbf v$ are computed in the representatives
used for $D$.

\begin{lemma}[Heine evaluation]\label{lem:HE}
Let $(\rho,z)\in\Hb^\circ$, let $y_k=\mu_k/\rho$, and let $(\alpha,\beta)\in\C^{2N}$. Let $D_{\rm I}$, $D_{\rm II}$ be as in
Lemma~\ref{lem:PI}\textup{(a)}, and $D_0(\lambda)$, $D_1(\lambda)$ as in Lemma~\ref{lem:PI}\textup{(b)}.
\begin{enumerate}[leftmargin=*,label=\textup{(\alph*)}]
\item The two factors of Lemma~\ref{lem:PI}(a) are
\begin{equation}\begin{gathered}
D_{\rm I}=(-1)^{m(m-1)/2}\Bigl(\frac2\rho\Bigr)^{m(m+1)}\Delta(w)\,\mathbf H_m(\mathbf r\omega),\\
D_{\rm II}=(-1)^{m(m+1)/2}\Bigl(\frac2\rho\Bigr)^{m(m+1)}\Delta(w)\prod_kr_k\,\mathbf H_{m+1}(\omega/\mathbf r).
\end{gathered}\end{equation}
Hence
\begin{equation}\label{eq:HEodd}
W(\mathcal P)=2(-1)^m\Bigl(\frac2\rho\Bigr)^{2m(m+1)}\Delta(w)^2\prod_kr_k\cdot\Pi(\rho,z;c).
\end{equation}
\item For $\delta\in\{0,1\}$, $D_\delta(\lambda)$ is a non-zero factor, which depends on $(\rho,z,w)$ but not on $(\alpha,\beta)$ or $\lambda$, times
\begin{equation}
\mathbf H_{m+1}(\varpi^{(-1)^\delta}\omega)+\lambda\,\mathbf H_m(\varpi^{(-1)^\delta}\omega).
\end{equation}
Hence, with a constant $K_2\ne0$ independent of $(\alpha,\beta)$ and $\lambda$,
\begin{equation}\label{eq:HEeven}
W(\mathcal P^{\rm ext}(\lambda))=K_2\,\bigl(\mathbf H_{m+1}(\varpi\omega)+\lambda\mathbf H_m(\varpi\omega)\bigr)\bigl(\mathbf H_{m+1}(\varpi^{-1}\omega)+\lambda\mathbf H_m(\varpi^{-1}\omega)\bigr).
\end{equation}
For real $(\alpha,\beta)$ and real $\lambda$, $\mathbf H_s(\varpi^{-1}\omega)=\overline{\mathbf H_s(\varpi\omega)}$, and the product in \eqref{eq:HEeven} is
$|\mathbf v_0+\lambda\mathbf v_1|^2$.
\end{enumerate}
\end{lemma}
\begin{proof}
Both parts are Laplace expansions of the determinants of Lemma~\ref{lem:PI} along the columns of one block (Lemma~\ref{lem:split}),
followed by a change of nodes.

(a) By \eqref{eq:yvarpi}, $\sigma_k=y_k-1/y_k=\frac2\rho w_k-\frac{2z}\rho$ is affine in $w_k$, with $\theta_1=2/\rho$, and $t_k=y_k+1/y_k=2r_k/\rho$. So
$D_{\rm I}=S_{m+1,m}(\sigma;\alpha,\beta t)$ and $D_{\rm II}=S_{m,m+1}(\sigma;\alpha t,\beta)$, and $N(N-1)/2-m(m+1)=m^2$. By Lemma~\ref{lem:split},
\begin{equation}
D_{\rm I}=(-1)^{m(m-1)/2}\Bigl(\frac2\rho\Bigr)^{m^2}\Delta(w)\,\mathbf H_m[\alpha,\beta t],\qquad
D_{\rm II}=(-1)^{m(m+1)/2}\Bigl(\frac2\rho\Bigr)^{m^2}\Delta(w)\,\mathbf H_{m+1}[\alpha t,\beta],
\end{equation}
where $\mathbf H_m[\alpha,\beta t]$ has the pairs $(\alpha_l,\hat b_lt_l)$ and $\mathbf H_{m+1}[\alpha t,\beta]$ the pairs $(\alpha_lt_l,\hat b_l)$. Each term of the first has
$m$ factors $t_l$, so $\mathbf H_m[\alpha,\beta t]=(2/\rho)^m\mathbf H_m(\mathbf r\omega)$. In the second, $\prod_{l\notin T}t_l=\prod_kt_k/\prod_{l\in T}t_l$ since
$t_l\ne0$, so
\begin{equation}
\mathbf H_{m+1}[\alpha t,\beta]=\prod_kt_k\cdot\Bigl(\frac\rho2\Bigr)^{m+1}\mathbf H_{m+1}(\omega/\mathbf r)=\Bigl(\frac2\rho\Bigr)^{m}\prod_kr_k\,\mathbf H_{m+1}(\omega/\mathbf r).
\end{equation}
For \eqref{eq:HEodd} multiply, and use Lemma~\ref{lem:PI}(a): the sign is $(-1)^{m(m-1)/2+m(m+1)/2}=(-1)^{m^2}=(-1)^m$.

(b) By \eqref{eq:rvarpi}, $\varpi_k^2=(w_k-z-i\rho)/(w_k-z+i\rho)=\psi(w_k)$ for the M\"obius map $\psi(x)=(x-z-i\rho)/(x-z+i\rho)$, and
$\psi(\infty)=1$. On the space $\mathcal V_m$ of complex polynomials of degree $\le m$ let $(\iota g)(x)=(x-z+i\rho)^m\,g(\psi(x))$. It maps $X^\ell$ to
$(x-z-i\rho)^\ell(x-z+i\rho)^{m-\ell}$, and these form a basis of $\mathcal V_m$, since $\rho\ne0$; so $\iota$ is an isomorphism,
$g(\varpi_k^2)=(\iota g)(w_k)/(w_k-z+i\rho)^m$, and $g(1)$ is the coefficient of $x^m$ in $\iota g$.

Since $J_\delta=\{\delta+2\ell:0\le\ell\le m\}$, the rows of $D_\delta$ are
$\bigl((\alpha_k\varpi_k^\delta\varpi_k^{2\ell})_{0\le\ell\le m}\mid(\beta_k\varpi_k^{1-\delta}\varpi_k^{2\ell})_{0\le\ell\le m}\bigr)$ and $(1\mid-\lambda)$. So $D_\delta$ is the determinant $E(\alpha',\beta')$ with the rows
$\bigl((\alpha'_k\varpi_k^{2\ell})_\ell\mid(\beta'_k\varpi_k^{2\ell})_\ell\bigr)$ and $(1\mid-\lambda)$, for $(\alpha',\beta')=(\alpha,\beta\varpi)$ if $\delta=0$ and
$(\alpha',\beta')=(\alpha\varpi,\beta)$ if $\delta=1$. In each block of $E$ the entries are the values of the monomials $X^0,\dots,X^m$ at $X=\varpi_k^2$,
resp.\ at $X=1$. Replace, in both blocks, the monomials by the basis $\iota^{-1}(x^\ell)$, $0\le\ell\le m$, of $\mathcal V_m$; this multiplies $E$ by a
non-zero factor that depends only on $(\rho,z,m)$. The rows become $(w_k-z+i\rho)^{-m}\bigl(\alpha'_kw_k^\ell\mid\beta'_kw_k^\ell\bigr)$, and the last row has
the entry $1$ at the position of $x^m$ in the first block, $-\lambda$ at the position of $x^m$ in the second block, and $0$ elsewhere. Expand
along the last row, row $N'$. The entry $1$ sits in column $m+1$, with the sign $(-1)^{N'+m+1}=(-1)^{m+1}$ and the minor $S_{m,m+1}(w;\alpha',\beta')$;
the entry $-\lambda$ sits in column $N'$, with the sign $+1$ and the minor $S_{m+1,m}(w;\alpha',\beta')$; both up to the row factors
$(w_k-z+i\rho)^{-m}$. By Lemma~\ref{lem:split} with $\theta_1=1$, and $(-1)^{m(m+1)/2}=(-1)^m(-1)^{m(m-1)/2}$,
\begin{multline}
(-1)^{m+1}S_{m,m+1}(w;\alpha',\beta')-\lambda S_{m+1,m}(w;\alpha',\beta')\\=-(-1)^{m(m-1)/2}\Delta(w)\bigl(\mathbf H_{m+1}[\alpha',\beta']+\lambda\mathbf H_m[\alpha',\beta']\bigr).
\end{multline}
Finally $\mathbf H_s[\alpha,\beta\varpi]=\mathbf H_s(\varpi\omega)$, and $\mathbf H_s[\alpha\varpi,\beta]=\prod_k\varpi_k\cdot\mathbf H_s(\varpi^{-1}\omega)$, since the term of $T$ is
$\prod_{l\notin T}\alpha_l\varpi_l\prod_{l\in T}\hat b_l=\prod_k\varpi_k\prod_{l\notin T}\alpha_l\prod_{l\in T}\hat b_l\varpi_l^{-1}$. All prefactors are non-zero and independent of
$(\alpha,\beta)$ and $\lambda$; this gives $D_\delta$. Now \eqref{eq:HEeven} follows from Lemma~\ref{lem:PI}(b), since $W(\mathcal P^{\rm ext}(\lambda))$ is the
determinant $W^{\rm ext}(\lambda)$ of that lemma for $y_k=\mu_k/\rho$, and $(y_k-i)/(y_k+i)=\varpi_k$ by \eqref{eq:yvarpi}. For real $(\alpha,\beta)$,
$\overline{\mathbf H_s(\varpi\omega)}=\mathbf H_s(\bar\varpi\omega)$ and $\bar\varpi_k=\varpi_k^{-1}$.
\end{proof}

\begin{theorem}[the closed formula]\label{thm:CF}
Let $\Pi$ and $\mathbf v$ be as in \eqref{eq:PiCF} and \eqref{eq:vCF}.
\begin{enumerate}[leftmargin=*,label=\textup{(\roman*)}]
\item For every $(\rho,z)\in\Hb^\circ$, as an identity of polynomials in $(\alpha,\beta)$,
\begin{equation}
\mathbf H_m(\varpi\omega)\mathbf H_{m+1}(\varpi^{-1}\omega)-\mathbf H_m(\varpi^{-1}\omega)\mathbf H_{m+1}(\varpi\omega)=2i\rho\,\mathbf H_m(\mathbf r\omega)\mathbf H_{m+1}(\omega/\mathbf r);
\end{equation}
for real $(\alpha,\beta)$ this says $\operatorname{Im}(\mathbf v_1\bar{\mathbf v}_0)=\rho\,\Pi$.
\item For all $c\in(\RP^1)^N$, in any representatives, $D=\Xi\,\Pi$ with
\begin{equation}
\Xi=2(2\rho^2)^{-N}\Bigl(\frac2\rho\Bigr)^{2m(m+1)}\Delta(w)^2\,\frac{\prod_ky_kr_k}{\prod_{k,l}(y_l-x_k)}>0,
\end{equation}
a function of $(\rho,z,w)$ alone. In particular $D\ne0$ at $(\rho,z)$ if and only if $\Pi(\rho,z;c)\ne0$. Where $\Pi\ne0$, in the null basis
and in real representatives,
\begin{equation}\label{eq:CFG}
\tilde G:=2(-1)^{m+1}G=\frac{\operatorname{Re}(\mathbf v\bar{\mathbf v}^T)}{\Pi},\qquad\text{i.e.}\qquad
\tilde G_{ab}\,d\phi^ad\phi^b=\frac{|\mathbf v_0\,d\phi^0+\mathbf v_1\,d\phi^1|^2}{\Pi}.
\end{equation}
\item If $c\in\mathcal A$, then $\Pi>0$ on $\Hb^\circ$, in any real representatives. Hence $D>0$; $G$ is real-analytic in $(\rho,z,w,c)$, in the
sense that in an affine chart of $(\RP^1)^N$ around any point of $\mathcal A$, with the corresponding representatives, it is the restriction of a
real-analytic function on a neighbourhood of each point of $\Hb^\circ\times\{w_1<\dots<w_N\}\times\mathcal A$; $\det G=\rho^2/4$; and $(-1)^{m+1}G$
is positive definite.
\item Let $c\in\mathcal A$. A function $\nu$ on $\Hb^\circ$ solves the quadrature \eqref{eq:nu} for $G$ if and only if $e^{2\nu}=C'\,\Pi$ with a
constant $C'>0$. For $c\in\mathcal A^\circ$, $\Pi$ in the standard representatives and $C'=e^{2\nu_0}/\Pi(c^*)$, which is a positive constant (Step~9 of the proof), with
$\nu_0$ a solution of \eqref{eq:nu} for the static seed, this is the conformal factor \eqref{eq:BZnu} of Belinski and Zakharov, with the BZ
vectors normalised by the standard representatives and $\Gamma^{(0)}=\Gamma(c^*)$.
\end{enumerate}
\end{theorem}

\begin{proof}[Proof of \textup{(i)--(iii)}]
The interpolation determinants of Proposition~\ref{prop:CW} factor by Lemmas~\ref{lem:PI} and~\ref{lem:HE}; this determines $\hat G$ up to a
factor $K_0(\rho,z,w)$, which is fixed by $\det\hat G=\frac14$ at a flat point and by Lemma~\ref{lem:Cbase} at the base point.

\emph{Step 1: $D$ and $W(\mathcal P)$.} By Lemma~\ref{lem:Creduced} and Proposition~\ref{prop:CW}(i), for all $(\alpha,\beta)$,
\begin{equation}
D=(2\rho^2)^{-N}\prod_ky_k^2\det M=(2\rho^2)^{-N}(-1)^{N(N-1)/2}\prod_ky_k\,W(\mathcal P)\Big/\prod_{k,l}(y_l-x_k).
\end{equation}
Insert \eqref{eq:HEodd}; since $N(N-1)/2=m(2m+1)\equiv m$ modulo $2$, the
signs $(-1)^m$ cancel, and $D=\Xi\Pi$ with $\Xi$ as stated. Here $\Xi>0$, because $y_k,r_k>0$ and $y_l-x_k=y_l+1/y_k>0$; and $\Xi$ involves neither $c$
nor the representatives. This proves the first part of (ii). In particular $\det M\ne0\iff W(\mathcal P)\ne0\iff\Pi\ne0$.

\emph{Step 2: the flat point.} A \emph{flat point} $c^\flat$ is defined by a choice of disjoint sets $Z,I\subset\{1,\dots,N\}$ with
$|Z|=|I|=m$ and the remaining index $k_0$: $c_k=0$ for $k\in Z$, $c_k=\infty$ for $k\in I$, and $c_{k_0}$ finite with $a_{k_0}>0$. Then
$c^\flat\in\mathcal A$. In the standard representatives, $(\alpha_k,\beta_k)=(1,0)$ for $k\in Z$, $(0,1)$ for $k\in I$, and $(1,c_{k_0})$ for $k=k_0$.
Lemma~\ref{lem:Hpos}(b), with these $Z$, $I$ and $\mathsf R=\{k_0\}$, gives
$\mathbf H_s(q\omega)=\kappa_Iq_I\,\mathbf H^{\{k_0\}}_{s-m}(q\omega^I)$ with $\kappa_I=\hat b_I\Delta_I^2$, $\hat b_I=\prod_{l\in I}\hat b_l=\prod_{l\in I}1/P'(w_l)$.
Here $\mathbf H^{\{k_0\}}_0=1$ and $\mathbf H^{\{k_0\}}_1(q\omega^I)=a^\flat q_{k_0}$ with $a^\flat=a_{k_0}\prod_{l\in I}(w_l-w_{k_0})^2>0$, so
\begin{equation}
\mathbf H_m(q\omega)=\hat b_I\Delta_I^2\prod_{l\in I}q_l,\qquad \mathbf H_{m+1}(q\omega)=\hat b_I\Delta_I^2\prod_{l\in I}q_l\cdot a^\flat q_{k_0}.
\end{equation}
Hence, using $|\varpi_l|=1$ and $\operatorname{Im}\bar\varpi_{k_0}=\rho/r_{k_0}$,
\begin{equation}\label{eq:CFflat}
\Pi(c^\flat)=\hat b_I^2\Delta_I^4\,\frac{a^\flat}{r_{k_0}}>0,\qquad \mathbf v=\hat b_I\Delta_I^2\prod_{l\in I}\varpi_l\cdot\bigl(a^\flat\varpi_{k_0},1\bigr),\qquad
\operatorname{Im}(\mathbf v_1\bar{\mathbf v}_0)=\hat b_I^2\Delta_I^4a^\flat\,\frac\rho{r_{k_0}}=\rho\,\Pi(c^\flat).
\end{equation}
Note that $\Pi(c^\flat)\,r_{k_0}=\hat b_I^2\Delta_I^4a^\flat$ does not depend on $(\rho,z)$.

\emph{Step 3: the shape of $\hat G$.} Fix $(\rho,z)$, and $c$ with $\Pi(\rho,z;c)\ne0$ in real representatives $(\alpha,\beta)$. By
Proposition~\ref{prop:CW}(ii), for all real $\lambda$,
\begin{equation}
(1,\lambda)\hat G(1,\lambda)^T=\frac{W(\mathcal P^{\rm ext}(\lambda))}{2W(\mathcal P)}=K_0\,\frac{|\mathbf v_0+\lambda\mathbf v_1|^2}{\Pi},
\end{equation}
with $K_0=K_0(\rho,z,w)\ne0$ independent of $c$ and $\lambda$, by \eqref{eq:HEodd} and \eqref{eq:HEeven}: $K_0$ is $K_2$ divided by twice the
prefactor of $\Pi$ in \eqref{eq:HEodd}, and neither depends on $c$ or $\lambda$, so one $K_0$ serves for all $c$. A priori $K_0$ is complex, since
$K_2$ is. It is real: at $c=c^\flat$, in its standard representatives, and
$\lambda=0$ the left side is real and $|\mathbf v_0|^2/\Pi>0$, by \eqref{eq:CFflat}. Since
$|\mathbf v_0+\lambda\mathbf v_1|^2=(1,\lambda)\operatorname{Re}(\mathbf v\bar{\mathbf v}^T)(1,\lambda)^T$, and a real symmetric $2\times2$ matrix is
determined by its quadratic form on the vectors $(1,\lambda)$, $\lambda\in\R$,
\begin{equation}\label{eq:CFshape}
\hat G=K_0\,\frac{\operatorname{Re}(\mathbf v\bar{\mathbf v}^T)}{\Pi}.
\end{equation}

\emph{Step 4: $|K_0|$ and (i).} For $\mathbf v\in\C^2$,
\begin{equation}\label{eq:detRe}
\operatorname{Re}(\mathbf v\bar{\mathbf v}^T)=\operatorname{Re}\mathbf v\,(\operatorname{Re}\mathbf v)^T+\operatorname{Im}\mathbf v\,(\operatorname{Im}\mathbf v)^T\ge0,\qquad
\det\operatorname{Re}(\mathbf v\bar{\mathbf v}^T)=(\operatorname{Im}\mathbf v_1\bar{\mathbf v}_0)^2,
\end{equation}
so $\operatorname{Re}(\mathbf v\bar{\mathbf v}^T)$ is positive definite if and only if $\operatorname{Im}(\mathbf v_1\bar{\mathbf v}_0)\ne0$. By
\eqref{eq:CFshape} and $\det\hat G=\frac14$ (Lemma~\ref{lem:Creduced}),
\begin{equation}
K_0^2\,(\operatorname{Im}\mathbf v_1\bar{\mathbf v}_0)^2=\tfrac14\,\Pi^2\qquad\text{wherever }\Pi\ne0.
\end{equation}
At $c^\flat$ this gives $K_0^2\rho^2=\frac14$, by \eqref{eq:CFflat}; since $K_0$ does not depend on $c$, $|K_0|=1/(2\rho)$. Now fix $(\rho,z)$. Then
$p_\pm=\operatorname{Im}(\mathbf v_1\bar{\mathbf v}_0)\mp\rho\Pi$ are real polynomials in $(\alpha,\beta)\in\R^{2N}$, and $p_+p_-=0$ on the set
$\{\Pi\ne0\}$, which is dense, since $\Pi(c^\flat)\ne0$. So $p_+p_-=0$ in $\R[\alpha,\beta]$, and $p_+=0$ or $p_-=0$. At $c^\flat$,
$p_-=2\rho\Pi(c^\flat)\ne0$. Hence $p_+=0$: $\operatorname{Im}(\mathbf v_1\bar{\mathbf v}_0)=\rho\Pi$ for all real $(\alpha,\beta)$. The difference of the
two sides of (i) is a polynomial in $(\alpha,\beta)$ that equals $2i\bigl(\operatorname{Im}(\mathbf v_1\bar{\mathbf v}_0)-\rho\Pi\bigr)=0$ on $\R^{2N}$, since
there $\mathbf H_s(\varpi^{-1}\omega)=\overline{\mathbf H_s(\varpi\omega)}$; so it vanishes identically. This proves (i).

\emph{Step 5: the sign of $K_0$.} By Lemma~\ref{lem:Hpos}(c), $\Pi(c^*)>0$, so $\operatorname{Re}(\mathbf v\bar{\mathbf v}^T)$ is positive definite at
$c^*$, by (i) and \eqref{eq:detRe}. By Lemma~\ref{lem:Cbase}, $(-1)^{m+1}\hat G$ is positive definite at $c^*$. By \eqref{eq:CFshape},
$\operatorname{sgn}K_0=(-1)^{m+1}$. So $K_0=(-1)^{m+1}/(2\rho)$, and $G=\rho\hat G$ (Lemma~\ref{lem:Creduced}) gives \eqref{eq:CFG}. This proves (ii).

\emph{Step 6: (iii).} Let $c\in\mathcal A$. By Lemma~\ref{lem:Hpos}(c), $\Pi>0$ on $\Hb^\circ$, and $D=\Xi\Pi>0$ by Step~1. In an affine
chart around a point of $\mathcal A$, with the corresponding representatives, \eqref{eq:CFG} writes $G$ as a quotient of polynomials in $(\alpha,\beta)$
whose coefficients are real-analytic in $(\rho,z,w)$ on $\{\rho>0\}\times\{w_1<\dots<w_N\}$ (the $r_k$, $\varpi_k$, $1/P'(w_l)$ and $\Delta_T^2$ are), with
the denominator $2\Pi$, which is positive on $\mathcal A$ and hence on a neighbourhood of each point; so $G$ is real-analytic in the stated sense.
By (i) and \eqref{eq:detRe}, $(-1)^{m+1}G=\operatorname{Re}(\mathbf v\bar{\mathbf v}^T)/(2\Pi)$ is positive definite, with determinant
$\rho^2\Pi^2/(4\Pi^2)=\rho^2/4$.
\end{proof}

\begin{proof}[Proof of \textup{(iv)}]
Let $\nu_0$ be a solution of \eqref{eq:nu} for the static seed $G_0$ of \eqref{eq:staticseed}, as in \eqref{eq:BZnu}; it exists since
\eqref{eq:hm} is the integrability condition of \eqref{eq:nu} and $\Hb^\circ$ is simply connected, and two solutions of \eqref{eq:nu} for the
same $G$ differ by a constant (Remark~\ref{rem:nuinv}). In this proof $\Pi$ is computed in the standard representatives; for $c\in\mathcal A$,
$\Pi(c)>0$ on $\Hb^\circ$, by (iii). Put $\nu_*=\nu_0-\tfrac12\log\Pi(c^*)$, a real-analytic function of $(\rho,z)$, independent of $c$, and
$\nu_c=\frac12\log\Pi(c)+\nu_*$ for $c\in\mathcal A$.

\emph{Step 7: the positive cone.} Let $c\in\mathcal A^\circ$; then no BZ vector is null. In \eqref{eq:BZnu}, $\Gamma$ is the matrix of
\eqref{eq:dressed}; we normalise the BZ vectors by the standard representatives, and take $\Gamma^{(0)}=\Gamma(c^*)$, in the standard
representatives $(1,c^*_k)$. By Lemma~\ref{lem:Cbase}, the BZ vectors of $c^*$ give back the static seed up to the constant factor
$(-1)^{m+1}$ and the constant change of frame $e_\pm$, as \eqref{eq:BZnu} allows (Remark~\ref{rem:nuinv}). So Theorem~\ref{thm:BZnu}, whose hypotheses hold for the static seed \eqref{eq:staticseed}, the removed poles $\mub_k$ and
the re-added poles of \S\ref{subsec:ALEsolitons}, says that \eqref{eq:BZnu} solves \eqref{eq:nu} for $G(c)$; this is the only use of
Theorem~\ref{thm:BZnu} in the paper. Now
$\det\Gamma$ is $D$ times a factor that depends on $(\rho,z)$ but not on $c$, since $\Gamma=4f\,SKS$ with $S$ independent of $c$
(\S\ref{subsec:ALEsolitons}). So $|\det\Gamma/\det\Gamma^{(0)}|=D(c)/D(c^*)=\Pi(c)/\Pi(c^*)$, by (ii), (iii): the factor $\Xi$ cancels, and
both determinants are positive since $c,c^*\in\mathcal A$. Thus \eqref{eq:BZnu} is $\nu_c$, and $\nu_c$ solves \eqref{eq:nu} for $G(c)$.
This also proves the identification with \eqref{eq:BZnu} stated in (iv).

\emph{Step 8: all of $\mathcal A$.} Fix an affine chart $\mathcal U$ of $(\RP^1)^N$ around a point of $\mathcal A$, with the corresponding
representatives. Computed in these representatives, $\Pi$ is real-analytic in $(\rho,z,c)$ and positive on
$\Hb^\circ\times(\mathcal A\cap\mathcal U)$, by (iii), and $d_{(\rho,z)}\nu_c=\tfrac12\,d_{(\rho,z)}\log\Pi+d\nu_*$, since $\Pi$ in the standard and in the chart
representatives differ by a factor $\prod_k\lambda_k(c)^2>0$ that does not depend on $(\rho,z)$. So $d_{(\rho,z)}\nu_c$ is continuous on
$\Hb^\circ\times(\mathcal A\cap\mathcal U)$, and so are $G$, with $\det G=\rho^2/4$, and its derivatives in $(\rho,z)$, by (iii). By Step~7, \eqref{eq:nu} holds for
$c\in\mathcal A^\circ\cap\mathcal U$, which is dense in $\mathcal A\cap\mathcal U$ (Definition~\ref{def:admissible}); so Remark~\ref{rem:density}, with the
parameter space $\mathcal A\cap\mathcal U$, gives \eqref{eq:nu} on $\Hb^\circ\times(\mathcal A\cap\mathcal U)$. Hence $\nu_c$ solves \eqref{eq:nu} for $G(c)$ for all $c\in\mathcal A$, phantom solitons included.

\emph{Step 9: $\nu_*$ is constant.} At a flat point $c^\flat$, by (ii) and \eqref{eq:CFflat}, with $u=w_{k_0}-z$, $\operatorname{Re}\varpi_{k_0}=u/r_{k_0}$ and
$r=r_{k_0}$,
\begin{equation}
(-1)^{m+1}G(c^\flat)=\frac{r}{2a^\flat}\begin{pmatrix}(a^\flat)^2&a^\flat u/r\\a^\flat u/r&1\end{pmatrix}
=\frac12\begin{pmatrix}a^\flat r&u\\u&r/a^\flat\end{pmatrix}=\frac14\,T^T\diag\Bigl(\mu_{k_0},\frac{\rho^2}{\mu_{k_0}}\Bigr)T,
\end{equation}
with $T=\bigl(\begin{smallmatrix}\sqrt{a^\flat}&1/\sqrt{a^\flat}\\ \sqrt{a^\flat}&-1/\sqrt{a^\flat}\end{smallmatrix}\bigr)$, since $\mu_{k_0}=r+u$ and
$\rho^2/\mu_{k_0}=r-u$. By Remark~\ref{rem:nuinv}, \eqref{eq:nu} for $G(c^\flat)$ is \eqref{eq:nu} for $\diag(e^{2U},\rho^2e^{-2U})$ with
$e^{2U}=\mu_{k_0}$, flat $\R^4$ in Weyl coordinates centred at $w_{k_0}$. For such a diagonal matrix \eqref{eq:nu} reads
$\partial_\rho\nu=\rho(U_\rho^2-U_z^2)-U_\rho$, $\partial_z\nu=2\rho U_\rho U_z-U_z$. A direct computation, using
$\partial_z\mu_{k_0}=-\mu_{k_0}/r$ and $\partial_\rho\mu_{k_0}=\rho/r$, shows that $\nu=-\frac12\log(2r)$ is a solution. By Step~8 and
Remark~\ref{rem:nuinv},
$\frac12\log\Pi(c^\flat)+\nu_*+\frac12\log(2r_{k_0})$ is constant on $\Hb^\circ$, and $\Pi(c^\flat)r_{k_0}$ is constant by \eqref{eq:CFflat}. So $\nu_*$ is
constant.

\emph{Conclusion.} So $e^{2\nu_*}=e^{2\nu_0}/\Pi(c^*)$ is a positive constant, and $\nu_c=\frac12\log(e^{2\nu_*}\Pi(c))$ solves \eqref{eq:nu} for every
$c\in\mathcal A$ by Step~8; by Remark~\ref{rem:nuinv} the solutions are exactly the functions $\frac12\log(C'\Pi)$ with constants $C'>0$.
\end{proof}

\subsection{Phantom solitons}\label{subsec:Cnondeg}
The parameters $c_k=0$ and $c_k=\infty$ are the phantom solitons (\S\ref{subsec:soliton}); their BZ vectors $m^{(k)}_0=(\beta_k,\alpha_k)$,
$(0,1)$ and $(1,0)$, are the null vectors of $\hat\eta$ \eqref{eq:etahat}. By the next lemma they drop out of the closed formula; so they are
not turning points (Theorem~\ref{thm:family}\textup{(a)}, Proposition~\ref{prop:axis}).

\begin{lemma}[phantom solitons]\label{lem:Cnull}
Let $c\in(\RP^1)^N$ and $\tilde G=2(-1)^{m+1}G$. Let $Z$ and $I$ be disjoint
sets of nodes with $c_k=0$ on $Z$ and $c_k=\infty$ on $I$, $\mathsf R$ the set of the other nodes, and $\omega^I$, $\kappa_I\ne0$ and $q_I$ as in
Lemma~\ref{lem:Hpos}\textup{(b)}.
\begin{enumerate}[leftmargin=*,label=\textup{(\alph*)}]
\item \emph{Reduction.} $\mathbf v=\kappa_I\varpi_I\,\mathbf v^{\rm red}$ and $\Pi=\kappa_I^2\,\Pi^{\rm red}$, with $\varpi_I=\prod_{h\in I}\varpi_h$ and
\begin{equation}
\mathbf v^{\rm red}=\bigl(\mathbf H^{\mathsf R}_{m+1-|I|}(\varpi\omega^I),\mathbf H^{\mathsf R}_{m-|I|}(\varpi\omega^I)\bigr),\qquad
\Pi^{\rm red}=\mathbf H^{\mathsf R}_{m-|I|}(\mathbf r\omega^I)\,\mathbf H^{\mathsf R}_{m+1-|I|}(\omega^I/\mathbf r);
\end{equation}
where $\Pi\ne0$, in real representatives, $\tilde G=\operatorname{Re}(\mathbf v^{\rm red}\bar{\mathbf v}^{\rm red\,T})/\Pi^{\rm red}$. If $Z\cup I$ contains all phantoms,
then in the standard representatives $\mathbf H^{\mathsf R}_s(q\omega^I)=H_s(q\omega^I)$ with
$\omega^I=\sum_{l\in\mathsf R}a_l\prod_{h\in I}(w_l-w_h)^2\,\delta_{w_l}$.
\item \emph{Moving phantoms.} Let $c\in\mathcal A$, with $Z\cup I$ the set of all its phantoms. Move the phantoms to other points of the axis,
distinct and off the nodes of $\mathsf R$, and give each node $w_l$ of $\mathsf R$ the parameter $c_l=a'_lP'(w_l)$, with
$a'_l=\omega^I_l\prod_{h\in I}(w_l-w'_h)^{-2}$, where $w'_h$ is the moved position of the phantom $h\in I$, and $P$ computed with the moved nodes, re-indexed in increasing order. Then the new parameters lie in $\mathcal A$, and $\omega^I$ and
$G$ do not change; hence neither do the solutions of \eqref{eq:nu}.
\end{enumerate}
\end{lemma}
\begin{proof}
(a) Lemma~\ref{lem:Hpos}(b) with $q=\varpi,\mathbf r,1/\mathbf r$ gives the formulas for $\mathbf v$ and $\Pi$: in $\Pi$ the factor is
$\kappa_I\mathbf r_I\cdot\kappa_I\mathbf r_I^{-1}=\kappa_I^2$, with $\mathbf r_I=\prod_{h\in I}r_h$. Since $|\varpi_I|=1$, the factor
$|\kappa_I\varpi_I|^2=\kappa_I^2\ne0$ cancels in \eqref{eq:CFG}. If $Z\cup I$ contains all phantoms, then in the standard representatives
the nodes of $\mathsf R$ have $\alpha_l=1$ and $\hat b_l=a_l$; then use Heine's formula.

(b) The new weights $c_k/P'(w_k)$ are $a'_l>0$ on $\mathsf R$, $\infty$ on $I$ and $0$ on $Z$, with the same numbers of infinities and zeros, so the
new parameters lie in $\mathcal A$; and the new reduced measure has the weights $a'_l\prod_{h\in I}(w_l-w'_h)^2=\omega^I_l$. By (a) and
Theorem~\ref{thm:CF}(ii),(iii), $\tilde G$ is given on $\Hb^\circ$ by Heine sums of the measure $\omega^I$ on the nodes of $\mathsf R$ of the orders
$m-|I|$ and $m+1-|I|$, which do not change; and \eqref{eq:nu} depends only on $G$.
\end{proof}

\section{The soliton family: the axis and infinity}\label{sec:alerods}
This section defines the soliton family, which contains the ALE soliton metrics and the far-node limits
(Theorem~\ref{thm:family}), analyses its axis and, in the ALE case, its end, and shows that the ALE soliton metrics lie in $\Ctor$
(Corollary~\ref{cor:CinCtor}).

\subsection{The family}\label{subsec:family}
For a function $q$ on the nodes, $q\omega$ is the measure with the weights $q_k\omega_k$ (\S\ref{subsec:heine}); we write $\mathbf r\omega$,
$\omega/\mathbf r$ and $\varpi\omega$ for $q_k=r_k$, $1/r_k$ and $\varpi_k$. We use the shifted determinants \eqref{eq:shifted} and the convention
$H_s=\Sigma_s=0$ for $s<0$ (\S\ref{subsec:atominf}).
\begin{definition}[the soliton family]\label{def:family}
\emph{Family data} $(\omega,r,\gamma,\mathfrak s)$ are: nodes $w_1<\dots<w_n$, $n\ge1$; a measure $\omega=\sum_k\omega_k\delta_{w_k}$ with
all $\omega_k>0$; $\gamma\ge0$; an integer $r$ with $0\le r\le n-1$ if $\gamma=0$ and $0\le r\le n$ if $\gamma>0$; and $\mathfrak s\in\R$. Let
$\mathcal U_\omega=\R^2\setminus\{(0,w_1),\dots,(0,w_n)\}$, let $r_k,\varpi_k$ be as in \eqref{eq:rvarpi}, and put on $\mathcal U_\omega$
\begin{equation}\label{eq:famXY}
\begin{aligned}
\mathcal X&=H_r(\varpi\omega),&\qquad \mathcal Y&=H^{(\gamma)}_{r+1}(\varpi\omega)+\gamma(\mathfrak s-i\rho)\,\mathcal X,\\
F_1&=H_r(\mathbf r\omega)+\gamma H_{r-1}(\mathbf r\omega),&\qquad F_2&=H_{r+1}(\omega/\mathbf r)+\gamma H_r(\omega/\mathbf r),
\end{aligned}
\end{equation}
and $\mathbf l=(-\mathcal Y,\mathcal X)$. The \emph{family matrix} is the real symmetric $2\times2$ matrix function
\begin{equation}\label{eq:famG}
\tilde G_{\omega,r,\gamma,\mathfrak s}=\frac{\operatorname{Re}(\mathbf l\,\mathbf l^*)}{F_1F_2},\qquad\text{i.e.}\qquad
\tilde G(x,x)=\frac{|\mathcal X x_1-\mathcal Y x_0|^2}{F_1F_2}\quad(x=(x_0,x_1)\in\R^2),
\end{equation}
in the \emph{family components} $x$. Its vectors are the family $(n,r,\gamma)$ of Definition~\ref{def:ALFfamily} for $\omega$,
$V_j=(H_r(\omega_{(j)}),H^{(\gamma)}_{r+1}(\omega_{(j)}))$, $0\le j\le n$ (for $\gamma=0$ the Hankel family $(n,r)$), and
$S_{\mathfrak s}(X,Y)=(X,Y+\gamma\mathfrak sX)$. For $\gamma=0$ the shift $\mathfrak s$ does not enter. Below, \eqref{eq:omegaj}, \eqref{eq:omegak} and \eqref{eq:wronskian} are used with the weights $\omega_k$ in place of $a_k$.
\end{definition}
The bound $r\le n-1$ for $\gamma=0$ is needed since $F_2=0$ for $(\gamma,r)=(0,n)$.

\begin{lemma}[the family identity]\label{lem:famid}
Let $(\omega,r,\gamma,\mathfrak s)$ be family data.
\begin{enumerate}[leftmargin=*,label=\textup{(\roman*)}]
\item On $\mathcal U_\omega$, $F_1\ge H_r(\mathbf r\omega)>0$ and $F_2>0$; $\mathcal X(-\rho,z)=\overline{\mathcal X(\rho,z)}$,
$\mathcal Y(-\rho,z)=\overline{\mathcal Y(\rho,z)}$, and $F_1,F_2$ are even in $\rho$. So $\tilde G$ is real-analytic on $\mathcal U_\omega$ and even
in $\rho$.
\item $\operatorname{Im}(\mathcal X\overline{\mathcal Y})=\rho\,F_1F_2$ on $\mathcal U_\omega$. This holds for every real measure $\omega$ on $n$
nodes, every $0\le r\le n$ and all real $\gamma,\mathfrak s$ (with $\mathcal X,\mathcal Y,F_1,F_2$ defined by \eqref{eq:famXY} for such data).
\item For $\rho>0$, $\tilde G$ is positive definite with $\det\tilde G=\rho^2$, and $\Phi=\rho^{-1}\tilde G=\Phi_\tau$ of \eqref{eq:Phitau} in
family components, with
\begin{equation}\label{eq:tauCF}
\tau=-\overline{\mathcal X}/\overline{\mathcal Y},\qquad \operatorname{Im}\tau=\rho F_1F_2/|\mathcal Y|^2>0 .
\end{equation}
\end{enumerate}
\end{lemma}
\begin{proof}
(i) On $\mathcal U_\omega$ all $r_k>0$, so $\mathbf r\omega$ and $\omega/\mathbf r$ have $n$ positive atoms. By Heine's formula \eqref{eq:heine},
$H_s$ of either is positive for $0\le s\le n$ and non-negative for all $s$. Hence $F_1\ge H_r(\mathbf r\omega)>0$. Next $F_2\ge H_{r+1}(\omega/\mathbf r)>0$
if $r\le n-1$, and $F_2\ge\gamma H_n(\omega/\mathbf r)>0$ if $r=n$, since then $\gamma>0$. Under $\rho\mapsto-\rho$, $r_k$ is fixed and
$\varpi_k\mapsto\bar\varpi_k$, and the weights, $\gamma$ and $\mathfrak s$ are real; this gives the symmetries, and under them
$\operatorname{Re}(\mathbf l_a\bar{\mathbf l}_b)\mapsto\operatorname{Re}(\bar{\mathbf l}_a\mathbf l_b)=\operatorname{Re}(\mathbf l_a\bar{\mathbf l}_b)$. The entries of
$\tilde G$ are polynomials in $\rho$, the $\varpi_k$, $\bar\varpi_k$ and $r_k^{\pm1}$, divided by $F_1F_2>0$.

(ii) At a fixed point both sides are polynomials in the weights, $\gamma$ and $\mathfrak s$. Replacing $\mathfrak s$ by $0$ changes $\mathcal Y$ by
$\gamma\mathfrak s\mathcal X$, and $\operatorname{Im}(\mathcal X\cdot\gamma\mathfrak s\overline{\mathcal X})=0$; so let $\mathfrak s=0$. Since the weights and
$\gamma$ are real, the symmetries in (i) hold, by the same argument, and both sides are odd in $\rho$; on $\rho=0$ the numbers $\mathcal X,\mathcal Y$ are real. So let
$\rho>0$.

\emph{The case $\gamma=0$, $r\le n-1$.} Let $m=\max(r,n-1-r)$, and choose a set $I$ of $m-r$ points and a set $Z$ of $m-n+1+r$ points of the
axis, distinct and different from the $w_k$; with the nodes of $\omega$ these are $N=2m+1$ nodes, with the polynomial $P_N$. Give the points
of $I$ the pairs $(0,1)$, those of $Z$ the pairs $(1,0)$, and the node $w_k$ the pair $(1,\beta_k)$, $\beta_k=\omega_kP_N'(w_k)/\prod_{k_0\in I}(w_k-w_{k_0})^2$.
Then $\hat b_k=\beta_k/P_N'(w_k)$, so the reduced measure of Lemma~\ref{lem:Hpos}(b) has the pairs $(1,\omega_k)$, and by Heine's formula
$\mathbf H_s(q\hat\omega)=\kappa_Iq_I\,H_{s-|I|}(q\omega)$ for the projective measure $\hat\omega$ of these $N$ nodes, with $\kappa_I\ne0$ and
$m-|I|=r$. Theorem~\ref{thm:CF}(i) is an identity of polynomials in $(\alpha,\beta)$, so it holds for these pairs, whatever the signs of the
$\omega_k$. With $q=\varpi^{\pm1},\mathbf r,1/\mathbf r$ and $H_s(\varpi^{-1}\omega)=\overline{H_s(\varpi\omega)}$ it reads
$2i\kappa_I^2\operatorname{Im}(\mathcal X\overline{\mathcal Y})=2i\rho\,\kappa_I^2F_1F_2$.

\emph{The case $\gamma=0$, $r=n$.} Here $\mathcal Y=H_{n+1}(\varpi\omega)=0$ and $F_2=H_{n+1}(\omega/\mathbf r)=0$.

\emph{The case $\gamma\ne0$.} Fix $(\rho,z)$, let $\Omega>w_n$, $a_\Omega=\gamma\Omega^{1-2r}$ and $\omega^\Omega=\omega+a_\Omega\delta_\Omega$, a real measure on $n+1$
nodes, with $r\le n$. Let $r_\Omega=|\Omega-z-i\rho|$ and $\varpi_\Omega=(\Omega-z-i\rho)/r_\Omega$, and put $X^\Omega=H_r(\varpi\omega^\Omega)$,
$Y^\Omega=H_{r+1}(\varpi\omega^\Omega)$. By the case $\gamma=0$, $\operatorname{Im}(X^\Omega\overline{Y^\Omega})=\rho H_r(\mathbf r\omega^\Omega)H_{r+1}(\omega^\Omega/\mathbf r)$,
and, since $\gamma\Omega$ is real, the left side is $\operatorname{Im}\bigl(X^\Omega\,\overline{Y^\Omega-\gamma\Omega X^\Omega}\bigr)$. Split off the atom at $\Omega$,
of weight $a_\Omega q_\Omega$ in $q\omega^\Omega$, by \eqref{eq:Hsplit}, and use the expansion
$H_{s-1}((t-\Omega)^2\tilde\omega)=\Omega^{2s-2}\bigl(H_{s-1}(\tilde\omega)-2\Omega^{-1}\Sigma_{s-1}(\tilde\omega)+O(\Omega^{-2})\bigr)$ of the proof of
Proposition~\ref{prop:ALFlimit}, which holds for complex weights. This gives
\begin{equation}
\begin{gathered}
X^\Omega=\mathcal X+\gamma\Omega^{-1}\varpi_\Omega H_{r-1}(\varpi\omega)+O(\Omega^{-2}),\\
Y^\Omega-\gamma\Omega X^\Omega=H_{r+1}(\varpi\omega)+\gamma\Omega(\varpi_\Omega-1)H_r(\varpi\omega)\\
-2\gamma\varpi_\Omega\Sigma_r(\varpi\omega)-\gamma^2\varpi_\Omega H_{r-1}(\varpi\omega)+O(\Omega^{-1}),\\
H_r(\mathbf r\omega^\Omega)=H_r(\mathbf r\omega)+\gamma\tfrac{r_\Omega}\Omega H_{r-1}(\mathbf r\omega)+O(\Omega^{-1}),\\
H_{r+1}(\omega^\Omega/\mathbf r)=H_{r+1}(\omega/\mathbf r)+\gamma\tfrac\Omega{r_\Omega}H_r(\omega/\mathbf r)+O(\Omega^{-1}).
\end{gathered}
\end{equation}
Now $\varpi_\Omega=1-i\rho/(\Omega-z)+O(\Omega^{-2})$, so $\Omega(\varpi_\Omega-1)\to-i\rho$, $\varpi_\Omega\to1$ and $r_\Omega/\Omega\to1$. As $\Omega\to\infty$ the
four quantities tend to $\mathcal X$, $\mathcal Y$, $F_1$ and $F_2$ (with $\mathfrak s=0$), which proves (ii).

(iii) By \eqref{eq:detRe} and (ii), $\operatorname{Re}(\mathbf l\,\mathbf l^*)$ is positive semi-definite, and its determinant is
$(\operatorname{Im}\mathbf l_1\bar{\mathbf l}_0)^2=(\operatorname{Im}\mathcal X\overline{\mathcal Y})^2=\rho^2F_1^2F_2^2$. For real $x$,
$|\mathcal Xx_1-\mathcal Yx_0|=|\overline{\mathcal Y}x_0-\overline{\mathcal X}x_1|$, and $\mathcal Y\ne0$ since $\operatorname{Im}(\mathcal X\overline{\mathcal Y})\ne0$; so
$\rho^{-1}\tilde G(x,x)=|x_0+\tau x_1|^2/\operatorname{Im}\tau$ with $\tau=-\overline{\mathcal X}/\overline{\mathcal Y}$ and
$\operatorname{Im}\tau=\operatorname{Im}(\mathcal X\overline{\mathcal Y})/|\mathcal Y|^2$.
\end{proof}

The soliton metrics of \S\S\ref{sec:ale} and~\ref{sec:far} are the family. Part (b) of the next theorem rests on
Theorem~\ref{thm:farclosed} of \S\ref{sec:far}, whose proof, with Lemmas~\ref{lem:FN} and~\ref{lem:FR}, uses \S\ref{sec:toric},
\S\S\ref{sec:bz}--\ref{sec:ale}, Lemma~\ref{lem:famid} and the field equations of Corollary~\ref{cor:CinCtor}(i), but neither part (b) nor any
result that depends on it.
\begin{theorem}[the soliton metrics are the family]\label{thm:family}
\begin{enumerate}[leftmargin=*,label=\textup{(\alph*)}]
\item \emph{ALE.} Let $c\in\cA$ have $N=2m+1$ nodes, let $Z=\{k:c_k=0\}$ and $I=\{k:c_k=\infty\}$ be its phantoms (\emph{light} and
\emph{heavy}), $\mathsf R$ the other, \emph{regular}, nodes, $n=|\mathsf R|$, $\omega^I$ the reduced measure of Lemma~\ref{lem:Cnull}\textup{(a)} and $r=m-|I|$. Then $(\omega^I,r,0)$ is family data,
and in the frame $(e_0,-e_1)$
\begin{equation}\label{eq:famALE}
2(-1)^{m+1}G=\tilde G_{\omega^I,r,0};
\end{equation}
the solutions of \eqref{eq:nu} for $G$ are $e^{2\nu}=C'F_1F_2$, $C'>0$. Conversely, let $(\omega,r,0)$ be family data, let
$m\ge\max(r,n-1-r)$, and choose a set $I$ of $m-r$ points and a set $Z$ of $m-n+1+r$ points of the axis, distinct and off the nodes of
$\omega$; with the nodes of $\omega$ these are $N=2m+1$ nodes, and $P_N$ is their polynomial $P$ of \S\ref{subsec:ALEsolitons}. With
$a_k=\omega_k\prod_{k_0\in I}(w_k-w_{k_0})^{-2}$ and $c_k=a_kP_N'(w_k)$ at the nodes of $\omega$, $c=\infty$ on $I$ and $c=0$ on $Z$,
we have $c\in\cA$, with the family data $(\omega,r,0)$.
\item \emph{ALF and AF.} Let $K=2m$ nodes $w_1<\dots<w_K$ carry parameters $c_k\in\RP^1$ and $\gamma>0$, admissible for the $N=2m+1$
nodes $w_1,\dots,w_K,\Omega$ with a far node $\Omega>w_K$ (the setting of \S\ref{sec:far}); the \emph{heavy} nodes are those with $c_k=\infty$,
$h$ in number, forming the set $I$, the \emph{light} nodes those with $c_k=0$, $l$ in number, and the $n=K-h-l$ others are
\emph{regular}; we assume $n\ge1$ (as in \S\ref{sec:far}; $n=0$ means $h=l=m$). Let $\omega$ be the positive measure \eqref{eq:farmeasure} on the regular nodes, $r=m-h$, and
$\mathfrak s=\sum_{k\le K}w_k-2\sum_{k_0\in I}w_{k_0}$. Then $(\omega,r,\gamma,\mathfrak s)$ is family data, and in the frame $(e_\gamma,e_1)$ the
far-node limit $G_\infty$ of Theorem~\ref{thm:farclosed}, the limit of the rescaled ALE soliton matrices $E_\Omega^TG_\Omega E_\Omega$ of the $N$ nodes as $\Omega\to\infty$,
satisfies
$2(-1)^{m+1}G_\infty=\tilde G_{\omega,r,\gamma,\mathfrak s}$, with the solutions $e^{2\nu}=C'F_1F_2$, $C'>0$, of \eqref{eq:nu}. Conversely, let
$(\omega,r,\gamma,\mathfrak s')$ be family data with $\gamma>0$, and let $m\ge\max(r,n-r)$. Choose a set $I$ of $m-r$ heavy and a set of
$m+r-n$ light nodes, distinct and off the nodes of $\omega$; with the nodes of $\omega$ these are $K=2m$ nodes, with the polynomial $P$. Put
$c_k=-a_kP'(w_k)$, $a_k=\omega_k\prod_{k_0\in I}(w_k-w_{k_0})^{-2}$, at the nodes of $\omega$, $c=\infty$ on $I$ and $c=0$ on the light nodes. These are
far-node data of \S\ref{sec:far} with the family data $(\omega,r,\gamma,\mathfrak s)$, $\mathfrak s$ as in \eqref{eq:farmeasure}. If
$m>\max(r,n-r)$, the positions can be chosen to give $\mathfrak s=\mathfrak s'$. In any case, if a metric has the Gram matrix $\tilde G_{\omega,r,\gamma,\mathfrak s}$
in a frame $(f_0,f_1)$, then it has the Gram matrix $\tilde G_{\omega,r,\gamma,\mathfrak s+\delta}$ in the frame $(f_0-\gamma\delta f_1,f_1)$, of
determinant $1$ relative to $(f_0,f_1)$.
\end{enumerate}
We call $(e_0,-e_1)$, resp.\ $(e_\gamma,e_1)$, the \emph{Hankel frame} of the soliton metric; in it the components of a vector are its
family components. From here on we fix a constant $C_\nu>0$ and take the conformal factor of every soliton metric to be $e^{2\nu}=C_\nu F_1F_2$; by Remark~\ref{rem:nuinv} no statement below depends on the value of $C_\nu$.
\end{theorem}
\begin{proof}[Proof of \textup{(a)}]
By Definition~\ref{def:admissible}(b),(c), $|Z|,|I|\le m$, so $n=N-|Z|-|I|\ge1$, $r\ge0$ and $n-1-r=m-|Z|\ge0$. In the standard
representatives a regular node has $\alpha_l=1$ and $\hat b_l=a_l>0$, so $\omega^I>0$. By Lemma~\ref{lem:Cnull}(a), with $Z\cup I$ all phantoms, in the
standard representatives $\mathbf v=\kappa_I\varpi_I(\mathcal Y,\mathcal X)$ and $\Pi=\kappa_I^2F_1F_2$ for the data $(\omega^I,r,0)$, and $\Pi>0$ on $\Hb^\circ$
(Theorem~\ref{thm:CF}(iii)). By Theorem~\ref{thm:CF}(ii), in the null basis, $2(-1)^{m+1}G(X,X)=|\mathcal Yx+\mathcal Xy|^2/(F_1F_2)$ for $X$ with null
components $(x,y)$. Its components in the frame $(e_0,-e_1)$ are $(x_0,x_1)=(x,-y)$, which gives \eqref{eq:famALE}. By
Theorem~\ref{thm:CF}(iv), the solutions of \eqref{eq:nu} are $e^{2\nu}=C''\Pi=C''\kappa_I^2F_1F_2$, $C''>0$.

\emph{Converse.} The weights $c_k/P_N'(w_k)$ are $a_k>0$ at the nodes of $\omega$, $\infty$ on $I$ and $0$ on $Z$, and $|I|=m-r\le m$,
$|Z|=m-(n-1-r)\le m$; so $c\in\cA$, with the phantoms $Z$ and $I$. Its reduced measure has the weights $a_k\prod_{k_0\in I}(w_k-w_{k_0})^2=\omega_k$, and
$m-|I|=r$.
\end{proof}
\begin{proof}[Proof of \textup{(b)}]
Theorem~\ref{thm:farclosed}\ref{it:fcCF} is \eqref{eq:famG} for the data of \eqref{eq:farmeasure}, and Theorem~\ref{thm:farclosed}\ref{it:fcnu}
gives $e^{2\nu}$. These are family data: $\omega>0$ and $n\ge1$ (Setting of \S\ref{sec:far}); the numbers $h$ of heavy and $l$ of light nodes
are at most $m$, so $r=m-h\ge0$ and $n=2m-h-l\ge m-h=r$.

\emph{Converse.} Here $h=m-r$ and $l=m+r-n$ lie in $[0,m]$, since $0\le r\le n$, and $K-h-l=n$. The weights $-c_k/P'(w_k)$ are $a_k>0$ at the
nodes of $\omega$, $\infty$ on $I$ and $0$ on the light nodes, so the data are admissible (Setting of \S\ref{sec:far}), and the measure
\eqref{eq:farmeasure} has the weights $a_k\prod_{k_0\in I}(w_k-w_{k_0})^2=\omega_k$. If $m>\max(r,n-r)$, there are a heavy node $u$ and a light node
$u'$. Moving $u'$ to $u'+t$ and $u$ to $u+t'$, off the other nodes and with the $c_k$ of the nodes of $\omega$ recomputed, keeps $\omega$ and
changes the shift $\mathfrak s=\sum_{k\le K}w_k-2\sum_{k_0\in I}w_{k_0}$ by $t-t'$; every real value of $t-t'$ is attained in this way.

\emph{The shift.} For $\mathfrak s+\delta$, $\mathcal Y$ is replaced by $\mathcal Y+\gamma\delta\mathcal X$, so
$\tilde G_{\omega,r,\gamma,\mathfrak s+\delta}(x_0,x_1)=\tilde G_{\omega,r,\gamma,\mathfrak s}(x_0,x_1-\gamma\delta x_0)$ by \eqref{eq:famG}; and
$x_0(f_0-\gamma\delta f_1)+x_1f_1=x_0f_0+(x_1-\gamma\delta x_0)f_1$.
\end{proof}

\subsection{The axis}\label{subsec:Caxis}
A function is \emph{real-analytic in $(\rho^2,z)$ near $(0,z_0)$} if, for $\rho>0$, it is the restriction of a real-analytic function on a
neighbourhood of $(0,z_0)$ in $\R^2$; likewise in $(s,t)$.
\begin{lemma}[local forms from analyticity]\label{lem:Clocal}
Let $\tilde G$ be a positive semi-definite real symmetric $2\times2$ matrix function with $\det\tilde G=\rho^2$, defined for $\rho>0$
near a point $(0,z_0)$ and written in a basis of $\mathfrak t$; determinants of two vectors are taken in this basis.
\begin{enumerate}[leftmargin=*,label=\textup{(\alph*)}]
\item If $\tilde G$ is real-analytic in $(\rho^2,z)$ near $(0,z_0)$ and $\tilde G(0,z)\mathbf e=0$ for a vector $\mathbf e$ and all $z$ near
$z_0$, then in every basis $(\mathbf e,\mathbf e')$ with $\det(\mathbf e,\mathbf e')=1$, $\tilde G$ has the rod form \eqref{eq:Grod}
near $(0,z_0)$, with $a,b,f$ real-analytic in $(\rho^2,z)$ and $f(0,z)>0$.
\item Let $\rho^2=4st$ and $z-z_0=t-s$ with $s,t\ge0$. If $\tilde G$ is real-analytic in $(s,t)$ near $0$, $\tilde G\mathbf e=0$ on $\{t=0\}$ and
$\tilde G\mathbf e'=0$ on $\{s=0\}$, with $\det(\mathbf e,\mathbf e')=\pm1$, then in the basis $(\mathbf e,\mathbf e')$, $\tilde G$ has the form
\eqref{eq:nutmodel} with $A,B,C$ real-analytic in $(s,t)$ near $0$, $A(0)B(0)=4$ and $A(0)>0$.
\end{enumerate}
\end{lemma}
\begin{proof}
A change of basis of determinant $\pm1$ does not change $\det\tilde G$. (a) Put $\sigma=\rho^2$ and work on any product neighbourhood of
$(0,z_0)$ on which $\tilde G$ is real-analytic in $(\sigma,z)$. The covector $\tilde G(\mathbf e,\cdot)$ is real-analytic in $(\sigma,z)$ and vanishes at $\sigma=0$, so it is $\sigma\chi$ with $\chi$
real-analytic. With $a=\chi(\mathbf e)$, $b=\chi(\mathbf e')$ and $f=\tilde G(\mathbf e',\mathbf e')$, $\tilde G$ is the matrix of \eqref{eq:Grod}
in the basis $(\mathbf e,\mathbf e')$, and $\rho^2(af-\rho^2b^2)=\det\tilde G=\rho^2$. So $af-\sigma b^2=1$ for $\sigma>0$, hence near
$(0,z_0)$. At $\sigma=0$ this gives $af=1$, so $f(0,z)\ne0$; and $f(0,z)\ge0$, as a limit of $\tilde G(\mathbf e',\mathbf e')\ge0$.

(b) The covector $\tilde G(\mathbf e,\cdot)$ is real-analytic in $(s,t)$ near $0$ and vanishes on $\{t=0,\ s\ge0\}$, hence, by analyticity in
$s$, on all of $\{t=0\}$ near $0$; so it is $t\chi_1$ with $\chi_1$ real-analytic near $0$. Likewise $\tilde G(\mathbf e',\cdot)=s\chi_2$. Then
$t\chi_1(\mathbf e')=\tilde G(\mathbf e,\mathbf e')=s\chi_2(\mathbf e)$ near $0$, so $\chi_1(\mathbf e')$ vanishes on $\{s=0\}$, first for $t\ne0$
and then by continuity, and it equals $sC$ with $C$ real-analytic. With $B=\chi_1(\mathbf e)$ and $A=\chi_2(\mathbf e')$ this is \eqref{eq:nutmodel}. Now $st(AB-stC^2)=\det\tilde G=4st$, so
$AB-stC^2=4$ and $A(0)B(0)=4$. Since $sA=\tilde G(\mathbf e',\mathbf e')\ge0$ for $s,t>0$, $A(0)\ge0$, hence $A(0)>0$.
\end{proof}

\begin{proposition}[the axis]\label{prop:axis}
Let $\tilde G$ be the family matrix of $(\omega,r,\gamma,\mathfrak s)$, $C_\nu>0$ and $e^{2\nu}=C_\nu F_1F_2$; let $V_j$ be the family $(n,r,\gamma)$,
$S=S_{\mathfrak s}$, and $I_0,\dots,I_n$ the rods of the nodes of $\omega$. Vectors and determinants are taken in family components, and
$U_p(w_k)=H_p\bigl(\sum_{l\ne k}\omega_l|w_l-w_k|\,\delta_{w_l}\bigr)$ for every integer $p$, as in \eqref{eq:wronskian} but with the
weights $\omega_l$ in place of $a_l$.
\begin{enumerate}[leftmargin=*,label=\textup{(\alph*)}]
\item\label{it:axrod} \emph{Rods.} Let $z_0\in I_j$. Near $(0,z_0)$, $\tilde G$ and $e^{2\nu}$ are real-analytic in $(\rho^2,z)$, and
$e^{2\nu}>0$. For $z\in I_j$, $\tilde G(0,z)$ has rank one, and its kernel is spanned by $SV_j\ne0$. In every basis $(\mathbf e,\mathbf e')$ with
$\mathbf e=SV_j$ and $\det(\mathbf e,\mathbf e')=1$, $\tilde G$ has the rod form \eqref{eq:Grod} near $(0,z_0)$, with $a,b,f$ real-analytic in
$(\rho^2,z)$ and $f(0,z)>0$. On $I_j$, $e^{2\nu}\adj\tilde G\to C_\nu\,SV_j\otimes SV_j$ as $\rho\to0$; so the normalised rod vectors are
$\pm C_\nu^{1/2}SV_j$.
\item\label{it:axtp} \emph{Turning points.} Let $w_k$ be a node, $\rho=2\varrho_1\varrho_2$, $z-w_k=\varrho_2^2-\varrho_1^2$ ($\varrho_i\ge0$) and
$(s,t)=(\varrho_1^2,\varrho_2^2)$. Then $\tilde G$ and $e^{2\nu}(s+t)$ are real-analytic in $(s,t)$ near $0$, and
$e^{2\nu}(s+t)\to C_\nu\,\omega_k\bigl(U_r(w_k)+\gamma U_{r-1}(w_k)\bigr)^2>0$ at $0$. Moreover
$\det(SV_{k-1},SV_k)=-2\omega_k\bigl(U_r(w_k)+\gamma U_{r-1}(w_k)\bigr)^2<0$, and in the basis $\mathbf e=SV_{k-1}$,
$\mathbf e'=SV_k/|\det(SV_{k-1},SV_k)|$, $\tilde G$ has the form \eqref{eq:nutmodel}, with $A,B,C$ real-analytic in $(s,t)$ and $A(0)>0$.
\item\label{it:axtame} \emph{Local tameness.} The map $\Phi=\rho^{-1}\tilde G$ is locally tamed by the rod lines $[SV_0],\dots,[SV_n]$ and
the turning points $w_1,\dots,w_n$ (Definition~\ref{def:tame}), and
\begin{equation}
\det(SV_0,SV_n)=\det(V_0,V_n)=2(-1)^{r+1}H_r(\omega)\bigl(H_{r+1}(\omega)-\gamma^2H_{r-1}(\omega)\bigr).
\end{equation}
\end{enumerate}
\end{proposition}
\begin{proof}
Everything follows from the closed form \eqref{eq:famG} and the identity of Lemma~\ref{lem:famid}(ii): on a rod, $\tilde G$ is analytic in
$(\rho^2,z)$ and its kernel is read off from $\mathcal X,\mathcal Y$ at $\rho=0$; at a turning point the square root of $\rho$ is handled by the
coordinates $(s,t)$, and the Wronskian of Proposition~\ref{prop:ALFstructure} gives the independence of the adjacent rod lines.

(a) \emph{Analyticity.} Near $(0,z_0)$ all $r_k$ are positive. By Lemma~\ref{lem:famid}(i), $\tilde G$ and $F_1F_2$ are real-analytic in $(\rho,z)$
and even in $\rho$ there, hence real-analytic in $(\rho^2,z)$; and $e^{2\nu}=C_\nu F_1F_2>0$, also at $\rho=0$.

\emph{Kernel.} At $\rho=0$, $\varpi_k=\operatorname{sgn}(w_k-z)$, which for $z\in I_j$ is $-1$ for $k\le j$ and $+1$ for $k>j$; so $\varpi\omega=\omega_{(j)}$ of
\eqref{eq:omegaj}. The term $-i\gamma\rho\mathcal X$ vanishes, and
$(\mathcal X,\mathcal Y)=\bigl(H_r(\omega_{(j)}),H^{(\gamma)}_{r+1}(\omega_{(j)})+\gamma\mathfrak sH_r(\omega_{(j)})\bigr)=SV_j$ is real. So
$\tilde G(0,z)=\mathbf l\,\mathbf l^T/(F_1F_2)$ with $\mathbf l=(-\mathcal Y,\mathcal X)$ real, and $\mathbf l\cdot SV_j=0$.

\emph{Rank one.} If $\tilde G(0,z)=0$, then $\tilde G(\rho,z)=O(\rho^2)$, by analyticity in $\rho^2$, and $\det\tilde G=O(\rho^4)$, contradicting
$\det\tilde G=\rho^2$ (Lemma~\ref{lem:famid}(iii)). So $\mathbf l\ne0$, i.e.\ $SV_j\ne0$, the rank is one, and $SV_j$ spans the kernel.

\emph{Rod form.} For $\rho>0$, $\tilde G$ is positive definite with $\det\tilde G=\rho^2$, so Lemma~\ref{lem:Clocal}(a) gives the rod form.

\emph{Adjugate.} For $2\times2$ matrices $\adj$ is linear, and $\adj\bigl(\begin{smallmatrix}a&b\\b&c\end{smallmatrix}\bigr)=\bigl(\begin{smallmatrix}c&-b\\-b&a\end{smallmatrix}\bigr)$.
Hence $\adj\operatorname{Re}(\mathbf l\,\mathbf l^*)=\operatorname{Re}(\mathbf l^\perp\mathbf l^{\perp*})$ with $\mathbf l^\perp=(\mathbf l_1,-\mathbf l_0)=(\mathcal X,\mathcal Y)$,
and $e^{2\nu}\adj\tilde G=C_\nu\operatorname{Re}(\mathbf l^\perp\mathbf l^{\perp*})$ on $\mathcal U_\omega$. At $\rho=0$ on $I_j$, $\mathbf l^\perp=SV_j$ is real.
By the rod form and \eqref{eq:Nrod}, the normalised rod vectors are $\pm C_\nu^{1/2}SV_j$; here $\adj$ transforms like $v\otimes v$ under changes of
basis of determinant $\pm1$.

(b) \emph{Splitting off the atom $k$.} Let $\omega'=\omega-\omega_k\delta_{w_k}$ and $\omega^{(k)}$ as in \eqref{eq:omegak}, with the weights
$\omega_l(w_l-w_k)^2$, $l\ne k$. By \eqref{eq:Hsplit} and \eqref{eq:splitSigma} for the atom $k$ of $q\omega$, $q\in\{\varpi,\mathbf r,1/\mathbf r\}$,
\begin{equation}
\mathbf l=\mathbf l^0+\omega_k\varpi_k\,\mathbf l^1,\qquad F_1=F_1^0+\omega_kr_kF_1^1,\qquad F_2=F_2^0+\frac{\omega_k}{r_k}F_2^1,
\end{equation}
where $\mathbf l^0,\mathbf l^1,F_i^0,F_i^1$ are polynomials in $\rho$, $\gamma$, $\mathfrak s$ and the $\varpi_l,r_l^{\pm1}$ with $l\ne k$, so real-analytic
near $(0,w_k)$, and $F_1^0=H_r(\mathbf r\omega')+\gamma H_{r-1}(\mathbf r\omega')$, $F_2^1=H_r(\omega^{(k)}/\mathbf r)+\gamma H_{r-1}(\omega^{(k)}/\mathbf r)$.

\emph{Square root.} Put $\mathsf q=\varrho_1-i\varrho_2$. Then $\mathsf q^2=w_k-z-i\rho$, $r_k=|\mathsf q|^2=s+t$ and $\varpi_k=\mathsf q/\bar{\mathsf q}$.
The map $(\varrho_1,\varrho_2)\mapsto(\rho,z)$, with the inverse $\varrho_2+i\varrho_1=(z-w_k+i\rho)^{1/2}$ (principal branch), is a homeomorphism of
the closed quadrant onto the closed half-plane, which maps $\{\varrho_2=0\}$ onto $\{\rho=0,\ z\le w_k\}$ and $\{\varrho_1=0\}$ onto
$\{\rho=0,\ z\ge w_k\}$. With $\mathsf n=\bar{\mathsf q}\,\mathbf l=\bar{\mathsf q}\,\mathbf l^0+\omega_k\mathsf q\,\mathbf l^1$ we have
$r_k\operatorname{Re}(\mathbf l\,\mathbf l^*)=\operatorname{Re}(\mathsf n\mathsf n^*)$, so $\tilde G=\mathcal N/\mathcal D$ with
$\mathcal N=\operatorname{Re}(\mathsf n\mathsf n^*)$ and $\mathcal D=r_kF_1F_2=(F_1^0+\omega_kr_kF_1^1)(r_kF_2^0+\omega_kF_2^1)$. Both are real-analytic in
$(\varrho_1,\varrho_2)$ near $0$.

\emph{Parity.} The reflections $\varrho_i\mapsto-\varrho_i$ act as $\rho\mapsto-\rho$ and fix $z$ and $r_k$. Since $\tilde G$ and $\mathcal D$ are even in
$\rho$ (Lemma~\ref{lem:famid}(i)), $\mathcal D$ and $\mathcal N=\tilde G\mathcal D$ are even in each $\varrho_i$, first where $r_k\ne0$ and then by
continuity. A real-analytic function that is even in each variable is real-analytic in $(\varrho_1^2,\varrho_2^2)=(s,t)$.

\emph{The value at $0$.} $\mathcal D(0,0)=\omega_kF_1^0(0,w_k)F_2^1(0,w_k)$. At $(0,w_k)$, $r_l=|w_l-w_k|$, so both measures $\mathbf r\omega'$ and
$\omega^{(k)}/\mathbf r$ have the weights $\omega_l|w_l-w_k|$, and $F_1^0=F_2^1=U_r(w_k)+\gamma U_{r-1}(w_k)$. This measure has $n-1$ positive atoms, so
$U_p(w_k)>0$ for $0\le p\le n-1$, and $U_n(w_k)=0$. If $r\le n-1$, then $U_r(w_k)>0$; if $r=n$, then $\gamma>0$ and $U_{n-1}(w_k)>0$. So
$\mathcal D(0,0)>0$, $\tilde G=\mathcal N/\mathcal D$ is real-analytic in $(s,t)$, and $e^{2\nu}(s+t)=C_\nu\mathcal D$ has the stated value at $0$.

\emph{Wronskian.} By Proposition~\ref{prop:ALFstructure}\ref{it:alfs1}, with the weights $\omega_k$ in place of $a_k$, and $\det S=1$.

\emph{Local form.} $\{t=0\}$ is the part $z\le w_k$ of the axis near $w_k$, in the closure of $I_{k-1}$. There $\tilde GSV_{k-1}=0$, by (a) for
$s>0$ and by continuity at $s=0$; likewise $\tilde GSV_k=0$ on $\{s=0\}$. Since $\det(\mathbf e,\mathbf e')=-1$, Lemma~\ref{lem:Clocal}(b) applies.

(c) By (b), consecutive rod lines differ. By (a), (b) and Lemma~\ref{lem:formtame}, $\Phi$ is within bounded distance of the rod model of
$[SV_j]$ near every point of $I_j$, and of the turning-point model of $[SV_{k-1}],[SV_k]$ near $w_k$. The determinant is
Proposition~\ref{prop:ALFstructure}\ref{it:alfs2}, with $\det S=1$.
\end{proof}

\subsection{Infinity}\label{subsec:Cinfty}
\begin{proposition}[infinity]\label{prop:Cinfty}
Let $(\omega,r,0)$ be family data, $\tilde G$ its family matrix, $\Phi=\rho^{-1}\tilde G$, $\bar w=\max_k|w_k|$ over the nodes of $\omega$, $\kappa_\infty=H_{r+1}(\omega)/H_r(\omega)$,
and, with $z+i\rho=Re^{i\theta}$,
\begin{equation}\label{eq:Ginftymodel}
\tilde G_\infty=\begin{pmatrix}\kappa_\infty R&z\\ z&R/\kappa_\infty\end{pmatrix}=\rho\,\Phi_{\tau_\infty},\qquad\tau_\infty=\frac{e^{i\theta}}{\kappa_\infty},
\end{equation}
in family components. For $R\ge4(r+1)\bar w$, uniformly in $\theta\in(0,\pi)$,
\begin{equation}\label{eq:Cinftybound}
\cosh d(\Phi,\rho^{-1}\tilde G_\infty)-1\le36(2r+1)^2\frac{\bar w^2}{R^2},\qquad\text{so}\qquad d(\Phi,\rho^{-1}\tilde G_\infty)\le\frac{9(2r+1)\bar w}R .
\end{equation}
\end{proposition}
\begin{proof}
\emph{The model.} $\operatorname{Re}\tau_\infty=\cos\theta/\kappa_\infty$, $\operatorname{Im}\tau_\infty=\sin\theta/\kappa_\infty$ and $|\tau_\infty|^2=\kappa_\infty^{-2}$; since
$\rho=R\sin\theta$ and $z=R\cos\theta$, $\rho\Phi_{\tau_\infty}$ is the matrix \eqref{eq:Ginftymodel}.

\emph{The phases.} Put $u=\bar w/R\le1/(4(r+1))$. Then $w_k-z-i\rho=-Re^{i\theta}(1-\frac{w_k}Re^{-i\theta})$, so $\varpi_k=-e^{i\theta}e^{i\varphi_k}$ with
$\varphi_k=\arg(1-\frac{w_k}Re^{-i\theta})$, and $R(1-u)\le r_k\le R(1+u)$. The number $1-\frac{w_k}Re^{-i\theta}$ has real part $\ge1-u\ge\frac12$
and imaginary part $\frac{w_k}R\sin\theta$, so
\begin{equation}
|\varphi_k|\le|\tan\varphi_k|\le2u\sin\theta=:\varphi_* .
\end{equation}

\emph{The Heine sums.} By Heine's formula, $H_s(\varpi\omega)=(-e^{i\theta})^s\sum_{|S|=s}P_S(\omega)e^{i\varphi_S}$ with $\varphi_S=\sum_{k\in S}\varphi_k$
and $P_S(\omega)>0$; since $|e^{ix}-1|\le|x|$,
\begin{equation}
H_s(\varpi\omega)=(-e^{i\theta})^sH_s(\omega)(1+\epsilon_s),\qquad|\epsilon_s|\le s\varphi_*,\qquad s\in\{r,r+1\}.
\end{equation}
Termwise, $F_1F_2=H_r(\mathbf r\omega)H_{r+1}(\omega/\mathbf r)=\lambda_\infty H_r(\omega)H_{r+1}(\omega)/R$ with $\lambda_\infty\ge(1-u)^r(1+u)^{-r-1}$. Here
$(1-u)^r\ge1-ru\ge\frac34$ and $(1+u)^{r+1}\le e^{(r+1)u}\le e^{1/4}<\frac32$, so $\lambda_\infty\ge\frac12$; and $(r+1)\varphi_*\le2(r+1)u\le\frac12$.

\emph{The half-plane coordinate.} By \eqref{eq:tauCF}, $\tau=-\overline{\mathcal X}/\overline{\mathcal Y}=\frac{e^{i\theta}}{\kappa_\infty}\,
\overline{\bigl(\tfrac{1+\epsilon_r}{1+\epsilon_{r+1}}\bigr)}$, so
\begin{equation}
|\tau-\tau_\infty|=\frac1{\kappa_\infty}\Bigl|\frac{\epsilon_r-\epsilon_{r+1}}{1+\epsilon_{r+1}}\Bigr|\le\frac{(2r+1)\varphi_*}{\kappa_\infty(1-(r+1)\varphi_*)}\le\frac{2(2r+1)\varphi_*}{\kappa_\infty},
\end{equation}
and, with $|\mathcal Y|=H_{r+1}(\omega)|1+\epsilon_{r+1}|$,
\begin{equation}
\operatorname{Im}\tau=\frac{\rho F_1F_2}{|\mathcal Y|^2}=\frac{\lambda_\infty\sin\theta\,H_r(\omega)H_{r+1}(\omega)}{H_{r+1}(\omega)^2|1+\epsilon_{r+1}|^2}
\ge\frac{\lambda_\infty\sin\theta}{\kappa_\infty(1+(r+1)\varphi_*)^2}\ge\frac{2\sin\theta}{9\kappa_\infty} .
\end{equation}

\emph{The distance.} By \eqref{eq:coshhalf},
\begin{equation}
\cosh d(\Phi_\tau,\Phi_{\tau_\infty})-1=\frac{|\tau-\tau_\infty|^2}{2\operatorname{Im}\tau\operatorname{Im}\tau_\infty}\le\frac{9(2r+1)^2\varphi_*^2}{\sin^2\theta}=36(2r+1)^2u^2 ;
\end{equation}
the factors $\sin\theta$ cancel, which gives the uniformity in $\theta$. Finally $\cosh d-1\ge d^2/2$ gives $d\le6\sqrt2\,(2r+1)u<9(2r+1)u$.
\end{proof}

\subsection{The ALE soliton metrics lie in \texorpdfstring{$\Ctor$}{C\_toric}}\label{subsec:CinCtor}
Let $(\omega,r,0)$ be family data, $\kappa_\infty$ as in Proposition~\ref{prop:Cinfty}, $\alpha>0$ and $\vartheta=(\alpha/(2\kappa_\infty))^{1/2}$. For
$c\in\cA$ with these family data, let $F_\alpha$ be the basis of $\mathfrak t$ in which a vector with Hankel components $(x_0,x_1)$ has the
components $(x_0/\vartheta,\vartheta x_1)$. It has determinant $1$ relative to the Hankel frame, and $-1$ relative to the null basis. We orient $g_c$
by $d\rho\wedge dz\wedge d\phi^a\wedge d\phi^b$ in $F_\alpha$, with the $(\rho,z)$ of the formulas; so the Weyl coordinates of $g_c$ in $F_\alpha$ are
$(\rho,z)$ (\S\ref{subsec:WP}). In $F_\alpha$, the Gram matrix of $\tilde G_\infty$ is
\begin{equation}\label{eq:GinftyFalpha}
\diag(\vartheta,\vartheta^{-1})\,\tilde G_\infty\,\diag(\vartheta,\vartheta^{-1})=\begin{pmatrix}\alpha R/2&z\\ z&2R/\alpha\end{pmatrix}=\rho\,\Phi^\sharp_{0,\alpha},
\end{equation}
since $\vartheta^2\kappa_\infty=\alpha/2$, and, by \eqref{eq:models} and \eqref{eq:PhiHA}, $\Phi^\sharp_{0,\alpha}=\Phi_{H,A}$ with $H=\alpha/(2R)$ and
$A=\alpha z/(2R)$, so that $\rho\Phi_{H,A}$ has the entries $H^{-1}A^2+\rho^2H=\alpha R/2$, $H^{-1}A=z$ and $H^{-1}=2R/\alpha$. The vector $V_j$ of the
Hankel family $(n,r)$ has the $F_\alpha$-components $(H_r(\omega)/\vartheta)\,D_\alpha^{-1}\tilde V_j$, with the normalised family $\tilde V_j$ of
\eqref{eq:Vnorm} and $D_\alpha=\diag(1,2/\alpha)$ as in Lemma~\ref{lem:match}; indeed the second component is
$\frac{H_r(\omega)}\vartheta\cdot\frac\alpha2\cdot\frac{H_{r+1}(\omega_{(j)})}{H_{r+1}(\omega)}=\vartheta H_{r+1}(\omega_{(j)})$.

\begin{corollary}[the ALE soliton metrics lie in $\Ctor$]\label{cor:CinCtor}
Let $c\in\cA$, with the family data $(\omega,r,0)$ (Theorem~\ref{thm:family}\textup{(a)}), the rods $I_0,\dots,I_n$ of the nodes of $\omega$,
$\bar w$ as in Proposition~\ref{prop:Cinfty},
$\tilde G=2(-1)^{m+1}G$, $e^{2\nu}=C_\nu F_1F_2$, $\Phi=\rho^{-1}\tilde G$ and $g_c=\tilde G_{ab}\,d\phi^ad\phi^b+e^{2\nu}(d\rho^2+dz^2)$.
\begin{enumerate}[leftmargin=*,label=\textup{(\roman*)}]
\item $g_c$ is a real-analytic Ricci-flat metric on $\Hb^\circ\times T^2$. Its turning points are the $n\ge1$ regular nodes; the phantom nodes are
interior points of rods.
\item In the Hankel frame the rod line of $I_j$ is spanned by $V_j=(H_r(\omega_{(j)}),H_{r+1}(\omega_{(j)}))$, and the normalised rod vectors are
$\pm C_\nu^{1/2}V_j$.
\item $(\Phi,\nu)\in\Ctor$, of type ALE: for every $\alpha>0$, in the basis $F_\alpha$, whose Weyl coordinates are the $(\rho,z)$ of the formulas,
$d(\Phi,\Phi^\sharp_{0,\alpha})\le9(2r+1)\bar w/R$ for $R\ge4(r+1)\bar w$, and the rod lines are spanned by $D_\alpha^{-1}\tilde V_j$.
\end{enumerate}
\end{corollary}
\begin{proof}
(i) \emph{The equations.} For $c$ in the positive cone $\mathcal A^\circ$ of Definition~\ref{def:admissible}, $G$ solves \eqref{eq:hm} where
$D\ne0$, by the $n$-soliton theorem of Belinski and Zakharov stated after \eqref{eq:dressed} (\cite{BZ1978}, \cite[\S8.2,
(8.19)--(8.26)]{BV2001}), and $D>0$ on $\Hb^\circ$ (Theorem~\ref{thm:CF}(iii)). By Theorem~\ref{thm:CF}(iii), for every
affine chart $\mathcal U$ of $(\RP^1)^N$ around a point of $\mathcal A$, $\tilde G$ is real-analytic in $(\rho,z,c)$ on a neighbourhood of
$\Hb^\circ\times(\mathcal A\cap\mathcal U)$, and for $c\in\mathcal A$ it is positive definite with $\det\tilde G=\rho^2$. Since $\mathcal A^\circ$ is
dense in $\mathcal A$ (Definition~\ref{def:admissible}), Remark~\ref{rem:density}, with the parameter space $\mathcal A\cap\mathcal U$, gives
\eqref{eq:hm} for all $c\in\mathcal A$. By Theorem~\ref{thm:family}(a), $e^{2\nu}=C_\nu F_1F_2$ solves \eqref{eq:nu}, and it is real-analytic and
positive (Lemma~\ref{lem:famid}(i)). The Weyl--Papapetrou reduction (\S\ref{subsec:WP}) gives a real-analytic Ricci-flat metric.

\emph{Turning points.} By Theorem~\ref{thm:family}(a), $\tilde G$ is the family matrix of $(\omega,r,0)$ in the Hankel frame, with $n\ge1$. By
Proposition~\ref{prop:axis}(a),(b), every point of the axis other than a node of $\omega$ is an interior point of a rod, and the rod lines of
the two rods that meet at a node of $\omega$ differ. The nodes of $\omega$ are the regular nodes, and the phantom nodes are not among them.

(ii) Proposition~\ref{prop:axis}(a), with $S=\mathrm{Id}$ since $\gamma=0$.

(iii) \emph{Local tameness} \ref{it:T1} is Proposition~\ref{prop:axis}(c). \emph{The model} \ref{it:T2}. By
Proposition~\ref{prop:Cinfty}, $d(\Phi,\rho^{-1}\tilde G_\infty)\le9(2r+1)\bar w/R$ for $R\ge4(r+1)\bar w$, uniformly in $\theta$, in family components; by \eqref{eq:GinftyFalpha}, the
Gram matrix of $\tilde G_\infty$ in $F_\alpha$ is $\rho\Phi^\sharp_{0,\alpha}$, as a function of the Weyl coordinates $(\rho,z)$ of $F_\alpha$. The change of
basis has determinant $1$, so it acts on $\mathcal H$ by an isometry, the same for $\Phi$ and for $\rho^{-1}\tilde G_\infty$. The rod lines in $F_\alpha$
are spanned by the $F_\alpha$-components of $V_j$, which are parallel to $D_\alpha^{-1}\tilde V_j$. The type is ALE, since $h_0=0$ and
$\det(V_0,V_n)=2(-1)^{r+1}H_r(\omega)H_{r+1}(\omega)\ne0$ (Proposition~\ref{prop:axis}(c)).
\end{proof}

\section{The ALF and AF soliton metrics: the far-node limit}\label{sec:far}
This section obtains the ALF and AF soliton metrics as limits of ALE soliton metrics in which one node goes to infinity: we
prove that the limit exists and is the case $\gamma>0$ of the soliton family (Theorem~\ref{thm:farclosed}), identify it with a
soliton on $\R^3\times S^1$ (Theorem~\ref{thm:mkn}), and show that it lies in $\Ctor$ (Proposition~\ref{prop:farCtor}).

\emph{Setting.} Throughout, $N=2m+1$ with $m\ge1$, and $K=N-1=2m$. Table~\ref{tab:fardata} collects the data of this section,
which are also used in \S\S\ref{sec:surj}--\ref{sec:proofs}.
\begin{table}[ht]\centering\footnotesize
\begin{tabular}{llll}\toprule
symbol & meaning & range & defined\\\midrule
$N=2m+1$ & number of nodes before the limit & $m\ge1$ & here\\
$K=2m$ & number of finite nodes $w_1<\dots<w_K$ & & here\\
$\Omega$ & the far node, $\Omega\to+\infty$ & $\Omega>w_K$ & here\\
$c_k$, $\gamma$ & parameters of the finite nodes and of $\Omega$ & $c\in(\RP^1)^K$, $\gamma>0$ & here\\
$h$, $l$ & numbers of heavy ($c_k=\infty$), resp.\ light ($c_k=0$), nodes & $h,l\le m$ & below\\
$n=K-h-l$ & number of regular nodes, the turning points & $n\ge1$ & below\\
$\omega$ & the measure on the regular nodes & $\omega>0$ & \eqref{eq:farmeasure}\\
$r=m-h$ & the degree of the Hankel determinants & $0\le r\le n$ & below\\
$\mathfrak s$ & the shift & $\mathfrak s\in\R$ & below\\
$y_k$, $\Upsilon=\prod_{k\le K}y_k$ & the ratios $\mu_k/\rho$ and their product & $y_k>0$ & here\\
$E_\Omega$ & the scaling of the frame $(e_\gamma,e_1)$ & $\det E_\Omega=1$ & here\\\bottomrule
\end{tabular}
\caption{The data of the far-node limit.}\label{tab:fardata}
\end{table} Let $w_1<\dots<w_K$ be real, $c_k=[\alpha_k:\beta_k]\in\RP^1$ for
$k\le K$, and $\gamma\in(0,\infty)$. We assume that $(c_1,\dots,c_K,\gamma)$ lies in the admissible set of
Definition~\ref{def:admissible} for the $N$ nodes $w_1,\dots,w_K,\Omega$ with $\Omega>w_K$. With $P(x)=\prod_{k\le K}(x-w_k)$ the weights of
these nodes are $c_k/(P'(w_k)(w_k-\Omega))$ and $\gamma/P(\Omega)$. So the condition says that $-c_k/P'(w_k)\in[0,\infty]$ for all $k$,
with at most $m$ zeros and at most $m$ infinities among the $c_k$; it does not depend on $\Omega$, and it depends on the nodes only
through their order, since $\operatorname{sgn}P'(w_k)=(-1)^{K-k}$. We use the standard representatives $(\alpha_k,\beta_k)$ of
\S\ref{subsec:split}, and $(\alpha_N,\beta_N)=(1,\gamma)$ for the far node. Every object below is independent of the representatives or
homogeneous in each pair ($\Pi_N$ of degree $2$, $\mathbf v$ of degree $1$). For $\rho>0$ let $y_k=\mu_k/\rho$ and
$x_k=-1/y_k$ for $k\le K$, and $\Upsilon=\prod_{k\le K}y_k$. Let $G_{\Omega}$ be the ALE soliton matrix of \S\ref{subsec:ALEsolitons} for the $N$
nodes $(w_1,\dots,w_K,\Omega)$ and parameters $(c_1,\dots,c_K,\gamma)$, in the null basis $(e_0,e_1)$, with $\det G_{\Omega}=\rho^2/4$; the
metric is built from $2(-1)^{m+1}G_{\Omega}$ (Corollary~\ref{cor:CinCtor}). Put $e_\gamma=-e_0+\gamma e_1$ and $\zeta=1/\Omega$, and let $E_{\Omega}$ be
the constant change of frame that multiplies $e_\gamma$ by $\Omega^{1/2}$ and $e_1$ by $\Omega^{-1/2}$. The limit is oriented by
$d\rho\wedge dz\wedge d\phi^\gamma\wedge d\phi^1$, with $(\phi^\gamma,\phi^1)$ dual to the frame $(e_\gamma,e_1)$.

\emph{The data of the limit.} Call $k$ \emph{heavy} if $c_k=\infty$ and \emph{light} if $c_k=0$, and let $h$ and $l$ be their
numbers, and $I$ the set of heavy indices; the other nodes are \emph{regular}. The \emph{measure} of the limit is
\begin{equation}\label{eq:farmeasure}
\omega=\sum_{k}a_k\prod_{k_0\in I}(w_k-w_{k_0})^2\,\delta_{w_k},\qquad a_k=-\frac{c_k}{P'(w_k)},
\end{equation}
the sum over the $n=K-h-l$ regular nodes. Its weights $\omega_k=a_k\prod_{k_0\in I}(w_k-w_{k_0})^2$ are positive. We assume $n\ge1$. Put
$r=m-h$, so $0\le r\le n$; the \emph{shift} is $\mathfrak s=\sum_{k\le K}w_k-2\sum_{k_0\in I}w_{k_0}$. The \emph{atom at infinity} is $\gamma\delta_\infty$, with weight $\gamma$.
We use the shifted Hankel family $(n,r,\gamma)$ of Definition~\ref{def:ALFfamily} for $\omega$, $V_j=(H_r(\omega_{(j)}),H^{(\gamma)}_{r+1}(\omega_{(j)}))$, on
the rods $I_0,\dots,I_n$ of the $n$ nodes of $\omega$.

\emph{The closed-form data.} Since $\omega>0$, $n\ge1$, $\gamma>0$ and $0\le r\le n$, $(\omega,r,\gamma,\mathfrak s)$ are family data
(Definition~\ref{def:family}). Let $\mathcal X$, $\mathcal Y$, $F_1$, $F_2$, $\mathbf l$ and $\mathcal U_\omega$ be those of
Definition~\ref{def:family} for these data (see \eqref{eq:famXY}), let $r_k$ and $\varpi_k$ be as in \eqref{eq:rvarpi} for all $K$ nodes, $\mathbf r$ the function $k\mapsto r_k$, and let
$\mathcal U_K\subset\mathcal U_\omega$ be the open set of $(\rho,z)\in\R^2$ different from $(0,w_k)$ for all $K$ nodes.

\emph{Projective form on the $K$ nodes} (the coefficients of the $N$-node Heine sums in $\zeta$, Lemma~\ref{lem:FN}). For $q=(q_1,\dots,q_K)\in\C^K$ and $T\subset\{1,\dots,K\}$ put
\begin{equation}\label{eq:farproj}
\mathbf t_T(q)=\prod_{l'\notin T}\alpha_{l'}\prod_{l\in T}\Bigl(-\frac{\beta_lq_l}{P'(w_l)}\Bigr)\Delta_T^2,\qquad
\mathbf H_s(q)=\sum_{|T|=s}\mathbf t_T(q),\qquad \boldsymbol\Sigma_s(q)=\sum_{|T|=s}\mathbf t_T(q)\,e_1(w_T),
\end{equation}
with $\mathbf H_s=\boldsymbol\Sigma_s=0$ for $s<0$ and $s>K$, and $\mathbf H^{(\gamma)}_{s+1}(q)=\mathbf H_{s+1}(q)-2\gamma\boldsymbol\Sigma_s(q)-\gamma^2\mathbf H_{s-1}(q)$. So $\mathbf H_s(q)$
is the sum \eqref{eq:Hproj} for the pairs $(\alpha_k,-\beta_k/P'(w_k))$, the weights $a_k$ in projective form. Put, on $\mathcal U_K$,
\begin{equation}
\begin{aligned}
\mathcal X^K&=\mathbf H_m(\varpi),&\qquad \mathcal Y^K&=\mathbf H^{(\gamma)}_{m+1}(\varpi)+\gamma\bigl(e_1(w)-i\rho\bigr)\mathcal X^K,\\
F_1^K&=\mathbf H_m(\mathbf r)+\gamma\mathbf H_{m-1}(\mathbf r),&\qquad F_2^K&=\mathbf H_{m+1}(1/\mathbf r)+\gamma\mathbf H_m(1/\mathbf r),
\end{aligned}
\end{equation}
and $\mathbf l^K=(-\mathcal Y^K,\mathcal X^K)$.

\subsection{The closed form}
For $q=(q_1,\dots,q_K,q_N)\in\C^N$ let $\mathbf H^{(N)}_s(q)$ be the projective Heine sum \eqref{eq:Hproj} of the ALE soliton of the $N$
nodes $w_1,\dots,w_K,\Omega$: the sum over the sets $T'$ of $s$ of the $N$ nodes of $\prod_{l'\notin T'}\alpha_{l'}\prod_{l\in T'}\beta_lq_l/P_N'(w_l)\,\Delta_{T'}^2$,
where $P_N(x)=P(x)(x-\Omega)$ and $w_N=\Omega$.
\begin{lemma}[the $N$-node Heine sums in $\zeta=1/\Omega$]\label{lem:FN}
For $q\in\C^N$ and $0\le s\le K$ put
\begin{equation}
\begin{gathered}
E^-_s(\zeta;q)=\sum_{|T|=s}\mathbf t_T(q)\prod_{l\in T}(1-w_l\zeta)^{-1},\qquad E^+_s(\zeta;q)=\sum_{|T|=s}\mathbf t_T(q)\prod_{l\in T}(1-w_l\zeta),\\
\pi(\zeta)=\prod_{k\le K}(1-w_k\zeta)^{-1},
\end{gathered}
\end{equation}
the sums over $T\subset\{1,\dots,K\}$, where $\mathbf t_T(q)$ uses only $q_1,\dots,q_K$, and $E^\pm_s=0$ for $s<0$ and $s>K$. Then for $\Omega=1/\zeta>w_K$ and $0\le s\le N$,
\begin{equation}\label{eq:FN}
\mathbf H^{(N)}_s(q)=\zeta^sE^-_s(\zeta;q)+\gamma\,q_N\,\zeta^{2m+1-s}\,\pi(\zeta)\,E^+_{s-1}(\zeta;q).
\end{equation}
Here $E^+_s$ is a polynomial in $\zeta$, and $E^-_s$ and $\pi$ are rational in $\zeta$ and real-analytic for $|\zeta|<1/\max_k|w_k|$, with
\begin{equation}\label{eq:FNexp}
E^\mp_s=\mathbf H_s(q)\pm\zeta\,\boldsymbol\Sigma_s(q)+O(\zeta^2),\qquad \pi=1+e_1(w)\,\zeta+O(\zeta^2).
\end{equation}
\end{lemma}
\begin{proof}
Split the sets $T'$ of $s$ of the $N$ nodes into those with $N\notin T'$, i.e.\ $T'=T\subset\{1,\dots,K\}$ with $|T|=s$, and those with
$T'=T\cup\{N\}$, $|T|=s-1$. For $l\le K$, $P_N'(w_l)=(w_l-\Omega)P'(w_l)=-\zeta^{-1}(1-w_l\zeta)P'(w_l)$; moreover $P_N'(\Omega)=P(\Omega)=\zeta^{-2m}\pi(\zeta)^{-1}$
and $\Delta_{T\cup\{N\}}^2=\Delta_T^2\prod_{l\in T}(\Omega-w_l)^2=\Delta_T^2\,\zeta^{-2|T|}\prod_{l\in T}(1-w_l\zeta)^2$. With $\alpha_N=1$ the term of $T'=T$ is
\begin{equation}
\prod_{l'\notin T}\alpha_{l'}\prod_{l\in T}\frac{\beta_lq_l}{P_N'(w_l)}\,\Delta_T^2=\mathbf t_T(q)\prod_{l\in T}\frac{\zeta}{1-w_l\zeta},
\end{equation}
and with $\beta_N=\gamma$ the term of $T'=T\cup\{N\}$ is
\begin{equation}
\begin{aligned}
\frac{\gamma q_N}{P(\Omega)}\prod_{l'\notin T}\alpha_{l'}\prod_{l\in T}\frac{\beta_lq_l}{P_N'(w_l)}\,\Delta_{T\cup\{N\}}^2
&=\gamma q_N\zeta^{2m}\pi(\zeta)\,\mathbf t_T(q)\prod_{l\in T}\frac{\zeta}{1-w_l\zeta}\cdot\frac{(1-w_l\zeta)^2}{\zeta^2}\\
&=\gamma q_N\zeta^{2m+1-s}\pi(\zeta)\,\mathbf t_T(q)\prod_{l\in T}(1-w_l\zeta).
\end{aligned}
\end{equation}
Summing over $T$ gives \eqref{eq:FN}. The functions $E^-_s$ and $\pi$ are sums of products of the factors $(1-w_l\zeta)^{-1}$, which are
real-analytic for $|\zeta|<1/\max_k|w_k|$, and $E^+_s$ is a polynomial. Finally $\prod_{l\in T}(1-w_l\zeta)^{\mp1}=1\pm e_1(w_T)\zeta+O(\zeta^2)$ and
$\pi=1+e_1(w)\zeta+O(\zeta^2)$, which gives \eqref{eq:FNexp}.
\end{proof}

\begin{lemma}[reduction of the phantom nodes]\label{lem:FR}
Let $\kappa_I=\prod_{l\notin I}\alpha_l\prod_{l\in I}\bigl(-\beta_l/P'(w_l)\bigr)\cdot\Delta_I^2$, a non-zero real constant, and for
$q\in\C^K$ let $q_I=\prod_{l\in I}q_l$ and $w_I=\sum_{l\in I}w_l$. Then for every $s$,
\begin{equation}\label{eq:FR}
\mathbf H_s(q)=\kappa_I\,q_I\,H_{s-h}(q\omega),\qquad
\boldsymbol\Sigma_s(q)=\kappa_I\,q_I\bigl(\Sigma_{s-h}(q\omega)+w_IH_{s-h}(q\omega)\bigr),
\end{equation}
where $\omega$ is the measure \eqref{eq:farmeasure} on the regular nodes and $q\omega$ uses the restriction of $q$ to them. Consequently,
on $\mathcal U_K$,
\begin{equation}
\mathcal X^K=\kappa_I\varpi_I\,\mathcal X,\qquad \mathcal Y^K=\kappa_I\varpi_I\,\mathcal Y,\qquad
F_1^K=\kappa_I\mathbf r_I\,F_1,\qquad F_2^K=\kappa_I\mathbf r_I^{-1}F_2,
\end{equation}
with $\varpi_I=\prod_{l\in I}\varpi_l$ and $\mathbf r_I=\prod_{l\in I}r_l$. Since
$|\varpi_I|=1$, it follows that
\[
\frac{\operatorname{Re}(\mathbf l^K\mathbf l^{K*})}{F_1^KF_2^K}=\frac{\operatorname{Re}(\mathbf l\,\mathbf l^*)}{F_1F_2},\qquad
F_1^KF_2^K=\kappa_I^2F_1F_2,
\]
and $F_1^KF_2^K>0$ on $\mathcal U_K$, by Lemma~\ref{lem:famid}\textup{(i)}.
\end{lemma}
\begin{proof}
A regular or light node has $\alpha_l\ne0$, and a heavy node has $\beta_l\ne0$, so $\kappa_I\ne0$. As in
Lemma~\ref{lem:Hpos}(b), applied to the pairs $(\alpha_l,-\beta_l/P'(w_l))$, $\mathbf t_T(q)=0$ unless $T=I\cup S$ with $S$ a set of regular
nodes (otherwise $T$ misses a heavy node, with $\alpha_{l'}=0$, or contains a light one, with $\beta_l=0$). For such $T$, since
$-\beta_l/P'(w_l)=\alpha_la_l$ for regular $l$ and $\Delta_T^2=\Delta_I^2\Delta_S^2\prod_{l\in S,\,k_0\in I}(w_l-w_{k_0})^2$, we get
$\mathbf t_T(q)=\kappa_Iq_IP_S(q\omega)$, with $|S|=|T|-h$ and $e_1(w_T)=e_1(w_S)+w_I$. Summing over $S$ gives \eqref{eq:FR}.
Insert \eqref{eq:FR} into the definitions, with $m-h=r$:
$\mathcal X^K=\kappa_I\varpi_IH_r(\varpi\omega)$ and
$\mathbf H^{(\gamma)}_{m+1}(\varpi)=\kappa_I\varpi_I\bigl(H^{(\gamma)}_{r+1}(\varpi\omega)-2\gamma w_IH_r(\varpi\omega)\bigr)$, so
$\mathcal Y^K=\kappa_I\varpi_I\bigl(H^{(\gamma)}_{r+1}(\varpi\omega)+\gamma(e_1(w)-2w_I-i\rho)H_r(\varpi\omega)\bigr)=\kappa_I\varpi_I\mathcal Y$, since
$e_1(w)-2w_I=\mathfrak s$. Likewise for $F_1^K$ and $F_2^K$, where $q_I=\mathbf r_I^{\pm1}$. The last assertions follow from
$|\varpi_I|=1$ and $\mathbf r_I\mathbf r_I^{-1}=1$.
\end{proof}

\begin{theorem}[the far-node limit]\label{thm:farclosed}
Let $G_\Omega$, $e_\gamma$ and $E_\Omega$ be as in the setting above, and let $\mathcal X$, $\mathcal Y$, $F_1$, $F_2$ and $\mathbf l$ be as in
\eqref{eq:famXY}.
\begin{enumerate}[leftmargin=*,label=\textup{(\roman*)}]
\item\label{it:fcCF} For every compact $\mathcal K\subset\Hb^\circ$ there are $\zeta_0>0$ and a real-analytic function
$\mathcal G$ of $(\rho,z,\zeta)$ on a neighbourhood of $\mathcal K\times[0,\zeta_0]$ such that $E_\Omega^TG_\Omega E_\Omega$, in the frame
$(e_\gamma,e_1)$, equals $\mathcal G(\rho,z,1/\Omega)$ for $(\rho,z)\in\mathcal K$ and $\Omega\ge1/\zeta_0$. Its value $G_\infty=\mathcal G|_{\zeta=0}$, the \emph{far-node limit},
satisfies on $\Hb^\circ$
\begin{equation}\label{eq:farCF}
2(-1)^{m+1}G_\infty=\frac{\operatorname{Re}(\mathbf l\,\mathbf l^*)}{F_1F_2},\qquad \mathbf l=(-\mathcal Y,\mathcal X),
\end{equation}
i.e.\ $2(-1)^{m+1}G_\infty(x,x)=|\mathcal Xx^1-\mathcal Yx^\gamma|^2/(F_1F_2)$ for $x=x^\gamma e_\gamma+x^1e_1$: in the frame $(e_\gamma,e_1)$,
$2(-1)^{m+1}G_\infty$ is the family matrix $\tilde G_{\omega,r,\gamma,\mathfrak s}$ of Definition~\ref{def:family}.
\item\label{it:fcreg}\label{it:fcnu} $G_\infty$ solves \eqref{eq:hm}, $\det G_\infty=\rho^2/4$, $(-1)^{m+1}G_\infty$ is positive definite,
and a function $\nu$ on $\Hb^\circ$ solves \eqref{eq:nu} for $G_\infty$ if and only if
\begin{equation}\label{eq:nuinfty}
e^{2\nu}=C'\,F_1F_2
\end{equation}
with a constant $C'>0$. With $C'=C_\nu$, the constant fixed for all soliton metrics in the last sentence of Theorem~\ref{thm:family}, in the
coordinates $(\phi^\gamma,\phi^1)$ dual to the frame $(e_\gamma,e_1)$,
\begin{equation}
g_\infty=\frac{|\mathcal X\,d\phi^1-\mathcal Y\,d\phi^\gamma|^2}{F_1F_2}+C_\nu\,F_1F_2\,(d\rho^2+dz^2)
\end{equation}
is a real-analytic Ricci-flat toric metric on $\Hb^\circ\times T^2$.
\end{enumerate}
\end{theorem}
\begin{proof}
The closed formula of Theorem~\ref{thm:CF} for the $N$ nodes is expanded in $\zeta=1/\Omega$ (Lemma~\ref{lem:FN}); after the scaling
$E_\Omega$ the leading terms are the data $\mathcal X,\mathcal Y,F_1,F_2$ of the $K$ finite nodes, the phantom nodes are removed by
Lemma~\ref{lem:FR}, and the field equations pass to the limit by analyticity in $\zeta$.

\emph{Step 1: the identity at finite $\Omega$.} Fix $\Omega>w_K$. The data $(c_1,\dots,c_K,\gamma)$ are admissible for the $N$ nodes. In the
representatives $(\alpha_k,\beta_k)$ and $(1,\gamma)$ let $\mathbf v=(\mathbf H^{(N)}_{m+1}(\varpi),\mathbf H^{(N)}_m(\varpi))$ and
$\Pi_N=\mathbf H^{(N)}_m(\mathbf r)\,\mathbf H^{(N)}_{m+1}(1/\mathbf r)$, where now $\varpi$ and $\mathbf r$ include the far node, $r_N=\sqrt{\rho^2+(\Omega-z)^2}$ and
$\varpi_N=(\Omega-z-i\rho)/r_N$. By Theorem~\ref{thm:CF}(ii),(iii) for the $N$ nodes, $\Pi_N>0$ on $\Hb^\circ$ (the representatives change $\Pi_N$ by a
positive factor), and in the null basis $2(-1)^{m+1}G_\Omega(u,u)=|\mathbf v_0u^0+\mathbf v_1u^1|^2/\Pi_N$ for real $u=u^0e_0+u^1e_1$. For real
$(x^\gamma,x^1)$ the vector $u=E_\Omega(x^\gamma e_\gamma+x^1e_1)=x^\gamma\Omega^{1/2}e_\gamma+x^1\Omega^{-1/2}e_1$ has, since $e_\gamma=-e_0+\gamma e_1$, the null components
$u^0=-\Omega^{1/2}x^\gamma$ and $u^1=\gamma\Omega^{1/2}x^\gamma+\Omega^{-1/2}x^1$; so $\mathbf v_0u^0+\mathbf v_1u^1=\Omega^{1/2}(\gamma\mathbf v_1-\mathbf v_0)x^\gamma+\Omega^{-1/2}\mathbf v_1x^1$, and
\begin{equation}\label{eq:farstep1}
2(-1)^{m+1}(E_\Omega^TG_\Omega E_\Omega)(x,x)=\frac{\bigl|\Omega^{1/2}(\gamma\mathbf v_1-\mathbf v_0)x^\gamma+\Omega^{-1/2}\mathbf v_1x^1\bigr|^2}{\Pi_N}.
\end{equation}

\emph{Step 2: expansion.} Put $\hat r_N=\sqrt{(1-z\zeta)^2+\rho^2\zeta^2}$. For $\zeta=1/\Omega>0$, $\hat r_N=\zeta r_N$ and $\varpi_N=(1-z\zeta-i\rho\zeta)/\hat r_N$;
the right sides are real-analytic in $(\rho,z,\zeta)$ where $\hat r_N>0$, in particular near $\zeta=0$, with $\hat r_N=1+O(\zeta)$ and
$\varpi_N=1-i\rho\zeta+O(\zeta^2)$. Apply Lemma~\ref{lem:FN} with $q=\varpi,\mathbf r,1/\mathbf r$, i.e.\ $q_N=\varpi_N,\ r_N=\hat r_N/\zeta,\ 1/r_N=\zeta/\hat r_N$. Writing
$E^\pm_s(q)$ for $E^\pm_s(\zeta;q)$, this gives
\begin{equation}
\begin{gathered}
\mathbf v_1=\zeta^m\,\mathsf X(\zeta),\qquad \gamma\mathbf v_1-\mathbf v_0=-\zeta^{m+1}\,\mathsf Y(\zeta),\\
\mathbf H^{(N)}_m(\mathbf r)=\zeta^m\mathsf F_1(\zeta),\qquad \mathbf H^{(N)}_{m+1}(1/\mathbf r)=\zeta^{m+1}\mathsf F_2(\zeta),
\end{gathered}
\end{equation}
with
\begin{equation}
\begin{aligned}
\mathsf X(\zeta)&=E^-_m(\varpi)+\gamma\varpi_N\zeta\,\pi E^+_{m-1}(\varpi),\\
\mathsf Y(\zeta)&=-\gamma\,\frac{E^-_m(\varpi)-\varpi_N\pi E^+_m(\varpi)}{\zeta}-\gamma^2\varpi_N\pi E^+_{m-1}(\varpi)+E^-_{m+1}(\varpi),\\
\mathsf F_1(\zeta)&=E^-_m(\mathbf r)+\gamma\hat r_N\pi E^+_{m-1}(\mathbf r),\qquad
\mathsf F_2(\zeta)=E^-_{m+1}(1/\mathbf r)+\gamma\hat r_N^{-1}\pi E^+_m(1/\mathbf r).
\end{aligned}
\end{equation}
For $\mathsf Y$ we used $\mathbf v_0=\zeta^{m+1}E^-_{m+1}(\varpi)+\gamma\varpi_N\zeta^m\pi E^+_m(\varpi)$. By \eqref{eq:FNexp} the numerator of the quotient in $\mathsf Y$ is
$\mathbf H_m(\varpi)-\mathbf H_m(\varpi)=0$ at $\zeta=0$, and its derivative there is
$\boldsymbol\Sigma_m(\varpi)-\bigl(-i\rho+e_1(w)\bigr)\mathbf H_m(\varpi)+\boldsymbol\Sigma_m(\varpi)$. So the quotient is real-analytic at $\zeta=0$, with the value
$2\boldsymbol\Sigma_m(\varpi)-(e_1(w)-i\rho)\mathbf H_m(\varpi)$, and
\begin{equation}
\begin{gathered}
\mathsf Y(0)=\mathbf H_{m+1}(\varpi)-2\gamma\boldsymbol\Sigma_m(\varpi)-\gamma^2\mathbf H_{m-1}(\varpi)+\gamma(e_1(w)-i\rho)\mathcal X^K=\mathcal Y^K,\\
\mathsf X(0)=\mathcal X^K,\qquad \mathsf F_1(0)=F_1^K,\qquad \mathsf F_2(0)=F_2^K.
\end{gathered}
\end{equation}
Insert this into \eqref{eq:farstep1}. Since $\Omega^{1/2}\zeta^{m+1}=\Omega^{-1/2}\zeta^m=\zeta^{m+1/2}$ and $\Pi_N=\zeta^{2m+1}\mathsf F_1\mathsf F_2$, for $\zeta=1/\Omega$ with $\Omega>w_K$,
\begin{equation}\label{eq:farstep2}
2(-1)^{m+1}(E_\Omega^TG_\Omega E_\Omega)(x,x)=\frac{|\mathsf X(\zeta)x^1-\mathsf Y(\zeta)x^\gamma|^2}{\mathsf F_1(\zeta)\mathsf F_2(\zeta)} .
\end{equation}

\emph{Step 3: analyticity and the limit; proof of \ref{it:fcCF}.} The functions $\mathsf X,\mathsf Y,\mathsf F_1,\mathsf F_2$ and their complex conjugates are
polynomials in $\gamma$, $\zeta$, $\pi$, $(1-w_l\zeta)^{\pm1}$, $\varpi_N$, $\bar\varpi_N$, $\hat r_N^{\pm1}$, $\rho$ and $\varpi_k$, $\bar\varpi_k$, $r_k^{\pm1}$ ($k\le K$), and of
the analytic quotient above. So they are real-analytic in $(\rho,z,\zeta)$ on the set where $(\rho,z)\in\mathcal U_K$, $|\zeta|\max_k|w_k|<1$ and
$\hat r_N>0$, which contains a neighbourhood of $\mathcal K\times\{0\}$. At $\zeta=0$, $\mathsf F_1\mathsf F_2=F_1^KF_2^K>0$ (Lemma~\ref{lem:FR}). By
compactness there is $\zeta_0>0$, with $\zeta_0\max_k|w_k|<1$ (so $1/\zeta_0>w_K$), such that $\mathsf F_1\mathsf F_2>0$ on a neighbourhood of $\mathcal K\times[0,\zeta_0]$. There the
real symmetric matrix $\mathcal G$ with the quadratic form $\frac{(-1)^{m+1}}2|\mathsf Xx^1-\mathsf Yx^\gamma|^2/(\mathsf F_1\mathsf F_2)$, i.e.\ with the entries
$\frac{(-1)^{m+1}}2\operatorname{Re}(\mathsf l_a\bar{\mathsf l}_b)/(\mathsf F_1\mathsf F_2)$ for $\mathsf l=(-\mathsf Y,\mathsf X)$, is real-analytic, and by \eqref{eq:farstep2} it equals
$E_\Omega^TG_\Omega E_\Omega$ for $0<\zeta\le\zeta_0$, since a real symmetric matrix is determined by its quadratic form. At $\zeta=0$ the quadratic form of
$2(-1)^{m+1}\mathcal G$ is $|\mathcal X^Kx^1-\mathcal Y^Kx^\gamma|^2/(F_1^KF_2^K)$, which equals $|\mathcal Xx^1-\mathcal Yx^\gamma|^2/(F_1F_2)$ by Lemma~\ref{lem:FR}. This
is \eqref{eq:farCF}. Since $\mathcal K$ is arbitrary and the limit does not depend on it, $G_\infty$ is defined, and real-analytic, on all
of $\Hb^\circ$.

\emph{Step 4: proof of \ref{it:fcreg}.} By \ref{it:fcCF} and Lemma~\ref{lem:famid}\textup{(iii)}, $\det G_\infty=\rho^2/4$ and
$(-1)^{m+1}G_\infty$ is positive definite for $\rho>0$. Let $U$ be the intersection with $\Hb^\circ\times[0,\zeta_0]$ of the neighbourhood
of Step 3; it is open in $\Hb^\circ\times[0,\zeta_0]$ and contains $\mathcal K\times[0,\zeta_0]$. On $U$, $\mathcal G$ and $\mathsf F_1\mathsf F_2>0$ are
real-analytic (Step 3), and $\mathcal G$ is invertible ($\det\mathcal G=\rho^2/4$). For $0<\zeta\le\zeta_0$, $E_{\Omega}^TG_{\Omega}E_{\Omega}$ solves
\eqref{eq:hm} (Corollary~\ref{cor:CinCtor}(i)), which is invariant under constant changes of frame. By Theorem~\ref{thm:CF}(iv) for the $N$
nodes, $\frac12\log\Pi_N$ solves \eqref{eq:nu} for $G_\Omega$ (in other representatives $\Pi_N$ changes by a constant factor); by
Remark~\ref{rem:nuinv}, and since \eqref{eq:nu} prescribes only $d\nu$, $\frac12\log(\zeta^{-2m-1}\Pi_N)=\frac12\log(\mathsf F_1\mathsf F_2)$ solves
\eqref{eq:nu} for $E_\Omega^TG_\Omega E_\Omega$, for the same $\zeta$. So Remark~\ref{rem:density}, with $\mathcal P=[0,\zeta_0]$ and the dense set $(0,\zeta_0]$, applied to
$\mathcal G$ and $\frac12\log(\mathsf F_1\mathsf F_2)$, shows that near $\mathcal K$, $G_\infty=\mathcal G|_{\zeta=0}$ solves \eqref{eq:hm} and
$\frac12\log(F_1^KF_2^K)=\frac12\log(\kappa_I^2F_1F_2)$, hence $\frac12\log(F_1F_2)$, solves \eqref{eq:nu} for $G_\infty$; as $\mathcal K$ is
arbitrary, on all of $\Hb^\circ$. Since two solutions differ by a constant (Remark~\ref{rem:nuinv}), this proves \eqref{eq:nuinfty}. The
Weyl--Papapetrou reduction (\S\ref{subsec:WP}) gives Ricci-flatness, and the formula for $g_\infty$ is \eqref{eq:farCF} in the dual
coordinates.
\end{proof}

The proof of Theorem~\ref{thm:mkn} uses a bordered form of the limit. Let $L$ be the symmetric $N\times N$ matrix
\begin{equation}\label{eq:Lgeneral}
L_{kl}=\frac{\alpha_k\beta_l+\beta_k\alpha_l}{1+y_ky_l}\ (k,l\le K),\qquad L_{kN}=\frac{\alpha_k\gamma+\beta_k}{y_k},\qquad L_{NN}=2\gamma ,
\end{equation}
and, for $\rho>0$,
\begin{equation}
\Xi_L=\frac{2\,(2/\rho)^{2m^2}\,\Delta(w)^2\prod_{k\le K}r_k}{\Upsilon^2\prod_{k,l\le K}(y_l-x_k)}>0 .
\end{equation}
\begin{lemma}[the bordered form of the limit]\label{lem:bordered}
Let $G_\infty$ be the far-node limit of Theorem~\ref{thm:farclosed}, and $y_k$ and $\Upsilon$ as in the setting above.
\begin{enumerate}[leftmargin=*,label=\textup{(\roman*)}]
\item\label{it:fcdetL} $\det L=\Xi_L\,F_1^KF_2^K$, in any representatives; in particular $\det L>0$ on $\Hb^\circ$.
\item\label{it:fcGinfty} In the frame $(e_\gamma,e_1)$, with $\hat w=(\gamma\alpha_k-\beta_k)_{k\le K}\oplus0$,
\begin{equation}\label{eq:Ginfty}
G_\infty=\frac1\Upsilon\begin{pmatrix}\dfrac{\rho^2}2\bigl(-\gamma-\tfrac12\,\hat w^TL^{-1}\hat w\bigr)&-\dfrac\rho2\,\hat w^TL^{-1}e_N\\[2mm]
-\dfrac\rho2\,\hat w^TL^{-1}e_N&-(L^{-1})_{NN}\end{pmatrix}.
\end{equation}
\end{enumerate}
\end{lemma}
\begin{proof}
\ref{it:fcdetL} We use the notation of the proof of Theorem~\ref{thm:farclosed}. For finite $\Omega$ let $M_N$ be the reduced matrix \eqref{eq:Mproj} of the $N$ nodes, in the representatives
$(\alpha_k,\beta_k)$ and $(1,\gamma)$, let $\mathsf D_N=\diag(1,\dots,1,y_N)$ and $\tau_N=1/y_N$. The matrix $L(\tau_N)=\mathsf D_NM_N\mathsf D_N$ has the entries $M_{kl}$
($k,l\le K$), $(\alpha_k\gamma+\beta_k)/(y_k+\tau_N)$ and $2\gamma/(1+\tau_N^2)$. These are analytic in $\tau_N$ near $0$, and $L(0)=L$. Here
$\mu_N=(\Omega-z)+r_N=\zeta^{-1}\hat\mu_N$ with $\hat\mu_N=1-z\zeta+\hat r_N$, so $y_N=\hat\mu_N/(\rho\zeta)\to\infty$ as $\zeta\to0$, and
$\det L=\lim_{\zeta\to0}y_N^2\det M_N$. By Lemma~\ref{lem:Creduced} and Theorem~\ref{thm:CF}(ii) for the $N$ nodes, in the same representatives,
$(2\rho^2)^{-N}\prod_{k\le N}y_k^2\det M_N=D=\Xi\,\Pi_N$, i.e.
\begin{equation}
\det M_N=2\Bigl(\frac2\rho\Bigr)^{2m(m+1)}\Delta_N(w)^2\prod_{k\le N}\frac{r_k}{y_k}\cdot\frac{\Pi_N}{\prod_{k,l\le N}(y_l-x_k)},
\end{equation}
where $\Delta_N$ is the Vandermonde product of the $N$ nodes and $x_N=-1/y_N$. As $\zeta\to0$: $y_N\sim2/(\rho\zeta)$;
$\Delta_N(w)^2=\Delta(w)^2P(\Omega)^2\sim\Delta(w)^2\zeta^{-4m}$; $r_N/y_N\to\rho/2$; $\Pi_N=\zeta^{2m+1}\mathsf F_1\mathsf F_2$ with
$\mathsf F_1\mathsf F_2\to F_1^KF_2^K$; and
$\prod_{k,l\le N}(y_l-x_k)\sim\prod_{k,l\le K}(y_l-x_k)\cdot\Upsilon\cdot y_N^{2m+1}$, since $y_l-x_N\to y_l$, $y_N-x_k\sim y_N$ and
$y_N-x_N\sim y_N$. Inserting this into the display for $\det M_N$, the powers of $\zeta$ cancel, and, with $\prod_{k\le K}r_k/y_k=\prod_{k\le K}r_k/\Upsilon$,
\begin{equation}
y_N^2\det M_N\to2\Bigl(\frac2\rho\Bigr)^{2m(m+1)}\Bigl(\frac\rho2\Bigr)^{2m}\frac{\Delta(w)^2\prod_{k\le K}r_k}{\Upsilon^2\prod_{k,l\le K}(y_l-x_k)}\,F_1^KF_2^K=\Xi_LF_1^KF_2^K .
\end{equation}
This proves the formula, and $\det L>0$ on $\Hb^\circ$ by Lemma~\ref{lem:FR}, since $\Xi_L>0$, as $r_k,y_k>0$ and $y_l-x_k=y_l+1/y_k>0$.

\ref{it:fcGinfty} With $\hat\mu_N$ as in the proof of \ref{it:fcdetL}, $\tau_N=\rho\zeta/\hat\mu_N$ and $y_N/\Omega=\hat\mu_N/\rho$ are
real-analytic in $(\rho,z,\zeta)$ near $\zeta=0$, with the values $0$ and $2/\rho$ at $\zeta=0$. By Lemma~\ref{lem:Creduced},
$G_\Omega=\rho\hat G$ with $\hat G_{ab}=(\hat\eta_{ab}-\tfrac12\hat c_a^TM_N^{-1}\hat c_b)/(\Upsilon y_N)$, where $\hat c_0=\beta$ and $\hat c_1=\alpha$ are the
vectors of all $N$ pairs. In the frame $(e_\gamma,e_1)$ the border vectors are $\hat c_{e_\gamma}=-\hat c_0+\gamma\hat c_1=(\gamma\alpha_k-\beta_k)_{k\le N}$,
whose last entry vanishes, and $\hat c_1$; since $M_N^{-1}=\mathsf D_NL(\tau_N)^{-1}\mathsf D_N$, $\mathsf D_N\hat c_{e_\gamma}=\hat w$ and
$\mathsf D_N\hat c_1=y_N\,v(\tau_N)$ with $v(\tau)=(\tau\alpha_k)_{k\le K}\oplus1$, and since $E_\Omega=\diag(\Omega^{1/2},\Omega^{-1/2})$ in this frame,
\begin{equation}\label{eq:EGEmatrix}
E_\Omega^TG_\Omega E_\Omega=\frac\rho\Upsilon\,S\Bigl(\begin{pmatrix}-\gamma&-\tfrac12\tau_N\\-\tfrac12\tau_N&0\end{pmatrix}-\tfrac12\,C(\tau_N)^TL(\tau_N)^{-1}C(\tau_N)\Bigr)S,
\end{equation}
where $S=\diag\bigl((\Omega/y_N)^{1/2},(y_N/\Omega)^{1/2}\bigr)$ and $C(\tau)$ is the $N\times2$ matrix with the columns $\hat w$ and $v(\tau)$. By \ref{it:fcdetL}, $\det L(0)>0$, so the right side is
real-analytic in $(\rho,z,\zeta)$ near $\zeta=0$, and at $\zeta=0$, where $\tau_N=0$, $S=\diag((\rho/2)^{1/2},(2/\rho)^{1/2})$ and $v(0)=e_N$, it is the right side of
\eqref{eq:Ginfty}. By Theorem~\ref{thm:farclosed}\ref{it:fcCF} the left side tends to $G_\infty$.
\end{proof}

\subsection{The limit is a soliton on the flat seed}
\begin{theorem}[the limit is a soliton on $\R^3\times S^1$]\label{thm:mkn}
Let $G_\infty$ be the far-node limit \eqref{eq:Ginfty}. For $k\le K$ put
\begin{equation} \beta^-_k=\beta_k-\gamma\alpha_k,\qquad \beta^+_k=\beta_k+\gamma\alpha_k,\qquad b_k=\frac{\beta^+_k}{\beta^-_k}=\frac{c_k+\gamma}{c_k-\gamma}\in\RP^1 . \end{equation}
Let $G_{\rm BZ}$ be the soliton \eqref{eq:BZAF} on $\R^3\times S^1$ with poles at $w_1,\dots,w_K$ and BZ vectors
$m^{(k)}=(\beta^-_k,\beta^+_k/\mu_k)$, i.e.\ BZ parameters $b_k$. (Formula \eqref{eq:BZAF}, with $\Gamma$ as in \eqref{eq:dressed}, does not change when $m^{(k)}$ is multiplied
by a constant, so $b_k=\infty$ is allowed.) Then
\begin{equation} G_\infty=-\frac1{2\gamma}\,\mathsf J_\gamma^TG_{\rm BZ}\,\mathsf J_\gamma,\qquad \mathsf J_\gamma=\begin{pmatrix}0&1\\-\gamma&0\end{pmatrix}, \end{equation}
in the frame $(e_\gamma,e_1)$ on the left and the frame of \eqref{eq:BZAF} on the right. So
$2(-1)^{m+1}G_\infty=(-1)^mB^TG_{\rm BZ}B$ with $B=\gamma^{-1/2}\mathsf J_\gamma$ and $\det B=1$. A heavy node has $b_k=1$ and a light node $b_k=-1$.
\end{theorem}
\begin{proof}
Both sides are rational functions of $(\rho,y,\alpha,\beta,\gamma)$, and the proof is an identity between them. It uses only that $\det L$,
and hence $\det \mathsf A$, is not the zero function, which holds by Lemma~\ref{lem:bordered}\ref{it:fcdetL}; so the identity holds for all data,
admissible or not, and for all $y_k$ at which both sides are defined; here $G_\infty$ denotes the rational expression on the right of \eqref{eq:Ginfty}, and the identity is one of rational functions. Write $\mathsf D_y=\diag(y_k)$ and
$C_{kl}=1/(1+y_ky_l)$, let $\circ$ be the entrywise product, and put
\begin{equation} p=(\beta^-_ky_k)_k,\qquad q=(\beta^+_k)_k,\qquad U=(p\ \ q),\qquad \mathsf A=(pp^T+qq^T)\circ C . \end{equation}
\begin{itemize}[leftmargin=*]
\item \emph{The far-node side.} In \eqref{eq:Lgeneral}, $L_{kN}=\beta^+_k/y_k$, and $\hat w=(-\beta^-)\oplus0$ with $\beta^-=(\beta^-_k)_k$. Since
$\beta^+_k\beta^+_l-\beta^-_k\beta^-_l=2\gamma(\alpha_k\beta_l+\beta_k\alpha_l)$, the Schur complement \eqref{eq:schur} of the corner $2\gamma$ is
\begin{equation} L'=(L_{kl})_{k,l\le K}-\frac{\ell\ell^T}{2\gamma}=-\frac1{2\gamma}\,\mathsf D_y^{-1}\mathsf A\mathsf D_y^{-1},\qquad \ell=(L_{kN})_{k\le K}. \end{equation}
So $\det L=2\gamma\,(2\gamma)^{-K}\Upsilon^{-2}\det \mathsf A$, and $\det \mathsf A\ne0$. The block inverse of $L$ has the blocks $L'^{-1}$,
$-L'^{-1}\ell/(2\gamma)$ and $\frac1{2\gamma}+\ell^TL'^{-1}\ell/(4\gamma^2)$. Inserting them into \eqref{eq:Ginfty}, with $L'^{-1}=-2\gamma \mathsf D_y\mathsf A^{-1}\mathsf D_y$,
$\ell=\mathsf D_y^{-1}q$ and $\mathsf D_y\beta^-=p$, gives
\begin{equation} G_\infty=-\frac1{2\Upsilon}\,\mathsf E\mathsf F\mathsf E,\qquad \mathsf F=\mathrm{Id}-U^T\mathsf A^{-1}U,\qquad \mathsf E=\diag\bigl(\rho\gamma^{1/2},\gamma^{-1/2}\bigr). \end{equation}
\item \emph{The soliton side.} Put $\mu_k=\rho y_k$ and $G_0=\diag(1,\rho^2)$. Then
\begin{equation} m^{(k)}G_0m^{(l)\mathsf T}=\frac{\beta^-_k\beta^-_ly_ky_l+\beta^+_k\beta^+_l}{y_ky_l},\qquad
\frac{G_0m^{(k)\mathsf T}}{\mu_k}=\Bigl(\frac{\beta^-_k}{\rho y_k},\frac{\beta^+_k}{y_k^2}\Bigr), \end{equation}
so $\Gamma=\rho^{-2}\mathsf D_y^{-1}\mathsf A\mathsf D_y^{-1}$. Inserting this into \eqref{eq:BZAF} gives
\begin{equation} G_{\rm BZ}=\Upsilon\,R\mathsf HR,\qquad \mathsf H=\mathrm{Id}-U^T\mathsf D_y^{-1}\mathsf A^{-1}\mathsf D_y^{-1}U,\qquad R=\diag(1,\rho). \end{equation}
\item \emph{Displacement.} Since $(1+y_ky_l)C_{kl}=1$,
\begin{equation} \mathsf A+\mathsf D_y\mathsf A\mathsf D_y=UU^T . \end{equation}
So $\mathsf D_y^{-1}\mathsf A^{-1}\mathsf D_y^{-1}=(UU^T-\mathsf A)^{-1}$, and $\det\mathsf F=\det(\mathsf A-UU^T)/\det \mathsf A=\det(-\mathsf D_y\mathsf A\mathsf D_y)/\det \mathsf A=\Upsilon^2$, since
$\det(-\mathsf D_y\mathsf A\mathsf D_y)=(-1)^K(\det\mathsf D_y)^2\det \mathsf A$ and $K$ is even. By Woodbury's identity \eqref{eq:woodbury}, $(\mathsf A-UU^T)^{-1}=\mathsf A^{-1}+\mathsf A^{-1}U\mathsf F^{-1}U^T\mathsf A^{-1}$. With $U^T\mathsf A^{-1}U=\mathrm{Id}-\mathsf F$ this gives
\begin{equation} \mathsf H=\mathrm{Id}+U^T(\mathsf A-UU^T)^{-1}U=\mathrm{Id}+(\mathrm{Id}-\mathsf F)+(\mathrm{Id}-\mathsf F)\mathsf F^{-1}(\mathrm{Id}-\mathsf F)=\mathsf F^{-1}. \end{equation}
\item \emph{Conclusion.} For a symmetric $2\times2$ matrix, $\mathsf F^{-1}=\mathsf J_1^T\mathsf F\mathsf J_1/\det\mathsf F$, with $\mathsf J_1$ the matrix $\mathsf J_\gamma$ at $\gamma=1$.
Hence $G_{\rm BZ}=\Upsilon^{-1}R\mathsf J_1^T\mathsf F\mathsf J_1R=-2\,\mathsf T^TG_\infty \mathsf T$, where
\begin{equation}
\mathsf T=\mathsf E^{-1}\mathsf J_1R=\begin{pmatrix}0&\gamma^{-1/2}\\-\gamma^{1/2}&0\end{pmatrix}=-\gamma^{1/2}\mathsf J_\gamma^{-1}.
\end{equation}
This is the claim.
\end{itemize}
For $c_k=\infty$, $(\alpha_k,\beta_k)=(0,1)$ gives $b_k=1$; for
$c_k=0$, $(1,0)$ gives $b_k=-1$.
\end{proof}

\begin{remark}[pole exchange]\label{rem:readd}
Write $\mu_k=\rho y_k$. At the point $(\rho,0)$, for the nodes $w_k=\frac\rho2(y_k-1/y_k)$, we have $r_k=\frac\rho2(y_k+1/y_k)$ and
$\varpi_k=(y_k-i)/(y_k+i)$, so the closed form of Theorem~\ref{thm:farclosed}, with $\bar\varpi_k$ written as $1/\varpi_k$, is a
rational function of $(\rho,y,\alpha,\beta,\gamma)$. The substitution $(y_k,\beta_k)\mapsto(-1/y_k,-\beta_k)$, for one $k$, fixes it: it
fixes $w_k$ and sends $r_k\mapsto-r_k$ and $\varpi_k\mapsto-\varpi_k$, so it fixes every factor $-\beta_kq_k/P'(w_k)$ in \eqref{eq:farproj},
hence $\mathbf l^K$, $F_1^K$, $F_2^K$ and $\operatorname{Re}(\mathbf l^K\mathbf l^{K*})/(F_1^KF_2^K)$, which is $2(-1)^{m+1}G_\infty$ by
Lemma~\ref{lem:FR} and \eqref{eq:farCF}. By
Theorem~\ref{thm:mkn}, this substitution is $(\mu_k,b_k)\mapsto(\mub_k,1/b_k)$; so replacing $\mu_k$ by $\mub_k$ and $b_k$ by $1/b_k$ in
\eqref{eq:BZAF}, in the normalising factor as well, does not change it. With Lemmas~\ref{lem:collapse} and~\ref{lem:scalar} it follows that
$(-1)^m$ times \eqref{eq:BZAF} is obtained from the AF seed by removing the anti-soliton poles $\mub_k$ and re-adding them with the
parameters $1/b_k$, so without the normalisation of Definition~\ref{def:soliton}(d). This is not used in the proofs.
\end{remark}

\subsection{The axis}
\begin{proposition}[the axis of the far-node limit]\label{prop:faraxis}
Let $G_\infty$ be the far-node limit, with the conformal factor $e^{2\nu}=C_\nu F_1F_2$ of \eqref{eq:nuinfty}, let $(\omega,r,\gamma,\mathfrak s)$
be its data \eqref{eq:farmeasure}, with the shifted Hankel family $V_j$ and the rods $I_0,\dots,I_n$ of the nodes of $\omega$, and put
$\Phi=\rho^{-1}\cdot2(-1)^{m+1}G_\infty$.
\begin{enumerate}[leftmargin=*,label=\textup{(\roman*)}]
\item\label{it:fa3} Proposition~\ref{prop:axis} applies to $2(-1)^{m+1}G_\infty$ in the frame $(e_\gamma,e_1)$, with the constant $C_\nu$. In
particular, the turning points are the nodes of $\omega$; every other point of the axis, phantom nodes included, is an interior point of
a rod; the rod line of $I_j$ is spanned by
\begin{equation}\label{eq:farrodvec}
H_r(\omega_{(j)})\,e_\gamma+\bigl(H^{(\gamma)}_{r+1}(\omega_{(j)})+\gamma\,\mathfrak s\,H_r(\omega_{(j)})\bigr)\,e_1 ,
\end{equation}
the vector with the components $S_{\mathfrak s}V_j$; the normalised rod vectors are $\pm\sqrt{C_\nu}\,S_{\mathfrak s}V_j$; and $\Phi$ is locally
tamed by these turning points and rod lines.
\item\label{it:fa2} \emph{Moving a phantom.} Moving a phantom node by $\delta$, to any point that is not another node, with $\omega$ and
$\gamma$ fixed, changes $2(-1)^{m+1}G_\infty$ by a constant change of frame: if it is the Gram matrix in the frame $(e_\gamma,e_1)$ before the move, it is the Gram matrix in the frame $(e_\gamma\mp\gamma\delta\,e_1,e_1)$ after it, of
determinant $1$ ($-$ for a light, $+$ for a heavy node), and leaves $F_1F_2$, hence $e^{2\nu}/C_\nu$, unchanged.
\end{enumerate}
\end{proposition}
\begin{proof}
(i) By Theorem~\ref{thm:farclosed}\ref{it:fcCF}, in the frame $(e_\gamma,e_1)$, $2(-1)^{m+1}G_\infty$ is the restriction to $\Hb^\circ$ of the
family matrix $\tilde G$ of $(\omega,r,\gamma,\mathfrak s)$, which is real-analytic on $\mathcal U_\omega$ (Lemma~\ref{lem:famid}\textup{(i)}), and
$e^{2\nu}=C_\nu F_1F_2$ by Theorem~\ref{thm:farclosed}\ref{it:fcnu}. So Proposition~\ref{prop:axis} applies. By its items \ref{it:axrod} and \ref{it:axtp}, the nodes of
$\omega$ are the turning points and every other axis point is an interior point of a rod; a phantom node is not a node of $\omega$.
The vector \eqref{eq:farrodvec} has the components $S_{\mathfrak s}V_j$. The rest is Proposition~\ref{prop:axis}\ref{it:axrod},\ref{it:axtame}.

(ii) A phantom node is not a node of $\omega$, and it enters $2(-1)^{m+1}G_\infty$ and $F_1F_2$ only through $\mathfrak s$
(Theorem~\ref{thm:farclosed}\ref{it:fcCF}, \eqref{eq:famXY}). Move the phantom node
$w_k$ to $w_k+\delta$, a point that is not another node, and re-index the $K$ nodes in increasing order. Give each regular node $w_l$ the new
parameter $c_l=-P'(w_l)\omega_l/\prod_{k_0\in I}(w_l-w_{k_0})^2$, computed with the moved nodes, where $\omega_l>0$ is the old weight of $\omega$ at
$w_l$. Then the new measure \eqref{eq:farmeasure} is $\omega$, the numbers of heavy and of light nodes, and $r$, are unchanged, and the new data are admissible, since
$-c_l/P'(w_l)>0$ for the regular nodes. The shift becomes $\mathfrak s+\epsilon\delta$ with $\epsilon=+1$ for a light and $\epsilon=-1$ for a heavy node. So $F_1$ and $F_2$ are
unchanged \eqref{eq:famXY}, the new $2(-1)^{m+1}G_\infty$ is $\tilde G_{\omega,r,\gamma,\mathfrak s+\epsilon\delta}$ in the frame $(e_\gamma,e_1)$
(Theorem~\ref{thm:farclosed}\ref{it:fcCF}), and by the last statement of Theorem~\ref{thm:family}(b), with $\delta$ replaced by
$\epsilon\delta$ and $(f_0,f_1)=(e_\gamma,e_1)$, this is the old $2(-1)^{m+1}G_\infty$ in the frame $(e_\gamma-\gamma\epsilon\delta e_1,e_1)$.
\end{proof}

\subsection{Infinity}
Write $z+i\rho=Re^{i\theta}$ with $R=\sqrt{\rho^2+z^2}$ and $\theta\in[0,\pi]$, and put $\varepsilon=1/R$. Let
\begin{equation}\label{eq:alphadef} D_\gamma=H_{r+1}(\omega)-\gamma^2H_{r-1}(\omega),\qquad \alpha(\omega,\gamma)=\frac{2D_\gamma}{\gamma H_r(\omega)}=\frac{2\bigl(H_{r+1}(\omega)-\gamma^2H_{r-1}(\omega)\bigr)}{\gamma H_r(\omega)}, \end{equation}
with $H_{-1}=0$ and $H_{n+1}=0$. For $\varepsilon_0>0$, a function of $(\rho,z)\in\Hb^\circ$ with $R>1/\varepsilon_0$ is called
\emph{real-analytic in $(\varepsilon,\theta)$} if it is the restriction to $0<\varepsilon<\varepsilon_0$, $0<\theta<\pi$ of a real-analytic function on
$\{|\varepsilon|<\varepsilon_0\}\times\R$ that is $2\pi$-periodic and even in $\theta$; negative $\varepsilon$ only serve to have an open set around $\varepsilon=0$.
\begin{proposition}[the end of the far-node limit]\label{prop:farinfty}
Let $f=G_\infty(e_\gamma,e_1)/G_\infty(e_1,e_1)$. There is $\varepsilon_0>0$ such that:
\begin{enumerate}[leftmargin=*,label=\textup{(\roman*)}]
\item $G_\infty(e_1,e_1)$ is real-analytic in $(\varepsilon,\theta)$, and $(-1)^{m+1}G_\infty(e_1,e_1)\to\frac1{2\gamma}$ as $\varepsilon\to0$.
\item $f$ is real-analytic in $(\varepsilon,\theta)$, and $f=-2\mathfrak N\cos\theta+C_f+O(\varepsilon)$, uniformly in $\theta$, with a constant $C_f$ and the
\emph{total NUT charge}
\begin{equation}\label{eq:NUTmu}
\mathfrak N=-\frac{\gamma\,\alpha(\omega,\gamma)}4=\frac{\gamma^2H_{r-1}(\omega)-H_{r+1}(\omega)}{2H_r(\omega)} .
\end{equation}
The real-analytic extension of $f$ to the axis equals $-2\mathfrak N+C_f$ on the upper outer rod and $2\mathfrak N+C_f$ on the lower one, for all $R$. Moreover
$C_f=2\gamma\Sigma/H_r(\omega)-\gamma\mathfrak s$, where $\Sigma=\Sigma_r(\omega)$ \eqref{eq:Sigmas}.
\item $e^{2\nu}=C_\nu F_1F_2$ is real-analytic in $(\varepsilon,\theta)$ and tends to the positive constant $C_\nu\gamma H_r(\omega)^2$ as $\varepsilon\to0$.
\end{enumerate}
\end{proposition}
\begin{proof}
Let $\tilde G=\tilde G_{\omega,r,\gamma,\mathfrak s}$ be the family matrix; it is real-analytic on $\mathcal U_\omega$, even in $\rho$, and equals
$2(-1)^{m+1}G_\infty$ on $\Hb^\circ$ in the frame $(e_\gamma,e_1)$ (Theorem~\ref{thm:farclosed}\ref{it:fcCF}, Lemma~\ref{lem:famid}\textup{(i)}).
Let $\bar w=\max_k|w_k|$ over the nodes of $\omega$, and parametrise $(\rho,z)=(R\sin\theta,R\cos\theta)$ with $\theta\in\R$; for $\sin\theta<0$ this is
the reflection $\rho\mapsto-\rho$, under which $\tilde G$ and $F_1F_2$ are invariant. For
$|\varepsilon|<1/\bar w$ put $u_k=1-\varepsilon w_ke^{-i\theta}$, so that
$|u_k|^2=1-2\varepsilon w_k\cos\theta+\varepsilon^2w_k^2\ge(1-|\varepsilon|\bar w)^2>0$, and $e^{i\varphi_k}:=u_k/|u_k|$. The functions $|u_k|$ and $e^{i\varphi_k}$ are
real-analytic in $(\varepsilon,\theta)$ on $\{|\varepsilon|<1/\bar w\}\times\R$, $2\pi$-periodic in $\theta$, and equal to $1$ at $\varepsilon=0$. For $0<\varepsilon<1/\bar w$,
i.e.\ $R=1/\varepsilon>\bar w$, we have $z+i\rho-w_k=Re^{i\theta}u_k$, hence
\begin{equation}
r_k=R\,|u_k|,\qquad \varpi_k=-e^{i\theta}e^{i\varphi_k} .
\end{equation}
With the measures
$\omega^\varphi=\sum_k\omega_ke^{i\varphi_k}\delta_{w_k}$, $|u|\omega=\sum_k\omega_k|u_k|\delta_{w_k}$ and $\omega/|u|$, the homogeneity of the
Heine terms gives, for $\varepsilon>0$,
\begin{equation}
\begin{gathered}
\mathcal X=(-e^{i\theta})^rH_r(\omega^\varphi),\qquad
\frac{\mathcal Y}{\mathcal X}=-e^{i\theta}\frac{H_{r+1}(\omega^\varphi)}{H_r(\omega^\varphi)}-2\gamma\frac{\Sigma_r(\omega^\varphi)}{H_r(\omega^\varphi)}+\gamma^2e^{-i\theta}\frac{H_{r-1}(\omega^\varphi)}{H_r(\omega^\varphi)}+\gamma\mathfrak s-i\gamma\rho,\\
F_1=R^r\bigl(H_r(|u|\omega)+\varepsilon\gamma H_{r-1}(|u|\omega)\bigr),\qquad F_2=R^{-r}\bigl(\varepsilon H_{r+1}(\omega/|u|)+\gamma H_r(\omega/|u|)\bigr),
\end{gathered}
\end{equation}
the second wherever $H_r(\omega^\varphi)\ne0$. At $\varepsilon=0$, $\omega^\varphi=|u|\omega=\omega/|u|=\omega$, and $H_r(\omega)>0$. By continuity and the compactness of
$\{0\}\times[0,2\pi]$ there is $\varepsilon_0\in(0,1/\bar w)$, depending on $\omega$ and $\gamma$, such that for $|\varepsilon|<\varepsilon_0$ and all $\theta$, $H_r(\omega^\varphi)\ne0$ and
\begin{equation}
\hat F:=\bigl(H_r(|u|\omega)+\varepsilon\gamma H_{r-1}(|u|\omega)\bigr)\bigl(\varepsilon H_{r+1}(\omega/|u|)+\gamma H_r(\omega/|u|)\bigr)>0;
\end{equation}
$\hat F$ is real-analytic in $(\varepsilon,\theta)$, $\hat F(0,\theta)=\gamma H_r(\omega)^2$, and $\hat F=F_1F_2$ for $\varepsilon>0$, since the factors $R^{\pm r}$ cancel.

(i) For $\varepsilon>0$, $\tilde G(e_1,e_1)=|\mathcal X|^2/(F_1F_2)=H_r(\omega^\varphi)\overline{H_r(\omega^\varphi)}/\hat F$, and the right side is real-analytic in $(\varepsilon,\theta)$ for $|\varepsilon|<\varepsilon_0$, and equal to $1/\gamma$
at $\varepsilon=0$. So $(-1)^{m+1}G_\infty(e_1,e_1)=\frac12\tilde G(e_1,e_1)\to\frac1{2\gamma}$, uniformly in $\theta$.

(ii) $f=\tilde G(e_\gamma,e_1)/\tilde G(e_1,e_1)=\operatorname{Re}(-\mathcal Y\overline{\mathcal X})/|\mathcal X|^2=-\operatorname{Re}(\mathcal Y/\mathcal X)$. The term $-i\gamma\rho$ is imaginary and drops
out, so for $\varepsilon>0$, $f$ is the real part of a quotient of the right sides above without that term, which is real-analytic in
$(\varepsilon,\theta)$ for $|\varepsilon|<\varepsilon_0$; the factor $\rho=\sin\theta/\varepsilon$ does not occur. At $\varepsilon=0$, with $H_s$ and $\Sigma_r$ of
$\omega$,
\begin{equation}
f_0(\theta)=\cos\theta\,\frac{H_{r+1}-\gamma^2H_{r-1}}{H_r}+\frac{2\gamma\Sigma_r}{H_r}-\gamma\mathfrak s=-2\mathfrak N\cos\theta+C_f .
\end{equation}
Analyticity on a neighbourhood of the compact set $\{0\}\times[0,2\pi]$ gives $f=f_0+O(\varepsilon)$ uniformly in $\theta$. Since $\tilde G$ is
real-analytic on $\mathcal U_\omega$, $f=-\operatorname{Re}(\mathcal Y/\mathcal X)$ at every point of $\mathcal U_\omega$ with $\mathcal X\ne0$, in particular on the
axis. On the upper outer rod ($\rho=0$, $z>w_n$) every $\varpi_k=-1$, so $\varpi\omega=-\omega$ and the term $-i\gamma\rho\mathcal X$ vanishes; since
$H_s(-\omega)=(-1)^sH_s(\omega)$ and $\Sigma_s(-\omega)=(-1)^s\Sigma_s(\omega)$, $\mathcal X=(-1)^rH_r(\omega)\ne0$ and $f=-\mathcal Y/\mathcal X=f_0(0)=-2\mathfrak N+C_f$.
Likewise, on the lower outer rod ($\rho=0$, $z<w_1$) every $\varpi_k=+1$, and $f=f_0(\pi)=2\mathfrak N+C_f$.

(iii) For $\varepsilon>0$, $e^{2\nu}=C_\nu F_1F_2=C_\nu\hat F$; it is real-analytic in $(\varepsilon,\theta)$ and tends to $C_\nu\gamma H_r(\omega)^2>0$.

Finally, under $\theta\mapsto-\theta$, $u_k$ and $\bar u_k$ are exchanged, so $|u_k|$ is invariant, $e^{i\varphi_k}$ and $e^{i\theta}$ are conjugated, and
$\rho\mapsto-\rho$. So $\mathcal X$ and $\mathcal Y$ are conjugated, and $\tilde G(e_1,e_1)$, $f$ and $\hat F$ are even in $\theta$.
\end{proof}

\subsection{The far-node limit lies in \texorpdfstring{$\Ctor$}{C\_toric}}
Let $T_0=\diag\bigl(1/H_r(\omega),1/(\gamma H_r(\omega))\bigr)$, and let
\begin{equation}\label{eq:Vhat} \hat V_j=\bigl(H_r(\omega_{(j)}),\ H^{(\gamma)}_{r+1}(\omega_{(j)})+2\gamma\,\tfrac{\Sigma}{H_r(\omega)}\,H_r(\omega_{(j)})\bigr) \end{equation}
be the Hankel vectors after the shear $(X,Y)\mapsto(X,Y+2\gamma\Sigma X/H_r(\omega))$. Then $\hat V_0=(H_r(\omega),D_\gamma)$ and
$\hat V_n=(-1)^r(H_r(\omega),-D_\gamma)$, so
\begin{equation} T_0\hat V_0=\bigl(1,\tfrac\alpha2\bigr),\qquad T_0\hat V_n=(-1)^r\bigl(1,-\tfrac\alpha2\bigr),\qquad \alpha=\alpha(\omega,\gamma). \end{equation}
\begin{proposition}[the far-node limit lies in $\Ctor$]\label{prop:farCtor}
Put $\Phi=\rho^{-1}\cdot2(-1)^{m+1}G_\infty$ and $e^{2\nu}=C_\nu F_1F_2$. In the basis
\begin{equation} a_\gamma=\gamma^{-1/2}(e_\gamma-C_fe_1),\qquad b_\gamma=\gamma^{1/2}e_1 , \end{equation}
which has determinant $1$ relative to $(e_\gamma,e_1)$, so that $\rho$ and $z$ are unchanged, $\Phi=\Phi_{H,A}$ as in \eqref{eq:PhiHA} with $H^{-1}=2(-1)^{m+1}\gamma\,G_\infty(e_1,e_1)$ and
$A=(f-C_f)/\gamma$, where $f=G_\infty(e_\gamma,e_1)/G_\infty(e_1,e_1)$ and $C_f$ are those of Proposition~\ref{prop:farinfty}, and:
\begin{enumerate}[leftmargin=*,label=\textup{(\roman*)}]
\item\label{it:ft1} $\Phi$ is locally tamed by the rod structure whose turning points are the nodes of $\omega$, whose rod lines are
spanned by $T_0\hat V_j$, and whose type is ALF$^-$ if $\alpha>0$, ALF$^+$ if $\alpha<0$ and AF$_\beta$ if $\alpha=0$, with $\beta$ recorded by the
lattice as in Definition~\ref{def:Ctor};
\item\label{it:ft2} $d(\Phi,\Phi^\sharp_{1,\alpha})=O(R^{-1})$ uniformly in $\theta$, with $\alpha=\alpha(\omega,\gamma)$.
\end{enumerate}
So $(\Phi,\nu)\in\Ctor$, with the model $\Phi^\sharp_{1,\alpha}$ in the basis $(a_\gamma,b_\gamma)$.
\end{proposition}
\begin{proof}
\ref{it:ft1} By \eqref{eq:farrodvec}, with $e_\gamma=\gamma^{1/2}a+C_f\gamma^{-1/2}b$ and $e_1=\gamma^{-1/2}b$, the rod line of $I_j$ is spanned by
\begin{equation} \gamma^{1/2}H_r(\omega_{(j)})\,a+\gamma^{-1/2}\bigl(H^{(\gamma)}_{r+1}(\omega_{(j)})+(C_f+\gamma\mathfrak s)H_r(\omega_{(j)})\bigr)b=\gamma^{1/2}H_r(\omega)\,T_0\hat V_j , \end{equation}
since $C_f+\gamma\mathfrak s=2\gamma\Sigma/H_r(\omega)$ by Proposition~\ref{prop:farinfty}(ii). Local tameness is Proposition~\ref{prop:faraxis}\ref{it:fa3}, since
$(a_\gamma,b_\gamma)$ has determinant $1$ relative to $(e_\gamma,e_1)$. The outer lines are $[(1,\alpha/2)]$ and $[(1,-\alpha/2)]$, and they coincide exactly when $\alpha=0$.

\ref{it:ft2} In the basis $(a_\gamma,b_\gamma)$ write $\Phi=\Phi_{H,A}$ as in \eqref{eq:PhiHA}. Then
\begin{equation}
\begin{gathered}
H^{-1}=2(-1)^{m+1}G_\infty(b_\gamma,b_\gamma)=2(-1)^{m+1}\gamma G_\infty(e_1,e_1),\\
A=2(-1)^{m+1}G_\infty(a_\gamma,b_\gamma)/H^{-1}=(f-C_f)/\gamma.
\end{gathered}
\end{equation}
The model $\Phi^\sharp_{1,\alpha}$ of \eqref{eq:models} is $\Phi_{H^\sharp,A^\sharp}$ with $H^\sharp=1+|\alpha|\varepsilon/2$ and $A^\sharp=\frac\alpha2\cos\theta$. By
Proposition~\ref{prop:farinfty}, $H^{-1}$ and $A-A^\sharp$ are real-analytic in $(\varepsilon,\theta)$, and $H^{-1}=1$ at $\varepsilon=0$; so there are
$\varepsilon_1\in(0,\varepsilon_0)$ and $C''$ with $|H-1|\le C''\varepsilon$, hence $|H-H^\sharp|\le(C''+|\alpha|/2)\varepsilon$, for $0<\varepsilon\le\varepsilon_1$ and all
$\theta$. Also $A-A^\sharp$ vanishes at $\varepsilon=0$, by Proposition~\ref{prop:farinfty}(ii) and $-2\mathfrak N/\gamma=\alpha/2$, so $A-A^\sharp=\varepsilon k_1$
with $k_1$ real-analytic on $\{|\varepsilon|<\varepsilon_0\}\times\R$, $2\pi$-periodic and even in $\theta$. For $0<\varepsilon<\varepsilon_0$, the value of $f$ at
$\theta=0$, resp.\ $\theta=\pi$, is its value at the point $(0,1/\varepsilon)$ of the upper, resp.\ $(0,-1/\varepsilon)$ of the lower, outer rod (it is computed
from $\tilde G$, which is real-analytic on $\mathcal U_\omega$), where $f=\mp2\mathfrak N+C_f$ (Proposition~\ref{prop:farinfty}(ii)); so
$A-A^\sharp=\mp(2\mathfrak N/\gamma+\alpha/2)=0$ there. Hence $k_1=0$ at $\theta=0$ and $\theta=\pi$ for $0<\varepsilon<\varepsilon_0$, and by analyticity for
$|\varepsilon|<\varepsilon_0$. Since $k_1$ is even about $\theta=0$ and about $\theta=\pi$ and vanishes there, $k_1/\sin^2\theta$ extends real-analytically
across $\theta\in\pi\Z$, so it is bounded on $[-\varepsilon_1,\varepsilon_1]\times[0,2\pi]$.
So $|A-A^\sharp|\le C'\varepsilon\sin^2\theta$ for $0<\varepsilon\le\varepsilon_1$ and all $\theta\in[0,\pi]$. By \eqref{eq:coshhalf},
with $\rho=R\sin\theta$ and $H,H^\sharp\ge\frac12$ for small $\varepsilon$,
\begin{equation} \cosh d-1=\frac{\rho^2(H-H^\sharp)^2+(A-A^\sharp)^2}{2\rho^2HH^\sharp}=O(\varepsilon^2)+O\Bigl(\frac{\varepsilon^2\sin^4\theta}{R^2\sin^2\theta}\Bigr)=O(\varepsilon^2), \end{equation}
so $d=O(R^{-1})$ uniformly in $\theta$.
\end{proof}

\begin{remark}\label{rem:farfacts}
\leavevmode
\begin{enumerate}[leftmargin=*,label=\textup{(\alph*)}]
\item \emph{The type.} $\alpha(\omega,\gamma)$ vanishes exactly on the \emph{AF locus} $\gamma=\gamma_*(\omega)=\sqrt{H_{r+1}(\omega)/H_{r-1}(\omega)}$, where
$V_n=(-1)^rV_0$ (Proposition~\ref{prop:ALFstructure}\ref{it:alfs2}); $\alpha>0$ for $\gamma<\gamma_*$ and $\alpha<0$ for $\gamma>\gamma_*$. For
$r=0$, $\gamma_*=\infty$, and for $r=n$, $\gamma_*=0$.
\item \emph{A coordinate.} For fixed $\omega$, $\gamma\mapsto\alpha$ is strictly decreasing, since
$\partial_\gamma\alpha=-2\bigl(H_{r+1}/\gamma^2+H_{r-1}\bigr)/H_r<0$. For $1\le r\le n-1$ its range is $\R$; for $r=0$ it is $(0,\infty)$, and
for $r=n$ it is $(-\infty,0)$. So $(\omega,\gamma)\mapsto(\omega,\alpha)$ is a change of coordinates.
\end{enumerate}
\end{remark}

\section{Surjectivity}\label{sec:surj}
Fix turning points $w_1<\dots<w_n$, $n\ge1$. In this section we show that every admissible configuration of rod lines,
normalised by the model at infinity, is attained by a soliton metric with these turning points (Propositions~\ref{prop:ALEreal}
and~\ref{prop:ALFreal}); this is the existence half of the proof of Theorem~\ref{thm:G}. Since the rod lines of the soliton
metrics are the lines of Hankel families (Corollary~\ref{cor:CinCtor}, Proposition~\ref{prop:farCtor}), it suffices to show that
the map from positive weights $\omega$ on the turning points to the rays of the normalised Hankel family is onto the corresponding cell ($C_r$ of
\S\ref{subsec:cells}, the exchanged cell $C_{r-1}$, or the AF cell $A_r$ of \S\ref{subsec:surjAF}), and we show by degree theory that it is proper of degree $\pm1$.

\emph{Notation.} We use the notation of \S\ref{sec:hankel}: for $\omega=\sum_ka_k\delta_{w_k}$ with $a\in\R^n_{>0}$ and a set $S$ of atoms,
$P_S(\omega)=\prod_{k\in S}a_k\Delta_S^2$ as in \eqref{eq:PS}, so that $H_s(\omega)=\sum_{|S|=s}P_S(\omega)$ by \eqref{eq:heine}, and
$\pi_s(S)=P_S(\omega)/H_s(\omega)$, a probability distribution on the $s$-subsets $S$ of $\{1,\dots,n\}$, for $0\le s\le n$. We write
$\varepsilon_j(S)=(-1)^{|S\cap[1,j]|}$, $\langle f\rangle_s=\sum_{|S|=s}\pi_s(S)f(S)$, with $\langle f\rangle_s=0$ for $s<0$ and $s>n$, and
$e_1(w_S)=\sum_{k\in S}w_k$ as in \eqref{eq:Sigmas}. Since $P_S(\omega_{(j)})=\varepsilon_j(S)P_S(\omega)$, Heine's formula gives
$H_s(\omega_{(j)})=H_s(\omega)\langle\varepsilon_j\rangle_s$.

\subsection{ALE}\label{subsec:surjALE}
Let $0\le r\le n-1$. By \eqref{eq:Vnorm} and Heine's formula, the normalised Hankel family of \S\ref{sec:hankel} is
\begin{equation}\label{eq:Vtilde} \tilde V_j=\bigl(H_r(\omega_{(j)})/H_r(\omega),\ H_{r+1}(\omega_{(j)})/H_{r+1}(\omega)\bigr)=\bigl(\langle\varepsilon_j\rangle_r,\langle\varepsilon_j\rangle_{r+1}\bigr). \end{equation}
It has $\tilde V_0=(1,1)$ and $\tilde V_n=(-1)^r(1,-1)$, and its steps are clockwise with sum $(r+\frac12)\pi$
(Propositions~\ref{prop:Awronskian} and~\ref{prop:Aturning}). Let
\begin{equation}\label{eq:Lw} L_w\colon\R^n_{>0}/\R_{>0}\to C_r,\qquad [a]\mapsto(s_1,\dots,s_n), \end{equation}
record the clockwise steps $s_j\in(0,\pi)$ from $\tilde V_{j-1}$ to $\tilde V_j$, with $C_r$ as in \eqref{eq:cell}. Both sides are connected
manifolds of dimension $n-1$.

All properness arguments of this section rest on the following lemma.
\begin{lemma}[degenerating weights]\label{lem:degen}
Let $\omega^{[i]}=\sum_ka^{[i]}_k\delta_{w_k}$ be such that $[a^{[i]}]$ leaves every compact subset of $\R^n_{>0}/\R_{>0}$, and let $1\le p\le n$.
After passing to a subsequence, every ratio $a_k/a_l$ converges in $[0,\infty]$. Call $a_k$ and $a_l$ of the same \emph{class} if the limit
is in $(0,\infty)$, and $a_k$ \emph{heavier} than $a_l$ if $a_l/a_k\to0$. Let $\mathfrak g$ be the $p$-th heaviest class, counted with
multiplicity, and let $\mathsf H$, $\mathsf M$, $\mathsf L$ be the sets of atoms whose weights are heavier than, of the class of, and lighter
than $\mathfrak g$. Then:
\begin{enumerate}[leftmargin=*,label=\textup{(\roman*)}]
\item $|\mathsf H|\le p-1$, $|\mathsf H|+|\mathsf M|\ge p$, and $\mathsf H\cup\mathsf L\ne\varnothing$.
\item Fix $k_0\in\mathsf M$, let $a'_k=\lim\,(a_k/a_{k_0})\prod_{h\in\mathsf H}(w_k-w_h)^2>0$ for $k\in\mathsf M$, $\omega'=\sum_{k\in\mathsf M}a'_k\delta_{w_k}$ and
$C_{\mathsf H}=\Delta_{\mathsf H}^2\prod_{h\in\mathsf H}a_h$. For $0\le s\le n$,
\begin{equation}\label{eq:degenH}
H_s(\omega)=C_{\mathsf H}\,a_{k_0}^{s-|\mathsf H|}\bigl(H_{s-|\mathsf H|}(\omega')+o(1)\bigr),
\end{equation}
where $H_t(\omega')=0$ for $t<0$ and $t>|\mathsf M|$. If $0\le s-|\mathsf H|\le|\mathsf M|$, then $\pi_s$ converges to the image of the distribution
$\pi'_{s-|\mathsf H|}$ of $\omega'$ under $S'\mapsto\mathsf H\cup S'$.
\item For $S'\subset\mathsf M$, $\varepsilon_j(\mathsf H\cup S')=\sigma_j\varepsilon_{j'}(S')$ with $\sigma_j=(-1)^{|\mathsf H\cap[1,j]|}$, where $j'$ is the number of atoms
of $\mathsf M$ to the left of the rod $j$ and $\varepsilon_{j'}$ refers to the atoms of $\mathsf M$; and $e_1(w_{\mathsf H\cup S'})=e_1(w_{\mathsf H})+e_1(w_{S'})$.
For $k\notin\mathsf M$ the rods $k-1$ and $k$ have the same $j'$, and $\sigma_k=\sigma_{k-1}$ if $k\in\mathsf L$, $\sigma_k=-\sigma_{k-1}$ if $k\in\mathsf H$.
\end{enumerate}
\end{lemma}
\begin{proof}
There are finitely many ratios, so the subsequence exists, and the classes are totally ordered.

(i) At most $p-1$ atoms are heavier than the $p$-th heaviest class, and at least $p$ are at least as heavy. If $\mathsf H\cup\mathsf L=\varnothing$,
all $a_k/a_{k_0}$ converge in $(0,\infty)$, so $[a^{[i]}]$ converges in $\R^n_{>0}/\R_{>0}$, which it does not.

(ii) Write $a_S=\prod_{k\in S}a_k$. If the factors of two products of $d$ weights can be paired so that each factor of the first is at
most as heavy as its partner, then the ratio of the products, a product of $d$ ratios of weights, converges in $[0,\infty)$, to $0$ if some
factor is strictly lighter than its partner. List the atoms by class, heaviest first: those of $\mathsf H$, then those of $\mathsf M$, whose
weights have the class of $a_{k_0}$, then those of $\mathsf L$. The $i$-th heaviest atom of a set $S$ is at most as heavy as the $i$-th atom of
the list. Let $|S|=s$ and $t=s-|\mathsf H|$. If $t<0$, pair the atoms of $S$ with the first $s$ atoms of the list and the $-t$ factors $a_{k_0}$
with the other atoms of $\mathsf H$: so $a_{k_0}^{-t}a_S/a_{\mathsf H}\to0$. If $t\ge0$, pair the atoms of $S$, heaviest first, with those of $\mathsf H$ and
then with $t$ factors $a_{k_0}$. If $t>|\mathsf M|$, the $(|\mathsf H|+|\mathsf M|+1)$-st atom of $S$ lies in $\mathsf L$, so $a_S/(a_{\mathsf H}a_{k_0}^t)\to0$. If
$0\le t\le|\mathsf M|$, then $a_S/(a_{\mathsf H}a_{k_0}^t)\to0$ unless $S=\mathsf H\cup S'$ with $S'\subset\mathsf M$, and then, since
\begin{equation}
\Delta^2_{\mathsf H\cup S'}=\Delta_{\mathsf H}^2\Delta_{S'}^2\prod_{k\in S',h\in\mathsf H}(w_k-w_h)^2,
\end{equation}
$P_S(\omega)=C_{\mathsf H}a_{k_0}^t(P_{S'}(\omega')+o(1))$. As the $\Delta_S^2$ are fixed positive constants, summing over $S$ gives \eqref{eq:degenH}. If
$0\le t\le|\mathsf M|$, then $H_t(\omega')>0$, and dividing $P_S(\omega)$ by \eqref{eq:degenH} gives the limit of $\pi_s$.

(iii) $|(\mathsf H\cup S')\cap[1,j]|=|\mathsf H\cap[1,j]|+|S'\cap[1,j]|$, and $S'\cap[1,j]$ consists of atoms among the first $j'$ atoms of $\mathsf M$. The
atoms to the left of the rods $k-1$ and $k$ differ only by $k$.
\end{proof}

\begin{proposition}\label{prop:Aproper}
The map $L_w$ of \eqref{eq:Lw} is proper.
\end{proposition}
\begin{proof}
Let $[a^{[i]}]$ leave every compact set. It suffices to show that, after passing to a subsequence, some step $s_k$ tends to $0$ or $\pi$.
Apply Lemma~\ref{lem:degen} with $p=r+1$. Then $|\mathsf H|\le r$, and $r'=r-|\mathsf H|$ satisfies $r'+1\le|\mathsf M|$. By Lemma~\ref{lem:degen}(ii),(iii)
for $s=r,r+1$, $\tilde V_j\to\sigma_j\tilde V'_{j'}$, where $\tilde V'$ is the normalised family $(|\mathsf M|,r')$ of $\omega'$. Its vectors are non-zero, by
Proposition~\ref{prop:Awronskian} for $\omega'$, since $r'\le|\mathsf M|-1$. For every atom $k\notin\mathsf M$, which exists by Lemma~\ref{lem:degen}(i),
the limits of $\tilde V_{k-1}$ and $\tilde V_k$ are $\sigma_{k-1}\tilde V'_{j'}$ and $\sigma_k\tilde V'_{j'}$ (Lemma~\ref{lem:degen}(iii)). So $s_k\to0$ if $k\in\mathsf L$ and
$s_k\to\pi$ if $k\in\mathsf H$.
\end{proof}

\begin{proposition}\label{prop:Adegree}
The map $L_w$ of \eqref{eq:Lw} has degree $\deg L_w=\pm1$ for all $n\ge r+1$.
\end{proposition}
\begin{proof}
\emph{Base case $n=r+1$.} By the proof of Proposition~\ref{prop:Aturning}, $\tilde V_j=(-1)^j(1-2F_j,1)$ with $F_j=\sum_{k\le j}\tau_k/\sum_k\tau_k$ and
$\tau_k=\Delta^2_{\{1,\dots,n\}\setminus k}/a_k$. The map $[a]\mapsto[\tau]$ is a diffeomorphism of $\R^n_{>0}/\R_{>0}$, and $[\tau]\mapsto(F_1<\dots<F_{n-1})$ is a
diffeomorphism onto the open simplex. The steps determine the $F_j$ and conversely, smoothly; the $F_j$ determine $[\tau]$, hence $[a]$.
So $L_w$ is injective and a diffeomorphism onto its image. Both sides have
dimension $n-1$, so the image is open by invariance of domain, and it is closed by Proposition~\ref{prop:Aproper}. Since $C_r$ is
connected, $L_w$ is a diffeomorphism onto $C_r$, and $\deg L_w=\pm1$.

\emph{Induction step $n\ge r+2$.} Fix an atom $k$.
\begin{enumerate}[leftmargin=*,label=(\arabic*)]
\item \emph{Extension to the face.} For $s\le r+1\le n-1$, $H_s(\omega)>0$ also when $a_k=0$. So the normalised vectors are
real-analytic in $a$ on a neighbourhood of the face $F_k=\{a_k=0\}$. On $F_k$ the family is the family $(n-1,r)$ with the rods
$I_{k-1},I_k$ merged (Lemma~\ref{lem:Climits}; for $r=0$ directly from $H_1(\omega_{(j)})=\sum_l\pm a_l$). So the extension $\bar L$ maps $F_k$ into the face $\{s_k=0\}$ of $\bar C_r$, and on it
$\bar L=L_{w'}$, $w'=w\setminus w_k$. This face is the cell $C_r$ of $n-1$ steps.
\item \emph{Transversality.} By Proposition~\ref{prop:Awronskian},
$\det(\tilde V_{k-1},\tilde V_k)=-2a_kU_r(w_k)^2/(H_r(\omega)H_{r+1}(\omega))$; here $U_r(w_k)$ does not involve $a_k$, and it and $H_r(\omega)$,
$H_{r+1}(\omega)$ are positive also at $a_k=0$, since $r+1\le n-1$. Since $\sin s_k=-\det(\tilde V_{k-1},\tilde V_k)/(|\tilde V_{k-1}||\tilde V_k|)$
and $\tilde V_{k-1}=\tilde V_k$ on $F_k$,
\begin{equation}
\frac{\partial\sin s_k}{\partial a_k}=\frac{2U_r(w_k)^2}{H_r(\omega)H_{r+1}(\omega)\,|\tilde V_k|^2}>0\qquad\text{on }F_k .
\end{equation}
Near the face we use the \emph{collar chart}: fix $l\ne k$ and, for $y'=(y'_j)$ in the face $\{s_k=0\}$ and $0\le\delta<y'_l$, let
\begin{equation}
\chi(y',\delta)=(s_1,\dots,s_n),\qquad s_k=\delta,\quad s_l=y'_l-\delta,\quad s_j=y'_j\ (j\ne k,l),
\end{equation}
which preserves the sum of the steps. Represent $[a]$ by the slice $a_l=1$ for a fixed $l\ne k$. In the coordinates (points of $F_k$, $a_k$) on the source in this slice and $\chi$ on the target, the Jacobian of
$\bar L$ at a point of the face is block triangular,
\begin{equation}
d\bar L=\begin{pmatrix}dL_{w'}&*\\0&\partial s_k/\partial a_k\end{pmatrix},\qquad \frac{\partial s_k}{\partial a_k}=\frac{\partial\sin s_k}{\partial a_k}>0 ,
\end{equation}
since $s_k=0$ on the face.
\item \emph{Localisation.} Let $y'$ be a regular value of $L_{w'}$, and $y_\delta=\chi(y',\delta)$ with $\delta>0$ small. Suppose $L_w(a^{[i]})=y_{\delta_i}$ with $\delta_i\to0$. If $a^{[i]}$ leaves every compact set, then by the proof of
Proposition~\ref{prop:Aproper} every atom outside $\mathsf M$ forces its own step to $0$ or $\pi$. Only $s_k\to0$, so the set outside $\mathsf M$ is
$\{k\}$, and $k$ is light. So $a^{[i]}$ converges to a point of $F_k$, which is a preimage of $y'$. If $a^{[i]}$ stays in a compact set,
a limit point would be a preimage of the point $(y',0)$ in the open cell, which does not exist.
\item \emph{Counting.} By (2), $d\bar L$ is invertible at each of the finitely many preimages $p$ of $y'$ in $F_k$, since $y'$ is
regular for $L_{w'}$, hence on a neighbourhood $U_p$ of $p$, and by the inverse function theorem $U_p$ can be chosen so that, for
small $\delta$, it contains exactly one preimage of $y_\delta$. Its local degree is the sign of $\det dL_{w'}$ times the fixed sign of the
normal block. By (3), for small $\delta$ all preimages of $y_\delta$ lie in $\bigcup_pU_p$; so $y_\delta$ is a regular value of $L_w$
and there are no other preimages.
\end{enumerate}
Hence $\deg L_w=\pm\deg L_{w'}=\pm1$.
\end{proof}

\begin{corollary}\label{cor:Lwonto}
The map $L_w$ of \eqref{eq:Lw} is onto the cell $C_r$ of \eqref{eq:cell}.
\end{corollary}
\begin{proof}
A proper map of non-zero degree between connected manifolds of the same dimension is onto.
\end{proof}

\begin{proposition}[ALE realisation]\label{prop:ALEreal}
Let $\alpha>0$, let $\ell_0,\dots,\ell_n$ be as in Lemma~\ref{lem:match}, and let $r$ be the integer of
Lemma~\ref{lem:match}(i); so $0\le r\le n-1$, and the clockwise steps of $D_\alpha\ell_j$, $D_\alpha=\diag(1,2/\alpha)$, lie in the cell $C_r$
of \eqref{eq:cell} (Lemma~\ref{lem:match}(ii)). Let $m=\max(r,n-1-r)$. Then there are
$N=2m+1$ nodes, consisting of $w_1,\dots,w_n$, $m-r$ further nodes with parameter $\infty$ and $m-n+1+r$ further nodes with
parameter $0$, and a parameter $c\in\mathcal A$, such that the ALE soliton metric $g_c$ of Corollary~\ref{cor:CinCtor} has the turning
points $w_1,\dots,w_n$, and in the basis $F_\alpha$ of \S\ref{subsec:CinCtor} for $c$ the model $\Phi^\sharp_{0,\alpha}$ and the rod lines $[\ell_0],\dots,[\ell_n]$.
\end{proposition}
\begin{proof}
By Corollary~\ref{cor:Lwonto} there is $\omega=\sum_ka_k\delta_{w_k}$, $a\in\R^n_{>0}$, with $L_w([a])$ equal to the steps of $D_\alpha\ell_j$. Choose
the $m-r$ and the $m-n+1+r$ further nodes anywhere on the axis, distinct and off $w$. By Theorem~\ref{thm:family}\textup{(a)} there is $c\in\cA$
on these $N$ nodes, with $c=\infty$, resp.\ $c=0$, at the further nodes, whose family data are $(\omega,r,0)$. By
Corollary~\ref{cor:CinCtor}\textup{(i)},\textup{(iii)}, $g_c$ has the turning points $w$, and in the basis $F_\alpha$ the model $\Phi^\sharp_{0,\alpha}$ and the rod
lines $[D_\alpha^{-1}\tilde V_j]$, with $\tilde V_j$ the normalised family $(n,r)$ of $\omega$. The vectors $\tilde V_j$ and $D_\alpha\ell_j$ have the same first
ray, that of $(1,1)$, and the same clockwise steps in $(0,\pi)$, so they have the same rays, and the rod lines are $[\ell_j]$.
\end{proof}

\subsection{ALF}\label{subsec:surjALF}
For $\omega$ as above, $0\le r\le n$ and $\gamma\in(0,\infty)$, let $V_j$ be the family $(n,r,\gamma)$ of Definition~\ref{def:ALFfamily}, $\hat V_j$ its
shear and $T_0$ as before Proposition~\ref{prop:farCtor}, and $\alpha=\alpha(\omega,\gamma)$ and $D_\gamma=H_{r+1}(\omega)-\gamma^2H_{r-1}(\omega)=\frac\alpha2\gamma H_r(\omega)$
as in \eqref{eq:alphadef}. For $\alpha\ne0$ let
$\Theta=\diag(1/H_r(\omega),1/D_\gamma)$, so that $\Theta\hat V_0=(1,1)$ and $\Theta\hat V_n=(-1)^r(1,-1)$, and $D_\alpha T_0=\Theta$. The \emph{Hankel rays} of the
family are the rays of $\Theta\hat V_j=D_\alpha T_0\hat V_j$, i.e.\ $D_\alpha$ applied to the rays of $T_0\hat V_j$, which are the normalised rays (\S\ref{subsec:cells}) of the far-node limit (Proposition~\ref{prop:farCtor} and Lemma~\ref{lem:famdegree}). We write $\mathrm{Cov}_r(f,g)=\langle fg\rangle_r-\langle f\rangle_r\langle g\rangle_r$. The sum $\Sigma$ in \eqref{eq:Vhat} is $\Sigma=\Sigma_r(\omega)=H_r(\omega)\langle e_1(w_S)\rangle_r$, by \eqref{eq:Sigmas}. Put
\begin{equation}\label{eq:kappaQ} \kappa=\frac{\gamma H_{r-1}(\omega)}{H_r(\omega)},\qquad Q=\frac{H_{r+1}(\omega)H_{r-1}(\omega)}{H_r(\omega)^2}. \end{equation}
The matrix $T_0$ gives the rod lines (Proposition~\ref{prop:farCtor}); for $\alpha\ne0$ we also use $\Theta=D_\alpha T_0$, for which the vectors
$\Theta\hat V_j$ stay bounded as $|\alpha|\to\infty$ (Lemma~\ref{lem:ALFnorm}(iii)).

\begin{lemma}[normalised far-node vectors]\label{lem:ALFnorm}
Let $\alpha$ be in the range of $\alpha(\omega,\cdot)$ (Remark~\ref{rem:farfacts}(b)), let $\gamma=\gamma(\omega,\alpha)$, let $T_0$, $\Theta$ and $\hat V_j$ be as above,
and let $\kappa$ and $Q$ be as in \eqref{eq:kappaQ}.
\begin{enumerate}[leftmargin=*,label=\textup{(\roman*)}]
\item $\kappa^2+\frac\alpha2\kappa=Q$, and $\kappa$ is the larger root: $\kappa=-\frac\alpha4+\bigl(\frac{\alpha^2}{16}+Q\bigr)^{1/2}$. Moreover $Q\le C(w)$ for all
$\omega>0$, with $C(w)$ depending only on $w$ and $r$.
\item $T_0\hat V_j=\bigl(\langle\varepsilon_j\rangle_r,\ (\frac\alpha2+\kappa)\langle\varepsilon_j\rangle_{r+1}-\kappa\langle\varepsilon_j\rangle_{r-1}-2\,\mathrm{Cov}_r(\varepsilon_j,e_1(w_S))\bigr)$.
In particular $T_0\hat V_j$ is a real-analytic function of $(\omega,\alpha)$. If $1\le r\le n-1$, it is defined for all $\alpha\in\R$, and at $\alpha=0$,
where $\gamma=\gamma_*(\omega)$, $\kappa=\sqrt Q$.
\item If $\alpha\ne0$, then $\Theta\hat V_j=(\xi_j,\eta_j)$ with
\begin{equation} \xi_j=\langle\varepsilon_j\rangle_r,\qquad \eta_j=\langle\varepsilon_j\rangle_{r+1}+\frac{2\kappa}\alpha\bigl(\langle\varepsilon_j\rangle_{r+1}-\langle\varepsilon_j\rangle_{r-1}\bigr)-\frac4\alpha\,\mathrm{Cov}_r\bigl(\varepsilon_j,e_1(w_S)\bigr) . \end{equation}
These vectors are bounded, uniformly in $\omega$ and in $|\alpha|\ge\alpha_0>0$.
\end{enumerate}
In particular the vectors of (ii) and (iii) depend only on $w$, $\pi_{r-1},\pi_r,\pi_{r+1}$, $Q$ and $\alpha$.
\end{lemma}
\begin{proof}
(i) $\kappa(\kappa+\frac\alpha2)=\frac{\gamma H_{r-1}(\omega)}{H_r(\omega)}\cdot\frac{H_{r+1}(\omega)}{\gamma H_r(\omega)}=Q$, since $\frac\alpha2+\kappa=H_{r+1}(\omega)/(\gamma H_r(\omega))$ by \eqref{eq:alphadef}. As $\kappa\ge0$ and $\kappa+\frac\alpha2\ge0$,
$\kappa\ge-\frac\alpha4$, so $\kappa$ is the larger root. For the bound, write $H_{r+1}(\omega)H_{r-1}(\omega)$ by
Heine's formula as a sum over pairs $(S_+,S_-)$ with $|S_+|=r+1$ and $|S_-|=r-1$. Choose $k\in S_+\setminus S_-$ and put $S=S_+\setminus\{k\}$, $S'=S_-\cup\{k\}$.
Then $a_{S_+}a_{S_-}=a_Sa_{S'}$, the Vandermonde factors change by a ratio bounded in terms of $w$, and each pair
$(S,S')$ arises at most $n$ times. So $H_{r+1}(\omega)H_{r-1}(\omega)\le C(w)H_r(\omega)^2$.

(ii) By \eqref{eq:shifted} and Heine's formula,
\begin{equation} H^{(\gamma)}_{r+1}(\omega_{(j)})=H_{r+1}(\omega)\langle\varepsilon_j\rangle_{r+1}-2\gamma H_r(\omega)\langle\varepsilon_je_1(w_S)\rangle_r-\gamma^2H_{r-1}(\omega)\langle\varepsilon_j\rangle_{r-1}, \end{equation}
and the shear adds
\begin{equation}
2\gamma\frac{\Sigma}{H_r(\omega)}H_r(\omega_{(j)})=2\gamma H_r(\omega)\langle e_1(w_S)\rangle_r\langle\varepsilon_j\rangle_r.
\end{equation}
Divide by $\gamma H_r(\omega)$ and use $H_{r+1}(\omega)/(\gamma H_r(\omega))=\frac\alpha2+\kappa$ and $\gamma H_{r-1}(\omega)/H_r(\omega)=\kappa$. By (i), $\kappa$ is a real-analytic
function of $(Q,\alpha)$ where $\alpha^2+16Q>0$, which holds on the range of $\alpha$: for $1\le r\le n-1$, $Q>0$, and for $r\in\{0,n\}$, $\alpha\ne0$. The $\langle\cdot\rangle_s$
and $Q$ are real-analytic in $\omega$. For $1\le r\le n-1$ the range of $\alpha$ is $\R$, and at $\alpha=0$, $\gamma=\gamma_*(\omega)$ (Remark~\ref{rem:farfacts}(a))
and $\kappa=\sqrt Q$ by (i).

(iii) $\Theta\hat V_j=D_\alpha T_0\hat V_j$; multiply the second component of (ii) by $2/\alpha$. For the bound: $|\langle\varepsilon_j\rangle_s|\le1$ and
$|\mathrm{Cov}_r(\varepsilon_j,e_1(w_S))|\le2\sum_k|w_k|$. For $\alpha>0$, $\kappa\le\sqrt Q$. For $\alpha<0$, $\kappa\ge|\alpha|/2$, and
$\frac{2\kappa}\alpha=-1+\frac{2Q}{\alpha\kappa}$ with $|\frac{2Q}{\alpha\kappa}|\le4C(w)/\alpha_0^2$.
\end{proof}

\begin{lemma}[degree of the Hankel family]\label{lem:famdegree}
Let $\Theta$, $T_0$, $\hat V_j$ and $\alpha=\alpha(\omega,\gamma)$ be as above, and let $\gamma_*$ be as in Remark~\ref{rem:farfacts}(a), so $\gamma_*=\infty$ if
$r=0$ and $\gamma_*=0$ if $r=n$. For every $\gamma\in(0,\infty)$ the rays $T_0\hat V_j$ satisfy Lemma~\ref{lem:match}(i) with the integer $r$. Let
$\gamma\ne\gamma_*$, so $\alpha\ne0$. Then the steps of $\Theta\hat V_j$ lie in $(0,\pi)$. If $\alpha>0$, then $0\le r\le n-1$, the steps are clockwise and
they lie in $C_r$. If $\alpha<0$, then $1\le r\le n$, the steps are counter-clockwise, and after exchanging the two components they
lie in $C_{r-1}$.
\end{lemma}
\begin{proof}
$\det T_0>0$ and $\det(V_{j-1},V_j)<0$ (Proposition~\ref{prop:ALFstructure}\ref{it:alfs1}), so the rays $\ell_j=T_0\hat V_j$ satisfy the hypotheses
of Lemma~\ref{lem:match}, with $\ell_0=(1,\alpha/2)$ and $\ell_n=(-1)^r(1,-\alpha/2)$. Let $r'(\gamma)$ be the integer of
Lemma~\ref{lem:match}(i), so $r'\equiv r$ modulo $2$. The steps are continuous in $\gamma\in(0,\infty)$ and stay in $(0,\pi)$, and $\alpha$ is
continuous, so $r'$ is constant on $(0,\infty)$.
\begin{itemize}[leftmargin=*]
\item If $r\le n-1$, then $\gamma_*>0$ (with $\gamma_*=\infty$ if $r=0$). For $0<\gamma<\gamma_*$, the steps of $\Theta\hat V_j=D_\alpha\ell_j$ sum to $(r'+\frac12)\pi$ by
Lemma~\ref{lem:match}(ii). As $\gamma\to0$, $\Theta\hat V_j$ tends to the normalised ALE family $(n,r)$, whose steps lie in $(0,\pi)$ and sum to
$(r+\frac12)\pi$ (Proposition~\ref{prop:Aturning}); the limit holds by Lemma~\ref{lem:ALFnorm}(iii), since $\kappa\to0$ and $\alpha\to\infty$. So $r'=r$.
\item If $r=n$, then $\gamma_*=0$ and $\alpha<0$ for all $\gamma$. Here $H_{n+1}=0$ and $\Sigma_n(\omega_{(j)})=e_1(w)H_n(\omega_{(j)})$, so $\Theta\hat V_j=(\langle\varepsilon_j\rangle_n,\langle\varepsilon_j\rangle_{n-1})$ for every $\gamma$, which is the normalised ALE family $(n,n-1)$
with the components exchanged, whose steps sum to $(n-\frac12)\pi$. By
Lemma~\ref{lem:match}(ii) the exchanged steps sum to $(r'-\frac12)\pi$. So $r'=n$.
\end{itemize}
The remaining claims are Lemma~\ref{lem:match}(ii).
\end{proof}

\begin{proposition}[properness]\label{prop:ALFproper}
Let $1\le r\le n-1$ and $\alpha_0>0$, and let the Hankel rays be those defined before Lemma~\ref{lem:ALFnorm}. The map that
sends $(\omega,\alpha)$ to the steps of the Hankel rays is proper to $C_r$ on
$\R^n_{>0}/\R_{>0}\times[\alpha_0,\infty]$, where $\alpha=\infty$ means the ALE map $L_w$ of the family $(n,r)$, and likewise proper to $C_{r-1}$ (components exchanged) on $\R^n_{>0}/\R_{>0}\times[-\infty,-\alpha_0]$, where $\alpha=-\infty$ means
the ALE map of $(n,r-1)$ with the components exchanged.
\end{proposition}
\begin{proof}
By Remark~\ref{rem:farfacts}(b), $\gamma=\gamma(\omega,\alpha)$ is a function of $(\omega,\alpha)$. For $\alpha>0$ let $\varsigma=1/\alpha$. Solving
$2\varsigma H_{r-1}\gamma^2+H_r\gamma-2\varsigma H_{r+1}=0$, $\gamma$ is real-analytic in $(\omega,\varsigma)$ near $\varsigma=0$, with $\gamma=2\varsigma H_{r+1}/H_r+O(\varsigma^3)$. So
$\kappa=O(\varsigma)$, and by Lemma~\ref{lem:ALFnorm}(iii) the normalised vectors are continuous up to $\varsigma=0$, where they are the
normalised ALE family $(n,r)$. For $\alpha<0$ let $\varsigma=-1/\alpha$; then $1/\gamma$ is analytic in $(\omega,\varsigma)$ near $\varsigma=0$, and by the
rewriting $\frac{2\kappa}\alpha=-1+\frac{2Q}{\alpha\kappa}$ the vectors tend to the exchanged ALE family $(n,r-1)$.

Let $\omega^{[i]}$ leave every compact set, with $\alpha_i$ in one of the two ranges. It suffices to show that, after passing to a
subsequence, some step tends to $0$ or $\pi$. Pass to a subsequence along which $\alpha_i\to\alpha^\infty$, and apply Lemma~\ref{lem:degen} with
$p=r+1$ if $\alpha_i>0$ and $p=r$ if $\alpha_i<0$. Then $|\mathsf H|\le r$, resp.\ $|\mathsf H|\le r-1$, and $r'=r-|\mathsf H|$ satisfies $r'+1\le|\mathsf M|$, resp.\
$1\le r'\le|\mathsf M|$.

\emph{Limits.} By Lemma~\ref{lem:degen}(ii), $\pi_s$ converges to the image of $\pi'_{s-|\mathsf H|}$ for $s\in\{r-1,r,r+1\}$ with
$0\le s-|\mathsf H|\le|\mathsf M|$, and, by \eqref{eq:degenH}, since the factors $C_{\mathsf H}$ and the powers of $a_{k_0}$ cancel,
\begin{equation}
Q\to Q'=H_{r'+1}(\omega')H_{r'-1}(\omega')/H_{r'}(\omega')^2 .
\end{equation}
The first statement covers all three values of $s$, with two exceptions.
\begin{itemize}[leftmargin=*]
\item For $\alpha>0$ and $|\mathsf H|=r$ it fails for $s=r-1$. But then $Q'=0$, since $H_{-1}=0$, so $\kappa\le\sqrt Q\to0$, and the term with
$\langle\varepsilon_j\rangle_{r-1}$, which is bounded by $1$, drops out.
\item For $\alpha<0$ and $|\mathsf M|=r'$ it fails for $s=r+1$. Again $Q'=0$, since $H_{|\mathsf M|+1}(\omega')=0$. After the rewriting,
\begin{equation}
\eta_j=\langle\varepsilon_j\rangle_{r-1}+\frac{2Q}{\alpha\kappa}\bigl(\langle\varepsilon_j\rangle_{r+1}-\langle\varepsilon_j\rangle_{r-1}\bigr)-\frac4\alpha\mathrm{Cov}_r(\varepsilon_j,e_1(w_S)),
\end{equation}
so $\langle\varepsilon_j\rangle_{r+1}$ enters only with the coefficient $\frac{2Q}{\alpha\kappa}$, and $|\frac{2Q}{\alpha\kappa}|\le4Q/\alpha_0^2\to0$.
\end{itemize}
So, by Lemma~\ref{lem:ALFnorm}(i), for finite $\alpha^\infty$, $\kappa$ converges to the larger root of $\kappa^2+\frac{\alpha^\infty}2\kappa=Q'$, which is the $\kappa$ of
the limit family below. By Lemma~\ref{lem:degen}(iii), $\langle\varepsilon_j\rangle_s\to\sigma_j\langle\varepsilon_{j'}\rangle'_{s-|\mathsf H|}$ for these $s$, and the
covariance does not see the constant $e_1(w_{\mathsf H})$, so
\begin{equation}
\mathrm{Cov}_r(\varepsilon_j,e_1(w_S))\to\sigma_j\mathrm{Cov}'_{r'}(\varepsilon_{j'},e_1(w_{S'})),
\end{equation}
where the prime refers to $\omega'$. Hence, by Lemma~\ref{lem:ALFnorm}(iii),
\begin{equation} \Theta\hat V_j\ \to\ \sigma_j\,\Theta'\hat V'_{j'}, \end{equation}
where $\Theta'\hat V'_{j'}$ are the normalised vectors of the family $(\omega',r',\gamma')$ with $\alpha(\omega',\gamma')=\alpha^\infty$ (Remark~\ref{rem:farfacts}(b)),
and the normalised ALE vectors if $\alpha^\infty=\pm\infty$. In the exceptions the dropped term is absent from the limit family as well,
since there $\kappa'=0$, resp.\ $2Q'/(\alpha^\infty\kappa')=0$. These vectors are non-zero, by Proposition~\ref{prop:ALFstructure}\ref{it:alfs1},
resp.\ Proposition~\ref{prop:Awronskian}; for $r'=0$ the first component is $1$.

\emph{A step degenerates.} By Lemma~\ref{lem:degen}(i) some atom $k\notin \mathsf M$ exists. The rods $k-1$ and $k$ have the same $j'$, so
their limit vectors are $\pm\Theta'\hat V'_{j'}$, and the step $s_k$ tends to $0$ or $\pi$.
\end{proof}

\begin{theorem}[surjectivity for ALF]\label{thm:ALFsurj}
Let $\alpha\ne0$, and let $C_r$ be the cell \eqref{eq:cell}. If $\alpha>0$ and $0\le r\le n-1$, the map $\omega\mapsto$ Hankel rays
at $(\omega,\alpha)$, defined before Lemma~\ref{lem:ALFnorm}, is a proper map of degree $\pm1$ to $C_r$, hence onto $C_r$. If $\alpha<0$ and $1\le r\le n$,
the same holds for the cell of the ALE family $(n,r-1)$ with the components exchanged.
\end{theorem}
\begin{proof}
If $r=0$, then $V_j=(1,H_1(\omega_{(j)}))$ and $\Theta\hat V_j=(1,H_1(\omega_{(j)})/H_1(\omega))$ do not depend on $\gamma$; the map is $L_w$, proper of degree
$\pm1$ by Propositions~\ref{prop:Aproper} and~\ref{prop:Adegree}. If $r=n$, then $H_{n+1}=0$, $\hat V_j=(H_n(\omega_{(j)}),-\gamma^2H_{n-1}(\omega_{(j)}))$ and
$\Theta\hat V_j$ is the normalised ALE family $(n,n-1)$ with the components exchanged, independently of $\gamma$; again
Propositions~\ref{prop:Aproper} and~\ref{prop:Adegree} apply. Let $1\le r\le n-1$ and
$\alpha>0$. With $\varsigma=1/\alpha\in[0,1/\alpha_0]$, Proposition~\ref{prop:ALFproper} gives a proper homotopy from $L_w$, which has degree $\pm1$
(Proposition~\ref{prop:Adegree}). So every slice is proper of degree $\pm1$, hence onto, as in the proof of Corollary~\ref{cor:Lwonto}. For $\alpha<0$ use
$\varsigma=-1/\alpha$ and the ALE map of $(n,r-1)$ at $\varsigma=0$.
\end{proof}

\subsection{AF}\label{subsec:surjAF}
Let $1\le r\le n-1$, so $n\ge2$. On the AF locus $\gamma=\gamma_*(\omega)$, $D_{\gamma_*}=0$, so $\hat V_0=(H_r(\omega),0)$ and
$\hat V_n=(-1)^r(H_r(\omega),0)$. So $T_0\hat V_0=(1,0)$ and $T_0\hat V_n=(-1)^r(1,0)$. The \emph{AF rays} are the rays of $T_0\hat V_1,\dots,T_0\hat V_{n-1}$.
By Lemma~\ref{lem:famdegree} at $\gamma=\gamma_*(\omega)$ and Lemma~\ref{lem:match}(i) with $\alpha=0$, their steps are clockwise and sum to $r\pi$, so they lie in the AF cell
\begin{equation}\label{eq:AFcell} A_r=\Bigl\{s\in(0,\pi)^n:\ \textstyle\sum_js_j=r\pi\Bigr\}. \end{equation}

\begin{theorem}[surjectivity for AF]\label{thm:AFsurj}
Let $n\ge2$ and $1\le r\le n-1$, let the AF rays be as above, and let $A_r$ be the AF cell \eqref{eq:AFcell}. The map
$L^{\AF}_w\colon\omega\mapsto$ AF rays at $(\omega,\gamma_*(\omega))$ is proper of degree $\pm1$ onto $A_r$.
\end{theorem}
\begin{proof}
We deform $L^{\AF}_w$ through the maps of Theorem~\ref{thm:ALFsurj}, in the variable $\alpha$. Let $X=\R^n_{>0}/\R_{>0}\times\R$. For $(\omega,\alpha)$
let $\gamma=\gamma(\omega,\alpha)$ (Remark~\ref{rem:farfacts}(b)) and $\ell_j=T_0\hat V_j$. By Lemma~\ref{lem:ALFnorm}(ii) the $\ell_j$ depend on $\omega$ only through
the $\pi_s$ and $Q$, which do not change under $\omega\mapsto\lambda\omega$; so they are real-analytic functions on $X$. At $\alpha=0$, $\gamma=\gamma_*(\omega)$, and
the $\ell_j$ span the AF rays.

\emph{The map.} By Proposition~\ref{prop:ALFstructure}\ref{it:alfs1}, $\det(\ell_{j-1},\ell_j)=\det T_0\det(V_{j-1},V_j)<0$, since the shear in \eqref{eq:Vhat}
has determinant $1$; and $\ell_0=(1,\frac\alpha2)$, $\ell_n=(-1)^r(1,-\frac\alpha2)$ (before Proposition~\ref{prop:farCtor}). So Lemma~\ref{lem:match} applies at every point of
$X$, and its integer $(\sum_js_j-2\arctan(\alpha/2))/\pi$ is $r$ everywhere, by Lemma~\ref{lem:famdegree}. Let $\Psi(\omega,\alpha)=(s_1,\dots,s_n)$ be the clockwise steps of the $\ell_j$. So $\Psi$ is real-analytic and maps $X$ into the
open convex set
\begin{equation}
Y=\Bigl\{s\in(0,\pi)^n:\ \bigl|\textstyle\sum_js_j-r\pi\bigr|<\pi\Bigr\},\qquad\textstyle\sum_js_j\circ\Psi=c(\alpha):=r\pi+2\arctan(\alpha/2).
\end{equation}
For $\alpha\in\R$ let $X_\alpha=\R^n_{>0}/\R_{>0}\times\{\alpha\}$, $Y_c=Y\cap\{\sum_js_j=c\}$, and $\Psi_\alpha\colon X_\alpha\to Y_{c(\alpha)}$ the restriction of $\Psi$. Then
$\Psi_0=L^{\AF}_w$ and $Y_{r\pi}=A_r$.

\emph{Properness.} Let $K\subset Y$ be compact and $(\omega^{[i]},\alpha_i)\in\Psi^{-1}(K)$. After passing to a subsequence, $\Psi(\omega^{[i]},\alpha_i)\to s^\infty\in K$.
Then $c(\alpha_i)\to\sum_js^\infty_j\in(r\pi-\pi,r\pi+\pi)$, so $\alpha_i\to\alpha^\infty=c^{-1}(\sum_js^\infty_j)\in\R$. Suppose that $\omega^{[i]}$ leaves every compact set.
We show that, after passing to a subsequence, there is an atom $k$ for which the limits of $\ell_{k-1}$ and $\ell_k$ are non-zero and parallel;
then $s_k\to0$ or $\pi$, which contradicts $s^\infty\in(0,\pi)^n$.
\begin{itemize}[leftmargin=*]
\item If $\alpha^\infty\ne0$, then for large $i$, $\alpha_i$ has the sign of $\alpha^\infty$ and $|\alpha_i|\ge|\alpha^\infty|/2$. The proof of
Proposition~\ref{prop:ALFproper} gives such an atom $k$ for the vectors $\Theta\hat V_j$. Since $\ell_j=D_{\alpha_i}^{-1}\Theta\hat V_j$ and $D_{\alpha_i}\to D_{\alpha^\infty}$, which is
invertible, $k$ has the same property for the $\ell_j$.
\item If $\alpha^\infty=0$, apply Lemma~\ref{lem:degen} with $p=r+1$, and let $r'=r-|\mathsf H|$, so $0\le r'$ and $r'+1\le|\mathsf M|$. As in the proof of
Proposition~\ref{prop:ALFproper}, $Q\to Q'$, with $Q'=0$ if $r'=0$; so, by Lemma~\ref{lem:ALFnorm}(i), $\kappa\to\sqrt{Q'}$ and $\frac{\alpha_i}2+\kappa\to\sqrt{Q'}$.
By Lemma~\ref{lem:degen}(ii),(iii), $\langle\varepsilon_j\rangle_s\to\sigma_j\langle\varepsilon_{j'}\rangle'_{s-|\mathsf H|}$ for $s=r,r+1$, and for $s=r-1$ if $r'\ge1$, and
$\mathrm{Cov}_r(\varepsilon_j,e_1(w_S))\to\sigma_j\mathrm{Cov}'_{r'}(\varepsilon_{j'},e_1(w_{S'}))$, where the prime refers to $\omega'$; if $r'=0$, then
$\kappa\langle\varepsilon_j\rangle_{r-1}\to0$, since $|\langle\varepsilon_j\rangle_{r-1}|\le1$. By Lemma~\ref{lem:ALFnorm}(ii), $\ell_j\to\sigma_j\ell'_{j'}$. Here, if $r'\ge1$, the
$\ell'_{j'}$ are the vectors $\ell_{j'}$ of $(\omega',r')$ at $\alpha=0$; they are non-zero, since consecutive ones are not parallel, as shown above
for $\omega$. If $r'=0$, then $\ell'_{j'}=(1,0)$, since $\langle\varepsilon_{j'}\rangle'_0=1$ and $\mathrm{Cov}'_0=0$. By Lemma~\ref{lem:degen}(i) there is an atom
$k\notin\mathsf M$, and by Lemma~\ref{lem:degen}(iii) the rods $k-1$ and $k$ have the same $j'$, so the limits of $\ell_{k-1}$ and $\ell_k$ are
$\sigma_{k-1}\ell'_{j'}$ and $\sigma_k\ell'_{j'}$.
\end{itemize}
So a subsequence of $(\omega^{[i]},\alpha_i)$ converges in $X$, and its limit lies in $\Psi^{-1}(K)$. Hence $\Psi$ is proper. For compact $K\subset Y_c$,
$\Psi_\alpha^{-1}(K)=\Psi^{-1}(K)$ with $\alpha=c^{-1}(c)$, so every $\Psi_\alpha$ is proper.

\emph{Degree.} $X$ and $Y$ are connected oriented $n$-manifolds and $\Psi$ is proper, so $\deg\Psi$ is defined; so are the slices
$X_\alpha$ and $Y_c$, which are open cells, hence connected, so that $\deg\Psi_\alpha$ is defined as well. Since
$\sum_js_j\circ\Psi=c(\alpha)$, $d\Psi$ maps $TX_\alpha$ into $TY_{c(\alpha)}$, and $d(\sum_js_j)(d\Psi\,\partial_\alpha)=c'(\alpha)=(1+\alpha^2/4)^{-1}>0$. In bases adapted to
$TX=TX_\alpha\oplus\R\partial_\alpha$ and $TY=TY_c\oplus\R v$, with $d(\sum_js_j)v=1$, $d\Psi$ is block upper triangular with the diagonal blocks $d\Psi_\alpha$ and
$c'(\alpha)$. So $s\in Y_c$ is a regular value of $\Psi$ if and only if it is one of $\Psi_\alpha$, $\alpha=c^{-1}(c)$, and $\Psi^{-1}(s)=\Psi_\alpha^{-1}(s)$. Orient $X_\alpha$
and $Y_c$ so that $\partial_\alpha$, resp.\ $v$, completes them to the orientations of $X$ and $Y$; then the local degrees of $\Psi$ and $\Psi_\alpha$ at
the points of $\Psi^{-1}(s)$ agree. By Sard's theorem $\Psi_\alpha$ has regular values, so $\deg\Psi_\alpha=\deg\Psi$ for every $\alpha$.

\emph{Value.} Let $\alpha>0$. For $s\in Y_{c(\alpha)}$ let $m_j(s)$ be the rays from $m_0=(1,\frac\alpha2)$ with the clockwise steps $s_j$. Since $\sum_js_j=c(\alpha)$,
$m_n(s)$ is the ray of $(-1)^r(1,-\frac\alpha2)$, so Lemma~\ref{lem:match} applies with the integer $r$, and by Lemma~\ref{lem:match}(ii) the steps of
$D_\alpha m_j(s)$ lie in $C_r$. This defines a real-analytic map $\phi_\alpha\colon Y_{c(\alpha)}\to C_r$. Its inverse, also real-analytic, sends $t\in C_r$ to the steps of
$D_\alpha^{-1}m'_j$, where $m'_j$ are the rays from $(1,1)$ with the clockwise steps $t_j$: since $\det D_\alpha^{-1}>0$, these steps are clockwise and
lie in $(0,\pi)$, $D_\alpha^{-1}m'_n$ is the ray of $(-1)^r(1,-\frac\alpha2)$, and by Lemma~\ref{lem:match}(i),(ii) the integer of Lemma~\ref{lem:match}(i) is
$r$, so the steps sum to $c(\alpha)$. So $\phi_\alpha$ is a diffeomorphism. As $\Theta\hat V_j=D_\alpha\ell_j$, $\phi_\alpha\circ\Psi_\alpha$ is the map $\omega\mapsto$
Hankel rays at $(\omega,\alpha)$, which has degree $\pm1$ by Theorem~\ref{thm:ALFsurj}. Hence $\deg L^{\AF}_w=\deg\Psi_0=\deg\Psi_\alpha=\pm1$, and $L^{\AF}_w$
is onto $A_r$, as in the proof of Corollary~\ref{cor:Lwonto}.
\end{proof}

\begin{proposition}[ALF and AF realisation]\label{prop:ALFreal}
Let $\alpha\in\R$, and let $\ell_0,\dots,\ell_n$ satisfy the hypotheses of Lemma~\ref{lem:match}, with the integer $r$ of
Lemma~\ref{lem:match}(i). Then for every $w_1<\dots<w_n$ there are $\omega=\sum_ka_k\delta_{w_k}$ with $a\in\R^n_{>0}$, $\gamma>0$ and a far-node configuration as in
\S\ref{sec:far}, with $m=\max(r,n-r)$, $m-r$ heavy and $m+r-n$ light nodes, whose far-node limit has the turning points $w$, and
in the basis $(a_\gamma,b_\gamma)$ of Proposition~\ref{prop:farCtor} the model $\Phi^\sharp_{1,\alpha}$ and the rod lines $[\ell_0],\dots,[\ell_n]$. If $\alpha=0$, then
$1\le r\le n-1$ and $\gamma=\gamma_*(\omega)$.
\end{proposition}
\begin{proof}
By Lemma~\ref{lem:match}, $0\le r\le n-1$ if $\alpha>0$, $1\le r\le n-1$ if $\alpha=0$, and $1\le r\le n$ if $\alpha<0$. If $\alpha\ne0$, Lemma~\ref{lem:match}(ii)
puts the steps of $D_\alpha\ell_j$, exchanged if $\alpha<0$, in $C_r$, resp.\ $C_{r-1}$, and Theorem~\ref{thm:ALFsurj} gives $\omega$ whose Hankel rays
at $(\omega,\alpha)$ are the rays of $D_\alpha\ell_j$. If $\alpha=0$, the steps of the $\ell_j$ lie in $A_r$, and Theorem~\ref{thm:AFsurj} gives $\omega$ whose AF rays
are the rays of the $\ell_j$. Let $\gamma=\gamma(\omega,\alpha)$ (Remark~\ref{rem:farfacts}(b)); so $\gamma=\gamma_*(\omega)$ if $\alpha=0$. In both cases the rays of
$D_\alpha T_0\hat V_j$ and of $D_\alpha\ell_j$ (if $\alpha=0$, of $T_0\hat V_j$ and of $\ell_j$) start at the same ray and have the same steps, so they coincide;
hence $[T_0\hat V_j]=[\ell_j]$. By Theorem~\ref{thm:family}\textup{(b)} there is a
far-node configuration with $m=\max(r,n-r)$, $m-r$ heavy and $m+r-n$ light nodes and the family data $(\omega,r,\gamma,\mathfrak s)$ for some $\mathfrak s$. By
Proposition~\ref{prop:farCtor} its far-node limit has the turning points $w$, and in the basis $(a_\gamma,b_\gamma)$ the model $\Phi^\sharp_{1,\alpha}$ and the rod lines
$[T_0\hat V_j]=[\ell_j]$.
\end{proof}

\begin{remark}[uniqueness of the soliton data]\label{rem:injectivity}
For fixed turning points, the maps of this section from soliton data, modulo scaling, to normalised rod lines are proper of
degree $\pm1$ onto convex open sets. A proper local diffeomorphism from a connected manifold onto a convex open set is a
diffeomorphism; so if their differentials are invertible everywhere, the soliton data of every element of $\Ctor$ are unique up to scaling, apart
from the number and positions of the phantom solitons, which are not parameters of the metric (\S\ref{subsec:soliton}). Theorem~\ref{thm:G} does
not need this; we shall consider it in future work. For the family of type AF in Aksteiner's parametrisation, the invertibility is
proved in \cite[Prop.~6.1]{Aksteiner2026}.
\end{remark}

\section{Proof of the main theorem}\label{sec:proofs}
This section states the detailed form of Theorem~\ref{thm:G} and proves Proposition~\ref{prop:GG}, from which both follow: by
uniqueness, every element of $\Ctor$ is the soliton metric of \S\ref{sec:surj} with the same rod structure and model.

\begin{theorem}[Theorem~\ref{thm:G}, detailed form]\label{thm:Gfull}
Let $(M,g)$ be a simply connected toric gravitational instanton (Definitions~\ref{def:toricRF}
and~\ref{def:GI}) whose torus action has finitely many fixed points, $n\ge1$ of them. Then the end of $M$ is of type
ALE, ALF$^\pm$ or AF$_\beta$, and $g$ is a gravitational soliton. More precisely, let $r$ be the degree index of $g$
(\S\ref{subsec:cells}; $g$ lies in $\Ctor$ by Propositions~\ref{prop:endtype} and~\ref{prop:smoothCtor}), so that the Li--Sun degrees of its rod structure (\S\ref{subsec:rods}) and of the reversed one are $\{n-1-r,r\}$
(ALE), resp.\ $\{n-r,r\}$ (ALF, AF), by Lemma~\ref{lem:degreeindex}.
\begin{enumerate}[leftmargin=*,label=\textup{(\alph*)}]
\item If $g$ is ALE, it is a soliton on $\R^4$ with $2m+1$ poles, $m=\max(r,n-1-r)$: one at each turning point, and $2m+1-n$ phantom
solitons.
\item If $g$ is ALF or AF, it is a Belinski--Zakharov soliton \eqref{eq:BZAF} on $\R^3\times S^1$, a single dressing of the flat seed,
with $2m$ real poles, $m=\max(r,n-r)$: one at each turning point, and $2m-n$ phantom solitons. It is AF exactly when its total NUT
charge \eqref{eq:NUTmu} vanishes.
\end{enumerate}
In both cases there are soliton data forming a positive measure $\omega$ on the turning points, with an atom $\gamma>0$ at infinity in
case (b); the proof produces them by a surjectivity argument, and their uniqueness is not proved here
(Remark~\ref{rem:injectivity}). In case (b) the BZ parameters of the poles at the turning points lie in $\RP^1\setminus\{\pm1\}$. In a suitable basis, the rod line of the $j$-th rod is spanned by $\bigl(H_r(\omega_{(j)}),H_{r+1}(\omega_{(j)})\bigr)$, where $\omega_{(j)}$
is $\omega$ with the signs of the atoms below that rod reversed and $H_s$ is the $s\times s$ Hankel determinant of the moments; in
case (b) the moment of order $2r-1$ in $H_{r+1}$ is shifted by $\gamma$, which we write $H^{(\gamma)}_{r+1}$ \eqref{eq:shifted}.
\end{theorem}

The proof also gives the metric itself in closed form. Let $r_k=\sqrt{\rho^2+(z-w_k)^2}$ be the distance to the $k$-th turning
point in Weyl coordinates, $\varpi_k=(w_k-z-i\rho)/r_k$, and write $\varpi\omega$, $\mathbf r\omega$ and $\omega/\mathbf r$ for the measures on the turning
points with the weights $\varpi_k\omega_k$, $r_k\omega_k$ and $\omega_k/r_k$, where $\omega_k>0$ are the weights of $\omega$ and the bold
$\mathbf r$ is the function $k\mapsto r_k$; $H_s$ and $H^{(\gamma)}_{r+1}$ of a measure with complex weights are defined by the same
determinants, and $H_{-1}=0$. In a suitable basis $(f_0,f_1)$ of the Lie algebra of the torus, with dual angles $\psi^0,\psi^1$, which has determinant $1$
relative to the basis of the model at infinity (Remark~\ref{rem:GGclosed}), and in the Weyl coordinates $(\rho,z)$ of this basis,
\begin{equation}\label{eq:introclosed}
\begin{gathered}
g=\frac{\bigl|\mathcal Y\,d\psi^0-\mathcal X\,d\psi^1\bigr|^2}{F_1F_2}+C'F_1F_2\,(d\rho^2+dz^2),\\
\begin{aligned}
\mathcal X&=H_r(\varpi\omega),&\qquad \mathcal Y&=H^{(\gamma)}_{r+1}(\varpi\omega)-i\gamma\rho\,H_r(\varpi\omega),\\
F_1&=H_r(\mathbf r\omega)+\gamma H_{r-1}(\mathbf r\omega),&\qquad F_2&=H_{r+1}(\omega/\mathbf r)+\gamma H_r(\omega/\mathbf r),
\end{aligned}
\end{gathered}
\end{equation}
with a constant $C'>0$ and with $\gamma=0$ in case (a) (Theorem~\ref{thm:family}, Remark~\ref{rem:GGclosed}). The factorisation of the interpolation determinants on which
\eqref{eq:introclosed} rests was pointed out to us by S.~Aksteiner \cite{Aksteiner2022}, who proves it, for both parities of the number of
poles, through the Cauchy--Pfaffian identity \cite[Prop.~2.2]{Aksteiner2026}.
On the $j$-th rod $\rho=0$ and the $\varpi_k$ are real, $-1$ for the turning points below the rod and $+1$ for those above, so there
$\varpi\omega=\omega_{(j)}$, and the rod line is spanned by $H_r(\omega_{(j)})\,f_0+H^{(\gamma)}_{r+1}(\omega_{(j)})\,f_1$, as in Theorem~\ref{thm:Gfull}. The
phantom solitons do not enter \eqref{eq:introclosed}. For $r=0$, \eqref{eq:introclosed} is the Gibbons--Hawking form of the multi-Eguchi--Hanson
metrics ($\gamma=0$) and of the multi-Taub--NUT metrics ($\gamma>0$) (Proposition~\ref{prop:GHmeh}).

\begin{proposition}[solitons in $\Ctor$]\label{prop:GG}
Every element of $\Ctor$, and every toric Ricci-flat metric built from it with any period lattice, is a gravitational
soliton, with the data described in Theorem~\ref{thm:Gfull}.
\end{proposition}
\begin{proof}
Let $(\Phi,\nu)\in\Ctor$, locally tamed by a rod structure with turning points $w_1<\dots<w_n$, $n\ge1$, and rod lines
$L_0,\dots,L_n$, and with the model $\Phi^\sharp_{h_0,\alpha}$ in a basis $(a,b)$ of $\mathfrak t$ (Definition~\ref{def:Ctor}). The Weyl
coordinates are those of the basis $(a,b)$, and the turning points are taken in these coordinates.

\emph{The soliton metric.} Let $\ell_0=(1,\alpha/2),\ell_1,\dots,\ell_n$ span the normalised rays of $(\Phi,\nu)$, and let $r$ be its
degree index (\S\ref{subsec:cells}); so $[\ell_j]=L_j$, and the $\ell_j$ satisfy the hypotheses of Lemma~\ref{lem:match}. If $h_0=0$,
then $\alpha>0$, the type is ALE, and Proposition~\ref{prop:ALEreal} applies. If $h_0=1$, the type is ALF$^-$ if $\alpha>0$, ALF$^+$ if $\alpha<0$, and AF if
$\alpha=0$, in which case $L_0=L_n$ and $n\ge2$; Proposition~\ref{prop:ALFreal} applies. It gives a soliton metric with the turning points $w$, the model $\Phi^\sharp_{h_0,\alpha}$ and the rod lines $L_j$: an ALE soliton
metric $g_c$ in the basis $F_\alpha$, whose Weyl coordinates are the $(\rho,z)$ of the null-basis formulas
(Corollary~\ref{cor:CinCtor}(iii)), resp.\ a far-node limit $g_\infty$ in the basis $(a_\gamma,b_\gamma)$ of Proposition~\ref{prop:farCtor}.
In each case, let $\Phi_s$ be the harmonic map of the soliton metric in the basis named, as a function of its Weyl coordinates.
Here $\Phi$ and $\Phi_s$ are compared as maps $\Hb^\circ\to\mathcal H$, in the Weyl coordinates of the basis $(a,b)$ and of the basis
named, respectively; the turning points $w$, the rod lines $L_j$ and the model $\Phi^\sharp_{h_0,\alpha}$ are read in these coordinates and
components, and no identification of the two bases of $\mathfrak t$ is made at this stage.

\emph{Uniqueness.} $\Phi$ and $\Phi_s$ are axisymmetric harmonic maps on $\Hb^\circ$. Both are locally tamed by the rod structure with
turning points $w$ and rod lines $L_j$: $\Phi$ by assumption, and $\Phi_s$ by Corollary~\ref{cor:CinCtor} and
Proposition~\ref{prop:farCtor}\ref{it:ft1}. Both satisfy $d(\cdot,\Phi^\sharp_{h_0,\alpha})\to0$ uniformly in $\theta$, for the same model: $\Phi$ by
\ref{it:T2}, and $\Phi_s$ with the rate $O(R^{-1})$, by Corollary~\ref{cor:CinCtor}(iii) and Proposition~\ref{prop:farCtor}\ref{it:ft2}.
Lemma~\ref{lem:uniq} gives $\Phi=\Phi_s$.

\emph{Conclusion.} Let $\tilde G_s=2(-1)^{m+1}G_c$ in the null basis for ALE (Corollary~\ref{cor:CinCtor}), and $\tilde G_s=(-1)^mG_{\rm BZ}$
in the frame of \eqref{eq:BZAF} for ALF and AF; so $\det\tilde G_s=\rho^2$. Let $B$ be the matrix of the basis named above relative to
these. For ALE, $\det B=-1$, since $F_\alpha$ has determinant $1$ relative to the Hankel frame $(e_0,-e_1)$. For ALF and AF,
$B=B_1B_2$, where $B_1$ is the matrix called $B$ in Theorem~\ref{thm:mkn} and $B_2$ is the matrix of $(a_\gamma,b_\gamma)$ relative
to $(e_\gamma,e_1)$; both have determinant $1$ (Theorem~\ref{thm:mkn}, Proposition~\ref{prop:farCtor}), so $\det B=1$. So $\rho$ is the same in both bases, and the Weyl coordinates of the soliton metric in the basis named
are the $(\rho,z)$ of its formulas (Corollary~\ref{cor:CinCtor}(iii), Proposition~\ref{prop:farCtor}). As $\Phi=\Phi_s$, the Gram matrix of $(\Phi,\nu)$ in the basis
$(a,b)$ is $G=B^T\tilde G_sB$, and $\tilde G_s$ is a constant multiple of the ALE soliton matrix \eqref{eq:bordered} with $N=2m+1$
poles, resp.\ of the soliton \eqref{eq:BZAF} on $\R^3\times S^1$ with $2m$ poles. Both are obtained from the flat model of their type by the operations of Definition~\ref{def:soliton},
the first by \S\ref{subsec:ALEsolitons} and the second by \S\ref{subsec:MKN}. By Remark~\ref{rem:nuinv}, $\nu$ solves \eqref{eq:nu}
for $\tilde G_s$ as well, so $e^{2\nu}$ is a constant multiple of $F_1F_2$ (Theorem~\ref{thm:family}). Hence $(\Phi,\nu)$, and every metric \eqref{eq:WP} built
from it with any lattice, is a gravitational soliton.

\emph{The data.} In all cases $(\omega,r,\gamma,\mathfrak s)$ are the family data of the soliton (Theorem~\ref{thm:family}), with $\gamma=0$
for ALE, and the rod lines are spanned by the vectors with the components $S_{\mathfrak s}V_j$ in its Hankel frame
(Proposition~\ref{prop:axis}\ref{it:axrod}); here $\omega$ is the positive measure on the turning points of the proof of
Proposition~\ref{prop:ALEreal}, resp.\ the measure \eqref{eq:farmeasure}. In the ALE case the soliton has $N=2m+1$ poles,
$m=\max(r,n-1-r)$: one at each turning point, with positive weight, $m-r$ phantom solitons with parameter $\infty$ and $m-n+1+r$ with
parameter $0$ at arbitrary further points. In the ALF and AF cases it has $2m$ poles, $m=\max(r,n-r)$: one at each turning point, with BZ
parameter $b_k=(c_k+\gamma)/(c_k-\gamma)$, $|b_k|\ne1$, and $m-r$ phantom solitons with $b=1$ and $m+r-n$ with $b=-1$ at arbitrary further
points (Theorem~\ref{thm:mkn}); the type is AF exactly when $\gamma=\gamma_*(\omega)$, i.e.\ when the total NUT charge \eqref{eq:NUTmu}
vanishes.
\end{proof}

\begin{remark}[closed formulas]\label{rem:GGclosed}
The proof gives every element of $\Ctor$ in closed form. Let $(\Phi,\nu)\in\Ctor$, and let $(\omega,r,\gamma,\mathfrak s)$ be the family data
of the soliton metric in the proof of Proposition~\ref{prop:GG} ($\gamma=0$ for ALE). Transport its Hankel frame to $\mathfrak t$ by
the identification of the basis named in that proof with $(a,b)$; the transported frame has determinant $1$ relative to $(a,b)$
(\S\ref{subsec:CinCtor}, resp.\ Proposition~\ref{prop:farCtor}). Since $\Phi=\Phi_s$, Theorem~\ref{thm:family} says that in it $\rho\Phi$ is the
family matrix $\tilde G_{\omega,r,\gamma,\mathfrak s}$, and $e^{2\nu}$ is a constant multiple of $F_1F_2$. Replace its first vector $f_0$ by $f_0+\gamma\mathfrak s f_1$; in the
new basis $(f_0,f_1)$, again of determinant $1$ relative to $(a,b)$ and with the same Weyl coordinates, the family has the shift $0$, and
$g$ is given by \eqref{eq:introclosed}.
\end{remark}

\begin{proof}[Proof of Theorems~\ref{thm:G} and~\ref{thm:Gfull}]
Let $(M,g)$ be as in Theorem~\ref{thm:G}. By Proposition~\ref{prop:endtype}, its end is of type ALE, ALF$^\pm$ or AF$_\beta$. By
Proposition~\ref{prop:smoothCtor}, $g$ has global Weyl coordinates, it has the form
\eqref{eq:WP} on $\Hb^\circ\times T^2$, with the lattice of $2\pi$-periodic elements, and its augmented harmonic map lies in $\Ctor$,
with the rod structure given by the fixed points and the stabilisers. Since the action has a fixed point, there is at least
one turning point. By Proposition~\ref{prop:GG}, $g$ is a gravitational soliton, with the data described there.
\end{proof}

\section{Examples}\label{sec:examples}
Nothing in this section is used in the proofs. We work out the smallest cases (\S\ref{subsec:smallest}, with Euclidean Kerr), the
Gibbons--Hawking metrics (\S\ref{subsec:GH}), the known smooth instantons (Table~\ref{tab:smooth}) and instantons found with computer
assistance (\S\ref{subsec:explicit}).

\subsection{The smallest cases}\label{subsec:smallest}
We work out the soliton families with the fewest poles. Coordinates are centred so that the formulas are short.

\emph{One ALE pole: flat $\R^4$.} For $N=1$, $m=0$, admissibility (Definition~\ref{def:admissible}) means $c\in(0,\infty)$, the
weight is $a_1=c$, and the soliton metric is the Gibbons--Hawking metric of Proposition~\ref{prop:GHmeh}(a) with $n=1$ and $\omega_1=c$,
i.e.\ flat $\R^4$ (in the frame
$(c^{-1/2}e_0,-c^{1/2}e_1)$ its Gram matrix is $\rho\Phi^\sharp_{0,2}$). Its rod lines are spanned by $V_0=(1,a_1)$ and $V_1=(1,-a_1)$ in the
Hankel frame (Corollary~\ref{cor:CinCtor}\textup{(ii)}).

\emph{Two far poles and one turning point: Taub--NUT.} In the setting of \S\ref{sec:far} let $K=2$, so $m=1$, and put $\delta=w_2-w_1$.
Admissibility means $c_1\in[0,\infty]$ and $c_2\in[-\infty,0]$, with at most one zero and at most one infinity. Let $c_1$ be a phantom
soliton and $c_2$ finite and non-zero. Then $n=1$, and the measure \eqref{eq:farmeasure} is $\omega_2\delta_{w_2}$ with a weight $\omega_2>0$; here $a_2=-c_2/\delta$.
If $c_1=\infty$ (heavy), then $r=0$, $\omega_2=-c_2\delta$, and the limit is the Taub--NUT metric of Proposition~\ref{prop:GHmeh}(b) with $n=1$ and
$c_0=\gamma$; since $\alpha=2\omega_2/\gamma>0$, its type is ALF$^-$ (anti-Taub--NUT). If $c_1=0$ (light), then $r=1$ and $\omega_2=-c_2/\delta$. After
the shear, $\hat V_0=(\omega_2,-\gamma^2)$ and $\hat V_1=(-\omega_2,-\gamma^2)$, and $\alpha=-2\gamma/\omega_2<0$, $\mathfrak N=\gamma^2/(2\omega_2)$. So the normalised
rays are those of $(1,\frac\alpha2)$ and $-(1,-\frac\alpha2)$, and the far-node limit, in the basis of Proposition~\ref{prop:farCtor}, has the turning
point $w_2$ and the rod lines and the model of $\Phi^\sharp_{1,\alpha}$, translated to $w_2$. By Lemma~\ref{lem:uniq} it is that model, $\Phi_{H,A}$
with $H=1+|\alpha|/(2R)$ and $A=\alpha z/(2R)$, where $R$ and $z$ are centred at $w_2$: Taub--NUT (ALF$^+$). TN$^+$ and TN$^-$ are the same metric
with opposite orientations (\S\ref{subsec:models}).

\emph{Two far poles and two turning points: Kerr.} Let $c_1,c_2$ be finite and non-zero. Then $n=2$, $r=1$, and the measure is
$a_1\delta_{w_1}+a_2\delta_{w_2}$ with $a_1=c_1/\delta>0$ and $a_2=-c_2/\delta>0$. Centre $z$ so that $a_1w_1+a_2w_2=0$. Then the shear is
trivial, and by \eqref{eq:shifted} the family $(2,1,\gamma)$ is
\begin{equation} \begin{gathered}
\hat V_0=\bigl(a_1+a_2,\ a_1a_2\delta^2-\gamma^2\bigr),\qquad \hat V_2=\bigl(-a_1-a_2,\ a_1a_2\delta^2-\gamma^2\bigr),\\
\hat V_1=\Bigl(a_2-a_1,\ -a_1a_2\delta^2-\frac{4\gamma a_1a_2\delta}{a_1+a_2}-\gamma^2\Bigr).
\end{gathered} \end{equation}
So $\hat V_2=-\hat V_0$ exactly on the AF locus $\gamma=\gamma_*=\delta\sqrt{a_1a_2}$, where the total NUT charge
$\mathfrak N=(\gamma^2-a_1a_2\delta^2)/(2(a_1+a_2))$ vanishes. By Theorem~\ref{thm:mkn} the limit is the soliton \eqref{eq:BZAF} with two poles
and $b_k=(c_k+\gamma)/(c_k-\gamma)$. On the AF locus these are the Euclidean Kerr metrics, possibly with conical singularities along
the axis (Proposition~\ref{prop:kerr}); off it they are ALF, with $\mathfrak N\ne0$.

\emph{Euclidean Kerr.} For $M>0$ and $\mathsf a\in\R$ put $\sigma=\sqrt{M^2+\mathsf a^2}$, $\mathsf r_\pm=M\pm\sigma$ and $\Omega=\mathsf a/(2M\mathsf r_+)$. The
\emph{Euclidean Kerr metric} \cite{Hawking1977,GH1979} is
\begin{equation}\label{eq:kerr}
g_{M,\mathsf a}=\frac\Sigma\Delta\,d\mathsf r^2+\Sigma\,d\vartheta^2+\frac\Delta\Sigma\bigl(d\tau+\mathsf a\sin^2\vartheta\,d\phi\bigr)^2
+\frac{\sin^2\vartheta}\Sigma\bigl((\mathsf r^2-\mathsf a^2)\,d\phi-\mathsf a\,d\tau\bigr)^2
\end{equation}
on $\{\mathsf r>\mathsf r_+,\ 0<\vartheta<\pi\}\times T^2$, with $\Delta=\mathsf r^2-2M\mathsf r-\mathsf a^2=(\mathsf r-\mathsf r_+)(\mathsf r-\mathsf r_-)$ and
$\Sigma=\mathsf r^2-\mathsf a^2\cos^2\vartheta$. It is obtained from the Kerr metric in Boyer--Lindquist coordinates by $t=-i\tau$ and $a=i\mathsf a$, and it
is Ricci-flat \cite{GH1979}. Here $(\mathsf r,\vartheta)$ are
Boyer--Lindquist coordinates, used only in this example. Note that $\Delta>0$,
$\mathsf r_+^2-\mathsf a^2=2M\mathsf r_+>0$, $\Sigma\ge\mathsf r^2-\mathsf a^2>0$ and $\Sigma-2M\mathsf r=\Delta+\mathsf a^2\sin^2\vartheta>0$ for $\mathsf r>\mathsf r_+$.
\begin{proposition}[Euclidean Kerr]\label{prop:kerr}
\leavevmode
\begin{enumerate}[leftmargin=*,label=\textup{(\roman*)}]
\item In the Weyl coordinates $\rho=\sqrt\Delta\sin\vartheta$, $z=(\mathsf r-M)\cos\vartheta$ and the basis $(\partial_\phi,\partial_\tau)$, the augmented
harmonic map of $g_{M,\mathsf a}$ lies in $\Ctor$. Its turning points are $\pm\sigma$, its rod lines are $[\partial_\phi]$, $[\Omega\partial_\phi+\partial_\tau]$,
$[\partial_\phi]$, its type is AF, and its model is $\Phi^\sharp_{1,0}$, at distance $O(R^{-1})$.
\item Let $G_\infty$ be a far-node limit of \S\ref{sec:far} of type AF with $n=2$ turning points $w_1<w_2$, and let $\Phi$ and $(a_\gamma,b_\gamma)$ be as in
Proposition~\ref{prop:farCtor}. Then its middle rod line is $[ta+b]$ for a unique $t\in\R$, and there are unique $M>0$ and $\mathsf a\in\R$ with
$\sigma=\frac12(w_2-w_1)$ and $\Omega=t$. For them, the Gram matrix $\rho\Phi$ in the basis $(a_\gamma,b_\gamma)$ at $(\rho,z)$ is the Gram matrix of $g_{M,\mathsf a}$ in the
basis $(\partial_\phi,\partial_\tau)$ at $(\rho,z-\frac12(w_1+w_2))$, and the two conformal factors differ by a constant factor.
\item Every $g_{M,\mathsf a}$ arises as in \textup{(ii)} from a far-node configuration with $K=2$ nodes, neither of them a phantom soliton.
\end{enumerate}
\end{proposition}
\begin{proof}
(i) \emph{Weyl coordinates.} A direct computation gives $d\rho^2+dz^2=\lambda\,(d\mathsf r^2/\Delta+d\vartheta^2)$ with
$\lambda=(\mathsf r-M)^2-\sigma^2\cos^2\vartheta>0$, and $R_\pm=\sqrt{\rho^2+(z\mp\sigma)^2}=\mathsf r-M\mp\sigma\cos\vartheta$. So $(\mathsf r,\vartheta)\mapsto(\rho,z)$ is a
real-analytic diffeomorphism of $(\mathsf r_+,\infty)\times(0,\pi)$ onto $\Hb^\circ$, with the inverse $\mathsf r-M=\frac12(R_++R_-)$,
$\sigma\cos\vartheta=\frac12(R_--R_+)$; these formulas extend $\mathsf r$ and $\cos\vartheta$ continuously to the closed half-plane, where
$\{\rho=0,\ |z|<\sigma\}$ is $\{\mathsf r=\mathsf r_+\}$, $\{\rho=0,\ z>\sigma\}$ is $\{\vartheta=0\}$ and $\{\rho=0,\ z<-\sigma\}$ is $\{\vartheta=\pi\}$. In the basis
$(\partial_\phi,\partial_\tau)$ the Gram matrix of \eqref{eq:kerr} is
\begin{equation}\label{eq:kerrG}
G=\frac1\Sigma\begin{pmatrix}\sin^2\vartheta\bigl(\Delta\mathsf a^2\sin^2\vartheta+(\mathsf r^2-\mathsf a^2)^2\bigr)&-2M\mathsf r\mathsf a\sin^2\vartheta\\
-2M\mathsf r\mathsf a\sin^2\vartheta&\Sigma-2M\mathsf r\end{pmatrix},\qquad \det G=\Delta\sin^2\vartheta=\rho^2,
\end{equation}
and $g_{M,\mathsf a}=G_{ab}\,d\phi^ad\phi^b+e^{2\nu}(d\rho^2+dz^2)$ with $e^{2\nu}=\Sigma/\lambda$. Since $g_{M,\mathsf a}$ is Ricci-flat, $(\rho^{-1}G,\nu)$ is an
augmented harmonic map (\S\ref{subsec:WP}). $G$ and $\nu$ are invariant under $\vartheta\mapsto\pi-\vartheta$, i.e.\ $z\mapsto-z$, so the sign
convention for $z$ in \S\ref{subsec:WP} plays no role.

\emph{\ref{it:T1}.} With the lattice generated by the normalised rod vectors $\pm\partial_\phi$ and $\pm(2M\mathsf r_+/\sigma)(\Omega\partial_\phi+\partial_\tau)$
(computed after the proof), $g_{M,\mathsf a}$ is the smooth Euclidean Kerr instanton on $\R^2\times S^2$, which is complete and simply connected,
with $|\mathrm{Rm}|=O(\mathfrak r^{-3})$ and two fixed points \cite{Hawking1977,GH1979}; so it is a toric gravitational instanton as in
Theorem~\ref{thm:G}. By Propositions~\ref{prop:endtype} and~\ref{prop:smoothCtor}, its augmented harmonic map, which does not depend on the
lattice, is locally tamed by the rod structure of its fixed points and stabilisers, in any basis (Definition~\ref{def:Ctor}). In the basis
$(\partial_\phi,\partial_\tau)$ the coordinates $(\rho,z)$ above are its Weyl coordinates: $\rho=(\det G)^{1/2}$, and $z$ is a harmonic conjugate of $\rho$,
since $(\mathsf r,\vartheta)\mapsto(\rho,z)$ is conformal for $d\mathsf r^2/\Delta+d\vartheta^2$, hence for the orbit-space metric $\Sigma(d\mathsf r^2/\Delta+d\vartheta^2)$;
the sign of $z$ plays no role, by the symmetry $z\mapsto-z$. The fixed points are
the zeros of $G$, over the turning points $\pm\sigma$, and the stabilisers are read off from the kernels of $G$: on the outer rods
($\vartheta\in\{0,\pi\}$) $G\partial_\phi=0$, and on the middle rod ($\mathsf r=\mathsf r_+$), since $\Delta=0$ and $\mathsf r_+^2-\mathsf a^2=2M\mathsf r_+$,
$\Sigma G=\sin^2\vartheta\,ww^T$ with $w=(\mathsf r_+^2-\mathsf a^2,-\mathsf a)$, and $w\cdot(\Omega,1)=0$. So the rod lines are $[\partial_\phi]$,
$[\Omega\partial_\phi+\partial_\tau]$, $[\partial_\phi]$, and the type is AF.

\emph{\ref{it:T2}.} Write $\rho^{-1}G=\Phi_{H,A}$ as in \eqref{eq:PhiHA}, in the basis $(\partial_\phi,\partial_\tau)$. Then $H=1/G_{22}=\Sigma/(\Sigma-2M\mathsf r)$ and
$A=G_{12}/G_{22}=-2M\mathsf r\mathsf a\sin^2\vartheta/(\Sigma-2M\mathsf r)$, while $\Phi^\sharp_{1,0}=\Phi_{1,0}$ by \eqref{eq:models}. Since $\Sigma-2M\mathsf r\ge\Delta$, we have
$0\le H-1\le2M\mathsf r/\Delta$ and $|A|/\rho\le2M\mathsf r|\mathsf a|/\Delta^{3/2}$, so by \eqref{eq:coshhalf}
\begin{equation}
\cosh d(\Phi_{H,A},\Phi_{1,0})-1=\frac{\rho^2(H-1)^2+A^2}{2\rho^2H}\le\frac{2M^2\mathsf r^2}{\Delta^2}+\frac{2M^2\mathsf r^2\mathsf a^2}{\Delta^3} .
\end{equation}
Finally $R^2=\rho^2+z^2=(\mathsf r-M)^2-\sigma^2\sin^2\vartheta$, so $R^2\le(\mathsf r-M)^2\le R^2+\sigma^2$ and $\Delta=(\mathsf r-M)^2-\sigma^2\ge R^2-\sigma^2$.
Hence $d=O(R^{-1})$ uniformly in $\vartheta$.

(ii) By Proposition~\ref{prop:farCtor}, $(\Phi,\nu)\in\Ctor$ with the model $\Phi^\sharp_{1,0}$ in the basis $(a_\gamma,b_\gamma)$, since the type is AF, and by
Lemma~\ref{lem:outerlines} its outer rod lines are $L_0=L_2=[a]$. The middle rod line differs from them, so it is $[ta+b]$ for a unique $t$.
Write $M=\sigma\cos\beta$, $\mathsf a=\sigma\sin\beta$ with $\beta\in(-\frac\pi2,\frac\pi2)$; then $\mathsf r_+=\sigma(1+\cos\beta)$ and
$\Omega=\sin\beta/\bigl(2\sigma\cos\beta(1+\cos\beta)\bigr)$. This is an odd function of $\beta$, and on $[0,\frac\pi2)$ it is the product of the
increasing function $\sin\beta$ and the increasing positive function $1/(2\sigma\cos\beta(1+\cos\beta))$; so it is a strictly increasing
bijection of $(-\frac\pi2,\frac\pi2)$ onto $\R$. Hence $M,\mathsf a$ exist and are unique. By (i), the map of $g_{M,\mathsf a}$, translated by $\frac12(w_1+w_2)$
in $z$ and written in the basis $(a_\gamma,b_\gamma)$ by identifying it with $(\partial_\phi,\partial_\tau)$, is locally tamed by the same rod structure as $\Phi$,
and tends to the same model; a translation of $z$ moves the model by $O(R^{-1})$ (after Definition~\ref{def:Ctor}). By Lemma~\ref{lem:uniq} the two
maps coincide, and so do the Gram matrices $\rho\Phi$. The conformal factors both solve \eqref{eq:nu} for this Gram matrix, so they differ by an
additive constant in $\nu$.

(iii) Given $M,\mathsf a$, let $\ell_0=(1,0)$, $\ell_1=-(\Omega,1)$ and $\ell_2=(-1,0)$. Then $\det(\ell_0,\ell_1)=\det(\ell_1,\ell_2)=-1$, and
Lemma~\ref{lem:match}(i) gives $r=1$. By Proposition~\ref{prop:ALFreal} with $\alpha=0$ and $n=2$, there is a far-node configuration with $m=\max(r,n-r)=1$,
i.e.\ $K=2$, with no heavy and no light node, whose far-node limit has the turning points $\pm\sigma$ and, in the basis $(a_\gamma,b_\gamma)$ of
Proposition~\ref{prop:farCtor}, the model $\Phi^\sharp_{1,0}$ and the rod lines $[\ell_j]$. Its middle rod line is $[\Omega a+b]$, so by (ii) and the
uniqueness of $(M,\mathsf a)$ it is $g_{M,\mathsf a}$.
\end{proof}
By the definition before \eqref{eq:Nrod}, a direct computation from \eqref{eq:kerrG} and $e^{2\nu}=\Sigma/\lambda$ gives $\mathfrak c(\partial_\phi)=1$ on the
outer rods and $\mathfrak c(\Omega\partial_\phi+\partial_\tau)=2M\mathsf r_+/\sigma$ on the middle rod. So the normalised rod vectors are $\pm\partial_\phi$ and
$\pm(2M\mathsf r_+/\sigma)(\Omega\partial_\phi+\partial_\tau)$. For the lattice that they generate, $g_{M,\mathsf a}$ is the smooth Euclidean Kerr instanton on
$\R^2\times S^2$ \cite{Hawking1977,GH1979}; for other lattices it has conical singularities along the axis.

The explicit rod lines give a first application of Theorem~\ref{thm:G}: the twist of an AF end cannot vanish unless the metric is
Schwarzschild. Let $(X,g)$ be as in Theorem~\ref{thm:G}, of type AF$_\beta$, with lattice $\Lambda$; we write $X$ for the manifold, since $M$ is
the mass parameter of $g_{M,\mathsf a}$. Let $v_0\in\Lambda$ be a
primitive generator of the stabiliser of the outer rods, and let $u$ be the Killing field that stays bounded at infinity, the
translation of the model, normalised to the period of the model. Since the end is equivariantly diffeomorphic to the end of
$(\R^3\times\R)/\langle q_\beta\rangle$, the lattice $\Lambda$ is generated by $v_0$ and an element $w=\beta v_0+u$, and $\beta$ is
determined by the metric modulo $\Z$ \cite[Defs.~4.6, 4.30]{LS2025}; $\beta\in\Z$ means that the translation itself is a lattice
vector, the end being that of $\R^3\times S^1$.
\begin{corollary}[toric AF$_0$ instantons]\label{cor:AF0}
Let $(X,g)$ be as in Theorem~\ref{thm:G}, of type AF$_\beta$. If $n\ge3$, then $\beta\notin\Z$. If $\beta\in\Z$, then $n=2$ and $g$ is the
Euclidean Schwarzschild metric. In particular, the Euclidean Schwarzschild metric is the only simply connected toric AF$_0$
gravitational instanton with finitely many fixed points, at least one, as conjectured by Li and Sun for all AF$_0$ gravitational instantons \cite[Conj.~8.1]{LS2025}. (Remark~\ref{rem:pi1} does not
remove the hypothesis of simple connectivity here, since an AF end has cubic, not Euclidean, volume growth.)
\end{corollary}
\begin{proof}
By Proposition~\ref{prop:GG} and Theorem~\ref{thm:Gfull}(b), $g$ is a far-node limit in a basis $(a_\gamma,b_\gamma)$ in which
the model is $\Phi^\sharp_{1,0}$ (Proposition~\ref{prop:farCtor}); its data are $(\omega,r,\gamma,\mathfrak s)$ with
$\gamma=\gamma_*(\omega)$, and $1\le r\le n-1$ by
Lemma~\ref{lem:match}(i); so $n\ge2$. Let $\tilde v_0,\dots,\tilde v_n$ be the normalised rod vectors, the primitive generators of the
stabilisers (Proposition~\ref{prop:smoothCtor}); consecutive ones form a basis of $\Lambda$ (Lemma~\ref{lem:axis}(i)), and
$\tilde v_0=\pm v_0$. By Proposition~\ref{prop:faraxis}\ref{it:fa3}, in the frame $(e_\gamma,e_1)$ of \S\ref{sec:far} the
$e_\gamma$-component of $\tilde v_j$ is $\pm\sqrt{C_\nu}\,H_r(\omega_{(j)})$. The vector $e_1$ is a multiple of $b_\gamma$, and
$b_\gamma$ is the direction in which the Gram matrix stays bounded at infinity, since $\rho\Phi_{H,A}(b_\gamma,b_\gamma)=H^{-1}$ in
\eqref{eq:PhiHA}, while $\rho\Phi_{H,A}(a_\gamma,a_\gamma)$ grows like $\rho^2$ and $\rho\Phi_{H,A}(a_\gamma,b_\gamma)$ stays bounded; so $u$ is a multiple of $e_1$. For $0<j<n$ write
$\tilde v_j=x_j\tilde v_0+y_ju$. Comparing $e_\gamma$-components,
\begin{equation}\label{eq:AF0strip}
|x_j|=\frac{|H_r(\omega_{(j)})|}{H_r(\omega)}<1,\qquad 0<j<n .
\end{equation}
The inequality holds because, by Heine's formula \eqref{eq:heine} with \eqref{eq:PS}, $H_r(\omega_{(j)})$ is the sum of the positive
terms $P_S(\omega)$ of $H_r(\omega)$ with the signs $(-1)^{|S\cap\{1,\dots,j\}|}$, and both signs occur: for $1\le j\le n-1$ and
$1\le r\le n-1$ there are $r$-sets $S$ of atoms with $k_0$ and with $k_0+1$ atoms among the first $j$, where $k_0=\max(0,r-n+j)$.

Now let $\beta\in\Z$. Replacing $w$ by $w-\beta v_0$, which does not change $\Lambda$, we may assume $\beta=0$, so that $(v_0,u)$ is a
basis of $\Lambda$. Then $\tilde v_j=p_jv_0+q_ju$ with $p_j,q_j\in\Z$, so $x_j=\pm p_j$, and \eqref{eq:AF0strip} forces $p_j=0$; since
$\tilde v_j$ is primitive, $\tilde v_j=\pm u$ for every $0<j<n$. Two consecutive interior rod vectors would then be parallel, which
contradicts $\det(\tilde v_j,\tilde v_{j+1})=\pm1$; so $n-1\le1$, i.e.\ $n=2$, and the middle rod line is $[u]=[b_\gamma]$. By
Proposition~\ref{prop:kerr}(ii), $g$ is a Euclidean Kerr metric $g_{M,\mathsf a}$ whose middle rod line is $[\Omega a_\gamma+b_\gamma]$;
so $\Omega=0$, i.e.\ $\mathsf a=0$, and $g_{M,0}$ is the Euclidean Schwarzschild metric.
\end{proof}
The general form of \eqref{eq:AF0strip} is a restriction on the twist: with $\tilde v_j=p_jv_0+q_jw$, every interior rod satisfies
$|p_j+q_j\beta|<1$. For the rod structures $X_n$ of Li and Sun this is $\beta\in(0,1)$, the range of their existence theorem
\cite[Thm.~1.1]{LS2025}, and for Euclidean Kerr it is the bound $|\Omega|<\kappa_{\mathrm s}=\sigma/(2M\mathsf r_+)$ of the angular velocity by the surface gravity, i.e.\
$|\mathsf a|<\sigma$, which always holds.

\subsection{Gibbons--Hawking metrics}\label{subsec:GH}
Let $w_1<\dots<w_n$. The Gibbons--Hawking metric with these collinear centres is
\begin{equation} g=V^{-1}(d\psi+A\,d\varphi)^2+V(d\rho^2+dz^2+\rho^2d\varphi^2), \end{equation}
where $V=c_0+\sum_l\omega_l/R_l$ and $A=\sum_l\omega_l(z-w_l)/R_l$, with $\omega_l=\lambda_l/2$, $\lambda_l>0$, $R_l=\sqrt{\rho^2+(z-w_l)^2}$ and $c_0\ge0$: multi-Eguchi--Hanson for $c_0=0$ and multi-Taub--NUT for $c_0>0$.
For equal $\lambda_l=\lambda$, and $\psi$ of period $2\pi\lambda$, these are smooth toric gravitational instantons, and by
Theorem~\ref{thm:G} they are solitons, with degree index $r=0$. They are all the toric hyperk\"ahler gravitational instantons
\cite[Thm.~7.4, \S7.1.1]{LS2025}; this also follows from the classification of the hyperk\"ahler gravitational instantons, which is
complete (Remark~\ref{rem:ALFends}). The following proposition shows directly, for all $\lambda_l>0$, that
they are soliton metrics: they are the case $r=0$ of the soliton family (Definition~\ref{def:family}), in which all Heine sums have
order $0$ or $1$. It uses neither the surjectivity of \S\ref{sec:surj} nor the uniqueness theorem.

\begin{proposition}[Gibbons--Hawking metrics]\label{prop:GHmeh}
Let $\omega=\sum_l\omega_l\delta_{w_l}$, and let $V$, $A$ and $c_0$ be as above. For every $\mathfrak s\in\R$, the family matrix of the family data
$(\omega,0,c_0,\mathfrak s)$ is the Gram matrix of $g$ in the basis $f_0=\partial_\varphi-c_0\mathfrak s\,\partial_\psi$, $f_1=\partial_\psi$, and $F_1F_2=V$.
Let $u_1,\dots,u_h$ be points of the axis, distinct and off the $w_l$, with $h=n-1$ if $c_0=0$ and $h=n$ if $c_0>0$, and put
\begin{equation} c_l=\pm\,\omega_l\,\frac{\prod_{l'\ne l}(w_l-w_{l'})}{\prod_{h'}(w_l-u_{h'})},\qquad +\text{ if }c_0=0,\quad -\text{ if }c_0>0 . \end{equation}
\begin{enumerate}[leftmargin=*,label=\textup{(\alph*)}]
\item If $c_0=0$, the parameters $c_l$ at the $w_l$ and $\infty$ at the $u_{h'}$ lie in $\cA$, and for the ALE soliton matrix $G$ of these $2n-1$
nodes, $2(-1)^nG$ in the Hankel frame is the Gram matrix of $g$ in the basis $(\partial_\varphi,\partial_\psi)$. So $g$ is a soliton on $\R^4$
with $2n-1$ poles, $n-1$ of them phantom solitons.
\item If $c_0>0$, then for the far-node limit of \S\ref{sec:far} with $\gamma=c_0$ and the $2n$ finite nodes $w_l$ (parameters $c_l$) and
$u_{h'}$ (heavy), $2(-1)^{n+1}G_\infty$ in the frame $(e_\gamma,e_1)$ is the Gram matrix of $g$ in the basis $(f_0,f_1)$ with
$\mathfrak s=\sum_lw_l-\sum_{h'}u_{h'}$. So, by Theorem~\ref{thm:mkn}, $g$ is the soliton \eqref{eq:BZAF} on $\R^3\times S^1$ with $2n$ poles: at the $w_l$
with $b_l=(c_l+c_0)/(c_l-c_0)$, and at the $u_{h'}$ with $b=1$, phantom solitons.
\end{enumerate}
In both cases the conformal factors $e^{2\nu}=C'F_1F_2$ of the soliton metric (Theorem~\ref{thm:family}) are the constant multiples of $V$,
the conformal factor of $g$.
\end{proposition}
\begin{proof}
For $r=0$, $H_0=1$ and $H_{-1}=\Sigma_0=0$, so $H^{(c_0)}_1=H_1$ by \eqref{eq:shifted}, and in \eqref{eq:famXY} $\mathcal X=1$, $F_1=1$,
$F_2=H_1(\omega/\mathbf r)+c_0=V$ and $\mathcal Y=\sum_l\omega_l\varpi_l+c_0(\mathfrak s-i\rho)$. Since $\operatorname{Re}\varpi_l=(w_l-z)/R_l$ and
$\operatorname{Im}\varpi_l=-\rho/R_l$, $\mathcal Y=-(A-c_0\mathfrak s)-i\rho V$. So, by \eqref{eq:famG},
\begin{equation}
\tilde G(x,x)=V^{-1}\bigl(x_1+(A-c_0\mathfrak s)x_0\bigr)^2+V\rho^2x_0^2=g(x_0f_0+x_1f_1,\,x_0f_0+x_1f_1),
\end{equation}
since $d\varphi(x_0f_0+x_1f_1)=x_0$ and $d\psi(x_0f_0+x_1f_1)=x_1-c_0\mathfrak sx_0$. Let $a_l=\omega_l\prod_{h'}(w_l-u_{h'})^{-2}$; the polynomial $P$ of
all nodes has $P'(w_l)=\prod_{l'\ne l}(w_l-w_{l'})\prod_{h'}(w_l-u_{h'})$, so $c_l=\pm a_lP'(w_l)$.

(a) $(\omega,0,0)$ are family data. Apply the converse of Theorem~\ref{thm:family}\textup{(a)} with $r=0$, $m=n-1=\max(r,n-1-r)$, $I=\{u_{h'}\}$
and $Z=\emptyset$: it gives exactly these parameters, in $\cA$, with the family data $(\omega,0,0)$. By Theorem~\ref{thm:family}\textup{(a)},
$2(-1)^{m+1}G=2(-1)^nG$ in the Hankel frame is the family matrix, which for $c_0=0$ is the Gram matrix of $g$ in $(f_0,f_1)=(\partial_\varphi,\partial_\psi)$.

(b) $(\omega,0,c_0)$ are family data with $c_0>0$. Apply the converse of Theorem~\ref{thm:family}\textup{(b)} with $r=0$, $m=n=\max(r,n-r)\ge1$,
$m-r=n$ heavy nodes $u_{h'}$ and $m+r-n=0$ light ones: it gives exactly these parameters, $-a_lP'(w_l)=c_l$, and the shift of
\eqref{eq:farmeasure} is $\sum_lw_l+\sum_{h'}u_{h'}-2\sum_{h'}u_{h'}=\sum_lw_l-\sum_{h'}u_{h'}$. By Theorem~\ref{thm:family}\textup{(b)},
$2(-1)^{m+1}G_\infty=2(-1)^{n+1}G_\infty$ in $(e_\gamma,e_1)$ is the family matrix with this shift. Theorem~\ref{thm:mkn} applies to these data, and the
heavy nodes have $b=1$.

In both cases Theorem~\ref{thm:family} gives the solutions $e^{2\nu}=C'F_1F_2=C'V$ of \eqref{eq:nu}.
\end{proof}

\subsection{Known smooth toric instantons}\label{subsec:known}
Table~\ref{tab:smooth} lists the smooth toric gravitational instantons known to us. By Theorem~\ref{thm:G} each of them is a
gravitational soliton; the number of poles and of phantom solitons follows from the proof of Proposition~\ref{prop:GG} and
Lemma~\ref{lem:degreeindex}. For several of them the soliton form was known before, as indicated.
\begin{table}[ht]\centering\small
\begin{tabularx}{\textwidth}{l>{\raggedright\arraybackslash}p{0.3\textwidth}cc>{\raggedright\arraybackslash}X}\toprule
type & instanton & $n$ & degrees & poles (phantom solitons)\\\midrule
ALE & flat $\R^4$ & $1$ & $\{0,0\}$ & $1$ $(0)$; Proposition~\ref{prop:GHmeh}\\
ALE & multi-Eguchi--Hanson $A_{n-1}$ & $n\ge2$ & $\{0,n-1\}$ & $2n-1$ $(n-1)$; Proposition~\ref{prop:GHmeh}\\
ALF & Taub--NUT & $1$ & $\{0,1\}$ & $2$ $(1)$; \cite{CT2010}, Proposition~\ref{prop:GHmeh}\\
ALF & multi-Taub--NUT & $n\ge2$ & $\{0,n\}$ & $2n$ $(n)$; Proposition~\ref{prop:GHmeh}\\
ALF & Taub--bolt \cite{Page1978,GibbonsPerry1980} & $2$ & $\{1,1\}$ & $2$ $(0)$\\
ALF & four turning points, $\kappa=(-1,1,-1)$, Remark~\ref{rem:compexamples}(a), \cite{Jones2026} & $4$ & $\{2,2\}$ & $4$ $(0)$; explicit, \eqref{eq:n4data}\\
AF & Euclidean Kerr (Schwarzschild: $\mathsf a=0$) & $2$ & $\{1,1\}$ & $2$ $(0)$; \cite{BZ1979}, Proposition~\ref{prop:kerr}\\
AF & Chen--Teo \cite{CT2011} & $3$ & $\{1,2\}$ & $4$ $(1)$; \cite{CT2011,Chen2016}, \cite[App.~A, 3-001]{Aksteiner2026}\\
AF & Li--Sun $X_3$ \cite{LS2025}, Teo \cite{Teo2026} & $4$ & $\{2,2\}$ & $4$ $(0)$; \cite{Teo2026}, \cite[App.~A, 4-001]{Aksteiner2026}\\
AF & Li--Sun $X_{n-1}$ \cite{LS2025} & $n\ge5$ & $\{\lfloor\frac n2\rfloor,\lceil\frac n2\rceil\}$ & $2\lceil\frac n2\rceil$ $(n\bmod2)$; numerically \cite[App.~A]{Aksteiner2026} for $n\le8$\\\bottomrule
\end{tabularx}
\caption{Smooth toric gravitational instantons known to us; $n$ is the number of fixed points. The last column gives the number
of poles of the soliton of Theorem~\ref{thm:G}, with the number of phantom solitons in brackets, and earlier soliton
realisations.}\label{tab:smooth}
\end{table}

\subsection{Explicit instantons from the soliton family}\label{subsec:explicit}
The closed form of \S\ref{sec:far} can be used to construct instantons. Since the rod data are polynomial in the soliton data,
smoothness is a system of polynomial equations, which can be solved numerically.

\begin{proposition}\label{prop:certsmooth}
Let $(\omega,\gamma)$ be data of the far-node family of \S\ref{sec:far} whose Hankel family $V_j$ has $\beta_j=1$ and $\kappa_j\in\Z$ for
$1\le j\le n-1$, in the sense of \eqref{eq:kappabeta}. Then the far-node limit, with the normalised augmentation \cite[Prop.~4.26]{LS2025}
and the lattice generated by two consecutive normalised rod vectors, is a smooth complete toric gravitational instanton
whose end is of type ALF$^-$, ALF$^+$ or AF according as $\alpha(\omega,\gamma)>0$, $<0$ or $=0$, in the sense of
Definition~\ref{def:asymptype} (for ALF a finite cover of the end is modelled on the end of anti-Taub--NUT, resp.\ Taub--NUT; for AF the
twist is recorded by the lattice).
\end{proposition}
\begin{proof}
By Proposition~\ref{prop:faraxis}\ref{it:fa3}, the normalised rod vectors are $\pm\sqrt{C_\nu}\,S_{\mathfrak s}V_j$ in the frame $(e_\gamma,e_1)$, i.e.\
the images of $\pm V_j$ under one invertible linear map, so they have the same $\kappa_j,\beta_j$. By Proposition~\ref{prop:faraxis}\ref{it:fa3} and Remark~\ref{rem:lattice}, the metric extends smoothly across the finite rods and
the turning points, and the outer normalised rod vectors, the images of $\pm V_0$ and $\pm V_n$, are primitive in the lattice, so it extends across the
outer rods as well. By Proposition~\ref{prop:farCtor}, $\Phi$ is locally tamed and $d(\Phi,\Phi^\sharp_{1,\alpha})=O(R^{-1})$, and by
Proposition~\ref{prop:farinfty}(iii) its conformal factor tends to a constant at infinity. By \cite[Thm.~4.11 and \S4.4]{LS2025} the
metric is complete, with an end modelled on the end of its type \cite[Prop.~4.36]{LS2025}.
\end{proof}

\begin{remark}[further examples]\label{rem:compexamples}
For a given rod structure, the conditions $\kappa_j=\kappa_j^0$, $\beta_j=1$ of Proposition~\ref{prop:certsmooth} are a square
polynomial system in the data $(w,a,\gamma)$ for ALF rod structures, after the scaling and translation of $w$ and the scaling of $a$
are fixed; for AF rod structures $\gamma=\gamma_*(\omega)$, and since $\tilde v_n=\pm\tilde v_0$ automatically, two of the equations are
redundant, so the solutions form curves along which the twist varies. Here
$\kappa_j$ is the self-intersection number of the sphere over $I_j$, for the orientation of \S\ref{subsec:WP}, cf.\
\cite[\S3.9, p.~544]{OR1970} and \cite[Thm.~1]{Nilsson2024}. Li and Sun show that for a smooth simply connected toric gravitational instanton with global Weyl
coordinates it lies in $\{-1,0,1\}$ unless the metric is hyperk\"ahler \cite[App.~A]{LS2026b}, and Aksteiner shows the same for the
smooth metrics of the positive soliton family of type AF, with the bound $|\det(v_a,v_b)|<b-a$ for the rod vectors of non-adjacent
rods \cite[Thm.~5.1]{Aksteiner2026}; so only the targets $\kappa\in\{-1,0,1\}^{n-1}$ need be considered. By numerical search we have
found solutions with positive weights for the following rod structures; the one with $n=4$ in (a) is known exactly.
\begin{enumerate}[leftmargin=*,label=\textup{(\alph*)}]
\item One ALF rod structure with $n=4$, two with $n=5$, five with $n=6$ and eleven with $n=7$, listed in Table~\ref{tab:compex}.
The one with $n=4$ (Figure~\ref{fig:rods4}(c)), which is also studied by Jones \cite{Jones2026}, has the turning points
$(0,1,\frac{\sqrt{10}}2,1+\frac{\sqrt{10}}2)$, so the finite
rods have the lengths $1$, $\frac{\sqrt{10}}2-1$, $1$; its weights are $(1,A,A,1)$ and its atom at infinity is $\gamma_0$, with
\begin{equation}\label{eq:n4data}
A=\frac{7+2\sqrt{10}}9,\qquad \gamma_0=\frac{40\sqrt2+36\sqrt5-19\sqrt{10}-65}{82},
\end{equation}
a root of $656T^4+2080T^3-120T^2-200T+25$, and $\kappa=(-1,1,-1)$. The moments are $H_1(\omega)=\frac{32+4\sqrt{10}}9$,
$H_2(\omega)=\frac{28+8\sqrt{10}}3$ and $H_3(\omega)=\frac{80+25\sqrt{10}}3$, so $\gamma_0/\gamma_*(\omega)=0.0447\ldots$ and
$\alpha=2(H_3-\gamma_0^2H_1)/(\gamma_0H_2)>0$: the end is of type ALF$^-$. These data are exact: the vectors $V_j=(H_2(\omega_{(j)}),H^{(\gamma_0)}_3(\omega_{(j)}))$ of the Hankel family ($r=2$) satisfy
the relations $V_{j+1}+V_{j-1}=\kappa_jV_j$ of \eqref{eq:kappabeta} exactly, a finite computation in
$\Q(\sqrt2,\sqrt5)$, which contains all the data. To check this, translate $z$ so that the turning points
are centred at $0$; this changes $\Sigma_2$ by a multiple of $H_2$, hence the $V_j$ by a common shear, which is harmless by
Remark~\ref{rem:lattice}. With $\frac92A=\frac72+\sqrt{10}$ and $A^2(\frac{\sqrt{10}}2-1)^2=\frac A2$ one finds, for $j=0,\dots,4$,
$H_2(\omega_{(j)})=4A\,(3,-1,-2,-1,3)$, $H_1(\omega_{(j)})=(2+2A,2A,0,-2A,-2-2A)$,
$\Sigma_2(\omega_{(j)})=(5+\sqrt{10})A\,(0,1,0,-1,0)$ and
$H_3(\omega_{(j)})=(\frac52A+\frac{45}2A^2,-\frac{45}2A^2,0,\frac{45}2A^2,-\frac52A-\frac{45}2A^2)$. So the first components satisfy the
relations with $\kappa=(-1,1,-1)$, and $Y_j=H^{(\gamma)}_3(\omega_{(j)})$ \eqref{eq:shifted} has the form $(Y_0,Y_1,0,-Y_1,-Y_0)$, for which the
three relations reduce to $Y_0+Y_1=0$, i.e.\ $(2+4A)\gamma^2+2(5+\sqrt{10})A\gamma=\frac52A$. Its positive root is $\gamma_0$, and
$\det(V_0,V_1)=8AY_1\ne0$. So by
Proposition~\ref{prop:certsmooth} they define a smooth complete toric gravitational instanton, without interval arithmetic. Its
normalised rod vectors are, in the basis of the lattice generated by the first two (this basis has the orientation opposite to that of the frame, so $\det(v_{j-1},v_j)=1$; the signs of the $v_j$ are
those fixed in Remark~\ref{rem:lattice} and are not re-chosen, so the linear relations \eqref{eq:kappabeta} give the same $\kappa_j$, $\beta_j$,
and $p,q'$ of Remark~\ref{rem:ALFends} are unchanged by this change of basis),
\begin{equation}\label{eq:n4rods}
v_0=(1,0),\quad v_1=(0,1),\quad v_2=(-1,-1),\quad v_3=(-1,-2),\quad v_4=(2,3),
\end{equation}
so $\det(v_0,v_4)=3$ and $v_4-v_0=(1,3)$ is primitive: the end is the quotient of the end of Taub--NUT by $\Z_3$ acting on the
base $\R^3$. The spheres over the three finite rods have self-intersections $-1$, $1$, $-1$; the manifold has $\chi=4$, $b_2=3$ and
the odd intersection form of determinant $3$ and signature $-1$, so it is not spin. The opposite orientation has $\kappa=(1,-1,1)$
at the same turning points, with the weights $(1,9A,9A,1)$ and the atom $45A/(4\gamma_0)$, and the Li--Sun degrees are $\{2,2\}$, so
the metric is Hermitian for neither orientation (\S\ref{subsec:rods}); in particular there are toric ALF gravitational instantons
that are not locally Hermitian, which answers a question of Li and Sun \cite[Q.~8.10]{LS2025} in the ALF case. Since $d+\bar d=n$
for ALF (Remark~\ref{rem:degid}), it is the first ALF rod structure in the family that is not of degree $0$ or $1$ in either
orientation.
\item The AF rod structure $X_4$ of Li and Sun, along a curve on which the twist varies.
\item Sample twists of $X_5$ and of $X_6$, and of two further AF rod structures with six turning points and two with seven
(Table~\ref{tab:compex}).
\end{enumerate}
Each of the other solutions is known numerically and passes the Krawczyk test, an interval version of Newton's method; if the
interval code is correct, the test proves that an exact solution lies in a small box around the approximate one, and by
Proposition~\ref{prop:certsmooth} each exact solution gives a smooth complete toric gravitational instanton, in the closed form
\eqref{eq:farCF}. The AF rod structures of (b) and (c), with
$X_4$, $X_5$ and $X_6$, are exactly the admissible AF words with five, six and seven turning points in the catalogue of Aksteiner
\cite[App.~A, Tables~7--9]{Aksteiner2026}, whose self-intersections are our $-\kappa_j$ up to reversal and orientation; there the
balance equations are solved with interval certificates at the midpoint of the admissible twist interval, with existence on the
whole interval under his comparison hypotheses.

Since the code has not been independently checked, and since for (b) the analysis of the curve near its ends is not given
here, we regard this as strong evidence rather than a proof, and, apart from the exact example with $n=4$, we do not claim the
existence of these instantons in this paper. Nothing else in the paper depends on them, and the other ALF examples in (a), whose rod
structures are not in Table~\ref{tab:smooth}, are not listed there. Details, with the data and the code, will be presented in a
future publication.

By Theorem~\ref{thm:G} every smooth toric instanton with one of these rod structures, in particular each of Li and Sun's metrics
with rod structure $X_4$, is a member of the family. Whether there is exactly one toric one for each twist is, for the rod structures
$X_n$, Li and Sun's uniqueness conjecture \cite[Conj.~8.5]{LS2025}, and the same question arises for the other AF rod structures of
(c); whether the AF$_\beta$ instanton on $X_n$ is unique up to isometry among all gravitational instantons, without the torus
symmetry, is their Question~8.6.
\end{remark}

\begin{remark}[the ends of the ALF examples]\label{rem:ALFends}
None of the ALF examples of Table~\ref{tab:compex} has an end of type ALF-$A_k$, the end of multi-Taub--NUT with $k+1$ centres. This
does not contradict the classification of ends. For Ricci-flat metrics with curvature decay faster than quadratic and controlled
holonomy, which means that geodesic loops of length less than $\epsilon\mathfrak r$ based far out have holonomy tending to the identity, the
ends are of the types ALE, ALF-$A_k$, ALF-$D_k$, ALG or ALH \cite{ChenLi2021}, extending the hyperk\"ahler classification
\cite{Minerbe2010,ChenChen2021}. For hyperk\"ahler instantons Sun and Zhang assume no decay rate, only the finite curvature energy
$\int_M|\mathrm{Rm}_g|^2\,d\mathrm{vol}_g<\infty$ of Definition~\ref{def:GI}: every non-flat hyperk\"ahler gravitational instanton is ALE,
ALF, ALG, ALG$^*$, ALH or ALH$^*$ \cite[Thm.~1.2]{SunZhang2024}, and together with the results for each type this classifies the
hyperk\"ahler gravitational instantons \cite[\S4]{LS2026}. Nilsson shows that a toric gravitational instanton with controlled
holonomy is ALE with cyclic group at infinity or ALF-$A_k$ \cite[Thm.~3]{Nilsson2024}. Our examples have curvature decay $O(\mathfrak r^{-3})$ (Remark~\ref{rem:asymptypes})
but not controlled holonomy. Their ends are quotients of the end of Taub--NUT by a cyclic group $\Z_p$ in the torus,
$p=|\det(v_0,v_n)|$, which acts on the base $\R^3$ through a rotation group $\Z_q$ with $q>1$: the cone at infinity is $\R^3/\Z_q$,
and the orbits of the asymptotic circle through the outer rods are short loops whose holonomy is a rotation of the base by a fixed
angle $2\pi l/q$. They are gravitational instantons in the sense of Definition~\ref{def:GI}, which asks only for quadratic curvature
decay, with ends of type ALF in the sense of Definition~\ref{def:asymptype}, which allows a finite cover; the end is of type
ALF-$A_{p-1}$ exactly when $q=1$. The boundary at infinity is the lens space $L(p,q')$ with
$q'=\operatorname{sgn}\det(v_0,v_n)\,\det(v_1,v_n)$ \cite[Lemma~3]{Nilsson2024}, and for every example of Table~\ref{tab:compex}
$q'\equiv\pm1$ modulo $p$, as for an $A_{p-1}$ end; the difference is geometric, not topological.
\end{remark}

\begin{table}[ht]\centering\small
\begin{tabularx}{\textwidth}{ccl>{\raggedright\arraybackslash}p{0.13\textwidth}>{\raggedright\arraybackslash}X}\toprule
$n$ & $r$ & $\kappa$ & $\gamma/\gamma_*$ & end\\\midrule
$4$ & $2$ & $(-1,1,-1)$ & $0.0447$ & $p=3$, $q=3$\\
$5$ & $2$ & $(-1,1,0,1)$ & $0.1012$ & $p=3$, $q=3$\\
$5$ & $2$ & $(1,-1,-1,1)$ & $0.1070$ & $p=3$, $q=3$\\
$6$ & $2$ & $(1,0,1,0,1)$ & $0.1465$ & $p=3$, $q=3$\\
$6$ & $3$ & $(-1,0,-1,-1,1)$ & $0.1545$ & $p=3$, $q=3$\\
$6$ & $3$ & $(-1,1,0,0,-1)$ & $0.1522$ & $p=3$, $q=3$\\
$6$ & $3$ & $(-1,1,-1,-1,1)$ & $0.0592$ & $p=4$, $q=2$\\
$6$ & $3$ & $(-1,1,0,1,-1)$ & $0.0493$ & $p=4$, $q=2$\\
$7$ & $3$ & $(-1,0,0,1,0,1)$ & $0.1880$ & $p=3$, $q=3$\\
$7$ & $3$ & $(-1,1,0,0,0,1)$ & $0.1953$ & $p=3$, $q=3$\\
$7$ & $3$ & $(-1,1,0,1,0,1)$ & $0.0813$ & $p=4$, $q=2$\\
$7$ & $3$ & $(-1,1,1,0,-1,1)$ & $0.2087$ & $p=3$, $q=3$\\
$7$ & $3$ & $(-1,1,1,1,1,-1)$ & $0.2088$ & $p=3$, $q=3$\\
$7$ & $3$ & $(1,-1,-1,0,0,1)$ & $0.1965$ & $p=3$, $q=3$\\
$7$ & $3$ & $(1,-1,-1,1,0,1)$ & $0.0867$ & $p=4$, $q=2$\\
$7$ & $3$ & $(1,-1,0,0,-1,1)$ & $0.2085$ & $p=3$, $q=3$\\
$7$ & $3$ & $(1,-1,1,1,-1,1)$ & $30.01$ & $p=5$, $q=5$\\
$7$ & $3$ & $(1,-1,1,1,0,1)$ & $9.641$ & $p=4$, $q=2$\\
$7$ & $3$ & $(1,0,1,1,0,1)$ & $5.240$ & $p=3$, $q=3$\\\midrule
$5$ & $2$ & $(1,0,0,1)$ & $1$ & AF; $X_4$; \cite[5-001]{Aksteiner2026}\\
$6$ & $3$ & $(1,0,0,0,-1)$ & $1$ & AF; $X_5$; \cite[6-001]{Aksteiner2026}\\
$6$ & $3$ & $(-1,1,0,-1,1)$ & $1$ & AF; \cite[6-002]{Aksteiner2026}\\
$6$ & $3$ & $(-1,1,1,1,-1)$ & $1$ & AF; \cite[6-003]{Aksteiner2026}\\
$7$ & $3$ & $(1,0,0,0,0,1)$ & $1$ & AF; $X_6$; \cite[7-003]{Aksteiner2026}\\
$7$ & $3$ & $(-1,1,1,1,0,1)$ & $1$ & AF; \cite[7-001]{Aksteiner2026}\\
$7$ & $3$ & $(1,-1,0,1,0,1)$ & $1$ & AF; \cite[7-002]{Aksteiner2026}\\\bottomrule
\end{tabularx}
\caption{The rod structures of the examples of Remark~\ref{rem:compexamples}: the number $n$ of turning points,
the degree index $r$, the integers $\kappa=(\kappa_1,\dots,\kappa_{n-1})$ of \eqref{eq:kappabeta} in the orientation
with the signs of Remark~\ref{rem:lattice}, one representative per metric up to the reflection $z\mapsto-z$ (which reverses $\kappa$)
and, for $n=4$ and $n=6$, orientation reversal (which changes the sign of $\kappa$ and replaces the data by those of the opposite
orientation, inverting the ratio $\gamma/\gamma_*$; for $n=4$ see Remark~\ref{rem:compexamples}(a)), and the ratio
$\gamma/\gamma_*(\omega)$ to four digits; for the ALF rows the end is the quotient of the end of Taub--NUT by a cyclic group of order
$p=|\det(v_0,v_n)|$ which acts on the base $\R^3$ through a cyclic group of order $q$ (so the end is of type ALF-$A_{p-1}$, the end of multi-Taub--NUT with $p$ centres, exactly when $q=1$, which does not occur
here; for $q>1$ the cone at infinity is $\R^3/\Z_q$). The AF rows are one-parameter families, with the twist varying; the last column names the Li--Sun
structures and the entries of Aksteiner's catalogue \cite[App.~A]{Aksteiner2026}, whose self-intersection words are $-\kappa$ up to
reversal and orientation.}\label{tab:compex}
\end{table}

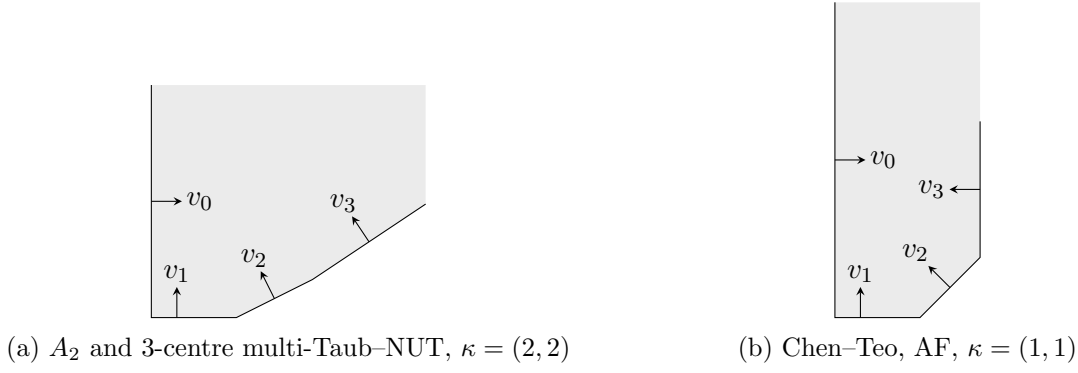
\begin{figure}[ht]
\centering
\begin{minipage}[b]{0.48\textwidth}\centering
\begin{tikzpicture}[scale=0.450,>=stealth]
    \fill[gray!50,opacity=0.3](0,6.837)--(0,0)--(2.5,0)--(4.736,1.118)--(8.064,3.337)--(8.064,6.837)--cycle;
    \draw (0,6.837)--(0,0)--(2.5,0)--(4.736,1.118)--(8.064,3.337);
    \draw[->] (0,3.418)--(0.9,3.418) node[anchor=180,inner sep=1.5pt] {$v_{0}$};
    \draw[->] (0.75,0)--(0.75,0.9) node[anchor=270,inner sep=1.5pt] {$v_{1}$};
    \draw[->] (3.618,0.559)--(3.216,1.364) node[anchor=297,inner sep=1.5pt] {$v_{2}$};
    \draw[->] (6.4,2.227)--(5.901,2.976) node[anchor=304,inner sep=1.5pt] {$v_{3}$};
\end{tikzpicture}

\smallskip\small (a) $A_2$ and $3$-centre multi-Taub--NUT, $\kappa=(2,2)$
\end{minipage}\hfill
\begin{minipage}[b]{0.48\textwidth}\centering
\begin{tikzpicture}[scale=0.450,>=stealth]
    \fill[gray!50,opacity=0.3](0,9.268)--(0,0)--(2.5,0)--(4.268,1.768)--(4.268,5.768)--(4.268,9.268)--cycle;
    \draw (0,9.268)--(0,0)--(2.5,0)--(4.268,1.768)--(4.268,5.768);
    \draw[->] (0,4.634)--(0.9,4.634) node[anchor=180,inner sep=1.5pt] {$v_{0}$};
    \draw[->] (0.75,0)--(0.75,0.9) node[anchor=270,inner sep=1.5pt] {$v_{1}$};
    \draw[->] (3.384,0.884)--(2.747,1.52) node[anchor=315,inner sep=1.5pt] {$v_{2}$};
    \draw[->] (4.268,3.768)--(3.368,3.768) node[anchor=0,inner sep=1.5pt] {$v_{3}$};
\end{tikzpicture}

\smallskip\small (b) Chen--Teo, AF, $\kappa=(1,1)$
\end{minipage}
\caption{Rod structures with three turning points, drawn as in \cite{LS2025}: the shaded region is the orbit space, each rod is a
side of its boundary, orthogonal to the rod vector, and the arrows are the oriented rod vectors $v_0,\dots,v_n$ of
\eqref{eq:kappabeta}, $v_{j+1}=\kappa_jv_j-v_{j-1}$, in the basis $v_0=(1,0)$, $v_1=(0,1)$ of the lattice; the finite sides are
drawn to scale with the first rod of length $1$, and the two semi-infinite rods are cut off. The outer rod vectors are parallel
exactly for the AF structures. (a) The multi-Eguchi--Hanson metric $A_2$ (ALE) and the multi-Taub--NUT metric with three centres
(ALF, of type ALF-$A_2$: the end is the quotient of the end of Taub--NUT by $\Z_3$ acting on the circle fibre) have the same rod
vectors $v_j=(1-j,j)$; their rod lengths are free parameters, drawn equal. (b) The Chen--Teo instanton,
$\kappa=(1,1)$, $v_3=-v_0$, at the point of its one-parameter family with equal rod lengths.}
\label{fig:rods3}
\end{figure}

\begin{figure}[ht]
\centering
\begin{minipage}[b]{0.31\textwidth}\centering
\begin{tikzpicture}[scale=0.399,>=stealth]
    \fill[gray!50,opacity=0.3](0,8.405)--(0,0)--(2.5,0)--(4.736,1.118)--(6.816,2.505)--(10.016,4.905)--(10.016,8.405)--cycle;
    \draw (0,8.405)--(0,0)--(2.5,0)--(4.736,1.118)--(6.816,2.505)--(10.016,4.905);
    \draw[->] (0,4.202)--(0.9,4.202) node[anchor=180,inner sep=1.5pt] {$v_{0}$};
    \draw[->] (0.75,0)--(0.75,0.9) node[anchor=270,inner sep=1.5pt] {$v_{1}$};
    \draw[->] (3.618,0.559)--(3.216,1.364) node[anchor=297,inner sep=1.5pt] {$v_{2}$};
    \draw[->] (5.776,1.811)--(5.277,2.56) node[anchor=304,inner sep=1.5pt] {$v_{3}$};
    \draw[->] (8.416,3.705)--(7.876,4.425) node[anchor=307,inner sep=1.5pt] {$v_{4}$};
\end{tikzpicture}

\smallskip\small (a) $A_3$ and $4$-centre multi-Taub--NUT
\end{minipage}\hfill
\begin{minipage}[b]{0.31\textwidth}\centering
\begin{tikzpicture}[scale=0.450,>=stealth]
    \fill[gray!50,opacity=0.3](0,9.268)--(0,0)--(2.5,0)--(4.268,1.768)--(6.768,1.768)--(6.768,5.768)--(6.768,9.268)--cycle;
    \draw (0,9.268)--(0,0)--(2.5,0)--(4.268,1.768)--(6.768,1.768)--(6.768,5.768);
    \draw[->] (0,4.634)--(0.9,4.634) node[anchor=180,inner sep=1.5pt] {$v_{0}$};
    \draw[->] (0.75,0)--(0.75,0.9) node[anchor=270,inner sep=1.5pt] {$v_{1}$};
    \draw[->] (3.384,0.884)--(2.747,1.52) node[anchor=315,inner sep=1.5pt] {$v_{2}$};
    \draw[->] (5.518,1.768)--(5.518,0.868) node[anchor=90,inner sep=1.5pt] {$v_{3}$};
    \draw[->] (6.768,3.768)--(7.668,3.768) node[anchor=180,inner sep=1.5pt] {$v_{4}$};
\end{tikzpicture}

\smallskip\small (b) $X_3$ (Teo), AF, $\kappa=(1,0,-1)$
\end{minipage}\hfill
\begin{minipage}[b]{0.31\textwidth}\centering
\begin{tikzpicture}[scale=0.440,>=stealth]
    \fill[gray!50,opacity=0.3](0,3.5)--(0,0)--(2.5,0)--(3.527,-1.027)--(5.763,-2.145)--(9.092,-4.364)--(9.092,3.5)--cycle;
    \draw (0,3.5)--(0,0)--(2.5,0)--(3.527,-1.027)--(5.763,-2.145)--(9.092,-4.364);
    \draw[->] (0,2.5)--(0.9,2.5) node[anchor=180,inner sep=1.5pt] {$v_{0}$};
    \draw[->] (0.75,0)--(0.75,0.9) node[anchor=270,inner sep=1.5pt] {$v_{1}$};
    \draw[->] (3.014,-0.514)--(2.377,-1.15) node[anchor=45,inner sep=1.5pt] {$v_{2}$};
    \draw[->] (4.645,-1.586)--(4.243,-2.391) node[anchor=63,inner sep=1.5pt] {$v_{3}$};
    \draw[->] (7.427,-3.255)--(7.927,-2.506) node[anchor=236,inner sep=1.5pt] {$v_{4}$};
\end{tikzpicture}

\smallskip\small (c) ALF, $\kappa=(-1,1,-1)$
\end{minipage}
\caption{Rod structures with four turning points, drawn as in Figure~\ref{fig:rods3}. (a) $A_3$ and the multi-Taub--NUT metric with
four centres (ALF-$A_3$), $\kappa=(2,2,2)$, equal rod lengths. (b) Li--Sun's $X_3$, the AF instanton of Teo, $v_4=v_0$, at the point of its
one-parameter family with three rods of equal length. (c) The ALF example of Remark~\ref{rem:compexamples}(a) and of \cite{Jones2026}, with the rod
vectors \eqref{eq:n4rods} and the rod lengths $1$, $\frac{\sqrt{10}}2-1$, $1$; the outer rods are not parallel, and
$\det(v_0,v_4)=3$. Its end is not of type ALF-$A_k$: it is the quotient of the end of Taub--NUT by $\Z_3$ acting on the base
$\R^3$ by rotations about the axis, so that the cone at infinity is $\R^3/\Z_3$ ($p=q=3$ in Table~\ref{tab:compex}), whereas an
ALF-$A_k$ end is the quotient by $\Z_{k+1}$ acting on the fibre only, with cone $\R^3$; see Remark~\ref{rem:ALFends}.}
\label{fig:rods4}
\end{figure}
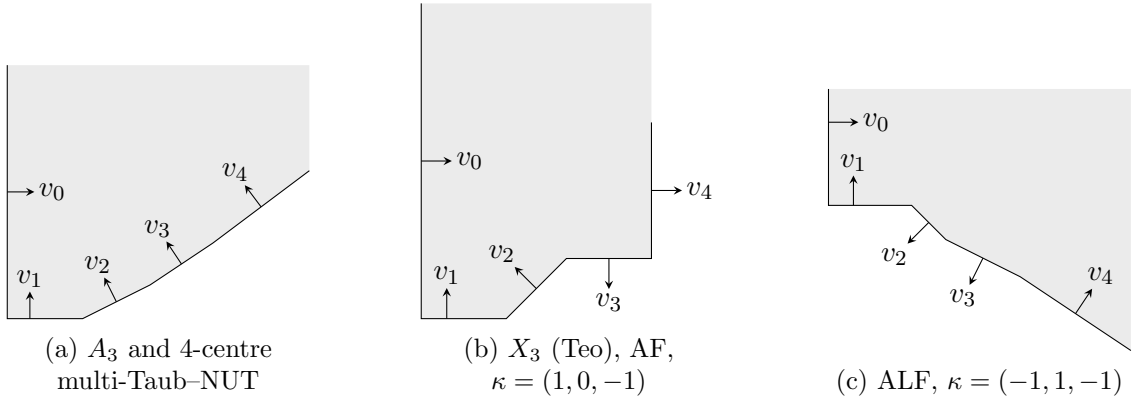

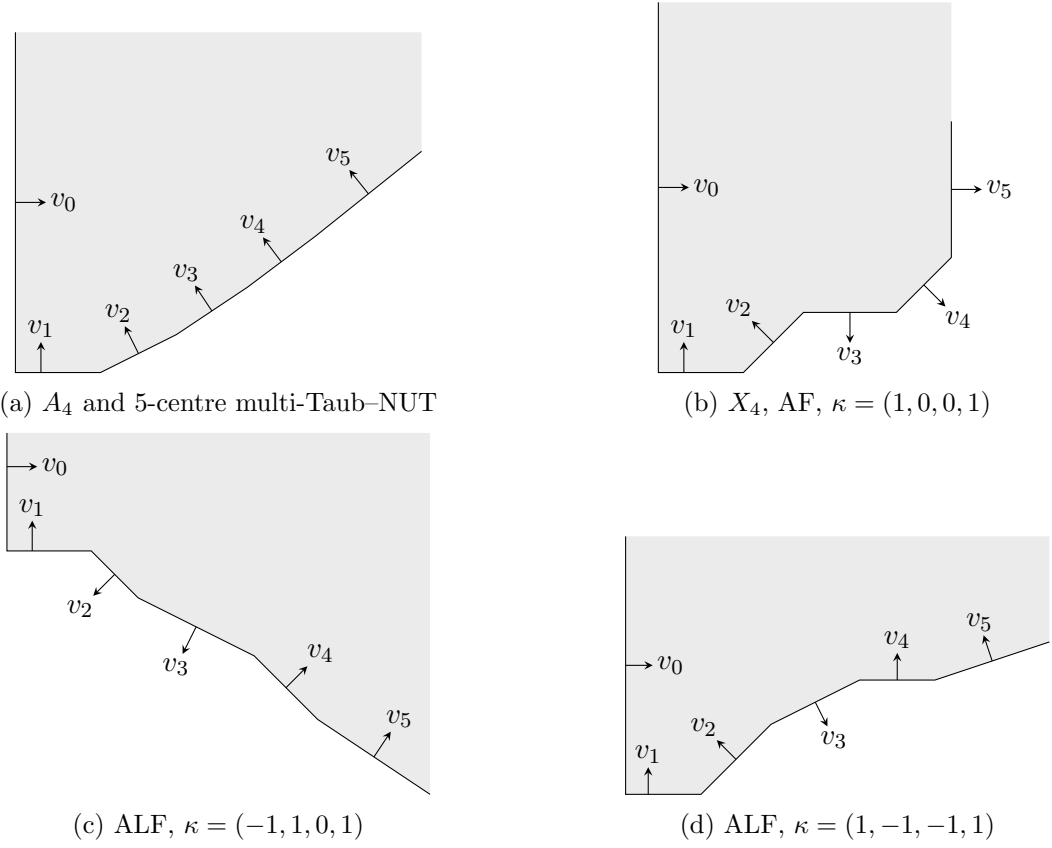
\begin{figure}[ht]
\centering
\begin{minipage}[b]{0.48\textwidth}\centering
\begin{tikzpicture}[scale=0.450,>=stealth]
    \fill[gray!50,opacity=0.3](0,10.004)--(0,0)--(2.5,0)--(4.736,1.118)--(6.816,2.505)--(8.816,4.005)--(11.94,6.504)--(11.94,10.004)--cycle;
    \draw (0,10.004)--(0,0)--(2.5,0)--(4.736,1.118)--(6.816,2.505)--(8.816,4.005)--(11.94,6.504);
    \draw[->] (0,5.002)--(0.9,5.002) node[anchor=180,inner sep=1.5pt] {$v_{0}$};
    \draw[->] (0.75,0)--(0.75,0.9) node[anchor=270,inner sep=1.5pt] {$v_{1}$};
    \draw[->] (3.618,0.559)--(3.216,1.364) node[anchor=297,inner sep=1.5pt] {$v_{2}$};
    \draw[->] (5.776,1.811)--(5.277,2.56) node[anchor=304,inner sep=1.5pt] {$v_{3}$};
    \draw[->] (7.816,3.255)--(7.276,3.975) node[anchor=307,inner sep=1.5pt] {$v_{4}$};
    \draw[->] (10.378,5.254)--(9.816,5.957) node[anchor=309,inner sep=1.5pt] {$v_{5}$};
\end{tikzpicture}

\smallskip\small (a) $A_4$ and $5$-centre multi-Taub--NUT
\end{minipage}\hfill
\begin{minipage}[b]{0.48\textwidth}\centering
\begin{tikzpicture}[scale=0.450,>=stealth]
    \fill[gray!50,opacity=0.3](0,10.884)--(0,0)--(2.5,0)--(4.268,1.768)--(6.997,1.768)--(8.614,3.384)--(8.614,7.384)--(8.614,10.884)--cycle;
    \draw (0,10.884)--(0,0)--(2.5,0)--(4.268,1.768)--(6.997,1.768)--(8.614,3.384)--(8.614,7.384);
    \draw[->] (0,5.442)--(0.9,5.442) node[anchor=180,inner sep=1.5pt] {$v_{0}$};
    \draw[->] (0.75,0)--(0.75,0.9) node[anchor=270,inner sep=1.5pt] {$v_{1}$};
    \draw[->] (3.384,0.884)--(2.747,1.52) node[anchor=315,inner sep=1.5pt] {$v_{2}$};
    \draw[->] (5.633,1.768)--(5.633,0.868) node[anchor=90,inner sep=1.5pt] {$v_{3}$};
    \draw[->] (7.805,2.576)--(8.442,1.939) node[anchor=135,inner sep=1.5pt] {$v_{4}$};
    \draw[->] (8.614,5.384)--(9.514,5.384) node[anchor=180,inner sep=1.5pt] {$v_{5}$};
\end{tikzpicture}

\smallskip\small (b) $X_4$, AF, $\kappa=(1,0,0,1)$
\end{minipage}

\medskip
\begin{minipage}[b]{0.48\textwidth}\centering
\begin{tikzpicture}[scale=0.446,>=stealth]
    \fill[gray!50,opacity=0.3](0,3.5)--(0,0)--(2.5,0)--(3.892,-1.392)--(7.333,-3.112)--(9.221,-5.001)--(12.55,-7.22)--(12.55,3.5)--cycle;
    \draw (0,3.5)--(0,0)--(2.5,0)--(3.892,-1.392)--(7.333,-3.112)--(9.221,-5.001)--(12.55,-7.22);
    \draw[->] (0,2.5)--(0.9,2.5) node[anchor=180,inner sep=1.5pt] {$v_{0}$};
    \draw[->] (0.75,0)--(0.75,0.9) node[anchor=270,inner sep=1.5pt] {$v_{1}$};
    \draw[->] (3.196,-0.696)--(2.559,-1.332) node[anchor=45,inner sep=1.5pt] {$v_{2}$};
    \draw[->] (5.612,-2.252)--(5.21,-3.057) node[anchor=63,inner sep=1.5pt] {$v_{3}$};
    \draw[->] (8.277,-4.057)--(8.913,-3.42) node[anchor=225,inner sep=1.5pt] {$v_{4}$};
    \draw[->] (10.886,-6.11)--(11.385,-5.361) node[anchor=236,inner sep=1.5pt] {$v_{5}$};
\end{tikzpicture}

\smallskip\small (c) ALF, $\kappa=(-1,1,0,1)$
\end{minipage}\hfill
\begin{minipage}[b]{0.48\textwidth}\centering
\begin{tikzpicture}[scale=0.399,>=stealth]
    \fill[gray!50,opacity=0.3](0,8.553)--(0,0)--(2.5,0)--(4.821,2.321)--(7.756,3.789)--(10.256,3.789)--(14.051,5.053)--(14.051,8.553)--cycle;
    \draw (0,8.553)--(0,0)--(2.5,0)--(4.821,2.321)--(7.756,3.789)--(10.256,3.789)--(14.051,5.053);
    \draw[->] (0,4.277)--(0.9,4.277) node[anchor=180,inner sep=1.5pt] {$v_{0}$};
    \draw[->] (0.75,0)--(0.75,0.9) node[anchor=270,inner sep=1.5pt] {$v_{1}$};
    \draw[->] (3.66,1.16)--(3.024,1.797) node[anchor=315,inner sep=1.5pt] {$v_{2}$};
    \draw[->] (6.289,3.055)--(6.691,2.25) node[anchor=117,inner sep=1.5pt] {$v_{3}$};
    \draw[->] (9.006,3.789)--(9.006,4.689) node[anchor=270,inner sep=1.5pt] {$v_{4}$};
    \draw[->] (12.154,4.421)--(11.869,5.275) node[anchor=288,inner sep=1.5pt] {$v_{5}$};
\end{tikzpicture}

\smallskip\small (d) ALF, $\kappa=(1,-1,-1,1)$
\end{minipage}
\caption{Rod structures with five turning points, drawn as in Figure~\ref{fig:rods3}. (a) $A_4$ and the multi-Taub--NUT metric with
five centres (ALF-$A_4$), $\kappa=(2,2,2,2)$, equal rod lengths. (b) Li--Sun's $X_4$, Remark~\ref{rem:compexamples}(b), $v_5=v_0$, at the
point of its family with rod lengths $(1,1,1.092,0.914)$. (c), (d) The two ALF examples of Remark~\ref{rem:compexamples}(a), with
rod lengths $(1,0.787,1.539,1.068)$ and $(1,1.313,1.313,1)$; (d) is symmetric under $z\mapsto-z$. Both have $\det(v_0,v_5)=3$ and
the same end as Figure~\ref{fig:rods4}(c), the quotient of the end of Taub--NUT by $\Z_3$ acting on the base, which is not an
ALF-$A_k$ end (Remark~\ref{rem:ALFends}).}
\label{fig:rods5}
\end{figure}

\clearpage
\section*{Computation and the use of an AI assistant}\label{sec:computation}
The proof of Theorem~\ref{thm:G} does not depend on numerical computation. The examples of
Remark~\ref{rem:compexamples}, except the exact one with four turning points, are computer-assisted, and none is used elsewhere. Computation nevertheless shaped the work: three facts were first found by computation and then
proved, namely the Hankel formula for the rod vectors (by exact computer algebra for small numbers of solitons), the closed
formula of Theorem~\ref{thm:CF} (numerically, as a factorisation of the interpolation determinants, Lemmas~\ref{lem:PI}
and~\ref{lem:HE}), and the identification of the far-node limit with a soliton on the flat seed $\R^3\times S^1$, with its rod lines
and the form of its end (numerically).%

The project was carried out with the help of an AI assistant, Claude (Anthropic), used through the Claude Code environment.
The assistant took part in developing and writing the arguments, wrote and ran the verification scripts and the searches
for examples, reviewed the literature, and carried out critical reviews of the drafts. The authors take responsibility for
the content.

\subsection*{Acknowledgements}
L.A. is grateful to Paul Tod, Lionel Mason, Maciej Dunajski, Mingyang Li, and Song Sun for helpful discussions. L.A. thanks the
Simons Center at Stony Brook University, where some of the work was done, for hospitality. We are grateful to Steffen Aksteiner
and Alfie Jones for sharing their work before publication, and for helpful discussions. We thank James Lucietti for comments on
a draft version of this paper.

\appendix
\section{Classical identities}\label{app:identities}
We collect the classical determinant identities used in \S\S\ref{sec:hankel}, \ref{sec:ale} and \ref{sec:far}, in the form in which they
are used. All but the last are in Chapter~0 of \cite{HornJohnson2013}; see also the notes \cite{Grinberg2022}, which state them
with complete proofs in the form used here.

\emph{Cauchy--Binet.} For an $s\times n$ matrix $A$ and an $n\times s$ matrix $B$, $\det(AB)$ is the sum of
$\det A_{\cdot T}\det B_{T\cdot}$ over the $s$-subsets $T$ of the columns of $A$ \cite[Ch.~I]{Gantmacher1959}, \cite[\S0.8]{HornJohnson2013}; see also \cite[Thm.~6.32]{Grinberg2022}.

\emph{Laplace expansion along several rows or columns} \cite[\S0.8]{HornJohnson2013}; see also \cite[Thm.~6.156]{Grinberg2022}. For an $N\times N$ matrix $A$ and a set $C$ of $a$ of its columns,
\begin{equation}\label{eq:laplace}
\det A=\sum_{|R|=a}(-1)^{\Sigma R+\Sigma C}\,\det A_{R,C}\,\det A_{R^c,C^c},
\end{equation}
the sum over the sets $R$ of $a$ rows. Here $A_{R,C}$ is the submatrix with the rows $R$ and the columns $C$, in their order in $A$;
$R^c$, $C^c$ are the complements; and $\Sigma R$, $\Sigma C$ are the sums of the indices in $R$, $C$. We call $(-1)^{\Sigma R+\Sigma C}$ the
\emph{Laplace sign}. It does not depend on whether rows and columns are numbered from $0$ or from $1$, since the two numberings change
$\Sigma R+\Sigma C$ by $2a$. Applied to $A^T$, \eqref{eq:laplace} gives the expansion along a set of $a$ rows, with the same sign. For
$a=1$ it is the expansion along one column, or one row. If $C$ consists of the first $a$ columns, then, numbering from $1$,
$(-1)^{\Sigma R+\Sigma C}=(-1)^{\Sigma R-a(a+1)/2}$, and
$\Sigma R-a(a+1)/2=\#\{(i,i'):i\in R,\ i'\notin R,\ i'<i\}$ is the number of inversions of the permutation that moves the rows $R$,
in their order, in front of the others. This is used for \eqref{eq:Csplit} in \S\ref{subsec:split}, in the proof of
Lemma~\ref{lem:HE}, and for \eqref{eq:shifted} in \S\ref{subsec:atominf}.

\emph{Schur complement} \cite{Schur1917}, \cite[\S0.8]{HornJohnson2013}; see also \cite[Exercise~6.36]{Grinberg2022}. For a square matrix $K$ with $\det K\ne0$, bordered by $U$, $V$ and a square block $E$,
\begin{equation}\label{eq:schur}
\det\begin{pmatrix}K&U\\V^T&E\end{pmatrix}=\det K\,\det\bigl(E-V^TK^{-1}U\bigr),
\end{equation}
so a bordered determinant divided by $\det K$ is the determinant of the Schur complement $E-V^TK^{-1}U$. Evaluated in both
orders, for a column $u$ and a row $v^T$, it gives the \emph{matrix determinant lemma}, Cauchy's expansion of a determinant with a
rank-one perturbation \cite[\S0.8]{HornJohnson2013}, see also \cite[Exercise~6.59]{Grinberg2022},
\begin{equation}\label{eq:mdl}
\det(A+uv^T)=\det A\,\bigl(1+v^TA^{-1}u\bigr),\qquad\det A\ne0 .
\end{equation}

\emph{Woodbury's identity} \cite{Woodbury1950}, \cite[(1)]{Hager1989}, \cite{HendersonSearle1981}, \cite[\S0.7]{HornJohnson2013}. For an
invertible $n\times n$ matrix $A$ and an $n\times s$ matrix $U$ with $\mathrm{Id}-U^TA^{-1}U$ invertible,
\begin{equation}\label{eq:woodbury}
(A-UU^T)^{-1}=A^{-1}+A^{-1}U\,(\mathrm{Id}-U^TA^{-1}U)^{-1}U^TA^{-1},
\end{equation}
as one checks by multiplying by $A-UU^T$; this is formula (1) of \cite{Hager1989}, where the history of the identity is given, with
$V=U^T$.

\emph{A Cauchy-type identity for interpolation determinants.} For $x,y,a,b$ in $\C^n$, let $V^{n,n}(x,y;a,b)$ be the $2n\times2n$
matrix whose row $i$ and row $n+j$, for $1\le i,j\le n$, are
\begin{equation}
\bigl(1,x_i,\dots,x_i^{n-1},\,a_i,a_ix_i,\dots,a_ix_i^{n-1}\bigr),\qquad
\bigl(1,y_j,\dots,y_j^{n-1},\,b_j,b_jy_j,\dots,b_jy_j^{n-1}\bigr).
\end{equation}
Then
\begin{equation} \det\Bigl[\frac{b_j-a_i}{y_j-x_i}\Bigr]_{i,j=1}^n=\frac{(-1)^{n(n-1)/2}}{\prod_{i,j=1}^n(y_j-x_i)}\,\det V^{n,n}(x,y;a,b). \end{equation}
This is \cite[Thm.~1.1(a), (1.3)]{IOTZ2006} with $p=q=0$: then $\det V^{p+1,q+1}(x_i,y_j;a_i,b_j)=b_j-a_i$ and $\det V^{p,q}=1$. The rows
of $V^{n,n}$ are those of the interpolation determinant $W$ of \S\ref{subsec:CW} for the pairs $[1:a_i]$ and $[1:b_j]$; since $W$ is
homogeneous of degree $1$ in each pair, the identity holds for the pairs $[\alpha_i:\beta_i]$ after multiplying both sides by the product of the $\alpha$'s of all
$2n$ pairs.

\bibliographystyle{amsplain}
\bibliography{Geroch}
\end{document}